\documentclass[lettersize,onecolumn]{IEEEtran}
\usepackage{amsmath,amsfonts}
\usepackage{algorithmic}
\usepackage{algorithm}
\usepackage{array}
\usepackage[caption=false,font=normalsize,labelfont=sf,textfont=sf]{subfig}
\usepackage{textcomp}
\usepackage{stfloats}
\usepackage{url}
\usepackage{verbatim}
\usepackage{graphicx}
\usepackage{cite}
\RequirePackage{amsthm,amsmath,amsfonts,amssymb}
\RequirePackage[colorlinks,citecolor=blue,urlcolor=blue]{hyperref}
\RequirePackage{bm}
\usepackage{xcolor}
\usepackage{subfig}
\usepackage{booktabs}
\usepackage{multirow}

\newtheorem{theorem}{Theorem}
\newtheorem{lemma}{Lemma}
\newtheorem{proposition}{Proposition}
\newtheorem{corollary}[lemma]{Corollary}

\newtheorem{condition}{Condition}
\newtheorem{remark}{Remark}

\newcommand{\indep}{\mathrel{\perp\!\!\!\!\perp}}
\newcommand{\op}{\textnormal{op}}
\newcommand{\var}{\operatorname{Var}}
\newcommand{\cov}{\operatorname{Cov}}
\newcommand{\vectorize}{\operatorname{Vec}}
\newcommand{\circulant}{\operatorname{Circ}}
\newcommand{\F}{\operatorname{F}}
\newcommand{\exptext}{\operatorname{exp}}

\begin{document}

\title{Simultaneous Inference for Covariance\\ and Precision Matrices of \\Long-Range Dependent Time Series}

\author{Percy S. Zhai, Mladen Kolar, Wei Biao Wu
\thanks{P. S. Zhai is with Booth School of Business, The University of Chicago, Chicago, IL 60637, USA.}
\thanks{M. Kolar is with Marshall School of Business, University of Southern California, Los Angeles, CA 90089, USA, and Department of Statistics and Data Science, Mohamed bin Zayed University of Artificial Intelligence, Abu Dhabi, United Arab Emirates.}
\thanks{W. B. Wu is with Department of Statistics, The University of Chicago, Chicago, IL 60637, USA. Wei Biao Wu's research is partially supported by the NSF (Grant NSF/DMS-2311249).}
}

\markboth{Accepted for Publication in IEEE Transactions on Information Theory}%
{Zhai, Kolar, Wu: Inference for Covariance and Precision Matrices of High-Dimensional Long-Range Dependent Time Series}


\maketitle

\begin{abstract}
For time series with long-range temporal dependence, inference for covariance and precision matrices is non-trivial.
We propose a Berry-Esseen type Gaussian approximation result that gives a finite-sample bound for the Kolmogorov distance between the infinity norms of the estimation error of sample covariance matrix and the corresponding Gaussian approximation.
The method utilizes martingale and $m$-dependent approximation and relies on constructing triadic blocks.
We also establish a bootstrapping result with block sampling method, which preserves validity despite strong temporal dependence.
Our results on covariance allow ultra-high-dimensional settings where the dimension of time series can grow sub-exponentially with sample size.
Similar results can be built for precision matrix under low-dimensional settings.
No assumption is required on the structure of covariance and precision matrices.
\end{abstract}


\section{Introduction}
\IEEEPARstart{L}{et} $\{X_i\}_{i=1}^n$ be a stationary process in $\mathbb{R}^p$ where $p$ may be much larger than $n$.
Assume that $X_i$ has mean 0 and covariance matrix $\Sigma = \mathbb{E}[X_i X_i^T] \in\mathbb{R}^{p\times p}$.
In this paper, we assume that $X_i$ is a Gaussian linear process
\begin{equation}\label{LinearProcess}
    X_i = \sum_{t=0}^\infty A_t \epsilon_{i-t}.
\end{equation}
Here, $\epsilon_t, t\in\mathbb{Z}$ are i.i.d.~Gaussian random vectors on $\mathbb{R}^d$ with mean zero and covariance matrix $\mathbb{E}[\epsilon_t\epsilon_t^\top]=I_d$.
The coefficients $A_t\in\mathbb{R}^{p\times d}$ decay with $t$, which is formally characterized in Condition \ref{cond.A}.
The decay rate of $A_t$ represents the extent of temporal dependence of the process.
Slower decay generally corresponds to longer memory.
Without loss of generality, we assume that the innovation term $\epsilon_t$ is Gaussian; the extension to sub-Gaussian innovations is straightforward.

A common estimator of $\Sigma$ is the sample covariance matrix $\hat{\Sigma}_n = n^{-1}\sum_{i=1}^n X_i X_i^\top$.
We are interested in the simultaneous inference problem for $\hat{\Sigma}_n$ even when the process \eqref{LinearProcess} has strong temporal dependence.
The purpose of this paper is to establish a finite-sample Gaussian approximation guarantee for
the maximum deviation $|\hat{\Sigma}_n - \Sigma|_\infty$.
To be precise, define the $p^2$-dimensional process $\mathcal{X}_i = \vectorize(X_i X_i^T-\Sigma)$,
and the standardized partial sum $S_n = n^{-1/2}\sum_{i=1}^n \mathcal{X}_i$.
Note that $S_n = n^{1/2}\vectorize(\hat{\Sigma}_n - \Sigma)$.
Let $Z$ be a $p^2$-dimensional Gaussian random vector with mean zero and the same covariance structure as $S_n$.
Since this paper focuses on the high-dimensional simultaneous inference setting, an appropriate metric for evaluating the validity of the Gaussian approximation is the Kolmogorov distance between $|S_n|_\infty$ and $|Z|_\infty$, defined as
\begin{equation}\label{EQ_primary_objective}
    \rho(n) = \sup_{u\in\mathbb{R}} \left|\mathbb{P}\left(n^{1/2}|\hat{\Sigma}_n-\Sigma|_\infty\geq u\right) - \mathbb{P}\left(|Z|_\infty\geq u\right)\right|.
\end{equation}
See, for example, \cite{chernozhukov2013gaussian, chernozhukov2015comparison, chernozhukov2017central, zhang2017gaussian}.
The primary objective of this work is to derive a non-asymptotic upper bound on $\rho(n)$. 
We show that this bound converges to zero as $n \to \infty$ under mild conditions, 
even when the process \eqref{LinearProcess} exhibits strong temporal dependence and the dimension $p$ grows exponentially with $n$.

Simultaneous inference on the sample covariance $\hat{\Sigma}_n$ has been well studied. For Gaussian populations with diagonal covariance, \cite{jiang2004asymptotic} found that the largest off-diagonal entry of sample correlation follows an asymptotic Gumbel distribution. Refined results are given by \cite{li2006some, zhou2007asymptotic, liu2008asymptotic, li2010necessary}.
\cite{cai2011limiting} extended this to banded covariance matrices.
\cite{xiao2013asymptotic} developed Gumbel convergence for non-Gaussian random vectors with general dependency, allowing tests for stationary, banded, or Toeplitz covariance matrices.
For Gaussian populations with equi-correlation, \cite{fan2019largest} found asymptotic Gaussianity of the largest off-diagonal sample correlation.
\cite{chen2018gaussian} proposed a two-step Gaussian approximation for U-statistics of order 2, enabling simultaneous inference on sample covariance for sub-Gaussian populations without structural assumptions on $\Sigma$.
These works require $X_1,\cdots,X_n$ to be i.i.d.~random vectors.
The extension to temporally dependent $X_i$ was achieved by \cite{zhang2017gaussian}. However, the results were limited to weakly dependent time series, with an admissible dimension rate of $p\left(\log(pn)\right)^{3q/2} = o(n^{q/2-1})$, covering only polynomially high dimensions.
\cite{xiao2011asymptotic} developed asymptotic inference for autocovariances of stationary processes using maximum and quadratic deviation statistics.
Under long-range temporal dependence, \cite{shu2019estimation} studied the estimation of large covariance and precision matrices. 
The estimation of long-run covariance and precision matrices was further developed by \cite{baek2023local}. 
However, simultaneous inference for the covariance matrix under both long-memory and high-dimensional settings remains an open problem.

A related but distinct line of work studies inference for high-dimensional regression models.
\cite{van2014asymptotically} and \cite{javanmard2014confidence} developed inference for low-dimensional components in high-dimensional regression models.
For temporally dependent settings, \cite{zhang2023statistical} studied inference for sparse high-dimensional VAR models with nonstationary innovations, while \cite{dai2025statistical} considered high-dimensional binary time series.
These works study regression coefficients rather than covariance or precision matrices, but they illustrate the broader development of high-dimensional inference under temporal dependence.

Our proposed simultaneous inference method satisfies the following three criteria:
First, it achieves exponentially high dimension, i.e., $\log(p)=o(n^\alpha)$ for some constant $\alpha$. Second, it is {\it free of structural assumptions} on the true covariance $\Sigma$. And third, it allows for strong temporal dependency (i.e., long-memory).
Our first main result, Theorem \ref{thm:GA.cov}, which provides a finite-sample upper bound on the Kolmogorov distance \eqref{EQ_primary_objective}, is the first theoretical guarantee satisfying these three criteria simultaneously that the authors are aware of.
We emphasize the importance of avoiding structural assumptions whenever possible. 
In the existing literature, estimation methods for large covariance and precision matrices often rely heavily on sparsity assumptions, 
which are typically difficult to verify empirically. 
Although the sample covariance matrix $\hat{\Sigma}_n$ may not provide an accurate estimate of $\Sigma$, 
simultaneous inference based on $\hat{\Sigma}_n$ can still reveal key features of the true underlying covariance structure. 
Our proposed method not only enables testing of structural assumptions and identification of important graphical connections, 
but also serves as a foundational tool for a broad class of accurate estimation procedures by providing a mechanism to validate their underlying conditions.

The covariance structure of the Gaussian approximation $Z$ in \eqref{EQ_primary_objective} is typically unknown in practice.
To achieve reliable empirical results, bootstrap methods are necessary.
Two popular techniques are the multiplier bootstrap and empirical bootstrap methods \cite{chernozhukov2013gaussian, chernozhukov2017central}.
However, these methods become invalid when dealing with time-dependent data.
\cite{zhang2017gaussian} introduced a method for estimating long-run covariance matrices as a remedy.
Another solution is the blockwise multiplier bootstrap method proposed by \cite{zhang2014bootstrapping}.
Both approaches require the time series to exhibit short memory.
To overcome this limitation, we employ a block resampling method suggested by \cite{zhang2013block}.
This technique, unlike many others, remains valid even in long-memory scenarios.
For implementation, we first choose a block length and slide an overlapping window through the observation.
From each window, we compute the quantity that drives the estimation error, so each window produces one bootstrap replicate. We then use the collection of these block replicates to form an empirical distribution, which serves as a fully data-driven approximation to the distribution of the estimation error.
Our result provides a finite-sample guarantee on the validity of block sampling method, which is characterized by \eqref{eq:rho.bootstrap.cov}.
The detailed result is given in Theorem \ref{thm:bootstrap.cov}.


The conditions for Gaussian approximation and bootstrap for sums of independent random vectors were identified by \cite{chernozhukov2013gaussian}, using a Slepian-Stein method and an anti-concentration inequality \cite{chernozhukov2015comparison}.
\cite{chernozhukov2017central} refined the result with Nazarov's Inequality and extended it to sparsely convex sets.
\cite{chernozhukov2019improved} enhanced the Gaussian approximation error bound and dimension rate using an iterative randomized Lindeberg method. For non-degenerate cases, \cite{fang2021high} achieved a sharper rate with Stein's method, proving near-optimality for log-concave $X_i$.
\cite{chernozhukov2020nearly} derived a near-optimal rate of Gaussian approximation for bounded $X_i$ in the non-degenerate case.
Developments in U-statistics are in \cite{song2019approximating}, \cite{chen2019randomized}, and \cite{chen2020jackknife}, requiring $X_1,\cdots,X_n$ to be independent random vectors.
\cite{koike2019gaussian} studied Gaussian approximation for smooth Wiener functionals with temporal dependency from nonsynchronicity.
\cite{zhang2014bootstrapping} developed Gaussian approximation for weakly dependent processes under a geometric moment contraction condition.
\cite{chernozhukov2019inference} extended this to dependent data with general mixing conditions.
\cite{zhang2017gaussian} addressed stationary time series using functional dependence measure.
\cite{chernozhukov2021lasso} applied these results to Lasso tuning parameter selection, considering strong temporal dependency.

The proof of the Gaussian approximation result relies on a three-step framework.
First, we derive tail bounds for the martingale approximation and $m$-dependence truncation using concentration inequalities.
Second, we construct conditionally independent temporal blocks using the triadic method of \cite{berkes2014komlos}, allowing extension to a strong temporal dependence.
Third, we apply Gaussian approximation results for i.i.d.~random vectors from \cite{chernozhukov2017central} twice: on conditionally independent blocks and their unconditional Gaussian approximations.
The ultra-high-dimensional property of \cite{chernozhukov2017central} is inherited, making the dimensionality bound $\log(p) = o(n^\alpha)$, up to a logarithmic factor.
We use an anti-concentration inequality from \cite{chernozhukov2015comparison} for Gaussian maxima to aggregate the results.

Apart from the Gaussian approximation result and bootstrap validity result for the covariance matrix, we develop a similar theoretical guarantee for the precision matrix $\Omega$ (the inverse of covariance matrix).
We focus on the low-dimensional setting, i.e. the regime where $p<n$, such that the sample covariance matrix is invertible, yielding a natural estimate $\hat{\Omega}_n = \hat{\Sigma}_n^{-1}$.
We show the validity of the Gaussian approximation to $n^{1/2}\vectorize(\hat{\Omega}_n - \Omega)$ by a $p^2$-dimensional Gaussian random vector $Z^\Omega$ with the same covariance structure, and provide a finite-sample bound on the Kolmogorov distance; see \eqref{eq:kolm.distance.Omega}.
Meanwhile, without knowing the specific covariance structure of $Z^\Omega$, we can approximate the distribution of the precision estimation error by block bootstrap.
We shall provide a finite-sample guarantee of its validity; see \eqref{eq:kolm.distance.Omega.boot}.
These results will be elaborated in Theorems \ref{thm:GA.prec} and \ref{thm:bootstrap.prec}, respectively.
For Gaussian processes, specifically, the conditional independence graph (CIG) is characterized by the precision matrix.
These results allow for simultaneous tests on the graph structure.
In the high-dimensional regime $p\geq n$, the precision matrix $\Omega$ can be estimated via graphical Lasso \cite{meinshausen2006high} or CLIME \cite{cai2011constrained}.
These methods require structural assumptions like sparsity on the true precision matrix $\Omega$.
While the asymptotic normality of these regularized estimators for time series is a very important problem, our goal in this paper is to impose as few structural assumptions as possible.
Therefore, we focus on the natural estimator $\hat\Omega_n = \hat\Sigma_n^{-1}$ in the $p<n$ regime, and leave the study of high-dimensional estimators for future study.

The theoretical results discussed in this paper offer new perspectives on a wide range of statistical topics. A notable example is the examination of covariance (or precision) structures in spatio-temporal processes, which often exhibit strong temporal dependence \cite{han2020test}. Applications in econometrics are also addressed, as seen in \cite{belloni2018high}. Furthermore, exploring extensions in the domain of spatio-temporal processes could be intriguing \cite{stein1999interpolation, wu2004simulating}.

The asymptotic theory for maximum deviations of sample covariance has practical applications. Using EEG or fMRI data as multivariate time series, researchers determine functional connectivity in brain graphical models through highest correlations between time series \cite{deco2013resting, hipp2012large, hutchison2013dynamic, larson2013adding, zhao2024high}. These issues are also relevant to financial topics. A significant characteristic of financial data is volatility clustering: large returns often follow large returns, and small returns follow small returns \cite{mandelbrot1967variation}. This shows strong temporal dependency in financial time series \cite{ding1993long}.

\subsection{Organization of the Paper}

The organization of this paper is as follows.
Section \ref{sec:overview} gives an overview of the long-range dependence settings and identifies two key technical conditions.
Section \ref{sec:GA.cov} establishes a finite-sample result to the Gaussian approximation problem for covariance matrices.
A block bootstrap method for covariance matrices is introduced in Section \ref{sec:bootstrap.cov}, along with a finite-sample result on its validity.
In the low-dimensional regime, we establish a Gaussian approximation result on sample precision matrices in Section \ref{sec:GA.prec}, and establish a bootstrap method for precision in Section \ref{sec:bootstrap.prec}.
All the main theorems are empirically examined with the simulation studies given in Section \ref{sec:simulation}.
Finally, Section \ref{sec:real.data} provides a real data analysis of fMRI data that studies the brain connectome structure of patients with autism.

\subsection{Notation}

We represent the matrix 2-norm of a vector or a matrix using $|\cdot|$.
The matrix 1-norm is denoted by $|\cdot|_1$, and the operator norm by $|\cdot|_\op$.
In this paper, the {\it infinity norm} of a matrix refers to its maximum absolute entry, denoted by $|\cdot|_\infty$.
The element located at the intersection of the $j$-th row and $l$-th column in a matrix is represented as $(\cdot)_{jl}$. The long column vector generated by concatenating columns of a matrix is denoted as $\vectorize(\cdot)$. The element of this extended column vector, $(\cdot)_{(s,t)}$, represents the element located at the $s$-th row and $t$-th column in the associated matrix. We let $[n]$ denote the set of integers $\{1,2,\cdots,n\}$.
We write $f(n) = o(g(n))$ if $f(n) / g(n) \rightarrow 0$, and $f(n) = \tilde{o}(g(n))$ if there exists $\alpha$ such that $f(n) = o(g(n)\log^\alpha(n))$.
Unless otherwise stated, all the constants in this paper may vary in value depending on the context.
We also define $\| \cdot\|_k$ as $\|X\|_k = (\mathbb{E}|X|^k)^{1/k}$.


\section{Overview}\label{sec:overview}

The temporal dependence of a stationary time series is characterized by its autocovariance matrices $\Gamma_k = \cov(X_t, X_{t-k})$, which typically decay with lag $k$. For a Gaussian linear process \eqref{LinearProcess}, $\Gamma_k = \sum_{t=0}^\infty A_t A_{t+k}^\top$. A slower decay of $\Gamma_k$ with $k$ indicates stronger temporal dependence. If $\sum_{k=0}^\infty \Gamma_k$ (the long-run covariance) diverges, we say that the time series exhibits {\it long-memory} or {\it long-range dependence}; see \cite{beran1994statistics}. All main results in this paper apply to long-range dependent time series.

\subsection{Technical Conditions}

Before introducing the main results, we impose two technical assumptions. The first specifies the decay rate for the coefficients in the linear process \eqref{LinearProcess}, $A_t = \{(A_{t})_{jl}\}_{1\leq j\leq p, 1\leq l\leq d}$, with respect to lag $t$.
\begin{condition}\label{cond.A}
For some $\beta>3/4$ and a constant $C_0>0$, the coefficient matrices satisfy
\begin{equation*}
    \max_{1\leq j\leq p}|(A_{t})_{j\cdot}| = \max_{1\leq j\leq p} \left(\sum_{l=1}^d (A_{t})^2_{jl}\right)^{1/2} \leq C_0(1\vee t)^{-\beta}.
\end{equation*}
\end{condition}
The parameter $\beta$ determines the decay rate of $A_t$.
When $\beta > 1$, the process exhibits short-range dependence.
Conversely, for $\beta < 1$, the autocovariances $\Gamma_k$ are not summable, which indicates long-range dependence.
It is important to note that the Gaussian approximation holds only if $\beta > 3/4$. When $1/2 < \beta \leq 3/4$, the asymptotic distribution of the sample covariance functions transitions from Gaussian to the Rosenblatt distribution, which is characterized by a lighter left tail and a heavier right tail; we refer to \cite{cai2011limiting,jiang2004asymptotic} for more details.

To motivate the second assumption, first consider $X_t$ as a unidimensional process.
Its lag-$k$ autocovariance is denoted by $\gamma_k$ and its sample version by $\hat{\gamma}_k$.
Bartlett's Formula \cite{bartlett1978introduction} states that the covariance of $\sqrt{n}\hat{\gamma}_k$ and $\sqrt{n}\hat{\gamma}_l$ converges to
\begin{equation*}
\sum_{i=-\infty}^\infty (\gamma_{i+l}\gamma_{i+k} + \gamma_{i+l}\gamma_{i-k}) + A_\kappa(k,l)
\end{equation*}
as $n\rightarrow\infty$. Here, the last term $A_\kappa(k,l)$ depends on the fourth order cumulants of the process $X_t$ and $A_\kappa(k,l)=0$ for Gaussian processes.
Specifically, by setting $k=l=0$, the sample covariance $\sqrt{n}\hat{\sigma}_n$ of a unidimensional Gaussian process has an asymptotic variance of $2\sum_{i=-\infty}^\infty \gamma_i^2$.
The following second assumption is a generalization to multidimensional case.
It necessitates that the asymptotic variance of all entries of $S_n$ are lower bounded, so that the limiting Gaussian distribution does not collapse to a point mass.
\begin{condition}\label{cond.G}
There exists a positive constant $c_0$ so that for all $1\leq s,t\leq p$,
\begin{equation}\label{eq:cond.G}
    \sum_{k=-\infty}^\infty \left((\Gamma_k)_{ss}(\Gamma_k)_{tt} + (\Gamma_k)_{st}(\Gamma_k)_{ts}\right) \geq c_0.
\end{equation}
\end{condition}
Observe that the expression on the left hand side represents the asymptotic variance of $n^{1/2}((\hat{\Sigma}_n)_{st} - \Sigma_{st})$. This condition is introduced to prevent the trivial scenario where both distributions in the Kolmogorov distance $\rho(n)$ are degenerate and concentrated on a single point.

\subsection{Main Results}


The first main result gives a finite sample upper bound for $\rho(n)$ in \eqref{EQ_primary_objective}. Theorem \ref{thm:GA.cov} shows $\rho(n)\rightarrow 0$ when $\log(p) = o(n^\alpha)$ with $\alpha$ explicitly related to $\beta$. See Section \ref{sec:GA.cov} for detailed discussions that lead to this result.

\begin{theorem}[Gaussian approximation for covariance]\label{thm:GA.cov}
Under Conditions \ref{cond.A} and \ref{cond.G}, 
\begin{equation*}
    \rho(n)\leq C \Psi(p,n),
\end{equation*}
where $C>0$ is a constant related to $C_0$ and $c_0$, and
\begin{equation}\label{eq:Theta.pn}
    \Psi(p,n) = \frac{\log^{(5\Tilde{\beta}+12)/(4\Tilde{\beta}+8)}(pn)}{n^{\Tilde{\beta}/(4\Tilde{\beta}+8)}},
\end{equation}
with $\tilde{\beta} = (4\beta-3)\wedge(2\beta-1)$.
In particular, if $\Psi(p,n) \rightarrow 0$, then $\rho(n) \rightarrow 0$.
\end{theorem}
\begin{remark}
For $\rho(n)$ to converge to zero, the dimension $p$ is required to satisfy $\log(p) = {o}(n^{\tilde{\beta}/(5\tilde{\beta}+12)})$.
A shorter memory generally corresponds to a sharper bound of $\rho(n)$ and a higher dimension allowed.
In the particular scenario where $X_t$ are independent, $\tilde{\beta}$ approaches infinity, resulting in the rate $\rho(n) = {o}(\log^{5/4}(pn)/n^{1/4})$, which matches the order of the bound found in \cite{chernozhukov2019improved}.
This rate converges to zero as the dimension satisfies $\log(p) = {o}(n^{1/5})$.
On the other hand, in the borderline scenario that $\tilde{\beta}$ is barely larger than $0$ (i.e., $\beta$ is barely larger than $3/4$, the phase transition point between asymptotic Gaussian and Rosenblatt distributions), the upper bound of $\rho(n)$ converges extremely slowly to $0$.
This agrees with the fact that the asymptotic distribution is no longer Gaussian when $\beta \leq 3/4$.

Also note that Theorem \ref{thm:GA.cov} can be extended to the case where the innovation $\epsilon_t$ follows a more general distribution.
One example is when $\epsilon_t$ is sub-Gaussian.
In this case, the sub-Gaussian norm $\|\epsilon_t\|_{\psi_2}$ is upper bounded, but is no longer a fixed value $\sqrt{2}$ as in the Gaussian case.
This will lead to a different constant in Hanson-Wright inequality (Lemma \ref{lm:hanson.wright}).
Passing it through the proof of Theorem \ref{thm:GA.cov} will still lead to a same rate $\Psi(p,n)$, but a different constant $C$ that increases with $\|\epsilon_t\|_{\psi_2}$.
Another example is when $\epsilon_t$ is sub-exponential.
In this case, the Hanson-Wright inequality, which requires sub-Gaussian innovations, no longer applies.
Instead, \cite{goetze2021concentration} developed an alternative concentration inequality for sub-exponential innovations, with a slower rate than the Hanson-Wright inequality.
That will lead to a slower rate in Theorem \ref{thm:GA.cov}.
In either case, the left hand side of \eqref{eq:cond.G} in Condition \ref{cond.G} need to be updated. It should contain an extra term related to the fourth order cumulants of the process $X_t$, which reflects the change in the asymptotic long-run covariance according to Bartlett's Formula.
\end{remark}


The second main result ensures the validity of the block bootstrap method for long-range dependent processes. Let the block size $l=l_n\rightarrow\infty$ and $l/n\rightarrow 0$ as $n\rightarrow\infty$. Construct blocks $\check{B}_{i,l} = \sum_{j=i-l+1}^i \mathcal{X}_j$, and define the empirical distribution function
\begin{equation*}
    \hat{F}_{n,l}(u) = \frac{1}{n-l+1}\sum_{i=l}^n \mathbf{1}\{l^{-1/2}|\check{B}_{i,l}|_\infty\leq u\}.
\end{equation*}
The validity of the block bootstrap is characterized by
\begin{equation}\label{eq:rho.bootstrap.cov}
    \rho_B(n,l) = \sup_{u\in\mathbb{R}}\left|\hat{F}_{n,l}(u)-P\left(n^{1/2}|\hat{\Sigma}_n-\Sigma|_\infty \leq u\right)\right|.
\end{equation}
The following theorem shows that for a certain order of $l$, $\rho_B(n,l)$ in \eqref{eq:rho.bootstrap.cov} is bounded by a quantity converging to 0 in probability, if $\log(p) = o(n^\alpha)$. See Section \ref{sec:bootstrap.cov} for details and sketch of proof.

\begin{theorem}[Block bootstrap consistency for covariance]\label{thm:bootstrap.cov}
Under Conditions \ref{cond.A} and \ref{cond.G},
for any $0<\epsilon<1$, if we take $l\asymp n^\phi\log^\psi(p)$, where $\phi = \frac{2\tilde{\beta}+4}{(3-\epsilon)\tilde{\beta} + 4}$ and $\psi = \frac{5\tilde{\beta}+12}{(3-\epsilon)\tilde{\beta} + 4}(1-\epsilon)$, then we have
\begin{equation*}
\rho_B(n,l) \leq C (\Psi_B(p,n,\epsilon))^{1-\epsilon}
\end{equation*}
with probability no less than $1-3 (\Psi_B(p,n,\epsilon))^{\epsilon}$, where $C>0$ is a constant related to $C_0$ and $c_0$, and
\begin{equation}\label{eq:Psi.B.pne}
    \Psi_B(p,n,\epsilon) = \frac{\log^{(5\tilde{\beta}+12)/[(6-2\epsilon)\tilde{\beta}+8]}(pn)} {n^{\tilde{\beta} / [(6-2\epsilon)\tilde{\beta}+8]}}.
\end{equation}
In particular, if $\Psi_B(p,n,\epsilon) \rightarrow 0$, then $\rho_B(n,l) \rightarrow 0$ in probability.
\end{theorem}

\begin{remark}
A smaller $\epsilon$ generally gives a tighter upper bound of $\rho_B(n,l)$ at a lower probability. When $\epsilon$ is close to 0, the upper bound becomes
\begin{equation*}
    \rho_B(n,l) \lesssim \frac{\log^{(5\tilde{\beta} + 12) / (6\tilde{\beta}+8)}(pn)}{n^{\tilde{\beta}/(6\tilde{\beta}+8)}}
\end{equation*}
with a probability that converges to 1 very slowly as $n$ increases. This is slightly relaxed compared to the bound for Kolmogorov distance $\rho(n)$ in Theorem \ref{thm:GA.cov}. When $\epsilon$ is closer to~1, the upper bound holds with a probability that converges to 1 faster as $n$ increases, but the upper bound converges slower.
For any $0<\epsilon<1$, when $n\rightarrow\infty$, $\rho_B(n,l)$ converges to zero with probability that converges to one, which shows that the block sampling method is valid. Theorem \ref{thm:bootstrap.cov} is the first finite-sample result on the validity of block sampling for long-memory time series that the authors are aware of.
\end{remark}

\begin{remark}
For $\rho_B(n,l)$ to probabilistically converge to zero, it is necessary that
\begin{equation*}
    \log(p) = o\left( n^{\tilde{\beta}/({5\tilde{\beta}+12})}\right),
\end{equation*}
which matches the dimensional condition for $\rho(n)$ to converge to zero as stated in Theorem \ref{thm:GA.cov}.
Note that both the Gaussian approximation result and this bootstrap validity guarantee are applicable to ultra-high-dimensional long-memory time series.
\end{remark}

\begin{remark}\label{rmk:choice.of.l}
Theorem \ref{thm:bootstrap.cov} specifies the theoretically optimal order of bootstrap block size $l$ with respect to $n$ and $\log(p)$, which demonstrates the best-case scenario when the finite-sample upper bound of $\rho_B(n,l)$ is minimized.
In practice, the specific value of block length $l$ can be chosen by data-driven methods; see, e.g. \cite{buhlmann1999block} and \cite{politis2004automatic}.
A trivial choice of $l = n^{1/3}$ is mentioned in \cite{buhlmann1999block}, which arises from general theoretical considerations in variance estimation.
Note that this trivial choice has a lower order than the optimal choice of $l$ in Theorem \ref{thm:bootstrap.cov}.
This is because long-memory time series generally requires a larger block size to capture the temporal dependence structure.
With $\epsilon$ close to zero, the optimal $\phi$ ranges between $2/3$ for short memory processes with very large $\tilde{\beta}$ and $1$ for long-memory processes with $\tilde{\beta}$ close to zero, suggesting a trivial choice of $l$ between $n^{2/3}$ and $n$.
\end{remark}

A valid bootstrap method enables the construction of data-driven simultaneous confidence regions for the unknown population covariance matrix $\Sigma$.
Consider any significance level $\alpha \in (0,1)$.
We define the $(1-\alpha)$-quantile of the bootstrap CDF $\hat F_{n,l}$ as
\[
    \hat{q}_{1-\alpha} = \inf \{ u \in \mathbb{R} : \hat{F}_{n,l}(u) \geq 1-\alpha \}.
\]
Note that $\hat q_{1-\alpha}$ is purely data-driven, and does not rely on any unknown quantities.
We can thus build a simultaneous confidence region $\hat{C}_\alpha(\hat{\Sigma}_n)$ centered at $\hat{\Sigma}_n$. Specifically,
\[
    \hat{\mathcal C}_\alpha(\hat{\Sigma}_n) = \left\{ \tilde\Sigma \in \mathbb{R}^{p \times p} : n^{1/2}|\tilde\Sigma - \hat{\Sigma}_n|_\infty \leq \hat{q}_{1-\alpha} \right\}.
\]
In practice, the event that the true covariance matrix $\Sigma$ landing within the confidence region $\hat{\mathcal C}_\alpha(\hat\Sigma_n)$ is equivalent to all of its entries landing within their respective confidence intervals,
\[
\Sigma_{jk} \in \left[ (\hat{\Sigma}_n)_{jk} - n^{-1/2}\hat{q}_{1-\alpha}, \quad (\hat{\Sigma}_n)_{jk} + n^{-1/2}\hat{q}_{1-\alpha} \right], \text{ for all }1\leq j,k\leq p.
\]
In fact, the probability of this event converges to $1-\alpha$.
The following corollary confirms that the rate of convergence is of order $\rho_B(n,l)$, ensuring the validity of the simultaneous confidence region $\hat{\mathcal C}_\alpha $.
\begin{proposition}[Confidence Regions for Covariance Matrices]\label{prop:CI.cov}
Under the conditions of Theorem \ref{thm:bootstrap.cov},
\[
\left|\mathbb{P}\left(\Sigma \in \hat{\mathcal C}_\alpha (\hat\Sigma_n)\right) - (1-\alpha)\right| \leq 3\rho_B(n,l).
\]
In particular, if $\Psi_B(p,n,\epsilon)\rightarrow 0$, then the left-hand side converges to zero in probability.
\end{proposition}
\begin{proof}[Proof of Proposition \ref{prop:CI.cov}]
Note that 
\[
\left|\mathbb{P}\left(\Sigma \in \hat{\mathcal C}_\alpha (\hat\Sigma_n)\right) - (1-\alpha)\right| \leq \left| \mathbb{P}\left( n^{1/2}|\hat{\Sigma}_n - \Sigma|_\infty \leq \hat{q}_{1-\alpha} \right) - \hat F_{n,l}(\hat q_{1-\alpha}) \right| + \left|\hat F_{n,l}(\hat q_{1-\alpha}) - (1-\alpha)\right|.
\]
The supremum of the first term over $\alpha \in (0,1)$ is tightly upper bounded by $\rho_B(n,l)$.
The second term stems from discreteness of the bootstrap CDF.
Note that $\lim_{q\uparrow \hat q_{1-\alpha}}\hat F_{n,l}(q) < 1-\alpha$, and the distribution of $n^{1/2}|\hat\Sigma_n - \Sigma|_\infty$ is continuous under our process \eqref{LinearProcess}, therefore
\begin{multline*}
    \hat F_{n,l}(\hat q_{1-\alpha}) - (1-\alpha) \leq \hat F_{n,l}(\hat q_{1-\alpha}) - \lim_{q\uparrow \hat q_{1-\alpha}}\hat F_{n,l}(q)\\
    \leq \left|\hat F_{n,l}(\hat q_{1-\alpha}) - \mathbb{P}\left(n^{1/2}|\hat\Sigma_n - \Sigma|_\infty\leq \hat q_{1-\alpha}\right) \right| + \left| \lim_{q\uparrow \hat q_{1-\alpha}}\mathbb{P}\left(n^{1/2}|\hat\Sigma_n - \Sigma|_\infty\leq q\right) -\lim_{q\uparrow \hat q_{1-\alpha}}\hat F_{n,l}(q) \right|
    \leq 2\rho_B(n,l).
\end{multline*}
Meanwhile, $\hat F_{n,l}(\hat q_{1-\alpha}) \geq 1-\alpha$.
The proof is finished by combining the above inequalities.
\end{proof}


The remaining two main results are related to the precision matrix $\Omega = \Sigma^{-1}$.
In this paper, we focus on the case where $p<n$, where the sample covariance matrices are invertible.
Let $\hat{\Omega}_n = \hat{\Sigma}_n^{-1}$.
To exclude the nondegenerate case where the covariance is barely invertible, we add a mild assumption throughout the precision-matrix-related discussions that all eigenvalues of the true underlying covariance matrix are upper and lower bounded, $0<c_e \leq \lambda_{\min}(\Sigma) \leq \lambda_{\max}(\Sigma) \leq C_e$.
The third main result of this paper is a finite-sample upper bound on the Kolmogorov distance
\begin{equation}\label{eq:kolm.distance.Omega}
    \rho^\Omega(n) = \sup_{u\in\mathbb{R}} \left|\mathbb{P}\left(n^{1/2}|\hat{\Omega}_n - \Omega|_\infty \geq u\right) - \mathbb{P}(|Z^\Omega|_\infty \geq u)\right|
\end{equation}
where $Z^\Omega$ is a $p^2$-dimensional Gaussian random vector with the same covariance structure as $n^{1/2}\vectorize(\hat{\Omega}_n - \Omega)$.
The following theorem shows that the Kolmogorov distance $\rho^\Omega(n)$ is upper bounded by a finite-sample rate that converges to $0$ if the dimension $p$ is low enough.
See Section \ref{sec:GA.prec} for detailed discussion.

\begin{theorem}[Gaussian approximation for precision]\label{thm:GA.prec}
    Under Conditions \ref{cond.A} and \ref{cond.G}, for a constant $C'>0$ that is related to $C_0$, $c_0$, $C_e$ and $c_e$,
    \begin{equation*}
        \rho^\Omega(n) \leq C' \Psi^\Omega(p,n),
    \end{equation*}
    where
    \begin{equation}
        \Psi^\Omega(p,n) = \frac{|\Omega|_1^2\log^{(5\tilde{\beta}+12)/(4\tilde{\beta}+8)}(pn)}{n^{\tilde{\beta}/(4\tilde{\beta}+8)}} 
        + \frac{p^4\log^{5/2}(n)}{n^{(2\beta-1/2)\wedge 3/2}}
        + \frac{|\Omega|_1 p^{2}\log(p)}{n^{(\beta-1/4)\wedge 3/4}}.
    \end{equation}
    In particular, if $\Psi^\Omega(p,n) \rightarrow 0$, then $\rho^\Omega(n) \rightarrow 0$.
\end{theorem}

\begin{remark}\label{rmk:GA.prec}
    With a growing dimension, it is likely that the norm $|\Omega|_1$ could also be growing.
    The rate could generally lie between $O(1)$ and $O(p)$, depending on the sparsity pattern of $\Omega$.
    $|\Omega|_1 = O(1)$ corresponds to the sparsest cases.
    Examples include
    (1) when each node of the graph represented by $\Omega$ is connected to a fixed number of other nodes,
    (2) when the precision matrix $\Omega$ is banded with a fixed bandwidth, and
    (3) when $\Omega$ is Toeplitz with a decay rate faster than $|\Omega_{st}| \leq C|s-t|^{-1}$.
    On the contrary, we could potentially have $|\Omega|_1 = O(p)$ in the very dense case, as the worst-case scenario.
    A lower order of $|\Omega|_1$ generally corresponds to a sharper rate $\Psi^\Omega(p,n)$.
    In order to have it converge to zero, the following two conditions are necessary.
    First, we need $|\Omega|_1 = \tilde o(n^{\tilde{\beta}/(8\tilde{\beta}+16)})$.
    Second, we need $|\Omega|_1 p^2 = \tilde o(n^{(\beta-1/4)\wedge 3/4})$.
    These conditions allow the dimension $p$ to grow at a polynomial rate of $n$ at most, up to a logarithm factor.
    For denser precision matrix $\Omega$, this requirement gets stricter.
    That being said, this Gaussian approximation result does not require a specific density assumption on the true underlying $\Omega$.
    Although the Gaussian approximation result for the precision matrix only holds valid in the low-dimensional settings, it still enjoys the desired property of allowing long-memory time series.
\end{remark}


For the fourth and final main result of this paper, we shall establish a data-adaptive and finite-sample validity result for block bootstrap method for precision.
Let
\[
\check{B}_{i,l}^\Omega = \sum_{j=i-l+1}^i \check{\mathcal{X}}_j^\Omega
\]
where
$\check{\mathcal{X}}_j^\Omega = \vectorize(\hat{\Omega}_n (X_j X_j^\top - \hat{\Sigma}_n) \hat{\Omega}_n)$, and let the empirical cdf be
\begin{equation}
    \hat{F}_{n,l}^\Omega(u) = \frac{1}{n-l+1} \sum_{i=l}^n \mathbf{1}\{l^{-1/2} |\check{B}_{i,l}^\Omega|_\infty \leq u\}.
\end{equation}
The validity of this bootstrap method is characterized by
\begin{equation}\label{eq:kolm.distance.Omega.boot}
    \rho^\Omega_B(n,l) = \sup_{u\in\mathbb{R}}\left|\hat{F}_{n,l}^\Omega (u) - \mathbb{P}\left(n^{1/2}|\hat{\Omega}_n - \Omega|_\infty \leq u\right)\right|.
\end{equation}
The following theorem shows that given a certain order of $l$, $\rho^\Omega_B(n,l)$ is upper bounded by a finite-sample rate that converges to $0$ in probability under the same dimensional requirement as in Theorem \ref{thm:GA.prec}.
More details are deferred to Section \ref{sec:bootstrap.prec}.

\begin{theorem}[Block bootstrap consistency for precision]\label{thm:bootstrap.prec}
    Assume that Conditions \ref{cond.A} and \ref{cond.G} hold. For any $0<\epsilon<1$, choose $l\asymp n^\phi \log^\psi(p)$, where $\phi = \frac{2\tilde{\beta}+4}{(3-\epsilon)\tilde{\beta} + 4}$ and $\psi = \frac{5\tilde{\beta}+12}{(3-\epsilon)\tilde{\beta} + 4}(1-\epsilon)$. Then with probability no less than $1-3(\Psi^\Omega_B(p,n,\epsilon))^{\epsilon}$,
    \begin{equation*}
        \rho_B^\Omega(n,l) \leq C'(\Psi^\Omega_B(p,n,\epsilon))^{1-\epsilon},
    \end{equation*}
    where $C'>0$ is a constant related to $C_0$, $c_0$, $C_e$ and $c_e$, and
    \begin{equation}
        \Psi^\Omega_B(p,n,\epsilon) = \frac{|\Omega|_1^2\log^{(5\tilde{\beta}+12)/[(6-2\epsilon)\tilde{\beta}+8]}(pn)}{n^{\tilde{\beta}/[(6-2\epsilon)\tilde{\beta}+8]}}
        + \frac{p^4\log^{5/2}(n)}{n^{(2\beta-1/2)\wedge 3/2}}
        + \frac{|\Omega|_1 p^{2}\log(p)}{n^{(\beta-1/4)\wedge 3/4}}.
    \end{equation}
    In particular, if $\Psi_B^\Omega(p,n,\epsilon) \rightarrow 0$, then $\rho_B^\Omega(n,l)\rightarrow 0$ in probability.
\end{theorem}

\begin{remark}\label{rmk:dim.condition.prec.boot}
    When $\epsilon$ approaches one, the rate $\Psi_B^\Omega(p,n,\epsilon)$ approaches the rate $\Psi^\Omega(p,n)$ in Theorem \ref{thm:GA.prec}.
    The dependence on $|\Omega|_1$ is also similar.
    For $\Psi_B^\Omega(p,n,\epsilon) \rightarrow 0$, we need
    (1) $|\Omega|_1 = \tilde o(n^{\tilde{\beta}/[(12-4\epsilon)\tilde{\beta}+16]})$ such that the first term converges, and
    (2) $|\Omega|_1 p^2 = \tilde o(n^{(\beta-1/4)\wedge 3/4})$ which controls the second and third terms.
    If we choose an $\epsilon$ close to one, the first condition is a similar rate as the first condition mentioned in Remark \ref{rmk:GA.prec}, while the second condition is identical to that in Remark \ref{rmk:GA.prec}.
    Similar to the Gaussian approximation result in Theorem \ref{thm:GA.prec}, the dimension $p$ is allowed to grow at a polynomial rate of $n$ at most, up to a logarithm factor, while denser precision matrix $\Omega$ makes it stricter.
    That being said, no specific structural assumption on the true underlying $\Omega$ is needed.
    Moreover, Theorem \ref{thm:bootstrap.prec} enjoys the desired property of allowing for long-memory time series.
\end{remark}

\begin{remark}
    The optimal order of block size $l$ for precision bootstrap in Theorem \ref{thm:bootstrap.prec} is identical to that for covariance bootstrap in Theorem \ref{thm:bootstrap.cov}.
    As pointed out in Remark \ref{rmk:choice.of.l}, the specific value of $l$ can be chosen in practice either by a data-driven method in the current literature or by trivially specifying $l$ between $n^{2/3}$ and $n$.
\end{remark}

Similar to Theorem \ref{thm:bootstrap.cov}, Theorem \ref{thm:bootstrap.prec} also enables the construction of data-driven simultaneous confidence regions for the precision matrix $\Omega$.
For any significance level $\alpha\in(0,1)$, we can also define the data-driven $(1-\alpha)$-quantile of the bootstrap CDF $\hat F_{n,l}^\Omega$ as
\[
\hat q_{1-\alpha}^\Omega = \inf\{u\in \mathbb{R}: \hat F_{n,l}^\Omega(u) \geq 1-\alpha\}.
\]
The simultaneous confidence region is thus defined as
\[
\hat{\mathcal C}_\alpha^\Omega (\hat\Omega_n) = \left\{\tilde\Omega\in \mathbb{R}^{p\times p}: n^{1/2} |\tilde\Omega - \hat\Omega_n|_\infty \leq \hat q_{1-\alpha}^\Omega \right\}.
\]
In practice, the event that the true precision matrix $\Omega$ landing within the confidence region $\hat{\mathcal C}_\alpha^\Omega (\hat\Omega_n)$ is equivalent to all of its entries landing within their respective confidence intervals,
\[
\Omega_{jk} \in \left[ (\hat{\Omega}_n)_{jk} - n^{-1/2}\hat{q}_{1-\alpha}^\Omega, \quad (\hat{\Omega}_n)_{jk} + n^{-1/2}\hat{q}_{1-\alpha}^\Omega \right], \text{ for all } 1\leq j,k\leq p.
\]
Following the same proof procedure as Proposition \ref{prop:CI.cov}, the accuracy of the confidence region $\hat{\mathcal C}_\alpha^\Omega$ is guaranteed by the following proposition, which is a corollary of Theorem \ref{thm:bootstrap.prec}.
\begin{proposition}[Confidence Regions for Precision Matrices]\label{prop:CI.prec}
    Under the conditions of Theorem \ref{thm:bootstrap.prec},
    \[
    \left|\mathbb{P}\left(\Omega \in \hat{\mathcal C}_\alpha^\Omega (\hat\Omega_n)\right) - (1-\alpha)\right| \leq 3\rho_B^\Omega(n,l).
    \]
    In particular, if $\Psi_B^\Omega(p,n,\epsilon)\rightarrow 0$, then the left-hand side converges to zero in probability.
\end{proposition}

\section{Gaussian Approximation for Covariance Matrix}\label{sec:GA.cov}

This section provides a finite-sample bound for $\rho(n)$ using a three-step framework similar to \cite{berkes2014komlos}. First, we use martingale and $m$-dependent approximations to $n^{1/2}|\hat{\Sigma}_n - \Sigma|_\infty$ and construct triadic blocks that are conditionally independent. Second, we exploit this conditional independence and apply Gaussian approximation results for independent random variables. Third, we apply Gaussian approximation on the conditioned sub-blocks to obtain an unconditional result. These steps lead to Theorem \ref{thm:GA.cov}, a Gaussian approximation result for $\rho(n)$.


We begin by creating a martingale approximation for $S_n$.
Let $\mathcal{F}^i$ represent the $\sigma$-algebra generated by the random vectors $\{\epsilon_{i}, \epsilon_{i-1}, \cdots\}$.
Note that $S_n$ is not a martingale, but is still an adapted process with respect to filtration $\{\mathcal{F}^n\}_{n\in\mathbb{Z}}$.
For $1\leq i\leq n$, define
\begin{equation}\label{eq:D}
D_i = \sum_{t=i}^\infty \mathcal{P}_i(\mathcal{X}_t) = \sum_{t=i}^\infty \left(\mathbb{E}[\mathcal{X}_t\mid\mathcal{F}^i] - \mathbb{E}[\mathcal{X}_t\mid\mathcal{F}^{i-1}]\right),
\end{equation}
see \cite{wu2007strong}.
It turns out that $T_n = \sum_{i=1}^n D_i$ is a martingale approximation of $n^{1/2}S_n$. The following lemma provides a tail bound for the approximation error.
\begin{lemma}[Martingale approximation]\label{lm:mtg.approx}
Assume Condition \ref{cond.A} holds, then for $\delta > 0$,
\begin{equation*}
    \mathbb{P}\left(|S_n-n^{-1/2}T_n|_\infty\geq\delta\right)\leq 2p^2\exp\left\{-c_1\left(\frac{\delta^2}{Q_1(n)}\wedge\frac{\delta}{\sqrt{Q_1(n)}}\right)\right\}
\end{equation*}
where $c_1$ is a constant related to $C_0$, and $Q_1(n)$ equals $n^{-1}$ for $\beta>1$, $n^{-1}\log^2 n$ for $\beta=1$, and $n^{-4\beta+3}$ for $3/4<\beta<1$.
\end{lemma}


We proceed with an $m$-dependent approximation for $T_n$. Define $\mathcal{F}^i_{i-k}$ as the $\sigma$-algebra generated by $\{\epsilon_{i}, \epsilon_{i-1}, \ldots, \epsilon_{i-k}\}$. The $m$-dependent approximation of $D_i$ is given by
\begin{equation}\label{eq:D.tilde}
\tilde{D}_{i,m} = \mathbb{E}[D_i \mid \mathcal{F}^i_{i-m+1}],
\end{equation}
where $m \in \mathbb{N}$ is a parameter associated with $n$ such that $m \leq n$ and $m \rightarrow \infty$ as $n \rightarrow \infty$. The partial sum $\tilde{T}_{n,m} = \sum_{i=1}^n \tilde{D}_{i,m}$ represents the $m$-dependent approximation of $T_n$. The following lemma provides a tail bound for the error resulting from the $m$-dependent approximation.

\begin{lemma}[$m$-dependent approximation]\label{lm:m.dep.approx}
Assume Condition \ref{cond.A} holds, then
\begin{equation*}
    \mathbb{P}\left(n^{-1/2}|T_n-\Tilde{T}_{n,m}|_\infty\geq\delta\right)\leq 2p^2\exp\left\{-c_2\left(\frac{\delta^2}{Q_2(m)}\wedge\frac{\delta}{\sqrt{Q_2(m)}}\right)\right\},
\end{equation*}
where $c_2$ is a constant related to $C_0$, and $Q_2(m)=m^{(-2\beta+1)\vee(-4\beta+3)}$.
\end{lemma}

In contrast to $D_i$, $\tilde{D}_{i,m}$ possesses the following characteristic: for any $i$ and $j$ such that $|i-j|\geq m$, $\tilde{D}_{i,m}$ is independent of $\tilde{D}_{j,m}$. To leverage this independence structure, we create conditionally independent triadic blocks (refer to \cite{berkes2014komlos}) and utilize the existing Gaussian approximation result from \cite{chernozhukov2017central} for independent random variables.

We decompose the partial sum $\tilde{T}_{n,m}$ into $q_n = \lfloor n/3m \rfloor$ many blocks $\{B_j\}_{j=0}^{q_n}$ as
\begin{equation}\label{eq:T.tilde.decomp}
\Tilde{T}_{n,m} = \sum_{j=0}^{q_n} B_j, \hspace{20pt}\text{where }B_j = \sum_{i=3jm+1}^{3(j+1)m \wedge n} \Tilde{D}_{i,m},\hspace{10pt} 0\leq j\leq q_n.
\end{equation}
The choice of block size depends on the strength of temporal dependence.
Intuitively, a larger block size is required to capture longer memory.
This implication is confirmed by the optimal order of $m$ in \eqref{eq:choice.of.eta.m}.
Each block $B_j$ can be further split into triadic sub-blocks,
\begin{align*}
    B_0 &= F_{0,1} + F_{0,2},\\
    B_j &= F_{j,0} + F_{j,1} + F_{j,2}, \quad 1\leq j\leq q_n,
\end{align*}
where
\begin{equation*}
F_{j,k} = \sum_{i=(3j+k)m+1}^{(3j+k+1)m \wedge n}\tilde{D}_{i,m}, \quad k=0,1,2.
\end{equation*}
Consider the partition of the index set
\[
[n] = \mathcal{I}^m_{0,1} \cup \mathcal{I}^m_{0,2} \cup \mathcal{I}^m_{1,0} \cup \mathcal{I}^m_{1,1} \cup \mathcal{I}^m_{1,2} \cup \cdots \cup \mathcal{I}^m_{q_n, 0} \cup \mathcal{I}^m_{q_n, 1} \cup \mathcal{I}^m_{q_n, 2} \cup \mathcal{I}^m_{q_n+1,0},
\]
where $\mathcal{I}^m_{j,k} = \{(3j+k-1)m+l: l\in[m]\} \cap [n]$.
Assume without loss of generality that $n$ is exactly divisible by $3m$, ensuring that the last three subsets are nonempty and each subset $\mathcal{I}^m_{j,k}$ has a cardinality of $m$.
We denote the set $\{\epsilon_i: i\in\mathcal{I}^m_{j,k}\}$ by $\bm{\epsilon}_{j,k}$. Recall that $\tilde{D}_{i,m}$ depends on $\epsilon_{i-m+1},\cdots,\epsilon_{i}$. Thus, $F_{j,0}$ depends on $\bm{\epsilon}_{j,0}$ and $\bm{\epsilon}_{j,1}$, $F_{j,1}$ on $\bm{\epsilon}_{j,1}$ and $\bm{\epsilon}_{j,2}$, and $F_{j,2}$ on $\bm{\epsilon}_{j,2}$ and $\bm{\epsilon}_{j+1,0}$. For $j \geq 1$, we have that $B_{j-1} \indep B_{j} \,\mid\, \bm{\epsilon}_{j,0}$.
Let
\[
\bm{\epsilon}_0 = \bigcup_{j=1}^{q_n+1}\bm{\epsilon}_{j,0}, \quad \bm{\epsilon}_1 = \bigcup_{j=0}^{q_n}\bm{\epsilon}_{j,1}, \quad \bm{\epsilon}_2 = \bigcup_{j=0}^{q_n}\bm{\epsilon}_{j,2}.
\]
Given $\bm{\epsilon}_0$, the blocks $\{B_j\}_{j=0}^{q_n}$ are independent. Thus, we derive a Gaussian approximation for
\[
\tilde{T}_{n,m} - \mathbb{E}[\tilde{T}_{n,m}\mid\bm{\epsilon}_0]
\]
using a theorem for independent random variables from~\cite{chernozhukov2017central}.

Conditional on $\bm{\epsilon}_{0}$, the variance of $B_0$ depends only on $\bm{\epsilon}_{1,0}$, and for $B_j$ ($j \geq 1$), it depends on $\bm{\epsilon}_{j,0}$ and $\bm{\epsilon}_{j+1,0}$. Thus, we define
\begin{align*}
    V^*_0(\bm{\epsilon}_{1,0}) &= \var(B_0 \mid \bm{\epsilon}_0),\\
    V^*_j(\bm{\epsilon}_{j,0},\bm{\epsilon}_{j+1,0}) &= \var(B_j \mid \bm{\epsilon}_0), \quad 1 \leq j \leq q_n.
\end{align*}
Consider $Y^*_0, Y^*_1, \ldots, Y^*_{q_n} \in \mathbb{R}^{p^2}$ as i.i.d.~standard Gaussian random vectors that are independent of $\bm{\epsilon}_0$. Given $\bm{\epsilon}_0$, the partial sum
\begin{equation}\label{eq:T.star}
T^*_n = [V^*_0(\bm{\epsilon}_{1,0})]^{1/2}Y^*_0 + \sum_{j=1}^{q_n} [V^*_j(\bm{\epsilon}_{j,0},\bm{\epsilon}_{j+1,0})]^{1/2}Y^*_j
\end{equation}
follows a Gaussian distribution with mean 0 and variance
\begin{equation*}
V^*_0(\bm{\epsilon}_{1,0}) + \sum_{j=1}^{q_n} V^*_j(\bm{\epsilon}_{j,0},\bm{\epsilon}_{j+1,0}),
\end{equation*}
and hence serves as a Gaussian approximation to $\tilde{T}_{n,m} - \mathbb{E}[\tilde{T}_{n,m}\mid\bm{\epsilon}_0]$ conditional on $\bm{\epsilon}_0$.
Note that the summands in \eqref{eq:T.star} are not independent, because the conditional variance of $B_j$ given $\bm{\epsilon}_0$ depends on both $\bm{\epsilon}_{j,0}$ and $\bm{\epsilon}_{j+1,0}$.
In order to build mutually independent triadic blocks, the following identity is useful.
As shown in Appendix \ref{app:var.identity}, we have
\begin{equation}\label{eq:var.identity}
V^*_0(\bm{\epsilon}_{1,0}) + \sum_{j=1}^{q_n} V^*_j(\bm{\epsilon}_{j,0},\bm{\epsilon}_{j+1,0}) = \sum_{j=1}^{q_n+1} V_j(\bm{\epsilon}_{j,0}),
\end{equation}
where $V_{q_n+1}(\bm{\epsilon}_{q_n+1,0}) = V^*_0(\bm{\epsilon}_{q_n+1,0})$ and $V_j(\bm{\epsilon}_{j,0}) = V^*_j(\bm{\epsilon}_{j,0},\bm{\epsilon}_{j,0})$ for $1 \leq j \leq q_n$. Let $Y_1,\cdots,Y_{q_n}, Y_{q_n+1}\in\mathbb{R}^{p^2}$  be a set of independent standard Gaussian random vectors. The partial sum
\begin{equation}\label{eq:TnY.def}
T_n^Y = \sum_{j=1}^{q_n+1} [V_j(\bm{\epsilon}_{j,0})]^{1/2}Y_j
\end{equation}
is identically distributed to $T_n^*$ conditional on $\bm{\epsilon}_0$. Thus, $T_n^Y$ is a Gaussian approximation to $\tilde{T}_{n,m} - \mathbb{E}[\tilde{T}_{n,m}\mid\bm{\epsilon}_0]$ if conditioned on $\bm{\epsilon}_0$.
The following lemma asserts that this Gaussian approximation holds with high probability in the probability space of $\bm{\epsilon}_0$.
The proof is given in Appendix \ref{sec:pf.cond.CCK}.

\begin{lemma}[Conditional Gaussian approximation]\label{lm:cond.CCK}
Assume that Conditions \ref{cond.A} and \ref{cond.G} hold. Then
\begin{equation*}
\sup_{u\in\mathbb{R}}\left|\mathbb{P}\left(n^{-1/2}|\tilde{T}_{n,m} - \mathbb{E}[\tilde{T}_{n,m}\mid \bm{\epsilon}_0]|_\infty \geq u \mid \bm{\epsilon_0}\right) - \mathbb{P}\left(n^{-1/2}|T^Y_n|_\infty\geq u\mid\bm{\epsilon}_0\right)\right|
\leq C\left(\frac{m^4 \log^5(pn)}{n}\right)^{1/4}
\end{equation*}
with probability no less than $1-4p^2\exp(-c_2\delta n^{1/2}/m^{1/2}) - 2\exp(-c_2 m^{2\epsilon})$, for any choice of $\delta>0$ and $0 < \epsilon < 1$. Here, $C>0$ is a constant related to $C_0$ and $c_0$, and $c_2>0$ is a constant related to $C_0$.
\end{lemma}

Without conditioning on $\bm{\epsilon}_0$, $T_n^Y$ is not Gaussian.
However, due to the $m$-dependence structure, $\{\bm{\epsilon}_{j,0}\}_{j=1}^{q_n+1}$ are mutually independent.
Thus, all terms on the right-hand side of \eqref{eq:TnY.def} are also mutually independent.
We can therefore reapply the Gaussian approximation result for independent variables by \cite{chernozhukov2019improved} to $T_n^Y$, as in the following lemma.

\begin{lemma}[Unconditional Gaussian approximation]\label{lm:uncond.CCK}
Assume that Conditions \ref{cond.A} and \ref{cond.G} hold.
For each $j \in [q_n+1]$, let $Z_j \in \mathbb{R}^{p^2}$ be a Gaussian random vector with zero mean and variance $\var([V_j(\bm{\epsilon}_{j,0})]^{1/2}Y_j)$.
Then the normalized partial sum $S_n^Z = n^{-1/2}\sum_{j=1}^{q_n+1}Z_j$ satisfies
\begin{equation*}
\sup_{u\in\mathbb{R}}\left| \mathbb{P}(n^{-1/2}|T^Y_n|_\infty \geq u) - \mathbb{P}(|S_n^Z|_\infty \geq u)\right| \leq C\left(\frac{m^4 \log^5(pn)}{n}\right)^{1/4},
\end{equation*}
where $C>0$ is a constant related to $C_0$ and $c_0$.
\end{lemma}
The proof of Lemma \ref{lm:uncond.CCK} is given in Appendix \ref{sec:pf.uncond.CCK}.
For finite $n$, $S_n^Z$ and its Gaussian approximation $Z$ are not identically distributed. However, their covariances converge as $n\rightarrow\infty$ (see Appendix \ref{sec:apdx.GA.comparison} for detailed discussion). The lemma below provides a finite sample bound for the Kolmogorov distance between $|Z|_\infty$ and $|S_n|_\infty$.
\begin{lemma}[Comparison of Gaussian approximations]\label{lm:GA.comparison}
    Assume that Conditions \ref{cond.A} and \ref{cond.G} hold. Then there exists a constant $C>0$ related to $C_0$ and $c_0$, such that
    \begin{equation*}
        \sup_{u\in \mathbb{R}}\left|\mathbb{P}(|Z|_\infty \geq u) - \mathbb P (|S_n^Z|_\infty \geq u)\right| \leq C \frac{\log(p)}{m^{\tilde\beta/2}}.
    \end{equation*}
\end{lemma}
Putting these lemmata together, we are ready to prove the Gaussian approximation result in Theorem \ref{thm:GA.cov}.

\begin{proof}[Proof of Theorem \ref{thm:GA.cov}]
We start by marginalizing out $\bm{\epsilon}_0$ in Lemma \ref{lm:cond.CCK}. Let
\begin{equation*}
\Delta(u) = \mathbb{P}\left(n^{-1/2}|\tilde{T}_{n,m} - \mathbb{E}[\tilde{T}_{n,m}\mid\bm{\epsilon}_0]|_\infty \geq u \mid \bm{\epsilon_0}\right) - \mathbb{P}\left(n^{-1/2}|T^Y_n|_\infty\geq u\mid\bm{\epsilon}_0\right).
\end{equation*}
Note that for any fixed $u$, $\Delta(u)$ is a random variable related to $\bm{\epsilon}_0$, and
\begin{equation*}
\mathbb{E}[\Delta(u)] = \mathbb{P}\left(n^{-1/2}|\tilde{T}_{n,m}|_\infty \geq u\right) - \mathbb{P}\left(n^{-1/2}|T^Y_n|_\infty\geq u\right).
\end{equation*}
As a first step, we shall bound $|\mathbb{E}[\Delta(u)]|$.
Note that $|\mathbb{E}[\Delta(u)]| \leq \mathbb{E}[|\Delta(u)|]$, and the convenient fact that $|\Delta(u)| \leq 1$.
We shall use the following property; for a bounded random variable $0\leq X\leq A$, $\mathbb{E}[X]\leq \mathbb{P}(X\geq x) A + \mathbb{P}(X\leq x) x$.
Indeed, for any $u$, $\delta>0$ and $0<\epsilon<1$,
\begin{align}
\begin{split}
    \mathbb{E}[|\Delta(u)|] &\leq \mathbb{P}\left(|\Delta(u)|\leq C\left(\frac{m^4\log^5(pn)}{n}\right)^{1/4}\right)C\left(\frac{m^4\log^5(pn)}{n}\right)^{1/4} + \mathbb{P}\left(|\Delta(u)|> C\left(\frac{m^4\log^5(pn)}{n}\right)^{1/4}\right)\\
    &\leq C\left(\frac{m^4\log^5(pn)}{n}\right)^{1/4} + 4p^2\exp(-c_2\delta n^{1/2}/m^{1/2}) + 2\exp(-c_2 m^{2\epsilon}).\label{eq:E.Delta}
\end{split}
\end{align}
The second inequality follows Lemma \ref{lm:cond.CCK}.
The first term on the right hand side of \eqref{eq:E.Delta} dominates for our choice of $m$ in \eqref{eq:choice.of.eta.m}. Therefore
\begin{equation}\label{eq:lemma.3.corollary}
\sup_{u\in\mathbb{R}}\left|\mathbb{P}\left(n^{-1/2}|\tilde{T}_{n,m}|_\infty \geq u\right) - \mathbb{P}\left(n^{-1/2}|T^Y_n|_\infty\geq u\right)\right| \lesssim \left(\frac{m^4\log^5(pn)}{n}\right)^{1/4}
\end{equation}
up to a constant related to $C_0$ and $c_0$.
Combine \eqref{eq:lemma.3.corollary} with Lemma \ref{lm:uncond.CCK}, we have the following Berry-Esseen bound for $S^Z_n$ as a Gaussian approximation to $n^{-1/2}\tilde{T}_{n,m}$:
\begin{equation}\label{eq:lemma.34}
\sup_{u\in\mathbb{R}}\left|\mathbb{P}\left(n^{-1/2}|\tilde{T}_{n,m}|_\infty \geq u\right) - \mathbb{P}\left(|S^Z_n|_\infty\geq u\right)\right| \lesssim \left(\frac{m^4\log^5(pn)}{n}\right)^{1/4}.
\end{equation}

As a result of \eqref{eq:Sigma.S.close}, the covariance matrix of $S^Z_n$ converges to that of $Z$ under our choice of $m$ in \eqref{eq:choice.of.eta.m}.
Condition \ref{cond.G} implies that the variances of all entries in $Z$ are lower bounded.
Therefore, the variances of all entries in $S^Z_n$ are also lower bounded by a positive constant related to $c_0$ when $n$ is large.
Similarly, under Condition \ref{cond.A}, the variances are upper bounded by a positive constant related to $C_0$ when $n$ is large.
By Theorem 3 of \cite{chernozhukov2015comparison}, for any $\eta>0$, we have the following anti-concentration inequality
\begin{equation}\label{eq:conc.ineq}
\sup_{u\in\mathbb{R}}\mathbb{P}\left(||S^Z_n|_\infty - u|<\eta\right) \lesssim \eta\sqrt{1\vee\log(p/\eta)}
\end{equation}
up to a constant that depends on $C_0$ and $c_0$.
Combining \eqref{eq:lemma.34} and \eqref{eq:conc.ineq} gives
\begin{equation}\label{eq:conc.ineq.Tnm}
\sup_{u\in\mathbb{R}}\mathbb{P}\left(|n^{-1/2}|\tilde{T}_{n,m}|_\infty - u|<\eta\right) \lesssim \eta\sqrt{1\vee\log(p/\eta)} + \left(\frac{m^4\log^5(pn)}{n}\right)^{1/4}.
\end{equation}
The validity of $n^{-1/2}\tilde{T}_{n,m}$ as an approximation of $S_n$ can be therefore verified as follows.
By triangle inequality, for every $\eta>0$,
\begin{align}
\begin{split}
&\sup_{u\in\mathbb{R}}\left|\mathbb{P}\left(|S_n|_\infty\geq u\right) - \mathbb{P}\left(n^{-1/2}|\tilde{T}_{n,m}|_\infty\geq u\right) \right|\\
\leq & \mathbb{P}\left(|S_n - n^{-1/2}\tilde{T}_{n,m}|_\infty \geq \eta\right) + \sup_{u\in\mathbb{R}}\mathbb{P}\left(|n^{-1/2}|\tilde{T}_{n,m}|_\infty - u|<\eta\right)\\
\leq & \mathbb{P}\left(|S_n - n^{-1/2}T_n|_\infty \geq \eta/2\right) + \mathbb{P}\left(n^{-1/2}|T_n - \tilde{T}_{n,m}|_\infty \geq \eta/2\right)
 + \sup_{u\in\mathbb{R}}\mathbb{P}\left(|n^{-1/2}|\tilde{T}_{n,m}|_\infty - u|<\eta\right).\label{eq:lemma.12}
\end{split}
\end{align}
Combining \eqref{eq:lemma.12} with \eqref{eq:lemma.34}, by Lemmas \ref{lm:mtg.approx}, \ref{lm:m.dep.approx}, \ref{lm:GA.comparison} and \eqref{eq:conc.ineq.Tnm}, for any choice of $\eta>0$ such that $\eta^2/Q_1(n)\gtrsim 1$ and $\eta^2/Q_2(m)\gtrsim 1$, up to a positive constant related to $C_0$ and $c_0$,
\begin{equation}
    \rho(n) \lesssim p^2\exp\left(-c_1 \eta / Q_1^{1/2}(n)\right) + p^2\exp\left(-c_2 \eta / Q_2^{1/2}(m)\right) + 
    \eta\sqrt{1\vee\log(p/\eta)} +\left(\frac{m^4\log^5(pn)}{n}\right)^{1/4} + \frac{\log(p)}{m^{\tilde\beta/2}}.\label{eq:rho.n.bound.eta}
\end{equation}

We finish the proof by selecting the optimal order of parameters $\eta$ and $m$, such that the bound of $\rho(n)$ given in \eqref{eq:rho.n.bound.eta} is minimized.
For constants $C_\eta>0$ and $C_m > 0$, let
\begin{equation}\label{eq:choice.of.eta.m}
\eta = C_\eta \frac{\log^{(3\tilde{\beta}+8) / (4\tilde{\beta}+8)}(pn)}{n^{\tilde{\beta} / (4\tilde{\beta}+8)}},\quad m=C_m(n\log(pn))^{1/(2\tilde{\beta}+4)}.
\end{equation}
The detailed reasoning of such a choice is given in Appendix \ref{sec:choice.of.eta.m}.
It can be verified that $m\rightarrow \infty$, $m/n \rightarrow 0$, and the right hand side of \eqref{eq:E.Delta} is dominated by the first term.
It also holds that $\eta^2/Q_1(n)\gtrsim 1$ and $\eta^2/Q_2(m)\gtrsim 1$, which is the assumption to \eqref{eq:rho.n.bound.eta}.
We can thus obtain the upper bound of $\rho(n)$ given in Theorem \ref{thm:GA.cov}.
A longer temporal dependence (i.e., smaller $\tilde\beta$) indeed corresponds to a larger choice of $m$.
\end{proof}

\section{Block Bootstrap for Covariance Matrix}\label{sec:bootstrap.cov}

In Theorem \ref{thm:GA.cov}, we have shown that $\mathbb{P}(n^{1/2}|\hat{\Sigma}_n - \Sigma|_\infty \geq u)$ can be well approximated by $\mathbb{P}(|Z|_\infty \geq u)$.
In practice, however, the covariance of the Gaussian approximation $Z$ contains true underlying autocovariance matrices that are generally unknown; see \eqref{eq:covariance.of.Z}.
This section provides an attainable data-dependent bootstrap method to approximate $\mathbb{P}(|Z|_\infty \geq u)$, and hence $\mathbb{P}(n^{1/2}|\hat{\Sigma}_n - \Sigma|_\infty \geq u)$.

As a vital step in realizing high-dimensional CLT, various bootstrap methods have been explored in the literature.
\cite{chernozhukov2013gaussian} and \cite{chernozhukov2017central} applied multiplier and empirical bootstrap techniques on independent random variables.
The same method was extended to short-range dependent time series by \cite{zhang2014bootstrapping} using blockwise bootstrap procedure.
Alternatively, \cite{zhang2017gaussian} implemented a batched-mean estimator of long-run covariance matrices for short-memory high-dimensional time series.
These bootstrap results require either independence or weak temporal dependence, and hence cannot be extended to the long-range dependence setting.
\cite{hall1998sampling} proposed a block sampling method for the sample mean of long-memory processes.
This result was generalized by \cite{zhang2013block} to nonlinear transforms of linear processes.
Both results on validity of block sampling method are asymptotic.
By applying block sampling technique on covariance of Gaussian linear process, we develop a finite sample result on validity.

Let $l=l_n$ be the size of sampling window.
Assume that $l_n\rightarrow\infty$ and $l_n/n\rightarrow 0$.
To ease the notation, the subscript will be omitted hereafter.
Let
\[
\check{B}_{i,l} = \sum_{j=i-l+1}^{i} \check{\mathcal{X}}_j \quad \text{where } \check{\mathcal{X}}_j = \vectorize(X_j X_j^\top - \hat{\Sigma}_n).
\]
Then $\mathbb{P}(|Z|_\infty \leq u)$ is approximated by the following empirical distribution function,
\begin{equation*}
    \hat{F}_{n,l}(u) = \frac{1}{n-l+1}\sum_{i=l}^n \mathbf{1}\{l^{-1/2}|\check{B}_{i,l}|_\infty \leq u\}.
\end{equation*}
The validity of block sampling method is measured by
\begin{equation*}
    \rho_B(n) = \sup_{u\in\mathbb{R}}\left|\hat{F}_{n,l}(u) - \mathbb{P}\left(n^{1/2}|\hat{\Sigma}_n-\Sigma|_\infty\leq u\right)\right|.
\end{equation*}
If $\rho_B(n)\rightarrow 0$ in probability, then the block sampling method is valid.
As a primary objective of this section, we will give an explicit rate of convergence for $\rho_B(n)$.

Recall \eqref{eq:D.tilde} for the definition of $\tilde{D}_{j,m}$. Let
\begin{equation*}
    \tilde{B}_{i,l,m} = \sum_{j=i-l+1}^{i} \tilde{D}_{j,m}.
\end{equation*}
The following lemma shows that $\tilde{B}_{i,l,m}$ is a good $m$-dependent estimate to $\check{B}_{i,l}$.
\begin{lemma}\label{lm:B.check.B.tilde}
    Under Conditions \ref{cond.A} and \ref{cond.G}, for any $\delta>0$, there exists a constant $C>0$ related to $C_0$ and $c_0$, and constants $c_1>0$ and $c_2>0$ that are related to $C_0$, such that
    \begin{multline*}
        \mathbb{P}\left(l^{-1/2}|\check{B}_{i,l} - \tilde{B}_{i,l,m}|_\infty > \delta\right) \leq C \Psi(p,n) + C\delta^{-1}\sqrt{\frac{l\log(p)}{n}} \\
        + 2p^2\exp\left(-c_1\left(\frac{\delta^2}{Q_1(l)} \wedge \frac{\delta}{\sqrt{Q_1(l)}}\right)\right) + 2p^2\exp\left(-c_2\left(\frac{\delta^2}{Q_2(m)} \wedge \frac{\delta}{\sqrt{Q_2(m)}}\right)\right).
    \end{multline*}
\end{lemma}
The proof of Lemma \ref{lm:B.check.B.tilde} is deferred to Appendix \ref{sec:apdx.B.check.B.tilde}.
Following the same proof strategy as Theorem \ref{thm:GA.cov}, the following Gaussian approximation result on $\tilde{B}_{i,l,m}$ holds.
\begin{lemma}[Gaussian Approximation for $\tilde{B}_{i,l,m}$]\label{lm:B.check.GA}
    Under Conditions \ref{cond.A} and \ref{cond.G}, for any $l\leq i\leq n$ and a constant $C>0$ related to $C_0$ and $c_0$,
    \begin{equation*}
    \sup_{u\in\mathbb{R}} \left|\mathbb{P}\left(l^{-1/2}|\tilde{B}_{i,l,m}|_\infty \leq u\right) - \mathbb{P}\left(|Z|_\infty \leq u\right) \right| \leq C\Psi(p,l).
    \end{equation*}
\end{lemma}

The proof of Lemma \ref{lm:B.check.GA} relies on the same technique as in Section \ref{sec:GA.cov}.
Analogous to the manipulation of $\Tilde{T}_{n,m}$, we divide $\Tilde{B}_{i,l,m}$ into triadic blocks given each fixed $i$, and exploit the $m$-dependent properties of $\tilde{D}_{j,m}$.
The detailed proof is deferred to Appendix \ref{sec:apdx.B.check.GA}.

The following corollary of Lemmas \ref{lm:B.check.B.tilde} and \ref{lm:B.check.GA} specifies an upper bound for the Kolmogorov distance between $l^{-1/2}|\check{B}_{i,l}|_\infty$ and $l^{-1/2}|\tilde{B}_{i,l,m}|_\infty$.
\begin{corollary}\label{cor:B.check.B.tilde}
    Under conditions \ref{cond.A} and \ref{cond.G}, up to a constant $C>0$ related to $C_0$ and $c_0$,
    \begin{equation}\label{eq:cor.B.check.B.tilde}
        \sup_{u\in\mathbb{R}}\left| \mathbb{P}\left(l^{-1/2}|\check{B}_{i,l}|_\infty \leq u\right) - \mathbb{P}\left(l^{-1/2}|\tilde{B}_{i,l,m}|_\infty \leq u\right) \right| \leq C\Psi(p,l).
    \end{equation}
\end{corollary}
\begin{proof}[Proof of Corollary \ref{cor:B.check.B.tilde}]
    By triangle inequality, for any $\eta>0$,
    \begin{multline}\label{eq:B.check.B.tilde.partition}
            \sup_{u\in\mathbb{R}}\left| \mathbb{P}\left(l^{-1/2}|\check{B}_{i,l}|_\infty \leq u\right) - \mathbb{P}\left(l^{-1/2}|\tilde{B}_{i,l,m}|_\infty \leq u\right) \right|\\
            \leq \mathbb{P}\left(l^{-1/2}|\check{B}_{i,l} - \tilde{B}_{i,l,m}|_\infty >\eta\right) + \sup_{u\in\mathbb{R}}\mathbb{P}\left(\left|l^{-1/2}|\tilde{B}_{i,l,m}|_\infty - u\right|<\eta\right).
    \end{multline}
    The first term on the right hand side of \eqref{eq:B.check.B.tilde.partition} is directly upper bounded by Lemma \ref{lm:B.check.B.tilde}.
    Recall that by Theorem 3 in \cite{chernozhukov2015comparison}, up to a constant $C>0$ related to $C_0$ and $c_0$, the following anti-concentration inequality holds for Gaussian random vector $Z$:
    \begin{equation}\label{eq:anti.conc.Z}
        \sup_{u\in\mathbb{R}}\mathbb{P}\left(\left||Z|_\infty-u\right|<\eta\right) \leq C \eta \sqrt{\log(p/\eta)}.
    \end{equation}
    By \eqref{eq:anti.conc.Z} and Lemma \ref{lm:B.check.GA}, the second term on the right hand side of \eqref{eq:B.check.B.tilde.partition} can be upper bounded by
    \begin{align*}
        &\sup_{u\in\mathbb{R}} \mathbb{P} \left(\left|l^{-1/2}|\tilde{B}_{i,l,m}|_\infty - u\right|<\eta\right)
        \leq C \eta \sqrt{\log(p/\eta)} + C\Psi(p,l).
    \end{align*}
    For any choice of $\eta>0$ such that $\eta^2/Q_1(n)\gtrsim 1$ and $\eta^2/Q_2(m)\gtrsim 1$,
    the upper bound of the Kolmogorov distance on the left hand side of \eqref{eq:cor.B.check.B.tilde} is
    \begin{equation}\label{eq:cor.B.check.B.tilde.target}
            C \eta \sqrt{\log(p/\eta)} + C\Psi(p,l) + C\eta^{-1}\sqrt{\frac{l\log(p)}{n}}\\
            +2p^2\exp\left(-c_1 \eta / Q_1^{1/2}(l)\right) + 2p^2\exp\left(-c_2 \eta / Q_2^{1/2}(m)\right).
    \end{equation}
    With the choice of
    \begin{equation}\label{eq:choice.eta.m.boot}
        \eta = C_\eta \frac{\log^{(3\tilde{\beta}+8) / (4\tilde{\beta}+8)}(pl)}{l^{\tilde{\beta} / (4\tilde{\beta}+8)}},\quad m=C_m (l\log(pl))^{1/(2\tilde{\beta}+4)},
    \end{equation}
    all five terms of \eqref{eq:cor.B.check.B.tilde.target} are dominated by the second term.
    This proves Corollary \ref{cor:B.check.B.tilde}.
\end{proof}

Consider the empirical CDF with respect to $\tilde{B}_{i,l,m}$,
\begin{equation}
    \tilde{F}_{n,l,m}(u) = \frac{1}{n-l+1}\sum_{i=l}^n \mathbf{1}\{l^{-1/2}|\tilde{B}_{i,l,m}|_\infty \leq u\}.
\end{equation}
Despite not being data adaptive, this intermediate quantity is nevertheless useful in the subsequent proof.
Corollary \ref{cor:B.check.B.tilde} implies the proximity of two empirical CDF's, $\hat{F}_{n,l}(u)$ and $\tilde{F}_{n,l,m}(u)$, with high probability.
This result is formalized by the following lemma.
Recall the definition of $\Psi(p,n)$ in \eqref{eq:Theta.pn}.
\begin{lemma}\label{lm:F.check.F.tilde}
    Under Conditions \ref{cond.A} and \ref{cond.G}, for any $0<\epsilon<1$, with probability no less than $1-(\Psi(p,l))^\epsilon$,
    \begin{equation*}
        \sup_{u\in\mathbb{R}}\left|\hat{F}_{n,l}(u) - \tilde{F}_{n,l,m}(u)\right| \leq C(\Psi(p,l))^{1-\epsilon},
    \end{equation*}
    where $C>0$ is a constant related to $C_0$ and $c_0$.
\end{lemma}
The proof of Lemma \ref{lm:F.check.F.tilde} relies on Markov inequality and Corollary \ref{cor:B.check.B.tilde}. See Appendix \ref{sec:apdx.F.check.F.tilde}.
Meanwhile, the following finite-sample Glivenko-Cantelli type result holds for empirical CDF $\tilde{F}_{n,l,m}$.
\begin{lemma}\label{lm:B.check.CDF}
    For any $\lambda>0$, with probability no less than $1-2e^{-2\lambda^2}$,
    \begin{equation}\label{eq:B.check.CDF}
        \sup_{u\in\mathbb{R}}\left|\tilde{F}_{n,l,m}(u) - \mathbb{P}\left(l^{-1/2}|\tilde{B}_{i,l,m}|_\infty \leq u\right)\right|
        \leq \sqrt{\frac{(n+m)(l+m-1)}{(n-l+1)^2}\left(\lambda^2 + \frac{1}{2}\log(l+m-1)\right)}.
    \end{equation}
\end{lemma}
The proof of Lemma \ref{lm:B.check.CDF} relies on a technique that exploits the $m$-dependent property of $\{\tilde{B}_{i,l,m}\}_{i=l}^n$.
Specifically, they are divided into groups such that all entries within the same groups are mutually independent.
Note that they are also identically distributed due to the stationarity of $\tilde{D}_{j,m}$ by construction.
The proof then relies on Corollary 1 in \cite{massart1990tight}, which is a finite-sample Glivenko-Cantelli-type result on i.i.d.~random variables.
The details are deferred to \ref{sec:apdx.B.check.CDF}.

Finally, we are ready to prove Theorem \ref{thm:bootstrap.cov} that verifies the consistency of block bootstrap.

\begin{proof}[Proof of Theorem \ref{thm:bootstrap.cov}]
    Combining the results of Corollary \ref{cor:B.check.B.tilde}, Lemmas \ref{lm:F.check.F.tilde} and \ref{lm:B.check.CDF}, for any $0<\epsilon<1$ and $\lambda>0$, up to a constant related to $C_0$ and $c_0$,
    \begin{equation}\label{eq:rho.B.ub}
            \rho_B(n,l) \lesssim (\Psi(p,l))^{1-\epsilon} + \sqrt{\frac{(n+m)(l+m-1)}{(n-l+1)^2}\left(\lambda^2 + \frac{1}{2}\log(l+m-1)\right)}
    \end{equation}
    with probability no less than
    \[
    1-(\Psi(p,l))^{\epsilon} - 2\exp(-2\lambda^2).
    \]
    Recall \eqref{eq:choice.eta.m.boot} that our choice of $m = C_m (l\log(pl))^{1/(2\tilde{\beta}+4)}$.
    Next we shall choose the order of $l$ with respect to $n$ and $p$ that minimizes the right hand side of \eqref{eq:rho.B.ub}, such that $m\ll l$ and $l\ll n$.
    Under these conditions, the second term on the right hand side of \eqref{eq:rho.B.ub} has order $O((\lambda^2 + \log(l))^{1/2} l^{1/2} / n^{1/2})$.
    Let $\lambda^2 = \epsilon \log(l) / 2$, and assume for some constants $\phi>0$ and $\psi>0$, $l = O( n^{\phi} \log^{\psi}(p))$.
    Then with probability no less than $1-3(\Psi(p,l))^\epsilon$,
    \begin{equation}\label{eq:rho.B.ub.new}
        \rho_B(n,l) \lesssim \frac{\log^{\frac{5\tilde{\beta} + 12}{4\tilde{\beta}+8}(1-\epsilon) - \frac{\tilde{\beta}}{4\tilde{\beta}+8}(1-\epsilon)\psi}(pn)}{n^{\frac{\tilde{\beta}}{4\tilde{\beta}+8}(1-\epsilon)\phi}} + \frac{\log^{\frac{\psi}{2}}(pn)}{n^{\frac{1}{2} - \frac{\phi}{2}}}.
    \end{equation}
    If we let the two terms have the same order for any $p$ and $n$, then the choice of $\phi>0$ and $\psi>0$ is unique, i.e.,
    \begin{equation*}
        \phi = \frac{2\tilde{\beta}+4}{(3-\epsilon)\tilde{\beta} + 4}, \quad \psi = \frac{5\tilde{\beta} + 12}{(3-\epsilon)\tilde{\beta} + 4}(1-\epsilon).
    \end{equation*}
    By definition of $\Psi_B(p,n,\epsilon)$ in \eqref{eq:Psi.B.pne}, \eqref{eq:rho.B.ub.new} therefore becomes
    \begin{equation*}
        \rho_B(n,l) \lesssim
        \left(\frac{\log^{5\tilde{\beta}+12}(p)}{n^{\tilde{\beta}}}\right)^{\frac{1-\epsilon}{(6-2\epsilon)\tilde{\beta} + 8}} = (\Psi_B(p,n,\epsilon))^{1-\epsilon}
    \end{equation*}
    with probability no less than $1-3(\Psi_B(p,n,\epsilon))^\epsilon$.
    This finishes the proof of Theorem \ref{thm:bootstrap.cov}.
\end{proof}

\section{Gaussian Approximation for Precision Matrix}\label{sec:GA.prec}

In this section, we develop a simultaneous Gaussian approximation result for the precision matrix under low-dimensional settings, that is, $p<n$ and $p/n\rightarrow 0$.
Under such settings, the inference remains free of any structural assumption, and both the
sample covariance matrix $\hat{\Sigma}_n$ and the true underlying covariance matrix $\Sigma$ are invertible in this case.
Recall the extra assumption in the following two sections that all eigenvalues of the true underlying covariance matrix are upper and lower bounded, $0<c_e \leq \lambda_{\min}(\Sigma) \leq \lambda_{\max}(\Sigma) \leq C_e$.
We let $\hat{\Omega}_n = \hat{\Sigma}_n^{-1}$ be the estimator of interest.
Analogous to the definition of $\rho(n)$ in \eqref{EQ_primary_objective}, let $Z^\Omega$ be a $p^2$-dimensional Gaussian random vector with mean zero and the same covariance structure as $n^{1/2}\vectorize(\hat{\Omega}_n - \Omega)$.
The primary objective of this section is to give an upper bound of the Kolmogorov distance
\begin{equation*}
    \rho^\Omega(n) = \sup_{u\in\mathbb{R}} \left|\mathbb{P}\left(n^{1/2}|\hat{\Omega}_n - \Omega|_\infty \leq u\right) - \mathbb{P}(|Z^\Omega|_\infty \leq u)\right|.
\end{equation*}
For brevity, denote $\mathfrak{W} = \hat{\Omega}_n - \Omega$ and $\mathfrak{S} = \hat{\Sigma}_n - \Sigma$.
We find that
\begin{equation}\label{eq:Delta.decomp}
    \mathfrak{W} = -(I+\Omega \mathfrak{S})^{-1} \Omega \mathfrak{S} \Omega.
\end{equation}
We shall first examine Gaussianity for the dominant linear term $\Omega \mathfrak{S} \Omega$.
Define
\begin{equation}\label{eq:X.j.Omega}
    S_n^\Omega = n^{-1/2}\sum_{i=1}^n \mathcal{X}_i^\Omega, \quad \text{where }{\mathcal{X}}_i^\Omega = \vectorize(\Omega(X_i X_i^\top - \Sigma)\Omega).
\end{equation}
Note that $S_n^\Omega = n^{1/2} \vectorize(\Omega \mathfrak{S} \Omega)$.
Similar to the procedure in Section \ref{sec:GA.cov}, we can construct the martingale approximation of $S_n^\Omega$ as
\begin{equation}\label{eq:D.j.Omega}
    T_n^\Omega = \sum_{i=1}^n D_i^\Omega, \quad \text{where } D_i^\Omega = \sum_{t=i}^\infty \mathcal{P}_i(\mathcal{X}_t^\Omega),
\end{equation}
and the $m$-dependent approximation
\begin{equation}\label{eq:D.tilde.j.Omega}
    \tilde{T}_{n,m}^\Omega = \sum_{i=1}^n \tilde{D}_{i,m}, \quad \text{where } \tilde{D}_{i,m}^\Omega = \mathbb{E}[D_i^\Omega \mid \mathcal{F}_{i-m+1}^m].
\end{equation}
Since all these quantities involve a linear transformation of the quantities discussed in Section \ref{sec:GA.cov}, we can construct a Gaussian approximation result for $\Omega \mathfrak{S} \Omega$, formalized by the following lemma.
\begin{lemma}\label{lm:GA.cov.linear}
    Let $Z^\mathfrak{S}$ be a $p^2$-dimensional Gaussian random vector with mean zero and the same covariance structure as $n^{1/2}\vectorize(\Omega \mathfrak{S} \Omega)$.
    Under Conditions \ref{cond.A} and \ref{cond.G}, there exists a constant $C'>0$ related to $C_0$, $c_0$, $C_e$ and $c_e$, such that
    \begin{equation*}
        \sup_{u\in\mathbb{R}} \left|\mathbb{P}\left(n^{1/2}|\Omega \mathfrak{S}\Omega|_\infty \leq u\right) - \mathbb{P}(|Z^\mathfrak{S}|_\infty \leq u)\right| \leq C' |\Omega|_1^2 \Psi(p,n).
    \end{equation*}
\end{lemma}

The proof of Lemma \ref{lm:GA.cov.linear} is given in Appendix \ref{sec:apdx.GA.cov.linear}.
Next, we shift our attention to building a tail bound for the residual, $\mathfrak{W} + \Omega \mathfrak{S} \Omega$.
The following lemma is a key step.

\begin{lemma}\label{lm:Sigma.op.small}
    Under Condition \ref{cond.A}, the operator norm of $\hat{\Sigma}_n - \Sigma$ satisfies the following tail bound,
    \begin{equation*}
        \mathbb{P}(n^{1/2}|\hat{\Sigma}_n - \Sigma|_{\op} > \delta) \leq 2\exp \left[-c_2 \left(\frac{n^{((2\beta-1) \wedge 1)}\delta^2}{p^2} \wedge \frac{n^{((\beta-1/2) \wedge 1/2)}\delta}{p}\right) + 3p\right],
    \end{equation*}
    where $c_2>0$ is a constant related to $C_0$.
\end{lemma}

Operator norm involves a supremum over the unit Euclidean sphere, which could be difficult to bound in probability.
Generic chaining method helps tackling this challenge \cite{koltchinskii2017concentration}.
In the proof of Lemma \ref{lm:Sigma.op.small} that is deferred to Appendix \ref{sec:apdx.Sigma.op.small}, the compact unit sphere is discretized by $1/4$-nets \cite{vershynin2010introduction}.
The proof is thus finished by union bound and Hanson-Wright Inequality.

An expansion to \eqref{eq:Delta.decomp} gives
\begin{equation*}
    \mathfrak{W} = -\left(\sum_{k=0}^\infty (\Omega \mathfrak{S})^k\right) \Omega \mathfrak{S} \Omega.
\end{equation*}
By triangle inequality, whenever $|\Omega \mathfrak{S}|_{\op} <1/2$, the infinity norm can be upper bounded by operator norm using
\begin{equation*}
    |\mathfrak{W} + \Omega \mathfrak{S}\Omega|_\infty \leq |\mathfrak{W} + \Omega \mathfrak{S}\Omega|_{\op} \leq \frac{1}{1-|\Omega \mathfrak{S}|_{\op}} |\Omega \mathfrak{S}|_{\op}^2 |\Omega|_{\op} \leq 2 |\Omega|_{\op}^3 |\mathfrak{S}|_{\op}^2.
\end{equation*}
Specifically, the operator norm is bounded by $|\Omega|_{\op} \leq \lambda_{\max}(\Omega) \leq c_e^{-1}$.
Lemma \ref{lm:Sigma.op.small} immediately implies the tail bound for the infinity norm of the residual,
\begin{multline}\label{eq:residual.tail.bound}
        \mathbb{P}\left(n^{1/2}|\mathfrak{W} + \Omega \mathfrak{S} \Omega|_\infty >\delta \right) \leq \mathbb{P}\left(n^{1/2}|\mathfrak{S}|_{\op} > (2 c_e^{-3})^{-1/2}n^{1/4}\delta^{1/2} \right)\\
        \leq 2\exp \left[-c_2' \left(\frac{n^{((2\beta-1/2) \wedge 3/2)}\delta}{p^2} \wedge \frac{n^{((\beta-1/4) \wedge 3/4)}\delta^{1/2}}{p}\right) + 3p\right],
\end{multline}
for a constant $c_2'>0$ related to $C_0$ and $c_e$.
Finally, the following lemma bridges the gap between $Z^\mathfrak{S}$, the Gaussian approximation of $n^{1/2}\Omega \mathfrak{S}\Omega$, and $Z^\Omega$, the Gaussian approximation of $n^{1/2}\mathfrak{W}$, in terms of Kolmogorov distance.
\begin{lemma}\label{lm:GA.comparison.prec}
    Under Conditions \ref{cond.A} and \ref{cond.G}, there exists a constant $C''>0$ that is related to $C_0$, $c_0$ and $c_e$, such that
    \begin{equation*}
        \sup_{u\in \mathbb{R}}\left|\mathbb{P}(|Z^\mathfrak{S}|_\infty \geq u) - \mathbb P (|Z^\Omega|_\infty \geq u)\right| \leq C'' \frac{|\Omega|_1 p^{2} \log(p)}{n^{(\beta - 1/4)\wedge 3/4}}.
    \end{equation*}
\end{lemma}
The proof of Lemma \ref{lm:GA.comparison.prec} is given in Appendix \ref{sec:apdx.GA.comparison.prec}.
We are now ready to provide the proof of Theorem \ref{thm:GA.prec}, the main theorem of this section, by combining the results in Lemma \ref{lm:GA.cov.linear}, Lemma \ref{lm:GA.comparison.prec}, and \eqref{eq:residual.tail.bound}.

\begin{proof}[Proof of Theorem \ref{thm:GA.prec}]
    For any $\eta>0$, by triangle inequality,
    \begin{multline}\label{eq:Delta.berry.esseen.bound}
        \sup_{u\in\mathbb{R}} \left|\mathbb{P} \left(n^{1/2}|\mathfrak{W}|_\infty \leq u\right) - \mathbb{P}\left( n^{1/2}|\Omega \mathfrak{S}\Omega|_\infty \leq u\right)\right|\\
        \leq \mathbb{P}\left(n^{1/2}|\mathfrak{W} + \Omega \mathfrak{S}\Omega|_\infty >\eta\right) + \sup_{u\in\mathbb{R}}\mathbb{P}\left(\left|n^{1/2}|\Omega \mathfrak{S}\Omega|_\infty - u\right|<\eta \right).
    \end{multline}
    The first term on the right hand side of \eqref{eq:Delta.berry.esseen.bound} is the tail probability for the residual, which is bounded by \eqref{eq:residual.tail.bound}.
    For the second term, by Lemma \ref{lm:GA.cov.linear},
    \begin{equation*}
        \sup_{u\in\mathbb{R}}\mathbb{P}\left(\left|n^{1/2}|\Omega \mathfrak{S}\Omega|_\infty - u\right|<\eta \right) \lesssim \sup_{u\in\mathbb{R}}\mathbb{P}\left(\left||Z^\mathfrak{S}|_\infty - u\right| < \eta \right) + C|\Omega|_1^2\Psi(p,n).
    \end{equation*}
    By Theorem 3 of \cite{chernozhukov2015comparison},
    \begin{align*}
            \sup_{u\in\mathbb{R}} \mathbb{P}\left(\left||Z^\mathfrak{S}|_\infty - u\right| < \eta \right) \lesssim \eta\sqrt{1\vee \log(p/\eta)}
    \end{align*}
    up to a constant $C>0$ related to $C_0$ and $c_0$.
    Combining these results, we have the upper bound for the desired Kolmogorov distance,
    \begin{multline*}
        \sup_{u\in\mathbb{R}} \left|\mathbb{P}\left(n^{1/2}|\mathfrak{W}|_\infty \leq u\right) - \mathbb{P}\left( |Z^\Omega|_\infty \leq u\right)\right|\\
        \lesssim \exp \left[-c'_2 \left(\frac{n^{((2\beta-1/2) \wedge 3/2)}\eta}{p^2} \wedge \frac{n^{((\beta-1/4) \wedge 3/4)}\eta^{1/2}}{p}\right) + 3p\right] + |\Omega|_1^2\Psi(p,n) + \eta\sqrt{1\vee \log(p/\eta)} + \frac{|\Omega|_1 p^{2}\log (p)}{n^{(\beta - 1/4)\wedge 3/4}},
    \end{multline*}
    up to a constant related to $C_0$, $c_0$, $C_e$ and $c_e$.
    For a constant $C_\eta \geq 1$, we can choose
    \[
    \eta = C_\eta \frac{p^4\log^2(n)}{n^{(2\beta - 1/2)\wedge 3/2}}.
    \]
    Then the first term in the bound above becomes
    \[
    \exp\left[-c'_2 \left(C_\eta \log^2(n)\cdot p^2 \wedge C_\eta^{1/2} \log(n)\cdot p \right) + 3p\right] = \exp\left[-\left( c'_2 C_\eta^{1/2} \log(n) - 3\right) p\right].
    \]
    For sufficiently large $n$, this term can be upper bounded by $n^{-\kappa p}$ for any constant $\kappa < c'_2 C_\eta^{1/2}$.
    As long as $C_\eta$ is chosen to be large enough, it can be absorbed by other terms that are polynomial in $n^{-1}$, even when $p$ is fixed.
    And the third term becomes of order $p^4\log^{5/2}(n) / n^{(2\beta - 1/2)\wedge 3/2}$.
    Then, for instance, if we choose $C_\eta > \left(\frac{(4\beta - 1)\wedge 3}{2c_2'}\right)^2$, the upper bound of the first term can be absorbed by the third term.
    This finishes the proof of the theorem.
\end{proof}

\section{Block Bootstrap for Precision Matrix}\label{sec:bootstrap.prec}
In Section \ref{sec:GA.prec}, we have given a finite-sample upper bound of the Kolmogorov distance between $n^{1/2}|\hat{\Omega}_n - \Omega|_\infty$ and the infinity norm of its Gaussian approximation, $|Z^\Omega|_\infty$.
The true underlying covariance structure of the Gaussian approximation, $Z^\Omega$, is unattainable.
The block bootstrap technique mentioned in Section \ref{sec:bootstrap.cov} is able to give a data-dependent approximation to $\mathbb{P}(|Z^\Omega|_\infty\leq u)$.

Recall that the size of sampling window, $l=l_n$, satisfies $l\rightarrow\infty$ and $l/n\rightarrow 0$.
We have shown in Section \ref{sec:GA.prec} that $\hat{\Omega}_n - \Omega$ can be approximated by a dominating linear term $\Omega (\hat{\Sigma}_n - \Sigma) \Omega$.
In light of this, define
\begin{equation*}
    \check{\mathcal{X}}_j^\Omega = \vectorize(\hat{\Omega}_n (X_j X_j^\top - \hat{\Sigma}_n) \hat{\Omega}_n), \quad 1 \leq j\leq n,
\end{equation*}
and
\begin{equation*}
    \check{B}_{i,l}^\Omega = \sum_{j=i-l+1}^i \check{\mathcal{X}}_j^\Omega, \quad l\leq i\leq n.
\end{equation*}
Note that $\check{B}_{i,l}^\Omega$ is completely data-driven, and requires no information of the true underlying distribution.
In this section, we will show that the distribution of $n^{1/2}|\hat{\Omega}_n - \Omega|_\infty$ can be approximated by the empirical distribution of $l^{-1/2}|\check{B}_{i,l}^\Omega|_\infty$, characterized by
\begin{equation*}
    \hat{F}_{n,l}^\Omega (u) = \frac{1}{n-l+1}\sum_{i=l}^n \mathbf{1}\{l^{-1/2}|\check{B}_{i,l}^\Omega|_\infty \leq u\}.
\end{equation*}
Specifically, we will prove the validity of block bootstrap by giving a data-dependent upper bound to
\begin{equation*}
    \rho^\Omega_B(n,l) = \sup_{u\in\mathbb{R}}\left|\hat{F}_{n,l}^\Omega (u) - \mathbb{P}\left(n^{1/2}|\hat{\Omega}_n - \Omega|_\infty \leq u\right)\right|.
\end{equation*}

As a first step, we shall approximate $\check{\mathcal{X}}_j^\Omega$ by a quantity that is linear with respect to $X_jX_j^\top$: let
\[
\check{B}_{i,l}^{\Omega,\diamond} = \sum_{j=i-l+1}^i \check{\mathcal{X}}_j^{\Omega,\diamond},
\]
where
\begin{equation*}
    \check{\mathcal{X}}_j^{\Omega,\diamond} = \vectorize(\Omega (X_j X_j^\top - \hat{\Sigma}_n) \Omega).
\end{equation*}
By simple manipulation (see Appendix \ref{app:X.check.diamond}), we find that
\begin{equation}\label{eq:B.check.diamond}
    |\check{B}_{i,l}^{\Omega} - \check{B}_{i,l}^{\Omega,\diamond}|_\infty \leq 2p|\Omega|_1 |\check{B}_{i,l}|_\infty |\hat{\Omega}_n - \Omega|_\infty + p^2 |\check{B}_{i,l}|_\infty |\hat{\Omega}_n - \Omega|_\infty^2.
\end{equation}
Recall that $\check{B}_{i,l}$ was defined at the beginning of Section \ref{sec:bootstrap.cov}.
The tail bound of this quantity is given by the following lemma.
\begin{lemma}\label{lm:B.check.diamond}
    For any $l\leq i\leq n$ and $\delta>0$,
    \begin{equation*}
            \mathbb{P}\left(l^{-1/2}|\check{B}_{i,l}^{\Omega} - \check{B}_{i,l}^{\Omega,\diamond}|_\infty > \delta \right)\leq 4C|\Omega|_1^{1/2}\frac{p^{1/2}\log^{1/2}(p)}{n^{1/4}\delta^{1/2}} + C\Psi(p,l) + C' \Psi^\Omega (p,n).
    \end{equation*}
    Here, $C>0$ is a constant related to $C_0$ and $c_0$, while $C'>0$ is a constant related to $C_0$, $c_0$, $C_e$ and $c_e$.
\end{lemma}
The proof of Lemma \ref{lm:B.check.diamond} is deferred to Appendix \ref{sec:apdx.B.check.diamond}.
The next step is $m$-dependent approximation for $\check{B}_{i,l}^{\Omega,\diamond}$.
Recall \eqref{eq:X.j.Omega} for the definition of $\mathcal{X}_j^\Omega$, which depends only on observations in terms of $X_j$.
From the definition of $D_j^\Omega$ and $\tilde{D}_{j,m}^\Omega$ in \eqref{eq:D.j.Omega} and \eqref{eq:D.tilde.j.Omega}, we can observe that $\tilde{D}_{j,m}^\Omega$ and $\tilde{D}_{k,m}^\Omega$ are independent if $|j-k|\geq m$.
Correspondingly, a good $m$-dependent approximation to $\check{B}_{i,l}^{\Omega,\diamond}$ is
\begin{equation*}
    \tilde{B}_{i,l,m}^\Omega = \sum_{j=i-l+1}^i \tilde{D}_{j,m}^\Omega.
\end{equation*}
The validity of this approximation is quantified by the following lemma analogous to Lemma \ref{lm:B.check.B.tilde}.
\begin{lemma}\label{lm:B.check.B.tilde.Omega}
    Under Conditions \ref{cond.A} and \ref{cond.G}, up to a constant $C>0$ related to $C_0$ and $c_0$, for any $\delta>0$,
    \begin{multline*}
    \mathbb{P}\left(l^{-1/2}|\check{B}_{i,l}^{\Omega,\diamond} - \tilde{B}_{i,l,m}^\Omega|_\infty > \delta\right) \leq C|\Omega|_1^2\delta^{-1}\sqrt{\frac{l\log(p)}{n}} + C\Psi(p,n) \\
            + 2p^2\exp\left(-c_1\left(\frac{\delta^2}{|\Omega|_1^4 Q_1(l)} \wedge \frac{\delta}{|\Omega|_1^2\sqrt{Q_1(l)}}\right)\right)
            + 2p^2\exp\left(-c_2\left(\frac{\delta^2}{|\Omega|_1^4 Q_2(m)} \wedge \frac{\delta}{|\Omega|_1^2\sqrt{Q_2(m)}}\right)\right).
    \end{multline*}
    Here, the constants $c_1>0$ and $c_2>0$ are related to $C_0$.
\end{lemma}

\begin{proof}[Proof of Lemma \ref{lm:B.check.B.tilde.Omega}]
    Observe that
    \begin{equation*}
        |\check{B}_{i,l}^{\Omega,\diamond} - \tilde{B}_{i,l,m}^\Omega|_\infty \leq |\Omega|_1^2 |\check{B}_{i,l} - \tilde{B}_{i,l,m}|_\infty.
    \end{equation*}
    Lemma \ref{lm:B.check.B.tilde.Omega} immediately follows Lemma \ref{lm:B.check.B.tilde}, only with extra factors of $|\Omega|_1$.
\end{proof}

Similar to Lemma \ref{lm:B.check.GA}, we have the following Gaussian approximation result for $\tilde{B}_{i,l,m}^\Omega$.
\begin{lemma}\label{lm:B.tilde.GA}
    Under Conditions \ref{cond.A} and \ref{cond.G}, there exists a constant $C''>0$ related to $C_0$, $c_0$ and $c_e$ such that
    \begin{equation*}
        \sup_{u\in\mathbb{R}} \left|\mathbb{P} \left( l^{-1/2} |\Tilde{B}_{i,l,m}^\Omega|_\infty \leq u\right) - \mathbb{P}(|Z^\mathfrak{S}|_\infty \leq u)\right| \leq C''|\Omega|_1^2 \Psi(p,l).
    \end{equation*}
\end{lemma}
\begin{proof}[Proof of Lemma \ref{lm:B.tilde.GA}]
    The proof technique is analogous to that of Lemma \ref{lm:B.check.GA} given in Appendix \ref{sec:apdx.B.check.GA}.
    With $\tilde{D}_{t,m}$ replaced by $\tilde{D}^\Omega_{t,m}$, we can similarly divide the index set, partition $\tilde B_{i,l,m}^\Omega$ into triadic blocks, and define an intermediate partial sum $B_{i,l}^{\Omega Y}$ similar to $B_{i,l}^Y$.
    Similar to \eqref{eq:boot.3.3}, for any $\delta>0$ and $0<\epsilon<1$, up to a constant $C>0$ related to $C_0$ and $c_0$ and $c_2>0$ related to $C_0$, with probability no less than $1-4p^2\exp(-c_2\delta l^{1/2}/m^{1/2}) - 2\exp(-c_2 m^{2\epsilon}|\Omega|_1^{4\epsilon})$, we have
    \[
    \sup_{u\in\mathbb{R}}\left|\mathbb{P}\left(l^{-1/2}|\tilde{B}^\Omega_{i,l,m} - \mathbb{E}[\tilde{B}^\Omega_{i,l,m}\mid \bm{\epsilon}_0^i]|_\infty \leq u\mid \bm{\epsilon}_0^i\right) - \mathbb{P}\left(l^{-1/2}|B_{i,l}^{\Omega Y}|_\infty \leq u\mid \bm{\epsilon}_0^i\right)\right|
        \leq C \left(\frac{m^4|\Omega|_1^8\log^5(pl)}{l}\right)^{1/4},
    \]
    up to a constant $C>0$ related to $C_0$ and $c_0$. The proof technique is same as Lemma \ref{lm:cond.CCK.prec}.
    Next, we can define a partial sum $S_l^{\Omega Z}$ analogous to $S_l^Z$, and show a Gaussian approximation result similar to \eqref{eq:boot.3.4} that
    \[
    \sup_{u\in\mathbb{R}}\left|\mathbb{P}(l^{-1/2}|B_{i,l}^{\Omega Y}|_\infty \leq u) - \mathbb{P}(|S_l^{\Omega Z}|_\infty \leq u)\right| \leq C''\left(\frac{m^4|\Omega|^8\log^5(pl)}{l}\right)^{1/4},
    \]
    up to a constant $C''>0$ related to $C_0$, $c_0$ and $c_e$. The proof technique is same as Lemma \ref{lm:uncond.CCK.prec}.
    Finally, similar to \eqref{eq:boot.3.5}, a direct application of \eqref{eq:like.lm.3.5} implies
    \[
    \sup_{u\in \mathbb{R}}\left|\mathbb{P}(|Z^\mathfrak{S}|_\infty \geq u) - \mathbb P(|S_l^{\Omega Z}|_\infty \geq u)\right| \leq C |\Omega|_1^2 \frac{\log(p)}{m^{\tilde\beta/2}}.
    \]
    All upper bounds contain an extra $|\Omega|_1^2$ factor. Using the same proof technique of Theorem \ref{thm:GA.cov}, we finish the proof of Lemma \ref{lm:B.tilde.GA}.
\end{proof}

The following corollary of Lemmas \ref{lm:B.check.diamond}, \ref{lm:B.check.B.tilde.Omega}, and \ref{lm:B.tilde.GA} gives an upper bound to the Kolmogorov distance between $l^{-1/2}|\check{B}_{i,l}^\Omega|_\infty$ and $l^{-1/2}|\tilde{B}_{i,l,m}^\Omega|_\infty$.

\begin{corollary}\label{cor:B.check.B.tilde.Omega}
    Under Conditions \ref{cond.A} and \ref{cond.G}, there exists a constant $C'>0$ related to $C_0$, $c_0$, $C_e$ and $c_e$ such that
    \begin{equation*}
        \sup_{u\in\mathbb{R}}\left|\mathbb{P}\left( l^{-1/2}|\check{B}_{i,l}^\Omega|_\infty \leq u\right) - \mathbb{P}(l^{-1/2}|\Tilde{B}_{i,l,m}^\Omega|_\infty \leq u)\right| \leq C' \Theta^\Omega(n,p,l),
    \end{equation*}
    where
    \begin{equation}
        \Theta^\Omega(n,p,l) = \frac{|\Omega|_1^2\log^{(5\tilde{\beta}+12)/(4\tilde{\beta}+8)}(pl)}{l^{\tilde{\beta}/(4\tilde{\beta}+8)}}
        + \frac{p^4 \log^{5/2}(n)}{n^{(2\beta-1/2)\wedge 3/2}}
        + \frac{|\Omega|_1p^{2}\log(p)}{n^{(\beta-1/4)\wedge 3/4}}
    \end{equation}
\end{corollary}
The proof of Corollary \ref{cor:B.check.B.tilde.Omega} is given in Appendix \ref{app:cor.B.check.B.tilde.Omega}.
Parallel to the proof in Section \ref{sec:bootstrap.cov}, we can now introduce the empirical CDF with respect to $\tilde{B}_{i,l,m}^\Omega$,
\begin{equation*}
    \tilde{F}_{n,l,m}^\Omega(u) = \frac{1}{n-l+1}\sum_{i=l}^n \mathbf{1}\{l^{-1/2}|\Tilde{B}_{i,l,m}^\Omega|_\infty \leq u\}.
\end{equation*}
The following result parallel to Lemma \ref{lm:F.check.F.tilde} guarantees the proximity between the two empirical CDF's, $\hat{F}_{n,l}^\Omega(u)$ and $\tilde{F}_{n,l,m}^\Omega (u)$.
The former measures the empirical distribution of a data-adaptive sequence, while the latter is focused on an artificial sequence that is easier to control.

\begin{lemma}\label{lm:F.check.F.tilde.Omega}
    Under Conditions \ref{cond.A} and \ref{cond.G}, for any $0<\epsilon<1$, we have
    \begin{equation*}
        \sup_{u\in\mathbb{R}}\left|\hat{F}_{n,l}^\Omega(u) - \Tilde{F}_{n,l,m}^\Omega(u)\right| \leq C' (\Theta^\Omega(n,p,l))^{1-\epsilon},
    \end{equation*}
    with probability no less than $1-(\Theta^\Omega(n,p,l))^\epsilon$.
    Here, $C'>0$ is a constant related to $C_0$, $c_0$, $C_e$ and $c_e$.
\end{lemma}
The proof is same as that of Lemma \ref{lm:F.check.F.tilde} in Appendix \ref{sec:apdx.F.check.F.tilde}, with only the upper bound of \eqref{eq:F.check.F.tilde.bound} replaced by $\Theta^\Omega(n,p,l)$.
Meanwhile, we have the following finite-sample Glivenko-Cantelli type result for $\Tilde{F}_{n,l,m}^\Omega$ that is parallel to Lemma \ref{lm:B.check.CDF}.
\begin{lemma}\label{lm:B.tilde.CDF.Omega}
    For any $\lambda>0$, with probability no less than $1-2e^{-2\lambda^2}$,
    \begin{equation*}
        \sup_{u\in\mathbb{R}}\left|\tilde{F}_{n,l,m}^\Omega(u) - \mathbb{P}\left(l^{-1/2}|\tilde{B}_{i,l,m}^\Omega|_\infty \leq u\right)\right|
        \leq \sqrt{\frac{(n+m)(l+m-1)}{(n-l+1)^2}\left(\lambda^2 + \frac{1}{2}\log(l+m-1)\right)}.
    \end{equation*}
\end{lemma}
The proof of Lemma \ref{lm:B.tilde.CDF.Omega} is same as that of Lemma \ref{lm:B.check.CDF} and is hereby omitted.
We are now ready to provide the proof of Theorem \ref{thm:bootstrap.prec}.

\begin{proof}[Proof of Theorem \ref{thm:bootstrap.prec}]
    By Corollary \ref{cor:B.check.B.tilde.Omega}, Lemma \ref{lm:F.check.F.tilde.Omega}, and Lemma \ref{lm:B.tilde.CDF.Omega}, we have for any $0<\epsilon<1$ and $\lambda>0$ that
    \begin{align}\label{eq:rho.B.Omega.ub}
            \rho_B^\Omega(n,l) \lesssim (\Theta^\Omega(n,p,l))^{1-\epsilon} + \sqrt{\frac{(n+m)(l+m-1)}{(n-l+1)^2}\left(\lambda^2 + \frac{1}{2}\log(l+m-1)\right)}
    \end{align}
    holds with probability no less than
    \begin{equation*}
        1-(\Theta^\Omega(n,p,l))^\epsilon - 2\exp(-2\lambda^2).
    \end{equation*}
    The choice of $m$ in the proof of Corollary \ref{cor:B.check.B.tilde.Omega} is of order $O\left( (l\log(pl))^{1/(2\tilde{\beta}+4)}\right)$, same as in \eqref{eq:choice.eta.m.boot}.
    We shall choose the order of $l$ with respect to $n$ and $p$ to minimize the right hand side of \eqref{eq:rho.B.Omega.ub}, such that $m\ll l$ and $l\ll n$.
    Under these conditions, the second term on the right hand side of \eqref{eq:rho.B.Omega.ub} has order $O((\lambda^2+\log(l))^{1/2}l^{1/2}/n^{1/2})$. Let $\lambda^2 = \epsilon \log(l)/2$, and assume for some constants $\phi>0$ and $\psi>0$, $l=O(n^\phi \log^\psi(p))$. Then with probability no less than $1-3(\Theta^\Omega(n,p,l))^\epsilon$,
    \begin{equation}\label{eq:rho.B.ub.Omega.new}
        \rho_B^\Omega(n,l) \lesssim \frac{|\Omega|_1^{2(1-\epsilon)}\log^{\frac{5\tilde{\beta} + 12}{4\tilde{\beta}+8}(1-\epsilon) - \frac{\tilde{\beta}}{4\tilde{\beta}+8}(1-\epsilon)\psi}(pn)}{n^{\frac{\tilde{\beta}}{4\tilde{\beta}+8}(1-\epsilon)\phi}}
        + \frac{p^{4(1-\epsilon)}\log^{5(1-\epsilon)/2}(n)}{n^{\frac{(4\beta-1)\wedge 3}{2}(1-\epsilon)}}
        + \frac{|\Omega|_1^{1-\epsilon} p^{2(1-\epsilon)}\log^{1-\epsilon}(p)}{n^{\frac{(4\beta-1)\wedge 3}{4}(1-\epsilon)}}
        + \frac{\log^{\frac{\psi}{2}}(pn)}{n^{\frac{1}{2} - \frac{\phi}{2}}}.
    \end{equation}
    We shall choose the value of $\phi$ and $\psi$ such that the right hand side of \eqref{eq:rho.B.ub.Omega.new} is minimized.
    Note that the second and third terms do not contain either $\phi$ or $\psi$, so it boils down to minimizing the sum of first and last terms, a problem we have already solved in the proof of Theorem \ref{thm:bootstrap.cov}, i.e.,
    \begin{equation*}
        \phi = \frac{2\tilde{\beta}+4}{(3-\epsilon)\tilde{\beta} + 4}, \quad \psi = \frac{5\tilde{\beta} + 12}{(3-\epsilon)\tilde{\beta} + 4}(1-\epsilon).
    \end{equation*}
    \eqref{eq:rho.B.ub.Omega.new} therefore becomes
    \begin{equation*}
        \rho_B^\Omega(n,l) \lesssim \left(\frac{|\Omega|_1^2\log^{\frac{5\tilde{\beta}+12}{(6-2\epsilon)\tilde{\beta}+8}}(pn)}{n^{\frac{\tilde{\beta}}{(6-2\epsilon)\tilde{\beta}+8}}}
        + \frac{p^4\log^{5/2}(n)}{n^{\frac{(4\beta-1)\wedge 3}{2}}}
        + \frac{|\Omega|_1p^{2}\log(p)}{n^{\frac{(4\beta-1)\wedge 3}{4}}}\right)^{1-\epsilon}
        = (\Psi_B^\Omega(p,n,\epsilon))^{1-\epsilon}
    \end{equation*}
    with probability no less than $1-3(\Psi_B^\Omega(p,n,\epsilon))^{\epsilon}$.
    This finishes the proof of Theorem \ref{thm:bootstrap.prec}.
\end{proof}

\section{Simulation}\label{sec:simulation}

A simulation study is carried out to verify the validity of our main theorems.
We shall generate several copies of Gaussian linear process $\bm{X}_n = (X_1,\cdots,X_n)$ as in \eqref{LinearProcess}, and study the empirical distribution of $n^{1/2}|\hat{\Sigma}_n - \Sigma|_\infty$, as well as $n^{1/2}|\hat{\Omega}_n - \Omega|_\infty$ for low-dimensional processes.

Although the simulation itself is straightforward, the data generating process may be extremely time-consuming.
Any single $X_i$ in \eqref{LinearProcess} involves an infinite sum, which requires an approximation by truncation, i.e., $X_i^* = \sum_{t=0}^{N-1} A_t \epsilon_{i-t}$.
Here, $N$ is the length of truncation, usually a large integer.
For a long-memory process, specifically, since the coefficient $A_t$ decays slowly with $t$, a large $N$ is necessary to maintain the accuracy of simulation.
In practice, a typical choice is $N=n^2$.
Moreover, to study the empirical distribution of the sample covariance and precision matrices, we need to generate multiple copies of processes.
These challenges make it infeasible to follow the conventional scheme by generating i.i.d.~standard Gaussian random vectors $\epsilon_t$ and summing them up one by one.
In this paper, we use a Fast-Fourier-Transform-based simulation technique for time series proposed by \cite{wood1994simulation,wu2004simulating}.
This algorithm is able to generate $N/n$ many copies of $\{X_i\}_{i=1}^n$ simultaneously.
We denote $K_{N,n} = N/n$.
The detailed data-generating algorithm is presented in Appendix \ref{sec:apdx.simulation}.

\subsection{Low-dimensional Regime}

We first consider the low-dimensional case when $p<n$.
The coefficients $A_t$ are set to be square Toeplitz matrices.
That is, $d=p$, and each entry of $A_t$ is set as
\[
(A_t)_{jk} = (t+1)^{-\beta}(|j-k|+1)^{-2}
\]
for all $1\leq j,k\leq p$.
We simulate three types of processes, distinguished by their temporal dependence.
The first type is short-memory with $\beta=2$.
The second type is long-memory with $\beta=0.9$.
As a comparison, we also simulate a third type with $\beta=0.55$, where the temporal dependence is so strong that Condition \ref{cond.A} is violated.
In general, we expect the covariance and precision to exhibit Gaussianity for the first two types of processes.
When $\beta=0.55$, the asymptotic distribution of covariance and precision is no longer Gaussian.
In this case, we expect the distribution to be skewed with a heavy right tail.

The simulation is done on a grid of $n\in\{200, 500, 1000, 2000, 4000, 10000\}$ and $p\in \{2,10,30,100\}$.
For each setup, we set $N = n^2$, and use 200 out of all $K_{N,n}$ realizations of $\{X_i\}_{i=1}^n$.
From each copy, we can compute one realization of the sample covariance matrix $\hat\Sigma_n$ and the sample precision matrix $\hat\Omega_n = \hat\Sigma_n^{-1}$.
Therefore, we obtain an empirical distribution of $n^{1/2}|\hat\Sigma_n - \Sigma|_\infty$, as well as $n^{1/2}|\hat\Omega_n - \Omega|_\infty$, consisting of 200 points.

To demonstrate the Gaussian approximation results in Theorems \ref{thm:GA.cov} and \ref{thm:GA.prec}, we shall generate their respective Gaussian approximations.
$Z$ is the Gaussian approximation of $n^{1/2}|\hat\Sigma_n - \Sigma|_\infty$.
By the discussion in Appendix \ref{sec:apdx.GA.comparison}, the covariance matrix $\Sigma_Z$ of $Z$ has a finite-sample closed-form expression \eqref{eq:covariance.of.Z}.
We can therefore generate 200 copies of $Z$ from $\Sigma_Z$, and obtain an empirical distribution of $|Z|_\infty$.
The Gaussian approximation for $n^{1/2}|\hat\Omega_n - \Omega|_\infty$, however, is tricky.
While it is hard to express covariance structure of the Gaussian approximation $Z^\Omega$ in closed form, we can use the intermediate Gaussian approximation, $Z^\mathfrak{S}$.
Following the discussion in Appendix \ref{sec:apdx.GA.cov.linear} that leads to \eqref{eq:like.lm.3.5}, the covariance matrix of $Z^\mathfrak{S}$ has a similar form as that of $Z$ mentioned above, as if the coefficients $A_t$ were replaced by $\Omega A_t$, or alternatively, as if the autocovariances $\Gamma_k$ were replaced by $\Omega \Gamma_k \Omega$.
A careful check of Theorem \ref{thm:GA.prec} and Lemma \ref{lm:GA.comparison.prec} suggests that the Kolmogorov distance between $n^{1/2}|\hat\Omega_n - \Omega|_\infty$ and $|Z^\mathfrak{S}|_\infty$ attains a similar rate as $\Psi^\Omega(p,n)$ with the same first term and a second term that contains a polynomial factor of $p$.
We can then generate 200 copies of $Z^\mathfrak{S}$, and obtain the empirical distribution of $|Z^\mathfrak{S}|_\infty$.
By comparing the empirical distribution of $n^{1/2}|\hat\Sigma_n - \Sigma|_\infty$ versus $|Z|_\infty$, and $n^{1/2}|\hat\Omega_n - \Omega|_\infty$ versus $|Z^\mathfrak{S}|_\infty$, we can assess the proximity of covariance and precision error to their respective Gaussian approximation.

Meanwhile, to demonstrate the validity of the block bootstrap method in Theorems \ref{thm:bootstrap.cov} and \ref{thm:bootstrap.prec}, we would like to draw empirical distributions of $l^{-1/2}|\check{B}_{i,l}|_\infty$ and $l^{-1/2}|\check{B}_{i,l}^\Omega|_\infty$.
To do so, for each copy of data $\{X_i\}_{i=1}^n$, we randomly choose a $l\leq i\leq n$, and compute $\check{B}_{i,l}$ and $\check{B}_{i,l}^\Omega$.
Recall that theoretically optimal orders of $l$ with respect to $n$ and $p$ are given in Theorems \ref{thm:bootstrap.cov} and \ref{thm:bootstrap.prec}.
However, a specific value of $l$ is not directly available from these theoretical results.
In practice, we take the simple nonadaptive choice of $l = n^{2/3}$ through all setups.
This is close to the order specified in Theorems \ref{thm:bootstrap.cov} and \ref{thm:bootstrap.prec}.
By comparing the empirical distribution of 200 instances of $l^{-1/2}|\check{B}_{i,l}|_\infty$ versus $n^{1/2}|\hat\Sigma_n - \Sigma|_\infty$, and $l^{-1/2}|\check{B}_{i,l}^\Omega|_\infty$ versus $n^{1/2}|\hat\Omega_n - \Omega|_\infty$, we can assess how close the bootstrap distribution approaches the distribution of covariance and precision error.

The covariance QQ-plots can be found in Figure \ref{fig:E2.QQ.cov.2} for short-memory processes with $\beta=2$, Figure \ref{fig:E2.QQ.cov.0.9} for long-memory processes with $\beta=0.9$, and \ref{fig:E2.QQ.cov.0.55} for ultra-long-memory processes with $\beta=0.55$.
For brevity, we have only chosen the grid over $n\in\{200, 2000, 10000\}$ and $p\in \{2, 10, 100\}$.
For each pair of $(n,p)$, we draw a QQ-plot for Gaussian approximation and a QQ-plot for block bootstrap.
If the points on a QQ-plot being close to the $y=x$ line, then the two distributions almost match each other.
As a supplemental view to QQ-plots, the empirical CDF plots are given in Appendix \ref{sec:apdx.supplementary.sim.results}.

For a more comprehensive view, we quantify the proximity of Gaussian approximation and block bootstrap for covariance matrix in Table \ref{tab:KS.cov}.
It provides exactly $\rho(n)$ and $\rho_B(n,l)$, the Kolmogorov distances of interest in Theorems \ref{thm:GA.cov} and \ref{thm:bootstrap.cov}.
As a supplemental view, we also show the corresponding Wasserstein-1 distances in Table \ref{tab:W1.cov} in Appendix \ref{sec:apdx.supplementary.sim.results}.
A Kolmogorov distance or Wasserstein-1 distance close to zero represents a valid approximation, while a Kolmogorov distance close to one or a large Wasserstein-1 value represents a mismatch, as expected in cases like $\beta=0.55$, for example.

\begin{figure}[t]
    \centering
    
    \includegraphics[width=0.8\linewidth]{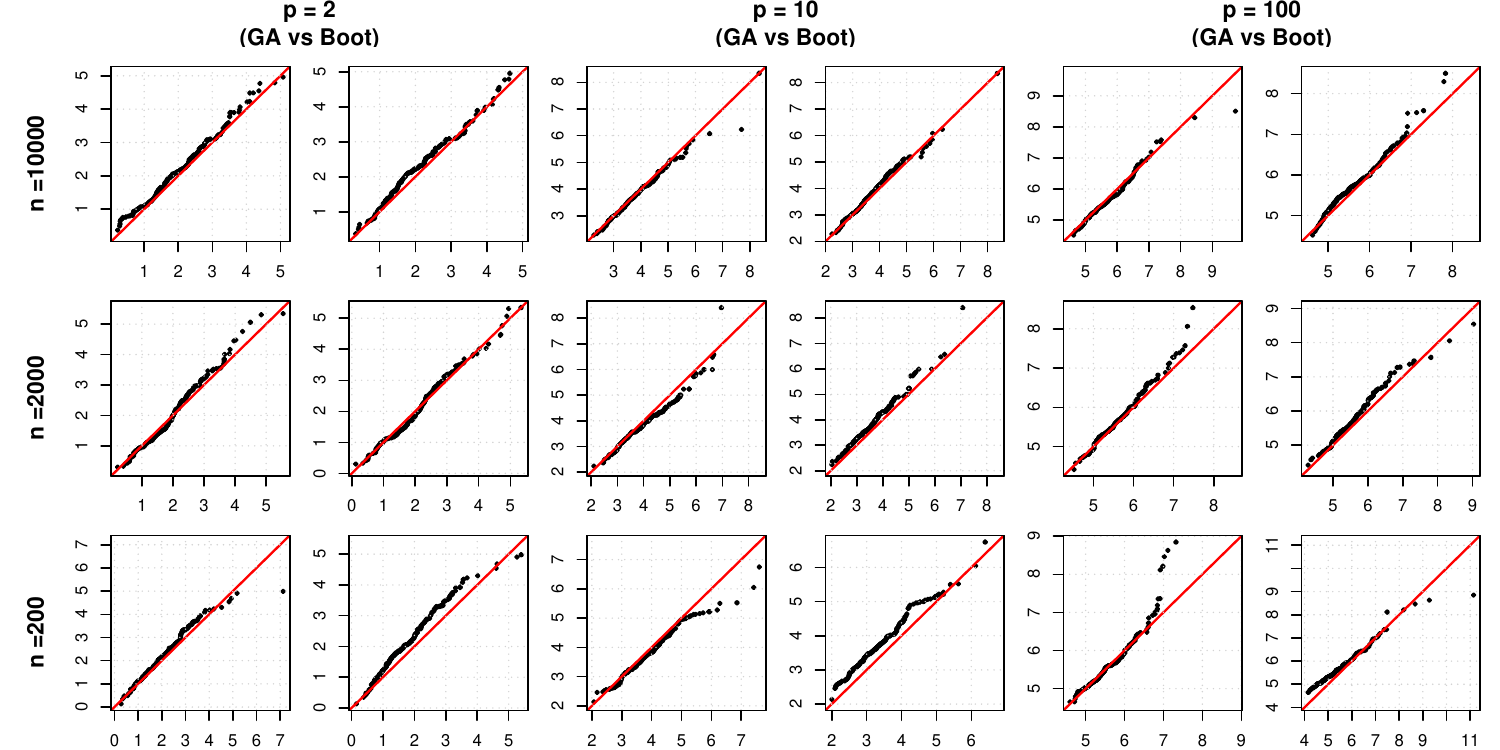}
    \caption{QQ-plots for covariance matrix in the low-dimensional regime, $\beta=2$ (short memory). Red line represents $y=x$. For each pair of $p$ and $n$ values, the plot on the left compares $n^{1/2}|\hat\Sigma_n - \Sigma|_\infty$ on $y$-axis and $|Z|_\infty$ on $x$-axis.
    Being closer to the $y=x$ line represents a smaller Kolmogorov distance $\rho(n)$ for Gaussian approximation.
    The plot on the right compares $n^{1/2}|\hat\Sigma_n - \Sigma|_\infty$ on $y$-axis and $l^{-1/2}|\check{B}_{i,l}|_\infty$ on $x$-axis.
    Being closer to the $y=x$ line represents a smaller Kolmogorov distance $\rho_B(n,l)$ for block bootstrap.}
    \label{fig:E2.QQ.cov.2}
    
    \vspace{1cm} 
    
    \includegraphics[width=0.8\linewidth]{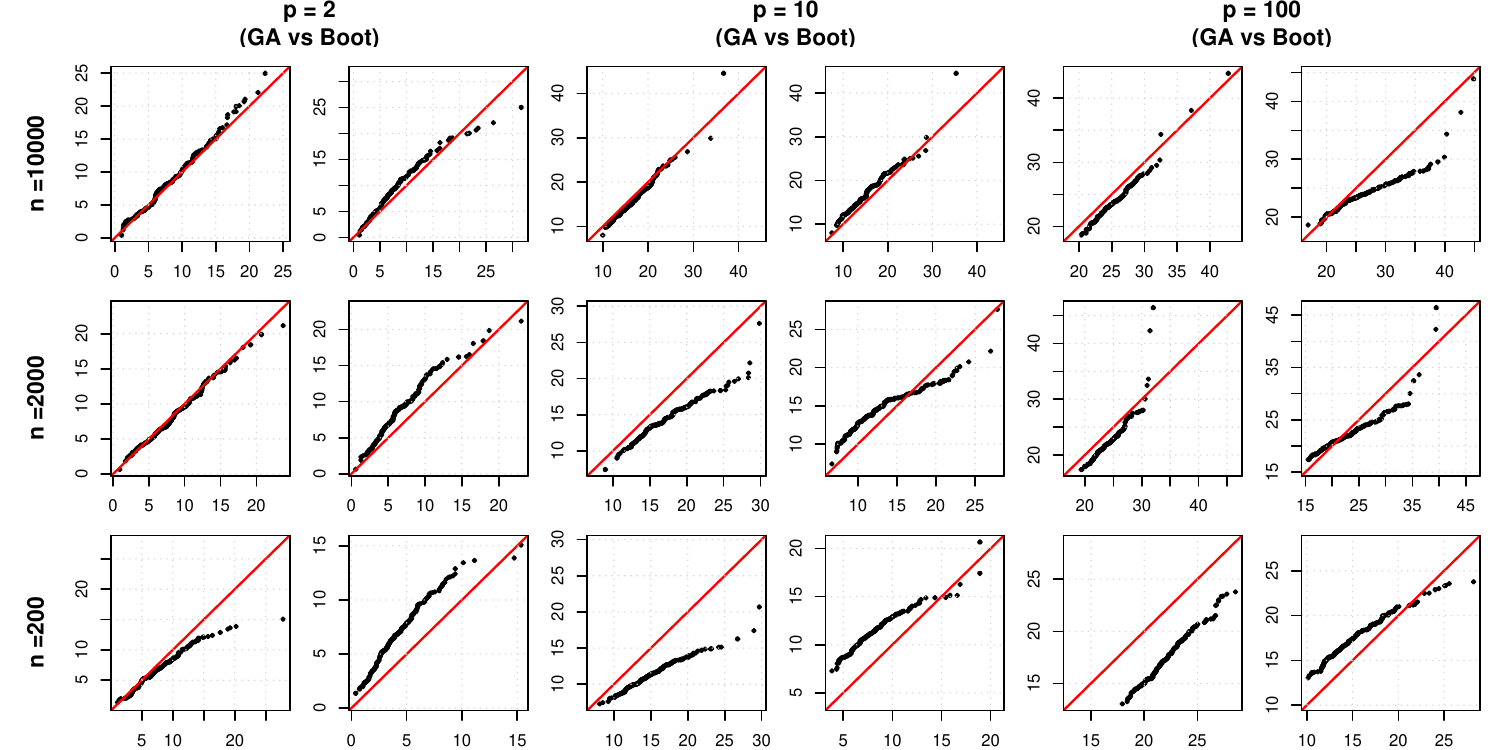}
    \caption{QQ-plots for covariance matrix in the low-dimensional regime, $\beta=0.9$ (long-memory). Red line represents $y=x$. For each pair of $p$ and $n$ values, the plot on the left compares $n^{1/2}|\hat\Sigma_n - \Sigma|_\infty$ on $y$-axis and $|Z|_\infty$ on $x$-axis.
    Being closer to the $y=x$ line represents a smaller Kolmogorov distance $\rho(n)$ for Gaussian approximation.
    The plot on the right compares $n^{1/2}|\hat\Sigma_n - \Sigma|_\infty$ on $y$-axis and $l^{-1/2}|\check{B}_{i,l}|_\infty$ on $x$-axis.
    Being closer to the $y=x$ line represents a smaller Kolmogorov distance $\rho_B(n,l)$ for block bootstrap.}
    \label{fig:E2.QQ.cov.0.9}
    
\end{figure}

\begin{figure}[t]
    \centering
    
    \includegraphics[width=0.8\linewidth]{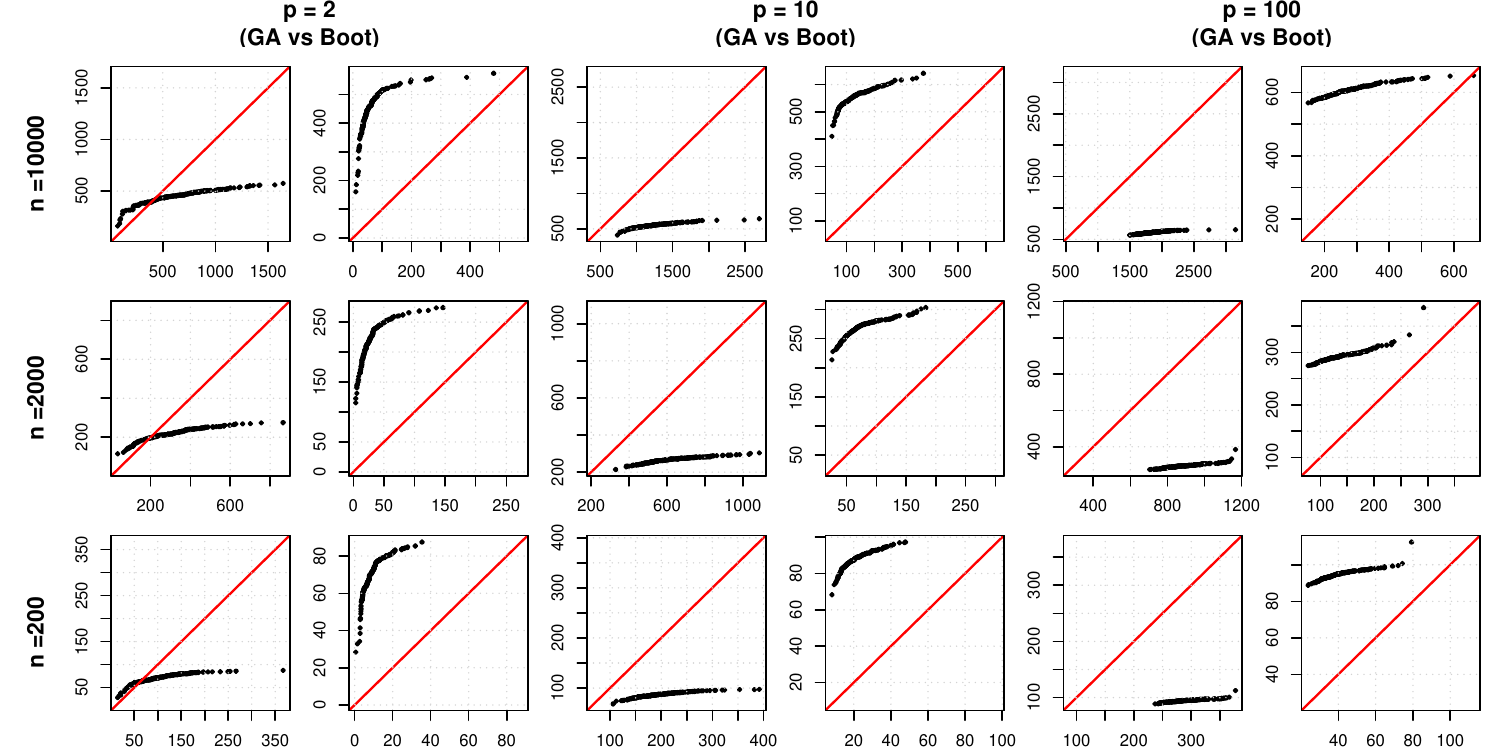}
    \caption{QQ-plots for covariance matrix in the low-dimensional regime, $\beta=0.55$ (ultra-long-memory). Red line represents $y=x$. For each pair of $p$ and $n$ values, the plot on the left compares $n^{1/2}|\hat\Sigma_n - \Sigma|_\infty$ on $y$-axis and $|Z|_\infty$ on $x$-axis.
    Being closer to the $y=x$ line represents a smaller Kolmogorov distance $\rho(n)$ for Gaussian approximation.
    The plot on the right compares $n^{1/2}|\hat\Sigma_n - \Sigma|_\infty$ on $y$-axis and $l^{-1/2}|\check{B}_{i,l}|_\infty$ on $x$-axis.
    Being closer to the $y=x$ line represents a smaller Kolmogorov distance $\rho_B(n,l)$ for block bootstrap.}
    \label{fig:E2.QQ.cov.0.55}
    
    

\end{figure}

\begin{table}[t]
\centering

\resizebox{\textwidth}{!}{
\begin{tabular}{llcccccccccccc}
\toprule
 & & \multicolumn{4}{c}{$\beta=2$} & \multicolumn{4}{c}{$\beta=0.9$} & \multicolumn{4}{c}{$\beta=0.55$} \\ 
 \cmidrule(lr){3-6} \cmidrule(lr){7-10} \cmidrule(lr){11-14}
& $n$ & $p=2$ & $p=10$ & $p=30$ & $p=100$ & $p=2$ & $p=10$ & $p=30$ & $p=100$ & $p=2$ & $p=10$ & $p=30$ & $p=100$ \\
\midrule
\multirow{6}{*}{GA} & 10000 & 0.07 & 0.06 & 0.10 & 0.09 & 0.06 & 0.17 & 0.18 & 0.28 & 0.54 & 1.00 & 1.00 & 1.00 \\
 & 4000 & 0.07 & 0.11 & 0.07 & 0.11 & 0.12 & 0.29 & 0.26 & 0.34 & 0.61 & 1.00 & 1.00 & 1.00 \\
 & 2000 & 0.09 & 0.11 & 0.09 & 0.07 & 0.07 & 0.33 & 0.44 & 0.34 & 0.60 & 1.00 & 1.00 & 1.00 \\
 & 1000 & 0.06 & 0.09 & 0.17 & 0.11 & 0.12 & 0.29 & 0.45 & 0.58 & 0.62 & 1.00 & 1.00 & 1.00 \\
 & 500 & 0.07 & 0.09 & 0.10 & 0.08 & 0.14 & 0.39 & 0.58 & 0.60 & 0.64 & 1.00 & 1.00 & 1.00 \\
 & 200 & 0.09 & 0.09 & 0.09 & 0.07 & 0.12 & 0.57 & 0.69 & 0.72 & 0.59 & 1.00 & 1.00 & 1.00 \\
\midrule
\multirow{6}{*}{Bootstrap} & 10000 & 0.12 & 0.07 & 0.07 & 0.16 & 0.18 & 0.22 & 0.20 & 0.28 & 0.97 & 1.00 & 1.00 & 0.99 \\
 & 4000 & 0.10 & 0.11 & 0.11 & 0.12 & 0.23 & 0.28 & 0.26 & 0.30 & 0.99 & 1.00 & 0.98 & 0.98 \\
 & 2000 & 0.07 & 0.15 & 0.13 & 0.15 & 0.26 & 0.34 & 0.20 & 0.19 & 0.99 & 1.00 & 0.99 & 1.00 \\
 & 1000 & 0.12 & 0.18 & 0.17 & 0.12 & 0.30 & 0.45 & 0.34 & 0.12 & 0.98 & 1.00 & 0.99 & 0.98 \\
 & 500 & 0.13 & 0.20 & 0.22 & 0.29 & 0.37 & 0.53 & 0.34 & 0.24 & 0.98 & 1.00 & 1.00 & 0.99 \\
 & 200 & 0.16 & 0.24 & 0.29 & 0.23 & 0.41 & 0.61 & 0.51 & 0.41 & 0.99 & 1.00 & 1.00 & 1.00 \\
\bottomrule
\end{tabular}
}
\caption{Kolmogorov distances for covariance matrix in the low-dimensional regime. The GA section compares the covariance error $n^{1/2}|\hat\Sigma_n - \Sigma|_\infty$ against its Gaussian approximation $|Z|_\infty$, while the Bootstrap section compares the covariance error against its block bootstrap approximation $l^{-1/2}|\check{B}_{i,l}|_\infty$. A value closer to zero represents proximity of Gaussian approximation or bootstrap.}
\label{tab:KS.cov}



\end{table}

From Figure \ref{fig:E2.QQ.cov.2}, i.e., the short-memory case $\beta=2$, we observe that over all choices of $p$ and $n$, both the distribution of the Gaussian approximation and the block bootstrap distribution match the distribution of the covariance error well.
The match becomes closer as $n$ increases or $p$ decreases, and is relatively robust for a small sample size like $n=200$ or a large dimension $p=100$.
In these cases, some discrepancies in the tail distribution appear, but the overall trend of the QQ-plots is still considerably close to $y=x$.
When the temporal dependence increases, like the long-memory case $\beta=0.9$ in Figure \ref{fig:E2.QQ.cov.0.9}, the proximity of both Gaussian approximation and block bootstrap becomes relatively weaker.
For large $p$ and small $n$, the distribution of Gaussian approximation and the bootstrap distribution start to drift away from the distribution of covariance error.
The bootstrap distribution appears to drift away more severely, especially on the tail.
This is an expected phenomenon, considering that the rates in Theorems \ref{thm:GA.cov} and \ref{thm:bootstrap.cov} increase with $p$ and decrease with $n$, and the rate of Gaussian approximation is faster than the rate of block bootstrap.
Despite the drift in distribution, it is still evident that almost all covariance QQ-plots in Figures \ref{fig:E2.QQ.cov.2} and \ref{fig:E2.QQ.cov.0.9} are reasonably close to $y=x$, and the distribution drift emerges from slow convergence.
We can argue that the limiting distribution of the covariance error is Gaussian.
However, this argument is not true in the ultra-long-memory case of $\beta = 0.55$.
Recall that the limiting distribution is no longer Gaussian as $\beta\leq 3/4$.
In this ultra-long-memory regime, we expect both Gaussian approximation and block bootstrap to fail.
This is confirmed by Figure \ref{fig:E2.QQ.cov.0.55}. All QQ-plots are significantly far from the $y=x$ line, even despite a large $n$ and a small $p$.

Meanwhile, the precision QQ-plots can be found in Figure \ref{fig:E2.QQ.prec.2} for short-memory processes with $\beta=2$, Figure \ref{fig:E2.QQ.prec.0.9} for long-memory processes with $\beta=0.9$, and \ref{fig:E2.QQ.prec.0.55} for ultra-long-memory processes with $\beta=0.55$.
The choice of grid is still $n\in\{200, 2000, 10000\}$ and $p\in \{2, 10, 100\}$.
In this setup, the proximity to the $y=x$ line represents how small the Kolmogorov distances $\rho^\Omega(n)$ and $\rho^\Omega_B(n,l)$ are.
We also compare the empirical CDF of these quantities of interest in Appendix \ref{sec:apdx.supplementary.sim.results}.

For a more comprehensive view, we quantify the proximity of Gaussian approximation and block bootstrap for precision matrix in Table \ref{tab:KS.prec}.
It provides the Kolmogorov distance between $n^{1/2}|\hat\Omega_n - \Omega|_\infty$ and its intermediate Gaussian approximation $|Z^\mathfrak{S}|_\infty$, a quantity that imitates the rate of $\rho^\Omega(n)$ in Theorem \ref{thm:GA.prec}.
The Kolmogorov distance of interest in Theorem \ref{thm:bootstrap.prec}, $\rho_B^\Omega(n,l)$, is also provided.
As a supplemental view, we also show the corresponding Wasserstein-1 distances in Table \ref{tab:W1.prec}, see Appendix \ref{sec:apdx.supplementary.sim.results}.

\begin{figure}[t]
    \centering
    \includegraphics[width=0.8\linewidth]{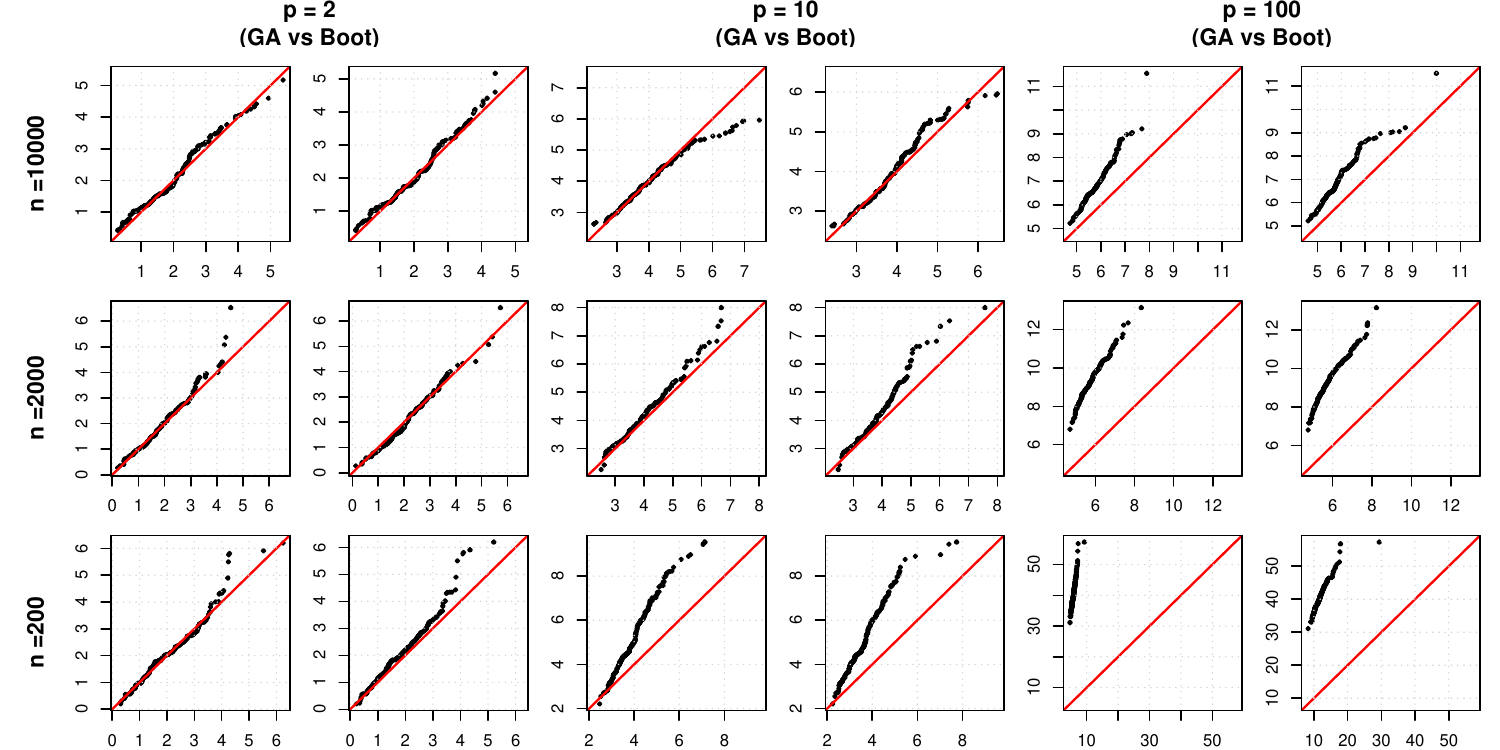}
    \caption{QQ-plots for precision matrix in the low-dimensional regime, $\beta=2$ (short memory). Red line represents $y=x$. For each pair of $p$ and $n$ values, the plot on the left compares $n^{1/2}|\hat\Omega_n - \Omega|_\infty$ on $y$-axis and $|Z^{\mathfrak{S}}|_\infty$ on $x$-axis.
    Being closer to the $y=x$ line represents a smaller Kolmogorov distance $\rho^\Omega(n)$ for Gaussian approximation.
    The plot on the right compares $n^{1/2}|\hat\Omega_n - \Omega|_\infty$ on $y$-axis and $l^{-1/2}|\check{B}^\Omega_{i,l}|_\infty$ on $x$-axis.
    Being closer to the $y=x$ line represents a smaller Kolmogorov distance $\rho_B^\Omega(n,l)$ for block bootstrap.}
    \label{fig:E2.QQ.prec.2}
\end{figure}

\begin{figure}[t]
    \centering
    \includegraphics[width=0.8\linewidth]{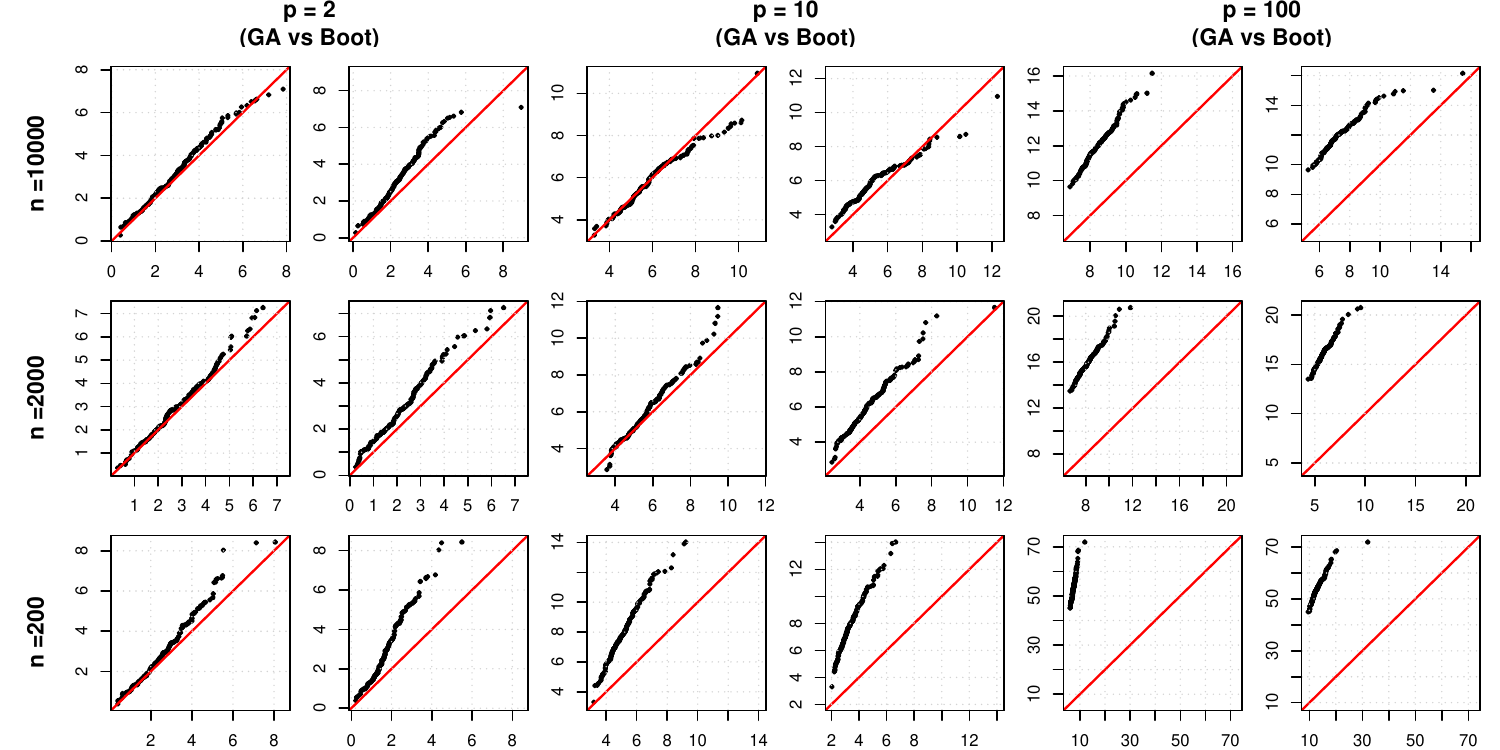}
    \caption{QQ-plots for precision matrix in the low-dimensional regime, $\beta=0.9$ (long-memory). Red line represents $y=x$. For each pair of $p$ and $n$ values, the plot on the left compares $n^{1/2}|\hat\Omega_n - \Omega|_\infty$ on $y$-axis and $|Z^{\mathfrak{S}}|_\infty$ on $x$-axis.
    Being closer to the $y=x$ line represents a smaller Kolmogorov distance $\rho^\Omega(n)$ for Gaussian approximation.
    The plot on the right compares $n^{1/2}|\hat\Omega_n - \Omega|_\infty$ on $y$-axis and $l^{-1/2}|\check{B}^\Omega_{i,l}|_\infty$ on $x$-axis.
    Being closer to the $y=x$ line represents a smaller Kolmogorov distance $\rho_B^\Omega(n,l)$ for block bootstrap.}
    \label{fig:E2.QQ.prec.0.9}
\end{figure}

\begin{figure}[t]
    \centering
    \includegraphics[width=0.8\linewidth]{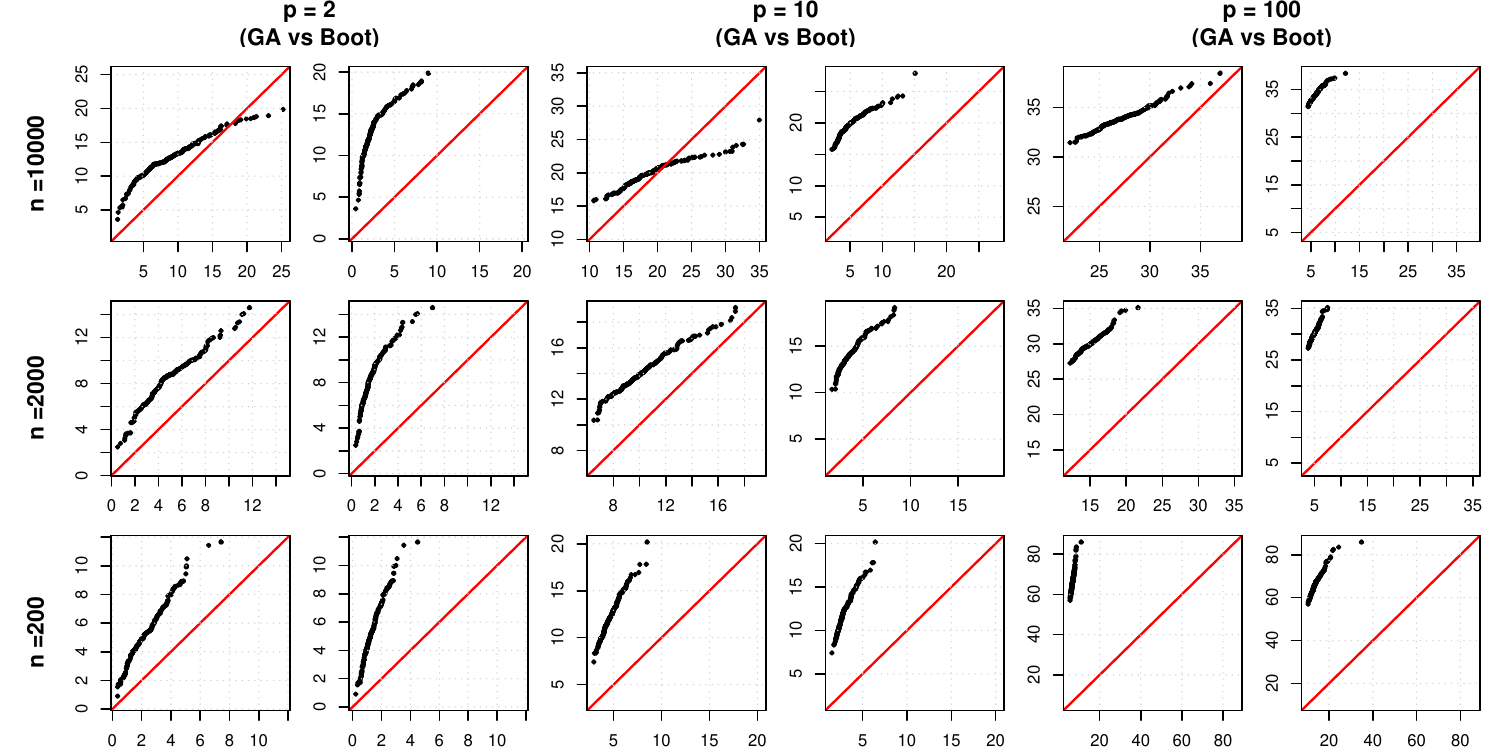}
    \caption{QQ-plots for precision matrix in the low-dimensional regime, $\beta=0.55$ (ultra-long-memory). Red line represents $y=x$. For each pair of $p$ and $n$ values, the plot on the left compares $n^{1/2}|\hat\Omega_n - \Omega|_\infty$ on $y$-axis and $|Z^{\mathfrak{S}}|_\infty$ on $x$-axis.
    Being closer to the $y=x$ line represents a smaller Kolmogorov distance $\rho^\Omega(n)$ for Gaussian approximation.
    The plot on the right compares $n^{1/2}|\hat\Omega_n - \Omega|_\infty$ on $y$-axis and $l^{-1/2}|\check{B}^\Omega_{i,l}|_\infty$ on $x$-axis.
    Being closer to the $y=x$ line represents a smaller Kolmogorov distance $\rho_B^\Omega(n,l)$ for block bootstrap.}
    \label{fig:E2.QQ.prec.0.55}
\end{figure}

\begin{table}[t]
\centering
\resizebox{\textwidth}{!}{
\begin{tabular}{llcccccccccccc}
\toprule
 & & \multicolumn{4}{c}{$\beta=2$} & \multicolumn{4}{c}{$\beta=0.9$} & \multicolumn{4}{c}{$\beta=0.55$} \\
 \cmidrule(lr){3-6} \cmidrule(lr){7-10} \cmidrule(lr){11-14}
& $n$ & $p=2$ & $p=10$ & $p=30$ & $p=100$ & $p=2$ & $p=10$ & $p=30$ & $p=100$ & $p=2$ & $p=10$ & $p=30$ & $p=100$ \\
\midrule
\multirow{6}{*}{GA} & 10000 & 0.09 & 0.05 & 0.08 & 0.57 & 0.09 & 0.09 & 0.17 & 0.96 & 0.47 & 0.35 & 0.42 & 0.95 \\
 & 4000 & 0.09 & 0.07 & 0.16 & 0.85 & 0.12 & 0.10 & 0.46 & 0.98 & 0.47 & 0.58 & 0.84 & 0.99 \\
 & 2000 & 0.06 & 0.12 & 0.33 & 0.97 & 0.09 & 0.12 & 0.67 & 1.00 & 0.56 & 0.73 & 0.96 & 1.00 \\
 & 1000 & 0.10 & 0.17 & 0.55 & 1.00 & 0.10 & 0.40 & 0.86 & 1.00 & 0.60 & 0.88 & 1.00 & 1.00 \\
 & 500 & 0.07 & 0.25 & 0.68 & 1.00 & 0.10 & 0.45 & 0.93 & 1.00 & 0.61 & 0.97 & 1.00 & 1.00 \\
 & 200 & 0.06 & 0.41 & 0.99 & 1.00 & 0.12 & 0.68 & 1.00 & 1.00 & 0.69 & 0.99 & 1.00 & 1.00 \\
\midrule
\multirow{6}{*}{Bootstrap} & 10000 & 0.09 & 0.11 & 0.17 & 0.61 & 0.21 & 0.28 & 0.36 & 0.96 & 0.94 & 1.00 & 1.00 & 1.00 \\
 & 4000 & 0.14 & 0.17 & 0.27 & 0.80 & 0.17 & 0.47 & 0.64 & 1.00 & 0.96 & 1.00 & 1.00 & 1.00 \\
 & 2000 & 0.10 & 0.16 & 0.40 & 0.96 & 0.26 & 0.51 & 0.86 & 1.00 & 0.94 & 1.00 & 1.00 & 1.00 \\
 & 1000 & 0.12 & 0.17 & 0.54 & 0.99 & 0.38 & 0.68 & 0.98 & 1.00 & 0.93 & 1.00 & 1.00 & 1.00 \\
 & 500 & 0.16 & 0.34 & 0.74 & 1.00 & 0.38 & 0.84 & 1.00 & 1.00 & 0.93 & 1.00 & 1.00 & 1.00 \\
 & 200 & 0.11 & 0.52 & 0.91 & 1.00 & 0.36 & 0.89 & 1.00 & 1.00 & 0.90 & 1.00 & 1.00 & 1.00 \\
\bottomrule
\end{tabular}
}
\caption{Kolmogorov distances for precision matrix in the low-dimensional regime. The GA section compares the precision error $n^{1/2}|\hat\Omega_n - \Omega|_\infty$ against its (intermediate) Gaussian approximation $|Z^\mathfrak{S}|_\infty$, while the Bootstrap section compares the precision error against its block bootstrap approximation $l^{-1/2}|\check{B}^\Omega_{i,l}|_\infty$. A value closer to zero represents proximity of Gaussian approximation or bootstrap.}
\label{tab:KS.prec}
\end{table}

The Gaussian approximation and block bootstrap of precision matrix exhibit larger Kolmogorov distances as $p$ increases, compared to those of covariance matrix.
For a large $p$ like $100$, the drift of distributions become more severe, for both short-memory case in Figure \ref{fig:E2.QQ.prec.2} and long-memory case in Figure \ref{fig:E2.QQ.prec.0.9}.
This is because of the extra polynomial factor of $p$ in the bounds of Theorem \ref{thm:GA.prec} (i.e., second term of $\Psi^\Omega(p,n)$) and Theorem \ref{thm:bootstrap.prec} (i.e., second term of $\Psi^\Omega_B(p,n,\epsilon)$), compared to their covariance counterparts.
If the increasing rate of $p$ is not significantly slower than $n$, then the bounds may diverge, even in the short-memory case.
This extra term accounts for the error incurred by approximating $\mathfrak{W} = \hat\Omega_n - \Omega$ with the dominating linear term $-\Omega \mathfrak{S}\Omega$, where $\mathfrak{S} = \hat\Sigma_n - \Sigma$.
Interestingly, since both the intermediate Gaussian approximation and the block bootstrap are based on this dominating linear term, their distributions drift to the same direction, as is shown in Figure \ref{fig:E2.ECDF.prec}.
In the large-$p$ cases, these two distributions tend to remain relatively close, especially when $n$ is large and the temporal dependence is relatively weak, despite both being far from the distribution of precision error.
Meanwhile, the ultra-long-memory case $\beta=0.55$ fails again, which is expected since our theoretical conditions are violated.

\subsection{High-dimensional Regime}

In the high-dimensional regime when $p$ exceeds $n$, we take the same setup on the structure of $A_t$, a square Toeplitz matrix, with each entry set as
\[
(A_t)_{jk} = (t+1)^{-\beta}(|j-k|+1)^{-2}.
\]
We consider the same three types of processes by temporal dependence, $\beta = 2$, $\beta = 0.9$, and $\beta=0.55$.
The simulation is done on a grid of $n\in\{200, 400, 800\}$, and $p\in\{2000, 4000, 8000\}$.
For each setup, we still let $N=n^2$, and use 200 out of all $K_{N,n}$ realizations of $\{X_i\}_{i=1}^n$.
Since we are in the high-dimensional regime, we can only compute the sample covariance matrix $\hat\Sigma_n$ from each copy of the data.
We can obtain the empirical distribution of $n^{1/2} |\hat\Sigma_n - \Sigma|_\infty$, consisting of 200 points.
We can then perform its Gaussian approximation, obtaining 200 instances of $|Z|_\infty$, and the block bootstrap procedure, resulting in 200 instances of $l^{-1/2}|\check{B}_{i,l}|_\infty$.

We can similarly draw QQ-plots for covariance matrix, demonstrating the proximity of Gaussian approximation and block bootstrap.
The QQ-plot for $\beta=2$, $\beta=0.9$, and $\beta=0.55$ are given in Figures \ref{fig:E3.QQ.cov.2}, \ref{fig:E3.QQ.cov.0.9}, and \ref{fig:E3.QQ.cov.0.55}, respectively.
We also draw the empirical CDF for $n^{1/2} |\hat\Sigma_n - \Sigma|_\infty$, $|Z|_\infty$, and $l^{-1/2}|\check{B}_{i,l}|_\infty$ in Figure \ref{fig:E3.ECDF.cov} --- see Appendix \ref{sec:apdx.supplementary.sim.results}.
A table that shows the Kolmogorov distances of interest in Theorems \ref{thm:GA.cov} and \ref{thm:bootstrap.cov}, $\rho(n)$ and $\rho_B(n,l)$, are given in Table \ref{tab:E3.KS.cov}.
As a supplement, Table \ref{tab:E3.W1.cov} in Appendix \ref{sec:apdx.supplementary.sim.results} shows the corresponding Wasserstein-1 distances.

\begin{figure}[t]
    \centering
    
    \includegraphics[width=0.8\linewidth]{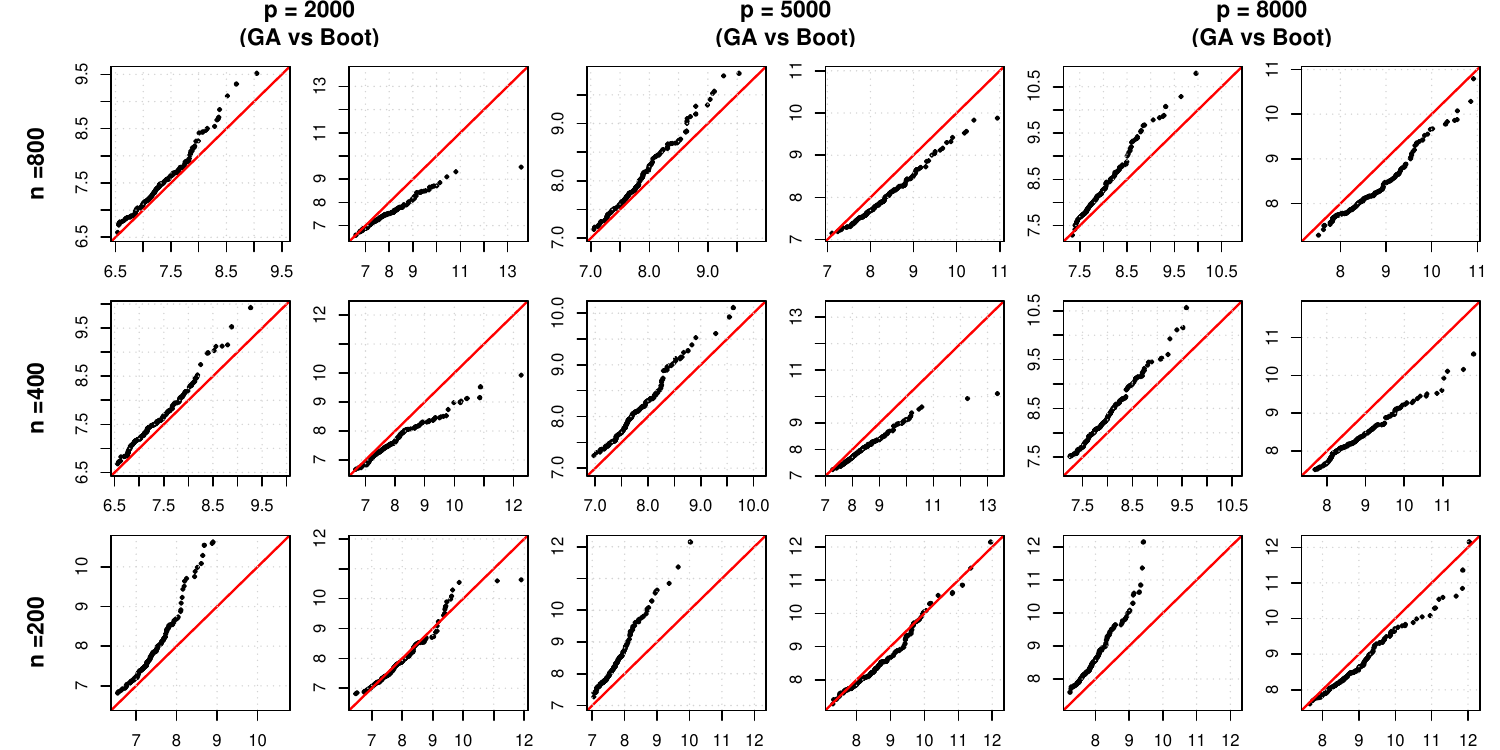}
    \caption{QQ-plots for covariance matrix in the high-dimensional regime, $\beta=2$ (short memory). Red line represents $y=x$. For each pair of $p$ and $n$ values, the plot on the left compares $n^{1/2}|\hat\Sigma_n - \Sigma|_\infty$ on $y$-axis and $|Z|_\infty$ on $x$-axis.
    Being closer to the $y=x$ line represents a smaller Kolmogorov distance $\rho(n)$ for Gaussian approximation.
    The plot on the right compares $n^{1/2}|\hat\Sigma_n - \Sigma|_\infty$ on $y$-axis and $l^{-1/2}|\check{B}_{i,l}|_\infty$ on $x$-axis.
    Being closer to the $y=x$ line represents a smaller Kolmogorov distance $\rho_B(n,l)$ for block bootstrap.}
    \label{fig:E3.QQ.cov.2}
    
    \vspace{1cm} 
    
    \includegraphics[width=0.8\linewidth]{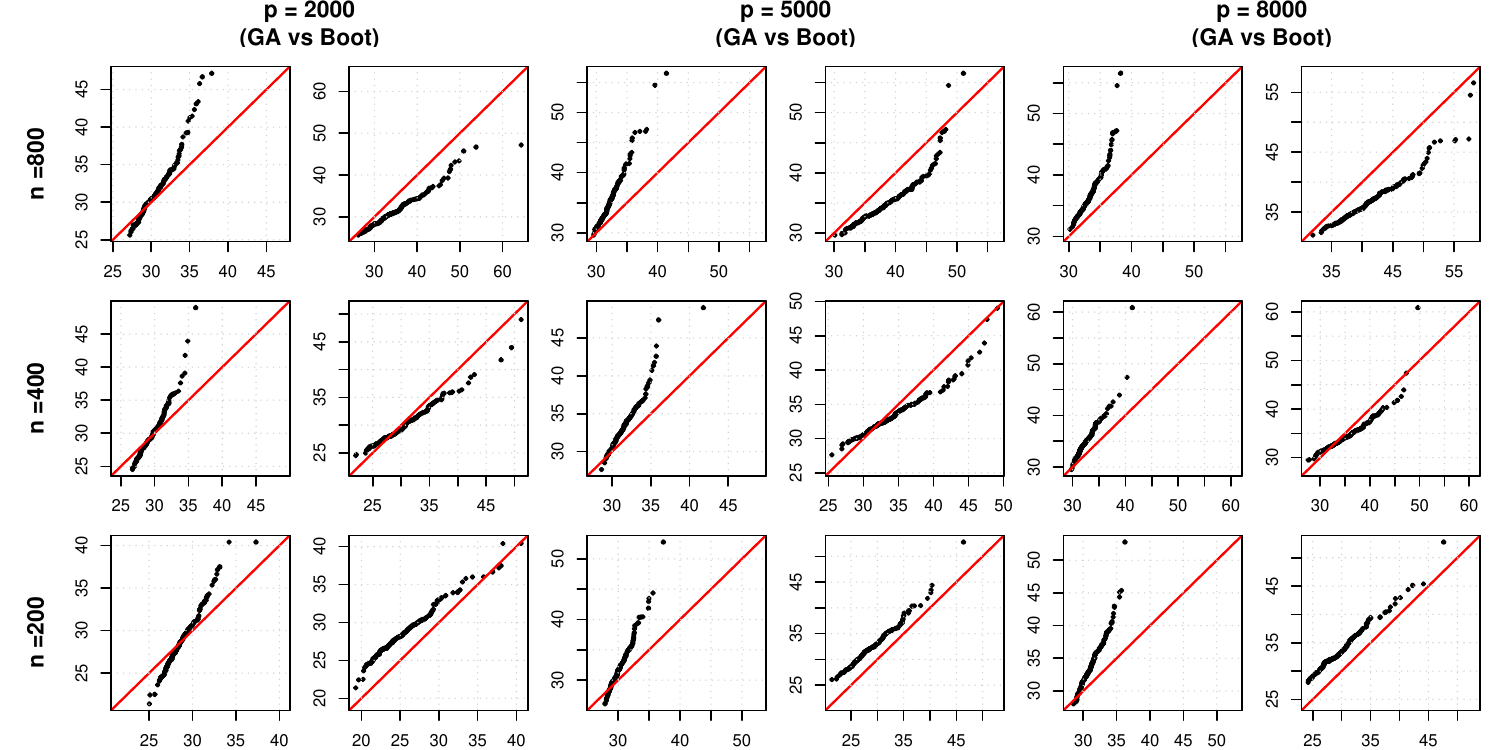}
    \caption{QQ-plots for covariance matrix in the high-dimensional regime, $\beta=0.9$ (long-memory). Red line represents $y=x$. For each pair of $p$ and $n$ values, the plot on the left compares $n^{1/2}|\hat\Sigma_n - \Sigma|_\infty$ on $y$-axis and $|Z|_\infty$ on $x$-axis.
    Being closer to the $y=x$ line represents a smaller Kolmogorov distance $\rho(n)$ for Gaussian approximation.
    The plot on the right compares $n^{1/2}|\hat\Sigma_n - \Sigma|_\infty$ on $y$-axis and $l^{-1/2}|\check{B}_{i,l}|_\infty$ on $x$-axis.
    Being closer to the $y=x$ line represents a smaller Kolmogorov distance $\rho_B(n,l)$ for block bootstrap.}
    \label{fig:E3.QQ.cov.0.9}
    
\end{figure}

\begin{figure}[t]
    \centering
    
    \includegraphics[width=0.8\linewidth]{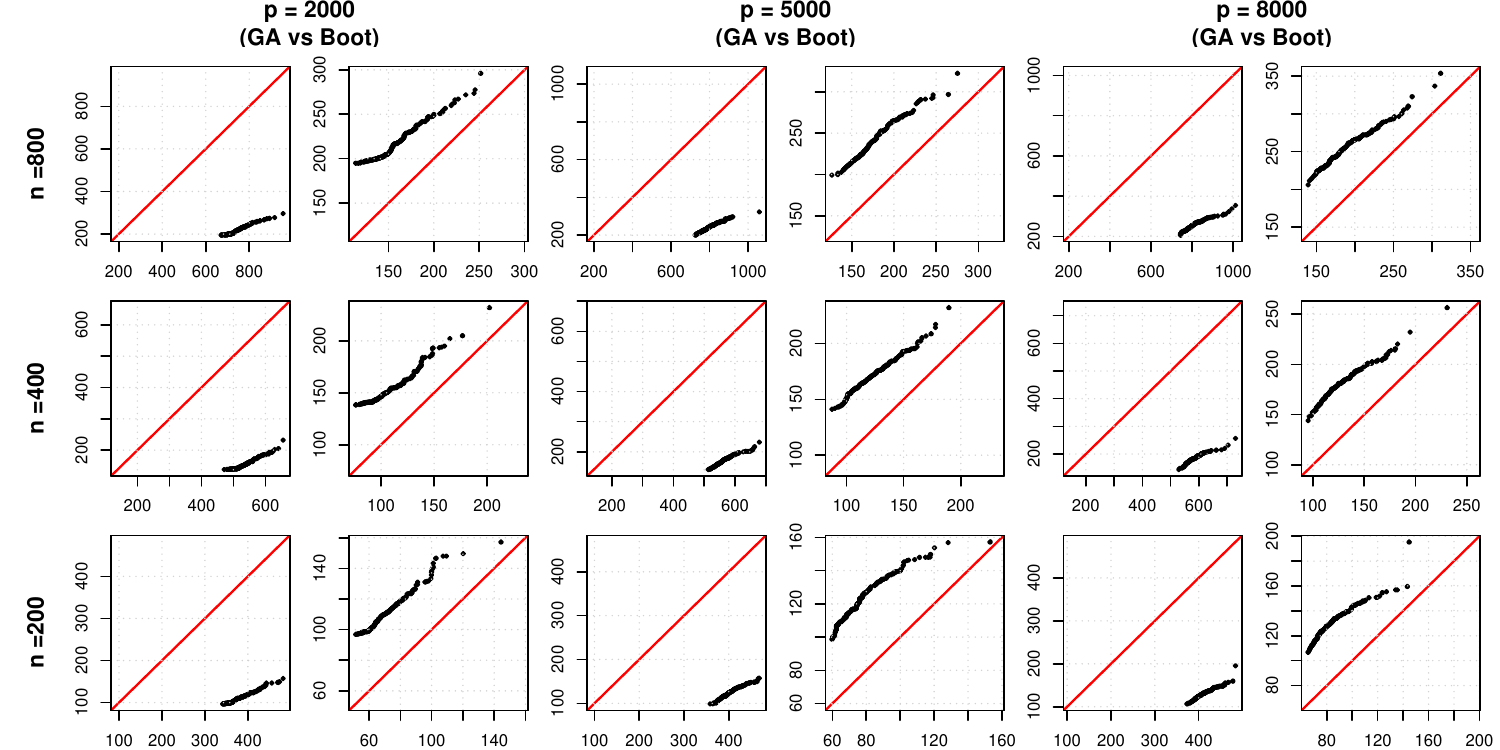}
    \caption{QQ-plots for covariance matrix in the high-dimensional regime, $\beta=0.55$ (ultra-long-memory). Red line represents $y=x$. For each pair of $p$ and $n$ values, the plot on the left compares $n^{1/2}|\hat\Sigma_n - \Sigma|_\infty$ on $y$-axis and $|Z|_\infty$ on $x$-axis.
    Being closer to the $y=x$ line represents a smaller Kolmogorov distance $\rho(n)$ for Gaussian approximation.
    The plot on the right compares $n^{1/2}|\hat\Sigma_n - \Sigma|_\infty$ on $y$-axis and $l^{-1/2}|\check{B}_{i,l}|_\infty$ on $x$-axis.
    Being closer to the $y=x$ line represents a smaller Kolmogorov distance $\rho_B(n,l)$ for block bootstrap.}
    \label{fig:E3.QQ.cov.0.55}
    
    

\end{figure}

\begin{table}[t]
\centering

\resizebox{\textwidth}{!}{
\begin{tabular}{llccccccccc}
\toprule
 & & \multicolumn{3}{c}{$\beta=2$} & \multicolumn{3}{c}{$\beta=0.9$} & \multicolumn{3}{c}{$\beta=0.55$} \\
 \cmidrule(lr){3-5} \cmidrule(lr){6-8} \cmidrule(lr){9-11}
& $n$ & $p=2000$ & $p=5000$ & $p=8000$ & $p=2000$ & $p=5000$ & $p=8000$ & $p=2000$ & $p=5000$ & $p=8000$ \\
\midrule
\multirow{3}{*}{GA} & 800 & 0.14 & 0.17 & 0.28 & 0.13 & 0.48 & 0.54 & 1.00 & 1.00 & 1.00 \\
 & 400 & 0.23 & 0.29 & 0.34 & 0.15 & 0.35 & 0.42 & 1.00 & 1.00 & 1.00 \\
 & 200 & 0.32 & 0.48 & 0.51 & 0.23 & 0.33 & 0.46 & 1.00 & 1.00 & 1.00 \\
\midrule
\multirow{3}{*}{Bootstrap} & 800 & 0.29 & 0.30 & 0.38 & 0.40 & 0.40 & 0.47 & 0.90 & 0.84 & 0.83 \\
 & 400 & 0.25 & 0.31 & 0.29 & 0.12 & 0.17 & 0.19 & 0.92 & 0.83 & 0.86 \\
 & 200 & 0.10 & 0.17 & 0.23 & 0.47 & 0.48 & 0.50 & 0.93 & 0.91 & 0.92 \\
\bottomrule
\end{tabular}
}
\caption{Kolmogorov distances for covariance matrix in the high-dimensional regime. The GA section compares the covariance error $n^{1/2}|\hat\Sigma_n - \Sigma|_\infty$ against its Gaussian approximation $|Z|_\infty$, while the Bootstrap section compares the covariance error against its block bootstrap approximation $l^{-1/2}|\check{B}_{i,l}|_\infty$.}
\label{tab:E3.KS.cov}


\end{table}

From these results, we see that Gaussian approximation and block bootstrap for covariance matrix still work relatively well in the short-memory case of $\beta=2$.
Both retain a reasonable accuracy even when $p$ is as large as 8,000.
However, noticeable distribution shifts do occur in the rather extreme cases when $p$ is very large or $n$ is very small.
On the other hand, the Gaussian approximation error and bootstrap error become relatively more prominent in the long-memory case of $\beta=0.9$.
That being said, the QQ-plots in Figure \ref{fig:E3.QQ.cov.0.9} and the empirical CDF in Figure \ref{fig:E3.ECDF.cov} suggest that block bootstrap still attains reasonable quality in this case, which is desirable in real-world application, especially when the dimension $p$ is much higher than $n$.
The increase of these distances is reasonable, considering the rates of $\Psi(p,n)$ and $\Psi_B(p,n,\epsilon)$ in Theorems \ref{thm:GA.cov} and \ref{thm:bootstrap.cov}.
In the ultra-long-memory case of $\beta=0.55$, both Gaussian approximation and bootstrap fail as expected, as shown in Figure \ref{fig:E3.QQ.cov.0.55}, as well as in Tables \ref{tab:E3.KS.cov} and \ref{tab:E3.W1.cov}.

\section{Real Data Analysis}\label{sec:real.data}
We demonstrate the practical application of our method on the data from the Autism Brain Imaging Data Exchange (ABIDE, \cite{di2014autism}).
Autism Spectrum Disorder (ASD) is a chronic neurodevelopmental disorder associated with challenges in social communication, repetitive behaviors, and atypical sensory processing, often linked to altered brain connectivity patterns \cite{ecker2015neuroimaging}.
The ABIDE dataset contains functional magnetic resonance imaging (fMRI) scans from patients with autism (ASD group) and a control group.
Using the parcellation approach of automatic anatomic labeling (AAL, \cite{tzourio2002automated}), the raw scanning image of the whole brain is segmented into temporal signals over 116 regions of interests (ROIs).
The connectivity between these temporal signals in these ROIs is characterized by the precision matrix of the multivariate time series.
We can spot connectome anomalies in individuals with ASD by comparing the behavior of precision matrices.


We select samples from 71 patients diagnosed with ASD and 95 controls.
For each individual, we have time series of length $n=176$ across $p=116$ many ROIs.
A noteworthy remark is that the majority of these time series are long-memory.
From a neuroscience perspective, this arises for fMRI signals due to the following factors.
First, human brain features the default mode networks that are characterized by intrinsic synchronization across brain regions, supporting persistent and organized patterns in the brain's functional structure, which contributes to long-memory behavior \cite{fransson2008precuneus}.
Second, repeating patterns introduced by neuronal oscillations reinforce long-memory properties of the fMRI signal \cite{becker2011ongoing}.
Third, the fMRI signal sometimes involves hemodynamic response (i.e., blood flow or changing level of hormone) that follows neural activity with a delay, lending a smoothing or long-memory effect to the fMRI signal \cite{friston1998event}.
To confirm long-range dependence in the dataset, we use Hurst exponent $H$ as a metric for temporal dependence.
A time series with $H > 0.5$ is considered to exhibit long-range dependence.
A histogram of Hurst exponents of all time series for all individuals and ROIs is given for each group in Figure \ref{fig:hurst.hist}.
We see that over a half of the time series in the dataset are long-range dependent.
An analysis on the autocorrelation function (ACF) also reveals the long-memory property.
86.2\% of all time series in the ASD group and 85.9\% of all the time series in the control group have at least one ACF of lag from 21 to 100 exceeding the threshold, $1.96 / \sqrt{n}$ (ACF exceeding this threshold is regarded significant at 95\% level).
A selection of ACF plots of the time series are given in Figure \ref{fig:acf.ind}.
These time series feature significant ACF at high lags with oscillations that decay slowly, indicating long-range dependence.

\begin{figure}
\centering
        \includegraphics[width=4cm]{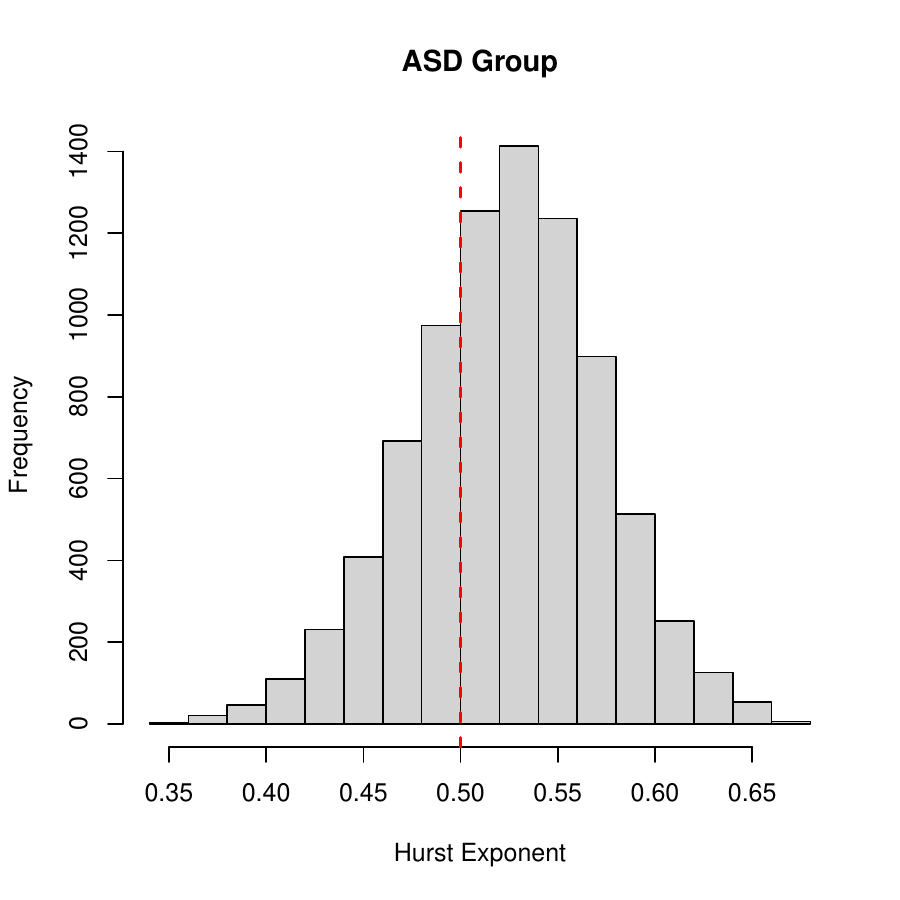}
        \includegraphics[width=4cm]{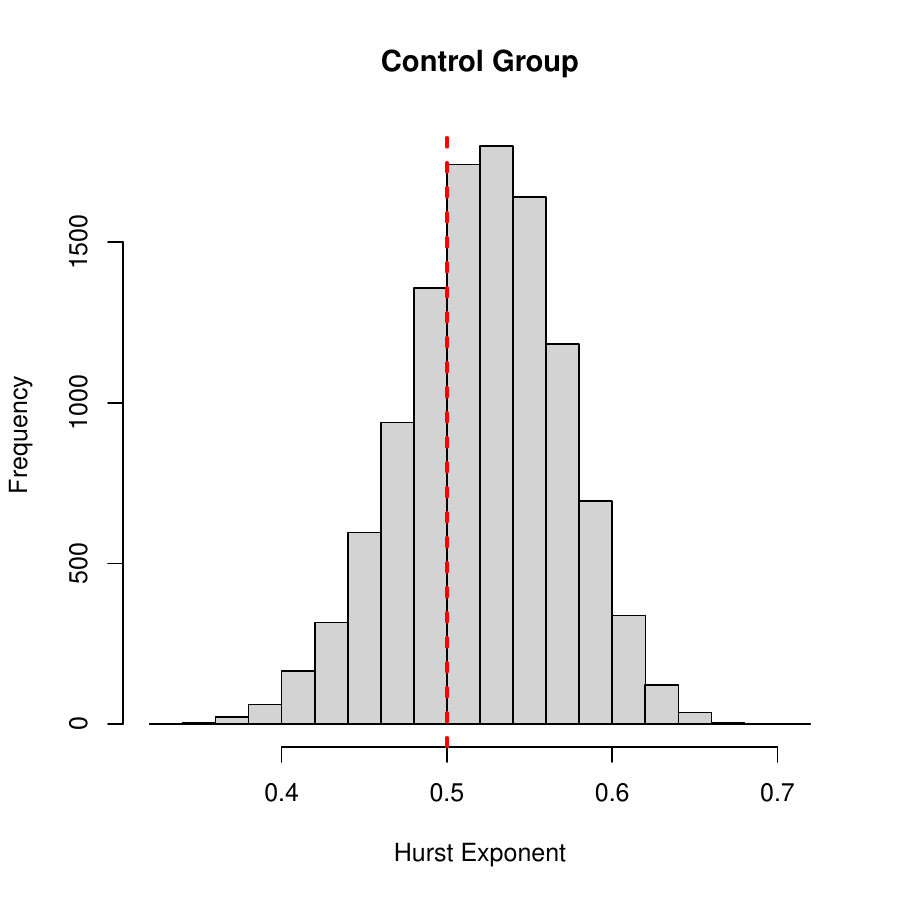}
\caption{Histograms of Hurst exponents of all time series within each group. The red dashed line is $H=0.5$. A time series with $H>0.5$ is considered long-memory.}
\label{fig:hurst.hist}
\end{figure}

\begin{figure}
\centering
        \includegraphics[width=3cm]{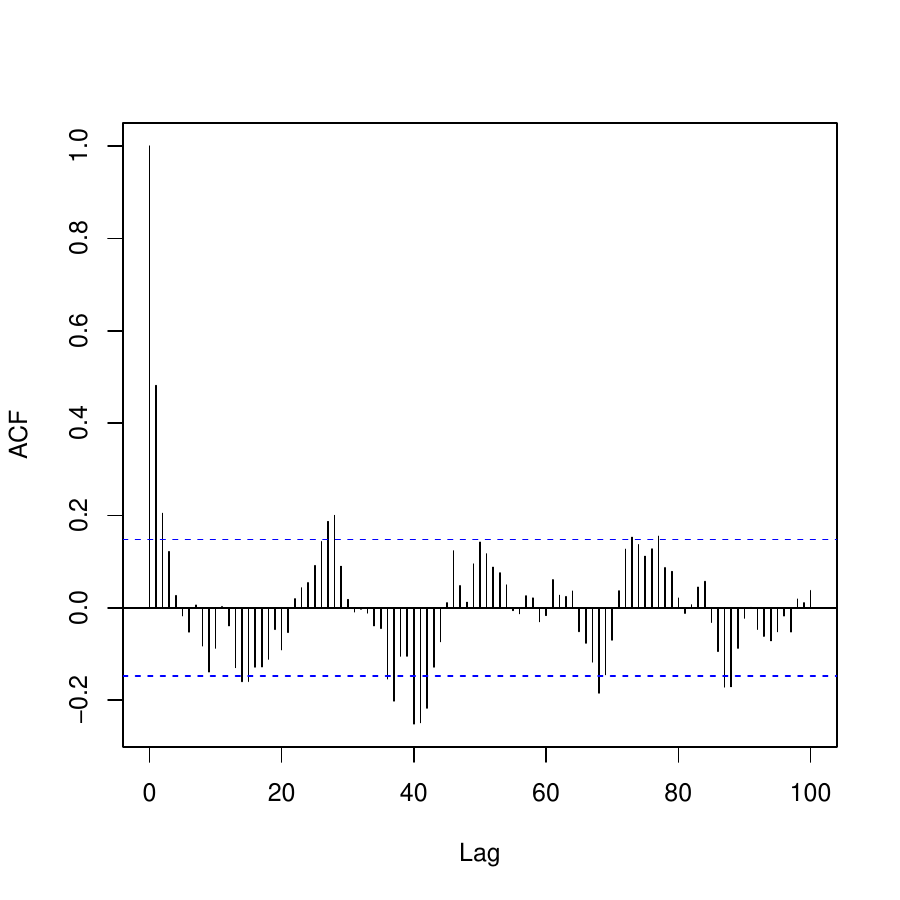}
        \includegraphics[width=3cm]{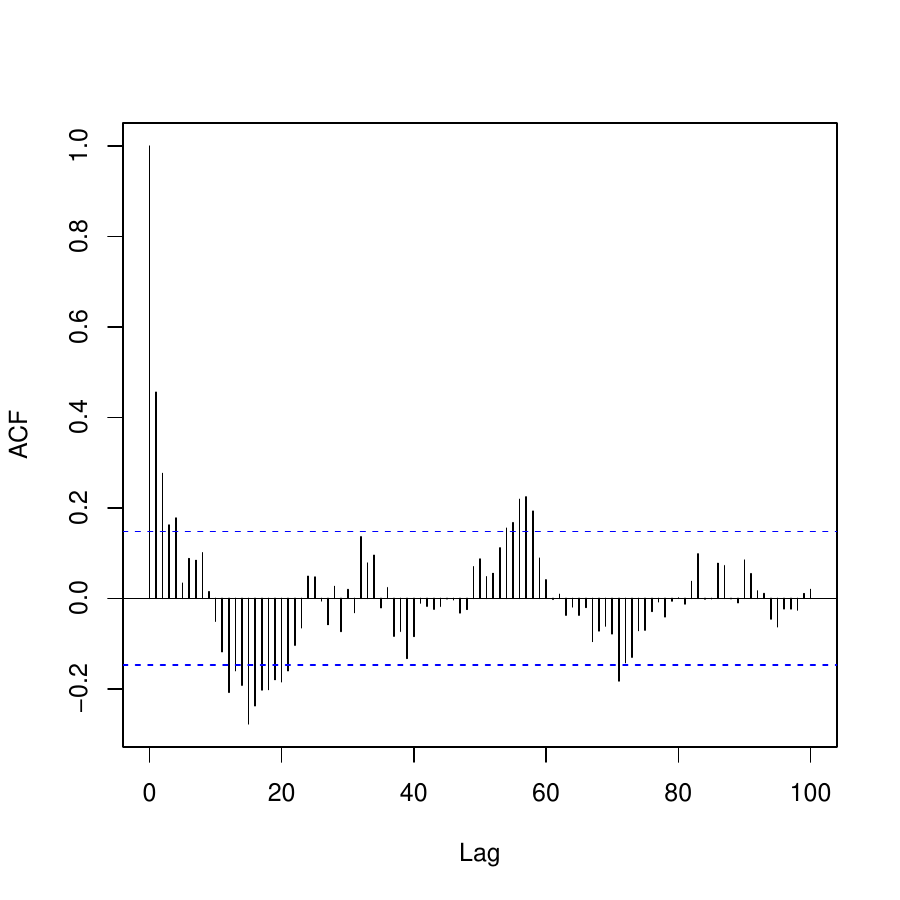}
        \includegraphics[width=3cm]{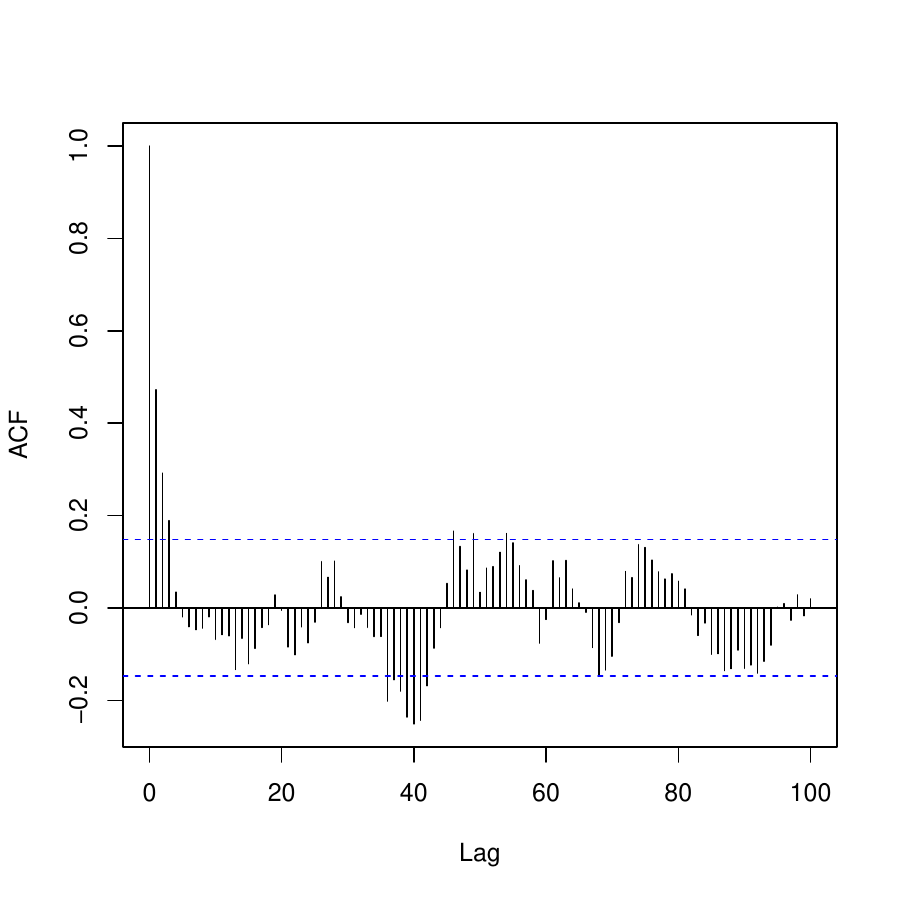}
        \includegraphics[width=3cm]{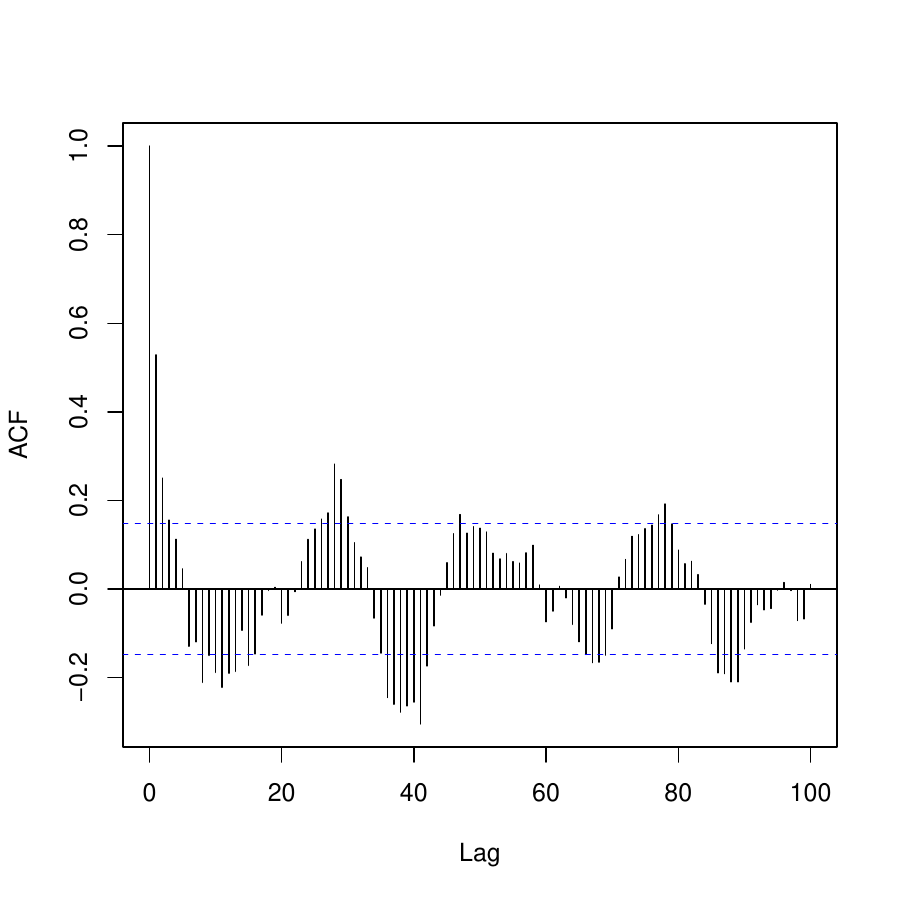}
\caption{ACF plots of select time series. The blue dashed lines are $\pm 1.96/\sqrt{n}$. autocorrelations that exceed these lines are considered significant at 95\% level.}
\label{fig:acf.ind}
\end{figure}

We perform a test on the precision matrix structure on the individual level, by first estimating $\hat \Omega_n = \hat \Sigma_n^{-1}$, and performing block bootstrap with $l=n^{2/3}$ for entrywise confidence intervals.
For any $1\leq j\neq k \leq p$, if the null hypothesis $\Omega_{jk} = 0$ is rejected (i.e., the confidence interval for the entry does not contain zero), then a connection is marked between node $j$ and node $k$ for this individual.
These connections characterize a conditional independence graph (CIG) for each individual.

The interpretation of a network may depend on its sparsity level, which should be treated as a tuning parameter.
On the individual level, it is controlled by the significance level of the hypothesis testing.
To discover common connectome patterns within both ASD and control groups, a graph can be drawn by aggregating the individual graphs in each group.
The sparsity level can be controlled by keeping the edges identified in most individuals.
In our case, we filter out the top 2\% most significant connections.
The aggregated graphs are given in Figure \ref{fig:fmri.connectome}.

\begin{figure}[t]
    \centering
    \subfloat[ASD group\label{fig:asd_group}]{
        \includegraphics[width=0.48\textwidth]{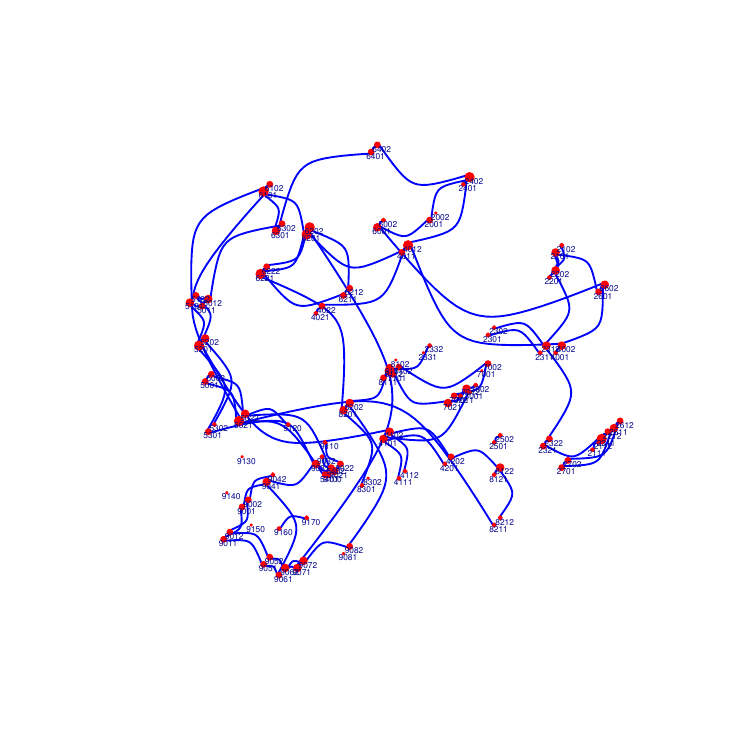}
    }
    \hfill
    \subfloat[Control group\label{fig:control_group}]{
        \includegraphics[width=0.48\textwidth]{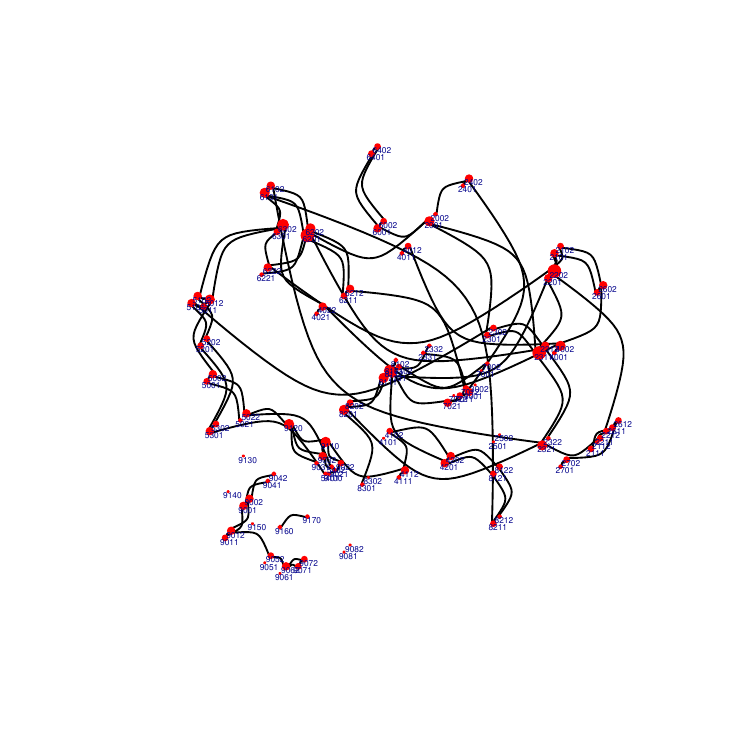}
    }
    
    \caption{Conditional independence graph (CIG) based on the precision matrix of brain fMRI signals over $p=116$ regions of interest, recovered by block bootstrap method. The red dots are ROIs plotted according to their spatial locations in the brain. The right hand side represents the front, while the left hand side represents the back. The size of the dots represents level of ROI activity. A larger dot indicates a larger number of connections out of the ROI. The graph sparsity is set to 2\%.}
    \label{fig:fmri.connectome}
\end{figure}

Our method is able to detect the basic patterns in both groups.
The first example is homotopic connectivity between the corresponding, symmetric ROIs on the left and right hemispheres of the brain.
Such interhemispheric coordination is ubiquitous in human brain, and essential for coordinated bilateral processing and integration \cite{zuo2010growing}.
From Figure \ref{fig:fmri.connectome}, we see that such connections are widely present for both ASD group and control group.
Indeed, out of the 54 left-right pairs of homotopic nodes, 29 of them are interconnected in ASD group, and 28 in control group.
The second example is local connectivity between neighboring ROIs that work together to process similar types of information.
Such connections contribute to the hierarchical structure of the brain, supporting both specialized and integrating functions \cite{sepulcre2010organization}.
Our method confirms that such local connections are prominent in both groups; most of the significant connections are short-ranged.

A comparison of the connectome patterns of the two groups in Figure \ref{fig:fmri.connectome} reveals that compared to the control group, the ASD group exhibits hypoactivity (i.e., fewer connections) and hyperactivity (i.e., more connections) at certain ROIs.
The most deactivated ROI is
{\it inferior frontal gyrus -- triangular part left} (2311), whose connections dropped from 6 in the control group to 1 in the ASD group, followed by
{\it middle frontal gyrus -- right} (2202, connections reduced from 6 to 3),
{\it precuneus -- right} (6302, connections reduced from 5 to 2),
{\it inferior parietal gyrus left} (6201, connections reduced from 6 to 4), and
{\it precentral gyrus left} (2001, connections reduced from 4 to 2).
The {\it inferior frontal gyrus} is critical for language processing, social cognition, and imitation, functions that are often disrupted in ASD \cite{fishman2014atypical}.
The {\it middle frontal gyrus} is central to the executive control network.
Its reduced connectivity explains the executive dysfunction and rigid cognitive style that are frequently observed in ASD \cite{hill2004executive}.
The {\it precuneus} is a central hub for the Default Mode Network, whose deactivation contributes to deficits in integrating individual knowledge with social context \cite{assaf2010abnormal}.
The {\it precentral gyrus} is a part of the primary motor cortex, which is involved in motor planning and execution \cite{nebel2014disruption}.
Disruption in this area has been linked to deficits in motor coordination and the presence of repetitive motor behaviors, both common symptoms of ASD.
The {\it inferior parietal gyrus} plays a role in social recognition, spatial awareness, and attentional control \cite{redcay2008deviant}.
Its reduced activity among ASD patients might be associated with their difficulties in interpreting social cues and coordinating attention.

Meanwhile, our analysis showed that the ROIs that gained the most connections are {\it lingual gyrus left} (5021, connections increased from 1 to 4),
{\it angular gyrus left} (6221, increased from 1 to 4),
{\it hippocampus} (4101, increased from 0 to 3),
and {\it pallidum right} (7022, increased from 0 to 2).
The {\it lingual gyrus} is involved in visual processing, particularly in visual memory and encoding, as well as processing visual aspects of language \cite{itahashi2014altered}.
The increased connectivity in {\it angular gyrus}, which is involved in language processing, number processing, and aspects of social cognition, such as theory of mind, reflects the compensatory mechanisms for language and social processing \cite{redcay2008deviant}.
The {\it hippocampus} is the crucial part of human brain for forming long-term memory. Its hyperactivity explains the distinct cognitive pattern typically seen in ASD patients using rote memory to compensate deficit in social processing skills \cite{ullman2015compensatory}.
The {\it pallidum} of the {\it lenticular nucleus} is part of the {\it basal ganglia}, which is involved in motor control, habit formation, and procedural learning.
Hyperactivity in this region is often linked to repetitive behaviors and routines, a hallmark of ASD \cite{turner2006atypically}.

Overall, most of the top deactivated nodes are located in the front of the brain, while the most deactivated nodes are located in the back or deep inside the brain.
The comparison in Figure \ref{fig:fmri.connectome} also suggests that the ASD group generally features a relatively silent frontal area and a hyperactive posterior area.
Additionally, the front and the back are generally less well-connected in ASD group compare to control.
The anterior network in the front of the brain is generally responsible for intuitive social processing, while the posterior network in the back and the subcortical network on the inside are associated with visual analysis and memory.
Collectively, our findings agree with those in \cite{just2004cortical}.
In order to compensate insufficient intuitive social functionality, the ASD brains use long-term rote memory to mechanically process raw sensory data.

\begin{figure}[t]
    \centering
    \subfloat[ASD group\label{fig:asd_group_FGLasso}]{
        \includegraphics[width=0.48\textwidth]{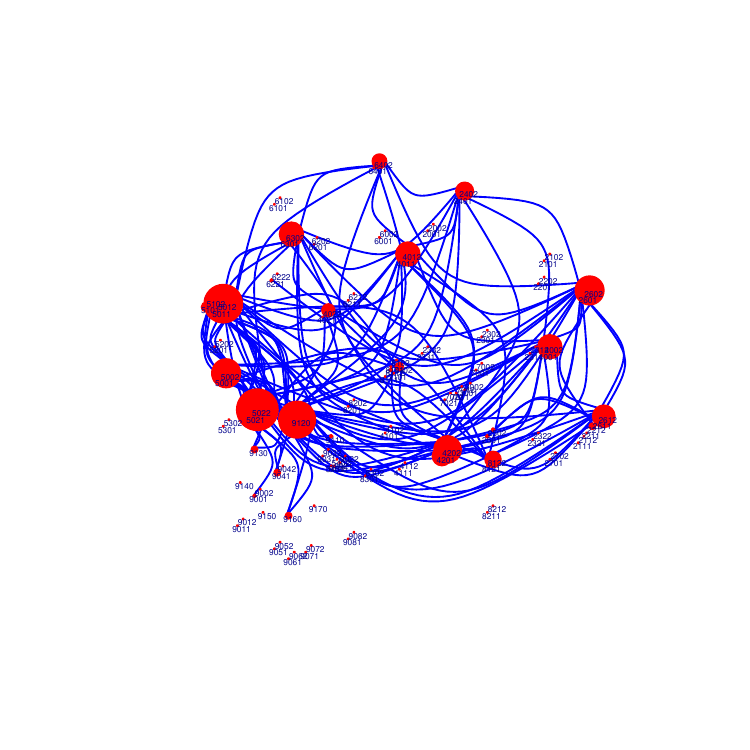}
    }
    \hfill
    \subfloat[Control group\label{fig:control_group_FGLasso}]{
        \includegraphics[width=0.48\textwidth]{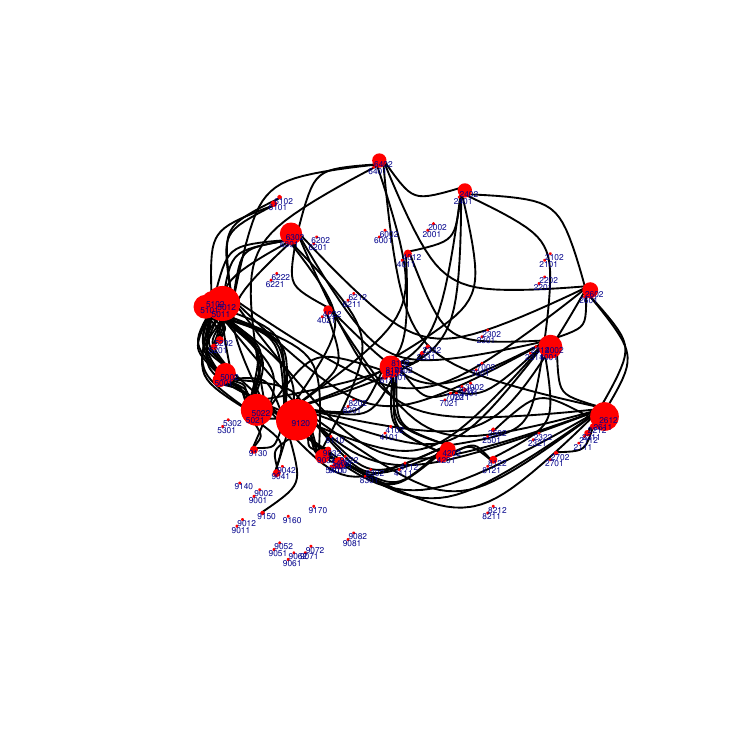}
    }
    
    \caption{Conditional independence graph (CIG) based on the precision matrix of brain fMRI signals over $p=116$ regions of interest, recovered by Functional Graphical Lasso (FGLasso). The red dots are ROIs plotted according to their spatial locations in the brain. The right hand side represents the front, while the left hand side represents the back. The size of the dots represents level of ROI activity. A larger dot indicates a larger number of connections out of the ROI. The graph sparsity is set to 2\%.}
    \label{fig:fmri.connectome.fglasso}
\end{figure}

As a technical comparison, we have also recovered the graph for ASD group and control group using Functional Graphical Lasso (FGLasso) by \cite{qiao2019functional}.
While our method detects edges by hypothesis testing, i.e., comparing the confidence interval obtained by block bootstrap with zero, FGLasso takes a fundamentally different approach of penalized precision matrix estimation.
Given an observation of functional data, it first performs functional principal component analysis (FPCA) for dimension reduction, and then estimate the block precision matrix of FPCA scores by maximizing a penalized log-likelihood (hence the name Lasso).
A more detailed discussion can be found in \cite{zhao2024high}, which also proposed a penalized regression method based on neighborhood selection instead of precision matrix estimation.
The resulting brain connectome graph by applying FGLasso to the ABIDE dataset, with also a 2\% sparsity enforced, is given in Figure \ref{fig:fmri.connectome.fglasso}.
For both ASD group and control group, the connections are concentrated on a few hub nodes, while the majority of the ROIs have no connection detected.
This phenomenon arises due to a known limitation of Lasso-type method in highly correlated data.
Out of a group of correlated signals (e.g., those from spatially proximal ROIs), it tends to select one representative while suppressing others.
Despite the advantages it brings when sparsity is present, it may lead to an overly condensed representation of the complex local networks for brain connectome data.
While FGLasso is a highly desirable method for problems where the true underlying network structure is sparse, our bootstrap inference-based method, on the other hand, tends to have a different emphasis of detecting detailed local network structure.

\bibliographystyle{IEEEtran}
\bibliography{percy-papers}

\makeatletter
\renewenvironment{IEEEbiographynophoto}[1]{%
  \par\medskip
  \normalfont\footnotesize\interlinepenalty500
  \parskip=0pt\par
  \noindent\textbf{#1\ }\ignorespaces
}{%
  \par
}
\makeatother

\begin{IEEEbiographynophoto}{Percy S. Zhai}
is currently a Ph.D.~student in econometrics and statistic at the University of Chicago Booth School of Business.
His research interests include time series analysis, graphical models, uncertainty quantification, statistical theory in generative AI, and Bayesian statistics.
He was selected as a recipient of the 2025 Arnold Zellner Doctoral Prize in recognition of his contribution to the application of Bayesian statistical ideas and methods to problems in time series of relevance to financial applications.
He has also served as a reviewer for the Journal of the American Statistical Association, and the Journal of the Royal Statistical Society, Series B.

\end{IEEEbiographynophoto}
\begin{IEEEbiographynophoto}{Mladen Kolar}
is a professor in the Department of Data Sciences and Operations at the USC Marshall School of Business and a Visiting Professor of Statistics and Data Science at the Mohamed bin Zayed University of Artificial Intelligence.
Mladen earned his PhD in Machine Learning from Carnegie Mellon University in 2013.
His research focuses on high-dimensional statistical methods, probabilistic graphical models, and scalable optimization methods, driven by the need to uncover interesting and scientifically meaningful structures from observational data.
Mladen was selected as a recipient of the 2024 Junior Leo Breiman Award for his outstanding contributions to these areas.
He currently serves as an associate editor for the Journal of Machine Learning Research, the Journal of Computational and Graphical Statistics, the Journal of the American Statistical Association, and IEEE Transactions on Pattern Analysis and Machine Intelligence.
He also served the Annals of Statistics and the New England Journal of Statistics in Data Science.
\end{IEEEbiographynophoto}
\begin{IEEEbiographynophoto}{Wei Biao Wu}
received the B.S. degree from Fudan University, Shanghai, China, in 1997, and the Ph.D. degree in statistics from the University of Michigan, Ann Arbor, MI, USA, in 2001.
He is currently a Professor with the Department of Statistics, The University of
Chicago, Chicago, IL, USA. His research interests include probability theory, time series analysis, high-dimensional statistical inference, and machine learning. His work focuses on the development of asymptotic theory for complex dependence structures, functional dependence measures, and the statistical analysis of stochastic gradient descent and online learning algorithms.
He was the recipient of the National Science Foundation CAREER Award in 2004, the Tjalling C.~Koopmans Econometric Theory Prize in 2009, and
the Humboldt Research Award from the Alexander von Humboldt Foundation in 2019. He has served as an Associate Editor for several journals, including the Annals of Statistics, Bernoulli and Journal of the American Statistical Association.
\end{IEEEbiographynophoto}

{\appendices

\renewcommand\thesubsection{\thesection.\arabic{subsection}}
\renewcommand\thesubsectiondis{\thesection.\arabic{subsection}}

\setcounter{equation}{0}
\renewcommand{\theequation}{A\arabic{equation}}

\setcounter{lemma}{0}
\renewcommand{\thelemma}{A\arabic{lemma}}
\setcounter{theorem}{0}
\renewcommand{\thetheorem}{A\arabic{theorem}}
\setcounter{remark}{0}
\renewcommand{\theremark}{A\arabic{remark}}
\setcounter{proposition}{0}
\renewcommand{\theproposition}{A\arabic{proposition}}
\setcounter{condition}{0}
\renewcommand{\thecondition}{A\arabic{condition}}

\section{Useful Results}\label{app:useful.results}

\subsection{A Corollary of Condition \ref{cond.A}}\label{app:cond.A.cor}

The following result following Condition \ref{cond.A} indicates that if a Gaussian linear process \eqref{LinearProcess} satisfies Condition \ref{cond.A}, then the left hand side of \eqref{eq:cond.G} is also upper bounded.

\begin{lemma}\label{prop:cond.A.cor}
Suppose that Condition \ref{cond.A} holds for a Gaussian linear process \eqref{LinearProcess}. Then for any $1\leq s,t\leq p$,
\begin{equation}\label{eq:cond.A.cor}
    \sum_{k=-\infty}^\infty \left[(\Gamma_k)_{ss}(\Gamma_k)_{tt} + (\Gamma_k)_{st}(\Gamma_k)_{ts}\right] \leq C_0',
\end{equation}
where $C_0'$ is a constant only related to $C_0$.
\end{lemma}

\begin{proof}[Proof of Lemma \ref{prop:cond.A.cor}]
We first observe a result that will be useful later. For every $j,k\in\mathbb{Z}$ and $1\leq s,t\leq p$, by Cauchy-Schwarz inequality and Condition \ref{cond.A},
\begin{equation}\label{EQ_CS}
    \left(\sum_{u=1}^p(A_{j})_{su}(A_{k})_{tu}\right)^2 \leq \left(\sum_{u=1}^p (A_{j})_{su}^2\right)\left(\sum_{u=1}^p (A_{k})_{tu}^2\right)\lesssim (j\vee 1)^{-2\beta}(k\vee 1)^{-2\beta},
\end{equation}
up to a constant depending only on $C_0$. Recall that $\Gamma_k = \sum_{l=0}^\infty A_l A_l^T$ for Gaussian linear process \eqref{LinearProcess}. By \eqref{EQ_CS} and Cauchy-Schwarz inequality,
\begin{align*}
    &\sum_{k=-\infty}^\infty \left[(\Gamma_k)_{ss}(\Gamma_k)_{tt} + (\Gamma_k)_{st}(\Gamma_k)_{ts}\right]\\
    &\lesssim \sum_{k=0}^\infty\sum_{l=1}^\infty\sum_{l'=1}^\infty l^{-\beta}(l+k)^{-\beta}l'^{-\beta}(l'+k)^{-\beta} + \left(1+\sum_{k=1}^\infty k^{-2\beta}\right)\\
    &\lesssim \sum_{l=1}^\infty l^{-\beta} \sum_{l'=1}^\infty l'^{-\beta} \sqrt{\sum_{k=0}^\infty (l+k)^{-2\beta} \sum_{k'=0}^\infty (l'+k')^{-2\beta}} +1 \\
    &\asymp\left(\sum_{l=1}^\infty l^{-2\beta+1/2}\right)^2+1 = O(1).
\end{align*}
\end{proof}

\subsection{A Convergence Result for m-dependent Approximation}\label{app:m.dep.conv}
The following lemma reveals a convergence result for $m$-dependent approximation. Recall that $m$ diverges to infinity as $n\rightarrow\infty$.
\begin{lemma}\label{lm:m.dep.conv}
$\lim_{m\rightarrow\infty} \mathbb{E}[(\tilde{D}_{i,m})_{st}^2] = \mathbb{E}[(D_i)^2_{st}]$ for any $1\leq s,t\leq p$.
\end{lemma}
\begin{proof}[Proof of Lemma \ref{lm:m.dep.conv}]
For the linear process with the form (\ref{LinearProcess}), we have
\begin{align}
    D_i &= \sum_{j=1}^\infty\sum_{l=0}^\infty A_{l+j}\epsilon_{i-j}\epsilon_i^\top A_l^\top + \sum_{j=1}^\infty \sum_{l=0}^\infty A_l\epsilon_i\epsilon_{i-j}^\top A_{l+j}^\top + \sum_{l=0}^\infty A_l\epsilon_i\epsilon_i^\top A_l^\top - \Sigma;\label{EQ_Di}\\
    \Tilde{D}_{i,m} &= \sum_{j=1}^{m-1}\sum_{l=0}^\infty A_{l+j}\epsilon_{i-j}\epsilon_i^\top A_l^\top + \sum_{j=1}^{m-1}\sum_{l=0}^\infty A_l\epsilon_i\epsilon_{i-j}^\top A_{l+j}^\top + \sum_{l=0}^\infty A_l\epsilon_i\epsilon_i^\top A_l^\top - \Sigma.\label{EQ_Di_tilde}
\end{align}
Observe that $(D_i-\Tilde{D}_{i,m})_{st}$ is independent to $(\Tilde{D}_{i,m})_{st}$. This implies that
\begin{equation*}
    \mathbb{E}[(D_i)_{st}^2] - \mathbb{E}[(\Tilde{D}_{i,m})_{st}^2] = \mathbb{E}[((D_i)_{st} - (\Tilde{D}_{i,m})_{st})^2].
\end{equation*}
The right-hand side can be rewritten as $\sum_{j=m}^\infty \zeta_j$, where 
\begin{equation*}
    \zeta_j = \sum_{l=0}^\infty \sum_{l'=0}^\infty \left(A_{l+j}A_{l'+j}^\top\right)_{ss} \left( A_l A_{l'}^\top\right)_{tt} + \sum_{l=0}^\infty \sum_{l'=0}^\infty \left(A_{l+j}A_{l'+j}^\top\right)_{tt} \left( A_l A_{l'}^\top\right)_{ss}
    + 2\sum_{l=0}^\infty \sum_{l'=0}^\infty \left(A_{l+j}A_{l'+j}^\top\right)_{st} \left( A_l A_{l'}^\top\right)_{ts}.
\end{equation*}
Note that
\begin{equation}\label{EQ_zeta}
    \sum_{j=1}^\infty \zeta_j = \sum_{k=-\infty}^\infty [(\Gamma_k)_{ss}(\Gamma_k)_{tt}+(\Gamma_k)_{st}(\Gamma_k)_{ts}].
\end{equation}
By Lemma \ref{prop:cond.A.cor}, under Condition \ref{cond.A}, for all $1\leq i\leq n, 1\leq s,t\leq p$, we have $\sum_{j=1}^\infty \zeta_j \leq C_0' < \infty$.
The proof of lemma is thus finished by the property of converging series, i.e., as $m\rightarrow\infty$, we have $\sum_{j=m}^\infty \zeta_j \rightarrow 0$.
\end{proof}

\begin{remark}\label{RMK_DD}
Observe that $\sum_{j=1}^\infty \zeta_j = \mathbb{E}[(D_i)^2_{st}]$. \eqref{EQ_zeta} indicates that under Condition \ref{cond.A}, we have 
$$\lim_{m\rightarrow\infty} \mathbb{E}[(\tilde{D}_{i,m})_{st}^2]\leq C_0'.$$
This convergence implies boundedness, i.e., for any $m\in\mathbb{N}$, $\mathbb{E}[(\tilde{D}_{i,m})_{st}^2]\leq C_0''$, where $C_0''$ is a constant related to $C_0$.
\end{remark}

\subsection{Hanson-Wright Inequality}\label{app:hanson.wright}

A concentration inequality of quadratic form is given by \cite{hanson1971bound}. The proof is later refined by \cite{rudelson2013hanson}, which also provides a stronger result. We shall refer to the latter. Let $|A|_{\op}=\sup_{|x|_2\leq 1}|Ax|_2$ be the operator norm of a matrix $A$, and $\|X\|_{\psi_2} = \inf\{K\in\mathbb{R}: \mathbb{E}[\exp(X^2/K^2)]\leq 2\}$ be the Orlicz 2-norm (sub-Gaussian norm) of a random variable $X$.

\begin{lemma}[Hanson-Wright Inequality, \cite{rudelson2013hanson}]\label{lm:hanson.wright}
Let $X=(X_1,\cdots,X_n)^T$ be a random vector with independent components $X_i$ which satisfy $\mathbb{E}[X_i]=0$ and $\|X_i\|_{\psi_2}\leq K$. Let $A$ be a $n\times n$ matrix. Then for every $t>0$,
\begin{equation}\label{eq:hanson.wright}
    \mathbb{P}\left(\left|X^\top AX - \mathbb{E}[X^\top AX]\right|>t\right)\leq 2\exp\left[-c\left(\frac{t^2}{K^4|A|_{\F}^2}\wedge\frac{t}{K^2|A|_{\op}}\right)\right].
\end{equation}
Here, $c>0$ is a universal constant.
\end{lemma}

\subsection{Moments of Gaussian Quadratic Form}\label{sec:Lk.norm}

The following lemma shows that any $L^k$ norm of a quadratic form of i.i.d.~Gaussian random vectors can be bounded by $L^2$ or $L^1$ norms under mild conditions.
It is a corollary of Marcinkiewicz-Zygmund Inequality \cite{marcinkiewicz1939quelques}.

\begin{lemma}\label{lm:Lk.norm}
Let $B=(b_{ij})_{1\leq i,j\leq n}$ be a symmetric matrix, $\eta = (\eta_1,\cdots,\eta_n)^T\sim N(0,I_n)$, and 
\begin{equation*}
    D = \sum_{i=1}^n \sum_{j=1}^n b_{ij}\eta_i\eta_j + \sum_{j=1}^n c_j\eta_j + d.
\end{equation*}
Then there exists a universal constant $C_u>0$ such that for any $k\geq 2$,
\begin{equation}\label{EQ_k-norms}
    \|D\|_k \leq C_u k\|D\|_2.
\end{equation}
Furthermore, when $B$ is positive semidefinite, we have
\begin{equation}\label{EQ_k-norms_ext}
    \|D\|_k \leq C_u k\|D\|_1.
\end{equation}
\end{lemma}

\begin{proof}[Proof of Lemma \ref{lm:Lk.norm}]
Since $B$ is symmetric, its eigendecomposition is $B=Q\Lambda Q^\top$, where $Q$ is orthonormal and $\Lambda=\text{diag}(\lambda_1,\cdots,\lambda_n)$ contains the eigenvalues. Let $\xi=Q^\top \eta$. Then, $E\xi_j=0$, $E\xi_j^2=1$, and $\xi_j$ are mutually independent for $j=1,\cdots,n$. Furthermore, let $c' = Q^{-1}c$, then
\begin{equation*}
    D = \xi^\top \Lambda\xi + c'^\top \xi + d = \sum_{j=1}^n \lambda_j \xi_j^2 + \sum_{j=1}^n c'_j\xi_j + d.
\end{equation*}
    
Without loss of generality, we shall assume that $c'=0$ and $d=0$ for simplicity in the following proof, i.e., $D=\sum_{j=1}^n \lambda_j \xi_j^2$. The general proof follows similarly. Thus,
\begin{equation}\label{EQ_D22}
    \|D\|_2^2 = \mathbb{E} D^2 = \left(\sum_{j=1}^n \lambda_j\right)^2 + 2\sum_{j=1}^n \lambda^2_j \geq 2\sum_{j=1}^n \lambda^2_j,
\end{equation}
and, for any $k\geq 2$,
\begin{equation}\label{EQ_Grand}
    \|D\|_k^k = \mathbb{E} D^k = \sum_{j_1=1}^n\sum_{j_2=1}^n\cdots\sum_{j_k=1}^n \lambda_{j_1}\lambda_{j_2}\cdots\lambda_{j_k}\mathbb{E} [\xi_{j_1}^2\xi_{j_2}^2\cdots\xi_{j_k}^2].
\end{equation}
The right-hand side of \eqref{EQ_Grand} can be expressed in the following manner. The indices $j_1, \cdots, j_k$ can assume $s$ distinct values, where $1 \leq s \leq k$. For any given $s$, let there be $\tau_s$ indices among $j_1, \cdots, j_k$ that take the same value, denoted as $j_{(s)}$. Note that $\tau_{1} + \tau_{2} + \cdots + \tau_{s} = k$. Hence, we can rewrite $\|D\|_k^k = \sum_{s=1}^k \Sigma_{(s)}$, where
\begin{equation*}
    \Sigma_{(s)} = \sum_{\tau_{1}+\tau_{2}+\cdots+\tau_{s} = k}  \sum_{j_{(1)}=1}^n\sum_{j_{(2)}=1}^n\cdots\sum_{j_{(s)}=1}^n \mathbb{E} [\xi_{j_{(1)}}^{2\tau_{1}}]\mathbb{E} [\xi_{j_{(2)}}^{2\tau_{2}}]\cdots \mathbb{E} [\xi_{j_{(s)}}^{2\tau_{s}}] \lambda_{j_{(1)}}^{\tau_1}\lambda_{j_{(2)}}^{\tau_2}\cdots\lambda_{j_{(s)}}^{\tau_s}.
\end{equation*}
The value of $\Sigma_{(s)}$ can be upper bounded by noticing the following three facts. First, since for a standard Gaussian random variable $Z$, $EZ^{2k} = (2k-1)!!$, for $\tau_{1}+\tau_{2}+\cdots+\tau_{s} = k$, we have
\begin{equation*}    
\mathbb{E} [\xi_{j_{(1)}}^{2\tau_{1}}]\mathbb{E} [\xi_{j_{(2)}}^{2\tau_{2}}]\cdots \mathbb{E} [\xi_{j_{(s)}}^{2\tau_{s}}] = (2\tau_1-1)!!(2\tau_2-1)!!\cdots(2\tau_s-1)!! \leq (2k-1)!!.
\end{equation*}
Next, we have
\begin{equation}\label{eq:lambda.combinations}
\sum_{j_{(1)}=1}^n\sum_{j_{(2)}=1}^n\cdots\sum_{j_{(s)}=1}^n\lambda_{j_{(1)}}^{\tau_1}\lambda_{j_{(2)}}^{\tau_2}\cdots\lambda_{j_{(s)}}^{\tau_s} = \prod_{i=1}^s \left(\sum_{j_{(i)}=1}^n \lambda_{j_{(i)}}^{\tau_i}\right)\leq (\|D\|_2/\sqrt{2})^k.
\end{equation}
This is because for even $\tau_i$, it is clear that $\sum_{j_{(i)}=1}^n \lambda_{j_{(i)}}^{\tau_i} \leq \left(\sum_{j=1}^n \lambda_j^2\right)^{\tau_i/2}$. For odd $\tau_i$, there exist odd $\tau_i^{(1)}$ and even $\tau_i^{(2)}$ such that $\tau_i = \tau_i^{(1)}+\tau_i^{(2)}$. By the Cauchy-Schwarz Inequality, $\sum_{j_{(i)}=1}^n \lambda_{j_{(i)}}^{\tau_i} \leq \left(\sum_{j_{(i)}=1}^n \lambda_{j_{(i)}}^{2\tau_i^{(1)}}\right)^{1/2}\left(\sum_{j_{(i)}=1}^n \lambda_{j_{(i)}}^{2\tau_i^{(2)}}\right)^{1/2}\leq \left(\sum_{j=1}^n \lambda_j^2\right)^{\tau_i/2}$. The second claim follows from \eqref{EQ_D22}. The number of terms in the summation $\sum_{\tau_{1}+\tau_{2}+\cdots+\tau_{s} = k}$ is ${k-1\choose s-1}$. Thus, 
\begin{equation*}
    \|D\|_k^k \leq \sum_{s=1}^k {k-1\choose s-1} (2k-1)!! (\|D\|_2/\sqrt{2})^k.
\end{equation*}
Using Stirling's formula, $(2k-1)!! = \frac{2k!}{k! 2^k} = \sqrt{2}(2k/e)^k(1+O(k^{-1}))$, up to a universal constant. Therefore, the inequality \eqref{EQ_k-norms} holds.
    
When $B$ is positive semidefinite, $\lambda_j \geq 0$ for all $j$. In this case, \eqref{eq:lambda.combinations} is replaced by
\begin{equation*}
    \prod_{i=1}^s \left(\sum_{j_{(i)}=1}^n \lambda_{j_{(i)}}^{\tau_i}\right)\leq \left(\sum_{j=1}^n \lambda_j\right)^k = \|D\|_1^k.
\end{equation*}
Following a similar proof, we get the inequality \eqref{EQ_k-norms_ext}.
\end{proof}

\subsection{Gaussian Approximation for Independent Random Variables}\label{sec:CCK}
The Gaussian approximation result presented by \cite{chernozhukov2019improved} is instrumental in our proof. Let $\Xi_1,\cdots,\Xi_n$ be independent random vectors in $\mathbb{R}^p$, with $\mathbb{E}[\Xi_{ij}]=0$ and $\mathbb{E}[\Xi_{ij}^2]<\infty$, $i=1,\cdots,n$, $j=1,\cdots,p$. Consider the Gaussian approximation to $S_n^\Xi=n^{-1/2}\sum_{i=1}^n \Xi_i$. For this purpose, let $\Upsilon_1,\cdots,\Upsilon_n$ be independent Gaussian random vectors in $\mathbb{R}^p$ such that $\Upsilon_i\sim N(0,\mathbb{E}[\Xi_i\Xi_i^T])$, and the normalized sum $S_n^\Upsilon=n^{-1/2}\sum_{i=1}^n \Upsilon_i$. Define $\mathcal{A}^{\text{re}}$ as the set of all hyperrectangles.
Assume that the following conditions hold:
\begin{itemize}
    \item (M.1) There exist a constant $b_1>0$ such that $n^{-1}\sum_{i=1}^n \mathbb{E}[\Xi_{ij}^2]\geq b_1^2$ for $j=1,\cdots,p$;
    \item (M.2) There exist a sequence $\beta_n \geq 1$ and a constant $b_2\geq b_1$, such that $n^{-1}\sum_{i=1}^n \mathbb{E}[|\Xi_{ij}|^4]\leq b_2^2\beta_n^2$ for $j=1,\cdots,p$. 
    \item (E.1) $\mathbb{E}[\exp(|\Xi_{ij}|/\beta_n)]\leq 2$ for all $i=1,\cdots,n$ and $j=1,\cdots,p$.
\end{itemize}
Note that in condition (M.2), the sequence $\beta_n \geq 1$ is allowed to diverge with $n$.

\begin{theorem}[Theorem 2.1 of \cite{chernozhukov2019improved}]\label{prop:CCK17}
Suppose that conditions (M.1), (M.2), and (E.1) are satisfied. Then we have
\begin{equation*}
    \sup_{A\in\mathcal{A}^{\text{re}}}\left|\mathbb{P}(S_n^\Xi\in A)-\mathbb{P}(S_n^\Upsilon\in A)\right|\leq C\left(\frac{\beta_n^2\log^5(pn)}{n}\right)^{1/4}.
\end{equation*}
where the constant $C$ depends only on $b_1$ and $b_2$.
\end{theorem}

\subsection{A Comparison Inequality for Gaussian Random Vectors}
A Berry-Esseen type bound between Gaussian approximations $Z$ and $S_n^Z$ can be implied by the following comparison inequality for Gaussian random vectors, developed by \cite{chernozhukov2015comparison} and extended by \cite{chernozhukov2019improved}.
\begin{theorem}[Proposition 2.1 of \cite{chernozhukov2019improved}]\label{thm:comparison}
    If $Z_1$ and $Z_2$ are centered Gaussian random vectors in $\mathbb{R}^p$ with covariance matrices $\Sigma^1$ and $\Sigma^2$, respectively, and $\Sigma^2$ is such that $\Sigma_{jj}^2 \geq \underline{\sigma}$ for all $j=1,\cdots,p$ for some constant $\underline{\sigma} > 0$, then
    \[
    \sup_{\mathcal{A}\subseteq \mathbb{R}^p} \left|\mathbb{P}(Z_1 \in \mathcal{A}) - \mathbb{P}(Z_2 \in \mathcal{A})\right| \leq C_\sigma \sqrt{\Delta_\Sigma} \log(p),
    \]
    where $C_\sigma>0$ is a constant depending only on $\underline{\sigma}$, and $\Delta_\Sigma = |\Sigma^1 - \Sigma^2|_\infty$.
\end{theorem}
An immediate corollary of Theorem \ref{thm:comparison} is
\[
    \sup_{u\in\mathbb{R}}\left|\mathbb{P}(|Z_1|_\infty \geq u) - \mathbb{P}(|Z_2|_\infty \geq u)\right| \leq C_\sigma \sqrt{\Delta_\Sigma} \log(p),
\]
which improves upon Theorem 2 in \cite{chernozhukov2015comparison}, where the upper bound is of order $\left(\Delta_\Sigma \log^2 p\right)^{1/3}$ under the same conditions.

\section{Proof of Section \ref{sec:GA.cov}}\label{appB}

\subsection{Proof of Lemma \ref{lm:mtg.approx}}\label{sec:pf.mtg.approx}

The approach to the proof involves the representation of $n^{1/2}(S_n)_{(s,t)} - (T_n)_{(s,t)}$ as a quadratic form for each $(s,t)$ pair, using Lemma \ref{lm:hanson.wright} to establish a concentration bound, and employing a union bound to derive Lemma \ref{lm:mtg.approx}.

From \eqref{eq:D}, the martingale approximation residual is
\begin{equation*}
n^{1/2}S_n - T_n = \sum_{i=1}^n \mathbb{E}[\mathcal{X}_i\mid \mathcal{F}^0] - \sum_{i=n+1}^\infty \mathbb{E}[\mathcal{X}_i\mid \mathcal{F}^n] + \sum_{i=n+1}^\infty \mathbb{E}[\mathcal{X}_i\mid \mathcal{F}^0].
\end{equation*}
For Gaussian linear process \eqref{LinearProcess}, for $1\leq s,t\leq p$, we have
\begin{equation*}
\mathbb{E}[(\mathcal{X}_i)_{(s,t)}\mid \mathcal{F}^0]
= \sum_{u=1}^d \sum_{v=1}^d \sum_{j=-\infty}^0 \sum_{k=-\infty}^0 (\epsilon_{j})_u (A_{i-j})_{su}(A_{i-k})_{tv}(\epsilon_k)_v
        - \sum_{u=1}^d\sum_{j=-\infty}^0 (A_{i-j})_{su}(A_{i-j})_{tu}
\end{equation*}
and
\begin{equation*}
\mathbb{E}[(\mathcal{X}_i)_{(s,t)}\mid \mathcal{F}^n]
= \sum_{u=1}^d \sum_{v=1}^d \sum_{j=-\infty}^n \sum_{k=-\infty}^n (\epsilon_{j})_u (A_{i-j})_{su}(A_{i-k})_{tv}(\epsilon_k)_v 
         - \sum_{u=1}^d\sum_{j=-\infty}^n (A_{i-j})_{su}(A_{i-j})_{tu}.  
\end{equation*}
Let $\bm{\epsilon} = [\epsilon_n^\top, \epsilon_{n-1}^\top, \epsilon_{n-2}^\top,\cdots]^\top$, where all entries are independent standard Gaussian variables.
Let $\bm{G}^{(s,t)}$ be a block matrix that consists of $d\times d$ blocks $G^{(s,t)}_{jk}$, $j,k\leq n$, with
\begin{equation*}
    G^{(s,t)}_{jk} = 
    \begin{cases}
        -\sum_{i=n+1}^\infty (A_{i-j})_{s\cdot}(A_{i-k})_{t\cdot}^\top, & j\vee k \geq 1;\\
        \sum_{i=1}^n (A_{i-j})_{s\cdot}(A_{i-k})_{t\cdot}^\top, & j\vee k \leq 0.
    \end{cases}
\end{equation*}
With this notation, we have
\begin{equation*}
    n^{1/2}(S_n)_{(s,t)} - (T_n)_{(s,t)} = \bm{\epsilon}^\top \bm{G}^{(s,t)}\bm{\epsilon} - \mathbb{E}[\bm{\epsilon}^\top \bm{G}^{(s,t)}\bm{\epsilon}].
\end{equation*}
We use the Hanson-Wright inequality (Lemma \ref{lm:hanson.wright}) to bound the quadratic form.
Note that for any standard Gaussian random variable $\epsilon_i$, the Orlicz 2-norm is $\|\epsilon_i\|_{\psi_2} = \sqrt{2}$.
We need to bound the operator and the Frobenius norms of $\bm{G}^{(s,t)}$. Since $|\bm{G}^{(s,t)}|_{\op} \leq |\bm{G}^{(s,t)}|_{\F}$, we focus on the latter. It suffices to compute $\sum_{j\leq n}\sum_{k\leq n} |G^{(s,t)}_{jk}|^2_{\F}$. By Condition \ref{cond.A} and Cauchy-Schwarz inequality, up to a constant related to $C_0$,
\begin{align*}
        \sum_{j=1}^n\sum_{k=1}^n |G_{jk}^{(s,t)}|_{\F}^2 &\lesssim \sum_{j=1}^n\sum_{k=1}^n\left(\sum_{i=n+1}^\infty (i-j)^{-\beta}(i-k)^{-\beta} \right)^2\\
        &\leq \sum_{j=1}^n\sum_{k=1}^n \sum_{i=n+1}^\infty (i-j)^{-2\beta} \sum_{i'=n+1}^\infty (i'-k)^{-2\beta} \\
        & = \left(\sum_{j=0}^{n-1}\sum_{i=1}^\infty (i+j)^{-2\beta}\right)^2\\
    &= \left(\sum_{j=0}^{n-1}\sum_{i=j+1}^\infty i^{-2\beta}\right)^2 \\
    & \asymp \left(\sum_{j=1}^nj^{-2\beta+1}\right)^2 \\
    & = \begin{cases}
    O(1), & \beta>1\\
    O(\log^2 n), & \beta=1\\
    O(n^{-4\beta+4}), & 3/4<\beta<1.
    \end{cases}
\end{align*}
Similarly,
\begin{align*}
        \sum_{j=-\infty}^0\sum_{k=-\infty}^0 |G_{jk}^{(s,t)}|_{\F}^2 
        &\lesssim \sum_{j=-\infty}^0\sum_{k=-\infty}^0\left(\sum_{i=1}^n (i-j)^{-\beta}(i-k)^{-\beta} \right)^2\\
        &\leq \sum_{j=-\infty}^0\sum_{k=-\infty}^0 \sum_{i=1}^n (i-j)^{-2\beta} \sum_{i'=1}^n (i'-k)^{-2\beta}\\
        &= \left(\sum_{j=0}^\infty\sum_{i=1}^n (i+j)^{-2\beta}\right)^2 \\
        & = \left(\sum_{i=1}^n\sum_{j=i}^\infty j^{-2\beta}\right)^2 \\
        & \asymp \left(\sum_{i=1}^n i^{-2\beta+1}\right)^2\\
    &= \begin{cases}
    O(1),& \beta>1\\
    O(\log^2 n), & \beta=1\\
    O(n^{-4\beta+4}), & 3/4<\beta<1
    \end{cases}
\end{align*}
and
\begin{align*}
        \sum_{j=1}^n\sum_{k=-\infty}^0 |G_{jk}^{(s,t)}|_{\F}^2 &\lesssim \sum_{j=1}^n\sum_{k=-\infty}^0\left(\sum_{i=n+1}^
        \infty(i-j)^{-\beta}(i-k)^{-\beta} \right)^2\\
        &= \sum_{j=0}^{n-1} \sum_{i=1}^\infty \sum_{i'=1}^\infty (i+j)^{-\beta}(i'+j)^{-\beta} \sum_{k=n}^\infty(i+k)^{-\beta}(i'+k)^{-\beta}\\
        &\leq \sum_{j=0}^{n-1} \sum_{i=1}^\infty \sum_{i'=1}^\infty (i+j)^{-\beta}(i'+j)^{-\beta} \sum_{k=j}^\infty(i+k)^{-\beta}(i'+k)^{-\beta}\\
        &\leq \sum_{j=0}^{n-1} \sum_{i=1}^\infty \sum_{i'=1}^\infty (i+j)^{-\beta}(i'+j)^{-\beta} \sqrt{\sum_{k=j}^\infty(i+k)^{-2\beta}\sum_{k'=j}^\infty(i'+k')^{-2\beta}}\\
        &\asymp \sum_{j=0}^{n-1} \left(\sum_{i=1}^\infty (i+j)^{-2\beta+1/2}\right)^2 \\
        & \asymp \sum_{j=1}^n j^{-4\beta+3}\\
        &= \begin{cases}
        O(1), & \beta>1\\
        O(\log(n)), & \beta=1\\
        O(n^{-4\beta+4}), & 3/4<\beta<1.
        \end{cases}
\end{align*}
Therefore, we have
\begin{equation*}
    |\bm{G}^{(s,t)}|^2_{\F} \lesssim 
    \begin{cases}
    O(1), & \beta>1\\
    O(\log(n)), & \beta=1\\
    O(n^{-4\beta+4}), & 3/4<\beta<1,
    \end{cases}
\end{equation*}
and $|\bm{G}^{(s,t)}|^2_{\F} \lesssim nQ_1(n)$. By Hanson-Wright inequality \cite{rudelson2013hanson}, we have
\begin{equation*}
    \mathbb{P}\left(|n^{1/2}(S_n)_{(s,t)} - (T_n)_{(s,t)}|>\delta\right) \leq 2\exp\left[-c\left(\frac{\delta^2}{nQ_1(n)} \wedge \frac{\delta}{n^{1/2}Q_1^{1/2}(n)}\right)\right],
\end{equation*}
where $c>0$ is a universal constant.
The proof concludes by substituting $\delta$ with $n^{1/2}\delta$ and using the union bound over all $p^2$ pairs $(s,t)$.

\subsection{Proof of Lemma \ref{lm:m.dep.approx}}\label{sec:pf.m.dep.approx}

Following the approach used in the proof of Lemma \ref{lm:mtg.approx}, the plan is to express $(T_n)_{(s,t)} - (\tilde{T}_{n,m})_{(s,t)}$ as a quadratic form and then utilize the Hanson-Wright inequality (Lemma \ref{lm:hanson.wright}).

Let $\bm{H}^{(s,t)}$ be a block matrix that consists of $d\times d$ blocks $H^{(s,t)}_{ij}$, $i,j \leq n$, with
\begin{equation*}
    H^{(s,t)}_{ij} = 
    \begin{cases}
        \sum_{l=0}^\infty (A_{l+j-i})_{s\cdot}(A_l)_{t\cdot}^\top, & 1\leq j\leq n, i\leq j-m\\
        \sum_{l=0}^\infty (A_l)_{s\cdot}^\top(A_{l+i-j})_{t\cdot}^\top, &1\leq i\leq n, j\leq i-m\\
        0, & \text{otherwise}
    \end{cases}
\end{equation*}
Recall \eqref{EQ_Di} and \eqref{EQ_Di_tilde} for the definitions of $D_i$ and $\Tilde{D}_{i,m}$ in the context of the linear process \eqref{LinearProcess}. 
Then 
\begin{align*}
        (T_n)_{(s,t)} - (\tilde{T}_{n,m})_{(s,t)} &= \sum_{i=1}^n \sum_{j=m}^\infty \sum_{u=1}^d \sum_{v=1}^d (\epsilon_{i-j})_u \left(\sum_{l=0}^\infty (A_{l+j})_{su} (A_l)_{tv}\right) (\epsilon_i)_v\\
        &\hspace{0.4cm} + \sum_{i=1}^n \sum_{j=m}^\infty \sum_{u=1}^d \sum_{v=1}^d (\epsilon_{i})_u \left(\sum_{l=0}^\infty (A_{l})_{su} (A_{l+j})_{tv}\right) (\epsilon_{i-j})_v\\
        &= \sum_{j=1}^n \sum_{i=-\infty}^{j-m}\sum_{u=1}^d\sum_{v=1}^d (\epsilon_i)_u \left(\sum_{l=0}^\infty (A_{l+j-i})_{su} (A_l)_{tv}\right)(\epsilon_j)_v\\
        &\hspace{0.4cm} + \sum_{i=1}^n \sum_{j=-\infty}^{i-m} \sum_{u=1}^d \sum_{v=1}^d (\epsilon_{i})_u \left(\sum_{l=0}^\infty (A_{l})_{su} (A_{l+i-j})_{tv}\right) (\epsilon_{j})_v\\
        &= \bm{\epsilon}^\top \bm{H}^{(s,t)} \bm{\epsilon},
\end{align*}
where $\bm \epsilon$ was defined in Appendix \ref{sec:pf.mtg.approx}.
Note that $(T_n)_{(s,t)} - (\tilde{T}_{n,m})_{(s,t)}$ has zero mean.
To apply Hanson-Wright inequality to bound this quadratic form, it is again necessary to bound the Frobenius norm and operator norm of the matrix $\bm H^{(s,t)}$.
Under Condition \ref{cond.A},
\begin{align*}
        &|\bm{H}^{(s,t)}|^2_{\F} = \sum_{j=1}^n \sum_{i=-\infty}^{j-m} |H^{s,t}_{ij}|^2_{\F} + \sum_{i=1}^n \sum_{j=-\infty}^{i-m}|H^{s,t}_{ij}|^2_{\F}\\
        &\lesssim \sum_{j=1}^n \sum_{i=-\infty}^{j-m} \left(\sum_{l=0}^\infty(l+j-i)^{-\beta}(1\vee l)^{-\beta}\right)^2  + \sum_{i=1}^n \sum_{j=-\infty}^{i-m} \left(\sum_{l=0}^\infty(l+i-j)^{-\beta}(1\vee l)^{-\beta}\right)^2 \\
        &= 2n\sum_{u=m}^\infty \left(\sum_{l=0}^\infty (l+u)^{-\beta} (1\vee l)^{-\beta}\right)^2\\
        &= 2n\sum_{u=m}^\infty \left(\sum_{l=0}^{u-1} (l+u)^{-\beta} (1\vee l)^{-\beta}\right)^2 + 2n\sum_{u=m}^\infty \left(\sum_{l=u}^\infty (l+u)^{-\beta} (l)^{-\beta}\right)^2\\
        &\leq 2n\sum_{u=m}^\infty u^{-2\beta}\left(1+\sum_{l=1}^{u-1} l^{-\beta}\right)^2 + 2n\sum_{u=m}^\infty \left(\sum_{l=u}^\infty l^{-2\beta}\right)^2\\
        &\asymp n\sum_{u=m}^\infty \left(u^{-2\beta}(1\vee u^{-2\beta+2}) + u^{-4\beta+2}\right) \asymp nm^{(-4\beta+3)\vee(-2\beta+1)}.
\end{align*}
Recall the definition of $Q_2(m)$, this is equivalent to $|\bm{H}^{(s,t)}|^2_{\F} \lesssim nQ_2(m)$.
The operator norm $|\bm{H}^{(s,t)}|_{\op}$ also has the same order, since it is upper bounded by Frobenius norm $|\bm{H}^{(s,t)}|_{\F}$.
By Hanson-Wright inequality, i.e., Lemma \ref{lm:hanson.wright}, we conclude that
\begin{equation*}
    \mathbb{P}\left(|(T_n)_{(s,t)} - (\tilde{T}_{n,m})_{(s,t)}|>\delta\right) \leq 2\exp\left[-c\left(\frac{\delta^2}{nQ_2(m)} \wedge \frac{\delta}{n^{1/2}Q_2^{1/2}(m)}\right)\right],
\end{equation*}
where $c>0$ is a universal constant.
The proof is finished by replacing $\delta$ by $n^{1/2}\delta$ and applying union bound over all $p^2$ possible pairs of $(s,t)$.

\subsection{Proof of Identity (\ref{eq:var.identity})}\label{app:var.identity}

The stochasticity of $F_{j,0}$ comes from $\bm{\epsilon}_{j,0}$ and $\bm{\epsilon}_{j,1}$, $F_{j,1}$ from $\bm{\epsilon}_{j,1}$ and $\bm{\epsilon}_{j,2}$, and $F_{j,2}$ from $\bm{\epsilon}_{j,2}$ and $\bm{\epsilon}_{j+1,0}$. When $\bm{\epsilon}_0$ is fixed as $\bm{\epsilon}_0^*$,  $F_{j,0}$ and $F_{j,2}$ can be written as $F_{j,0}(\bm{\epsilon}_{j,0}^*)$ and $F_{j,2}(\bm{\epsilon}_{j+1,0}^*)$. For the Gaussian linear process \eqref{LinearProcess}, each entry is
\begin{align}
    \begin{split}
    (F_{j,0}&(\bm{\epsilon}_{j,0}^*))_{(s,t)} = \sum_{i=3jm+1}^{(3j+1)m} \Bigg[ \Bigg.\sum_{l=0}^\infty (A_l)_{s\cdot}\epsilon_{i}(A_l)_{t\cdot}\epsilon_{i}-\Sigma_{st}\\
    &+\sum_{k=1}^{i-3jm-1}\sum_{l=0}^\infty (A_{l+k})_{s\cdot}\epsilon_{i-k}(A_l)_{t\cdot}\epsilon_{i} + \sum_{k=1}^{i-3jm-1}\sum_{l=0}^\infty (A_l)_{s\cdot}\epsilon_{i}(A_{l+k})_{t\cdot}\epsilon_{i-k}\\
    &+\sum_{k=i-3jm}^{m-1}\sum_{l=0}^\infty (A_{l+k})_{s\cdot}\epsilon^*_{i-k}(A_l)_{t\cdot}\epsilon_{i} +\sum_{k=i-3jm}^{m-1}\sum_{l=0}^\infty (A_l)_{s\cdot}\epsilon_{i}(A_{l+k})_{t\cdot}\epsilon^*_{i-k}\Bigg. \Bigg].\label{eq:F.j.0}
    \end{split}
\end{align}
Note that $F_{j,0}(\bm{\epsilon}_{j,0}^*)$ is still mean-zero given any fixed $\bm{\epsilon}^*_0$. We also have
\begin{align}
    \begin{split}
     (F_{j,1}&)_{(s,t)} = \sum_{i=(3j+1)m+1}^{(3j+2)m} \Bigg[ \Bigg.\sum_{l=0}^\infty (A_l)_{s\cdot}\epsilon_{i}(A_l)_{t\cdot}\epsilon_{i}-\Sigma_{st}\\
     &+\sum_{k=1}^{m-1}\sum_{l=0}^\infty (A_{l+k})_{s\cdot}\epsilon_{i-k}(A_l)_{t\cdot}\epsilon_{i} + \sum_{k=1}^{m-1}\sum_{l=0}^\infty (A_l)_{s\cdot}\epsilon_{i}(A_{l+k})_{t\cdot}\epsilon_{i-k}\Bigg. \Bigg],\label{eq:F.j.1}
    \end{split}
\end{align}
which is mean-zero, and unrelated to $\bm{\epsilon}^*_0$. 
Note that $F_{j,2}(\bm{\epsilon}_{j+1,0}^*)$ is not mean-zero conditioned on $\bm{\epsilon}_0^*$. Denote $\mathbb{E}[F_{j,2}(\bm{\epsilon}_{j+1,0})\mid \bm{\epsilon}_{j+1,0}]$ by $\Lambda_{j,2}(\bm{\epsilon}_{j+1,0})$. Then we have
\begin{align}
    \begin{split}
    &(F_{j,2}(\bm{\epsilon}_{j+1,0}^*))_{(s,t)} - (\Lambda_{j,2}(\bm{\epsilon}^*_{j+1,0}))_{(s,t)}\\ 
    &\quad = \sum_{i=(3j+2)m+1}^{(3j+3)m} \Bigg[ \Bigg. \sum_{k=i-(3j+2)m}^{m-1} \sum_{l=0}^\infty (A_{l+k})_{s\cdot}\epsilon_{i-k}(A_l)_{t\cdot}\epsilon^*_i\\
    &\quad \hspace{50pt}+ \sum_{k=i-(3j+2)m}^{m-1} \sum_{l=0}^\infty (A_{l})_{s\cdot}\epsilon^*_{i}(A_{l+k})_{t\cdot}\epsilon_{i-k} \Bigg. \Bigg].\label{eq:F.j.2}
    \end{split}
\end{align}
Therefore, for any fixed $\bm{\epsilon}_0 = \bm{\epsilon}_0^*$, we can rewrite
\begin{align}
    \begin{split}
        V_0^*(\bm{\epsilon}^*_{1,0}) &= \mathbb{E}[F_{0,1} F_{0,1}^\top] + \mathbb{E}[F_{0,1} (F_{0,2}(\bm{\epsilon}^*_{1,0}) - \Lambda_{0,2}(\bm{\epsilon}^*_{1,0}))^\top]\\
        &+ \mathbb{E}[(F_{0,2}(\bm{\epsilon}^*_{1,0}) - \Lambda_{0,2}(\bm{\epsilon}^*_{1,0})) F_{0,1}^\top]\\
        &+ \mathbb{E}[(F_{0,2}(\bm{\epsilon}^*_{1,0}) - \Lambda_{0,2}(\bm{\epsilon}^*_{1,0}))(F_{0,2}(\bm{\epsilon}^*_{1,0}) - \Lambda_{0,2}(\bm{\epsilon}^*_{1,0}))^\top]. \label{eq:V0.cond}
    \end{split}
\end{align} 
And for any $1\leq j\leq q_n$, note that $F_{j,0}$ and $F_{j,2}$ are independent, we can rewrite
\begin{align}
    \begin{split}
        V_j^*(\bm{\epsilon}_{j,0}, \bm{\epsilon}_{j+1,0}) &= \mathbb{E}[F_{j,0}(\bm{\epsilon}_{j,0}) F_{j,0}(\bm{\epsilon}_{j,0})^\top] + \mathbb{E}[F_{j,0}(\bm{\epsilon}_{j,0}) F_{j,1}^\top] + \mathbb{E}[F_{j,1} F_{j,0}(\bm{\epsilon}_{j,0})^\top]\\
        &+ \mathbb{E}[F_{j,1}F_{j,1}^\top] + \mathbb{E}[F_{j,1}(F_{j,2}(\bm{\epsilon}^*_{j+1,0}) - \Lambda_{j,2}(\bm{\epsilon}^*_{j+1,0}))^\top]\\
        &+ \mathbb{E}[(F_{j,2}(\bm{\epsilon}^*_{j+1,0}) - \Lambda_{j,2}(\bm{\epsilon}^*_{j+1,0})) F_{j,1}^\top]\\
        &+ \mathbb{E}[(F_{j,2}(\bm{\epsilon}^*_{j+1,0}) - \Lambda_{j,2}(\bm{\epsilon}^*_{j+1,0}))(F_{j,2}(\bm{\epsilon}^*_{j+1,0}) - \Lambda_{j,2}(\bm{\epsilon}^*_{j+1,0}))^\top].\label{eq:Vj.cond}
    \end{split}
\end{align}
Thus we have
\begin{align}
    \begin{split}
        V_{q_n+1}&(\bm{\epsilon}^*_{q_n+1,0}) = \mathbb{E}[F_{0,1} F_{0,1}^\top] + \mathbb{E}[F_{0,1} (F_{0,2}(\bm{\epsilon}^*_{q_n+1,0}) - \Lambda_{0,2}(\bm{\epsilon}^*_{q_n+1,0})^\top]\\
        &+ \mathbb{E}[(F_{0,2}(\bm{\epsilon}^*_{q_n+1,0}) - \Lambda_{0,2}(\bm{\epsilon}^*_{q_n+1,0})) F_{0,1})^\top]\\
        &+ \mathbb{E}[(F_{0,2}(\bm{\epsilon}^*_{q_n+1,0}) - \Lambda_{0,2}(\bm{\epsilon}^*_{q_n+1,0}))(F_{0,2}(\bm{\epsilon}^*_{q_n+1,0}) - \Lambda_{0,2}(\bm{\epsilon}^*_{q_n+1,0}))^\top]. \label{eq:V0.uni}
    \end{split}
\end{align} 
and
\begin{align}
    \begin{split}
        V_j(\bm{\epsilon}_{j,0}) &= \mathbb{E}[F_{j,0}(\bm{\epsilon}_{j,0}) F_{j,0}(\bm{\epsilon}_{j,0})^\top] + \mathbb{E}[F_{j,0}(\bm{\epsilon}_{j,0}) F_{j,1}^\top] + \mathbb{E}[F_{j,1} F_{j,0}(\bm{\epsilon}_{j,0})^\top]\\
        &+ \mathbb{E}[F_{j,1}F_{j,1}^\top] + \mathbb{E}[F_{j,1}(F_{j,2}(\bm{\epsilon}^*_{j,0}) - \Lambda_{j,2}(\bm{\epsilon}^*_{j,0}))^\top]\\
        &+ \mathbb{E}[(F_{j,2}(\bm{\epsilon}^*_{j,0}) - \Lambda_{j,2}(\bm{\epsilon}^*_{j,0})) F_{j,1}^\top]\\
        &+ \mathbb{E}[(F_{j,2}(\bm{\epsilon}^*_{j,0}) - \Lambda_{j,2}(\bm{\epsilon}^*_{j,0}))(F_{j,2}(\bm{\epsilon}^*_{j,0}) - \Lambda_{j,2}(\bm{\epsilon}^*_{j,0}))^\top].\label{eq:Vj.uni}
    \end{split}
\end{align}
From \eqref{eq:F.j.0}, \eqref{eq:F.j.1} and \eqref{eq:F.j.2}, we observe that $\mathbb{E}[(F_{j,2}(\cdot) - \Lambda_{j,2}(\cdot))(F_{j,2}(\cdot) - \Lambda_{j,2}(\cdot))^\top]$ and $\mathbb{E}[(F_{j,2}(\cdot) - \Lambda_{j,2}(\cdot))F_{j,1}^\top]$
are invariant with $j$. Therefore by \eqref{eq:V0.cond}, \eqref{eq:Vj.cond}, \eqref{eq:V0.uni} and \eqref{eq:Vj.uni}, we conclude that
\begin{equation*}
V^*_0(\bm{\epsilon}^*_{1,0}) + \sum_{j=1}^{q_n} V^*_j(\bm{\epsilon}^*_{j,0},\bm{\epsilon}^*_{j+1,0}) = \sum_{j=1}^{q_n} V_j(\bm{\epsilon}^*_{j,0}) + V_{q_n+1}(\bm{\epsilon}^*_{q_n+1,0}).
\end{equation*}

\subsection{Proof of Lemma \ref{lm:cond.CCK}}\label{sec:pf.cond.CCK}

Given any fixed value of $\bm{\epsilon}_0$, define
\begin{equation*}
    B_j^{\bm{\epsilon}_0} = B_j - \mathbb{E}[B_j\mid \bm{\epsilon}_0].
\end{equation*}
and
\begin{equation*}
    \tilde{T}_{n,m}^{\bm{\epsilon}_0} = \sum_{j=0}^{q_n}B_j^{\bm{\epsilon}_0} = \tilde{T}_{n,m} - \mathbb{E}[\tilde{T}_{n,m}\mid \bm{\epsilon}_0].
\end{equation*}
Also note that given any $\bm{\epsilon}_0$, the variance of $B_0^{\bm{\epsilon}_0}$ is $V_0^*(\bm{\epsilon}_{1,0})$, and the variance of $B_j^{\bm{\epsilon}_0}$ for $1\leq j\leq q_n$ is $V_j^*(\bm{\epsilon}_{j,0}, \bm{\epsilon}_{j+1,0})$.
If we can show that for all $1\leq s,t\leq p$,
\begin{equation}\label{eq:cond.CCK.M1}
    (q_n+1)^{-1}\sum_{j=0}^{q_n}\mathbb{E}[(B_j^{\bm{\epsilon}_0})_{(s,t)}^2 \mid \bm{\epsilon}_0] \geq b_1^2
\end{equation}
for some constant $b_1>0$ related to $C_0$ and $c_0$,
\begin{equation}\label{eq:cond.CCK.M2}
    (q_n+1)^{-1}\sum_{j=0}^{q_n}\mathbb{E}[(B_j^{\bm{\epsilon}_0})_{(s,t)}^4 \mid \bm{\epsilon}_0] \leq b_2^2\beta_n^2
\end{equation}
for a sequence of constants $\beta_n$ with a constant $b_2>0$ related to $C_0$ and $c_0$ such that $b_2 \geq b_1$, and
\begin{equation}\label{eq:cond.CCK.E1}
    \mathbb{E}[\exptext(|(B_j^{\bm{\epsilon}_0})_{(s,t)}| / \beta_n) \mid \bm{\epsilon}_0] \leq 2
\end{equation}
for any $0\leq j\leq q_n$,
then by Theorem \ref{prop:CCK17}, we have
\begin{equation*}
    \sup_{u\in\mathbb{R}}\left| \mathbb{P}\left({(q_n+1)}^{-1/2}|\tilde{T}_{n,m}^{\bm{\epsilon}_0}|_\infty \geq u \mid \bm{\epsilon}_0\right) - \mathbb{P}\left({(q_n+1)}^{-1/2}|{T}_{n,m}^*|_\infty \geq u \mid \bm{\epsilon}_0\right)\right|
    \leq C\left(\frac{\beta_n^2\log^5(pn)}{n}\right)^{1/4},
\end{equation*}
where $C>0$ is a constant only related to $b_1$ and $b_2$, therefore to $C_0$ and $c_0$.
The proof of Lemma \ref{lm:cond.CCK} is then finished by rescaling $u$ and observing that $T^*_{n,m}$ is identically distributed with $T^Y_{n,m}$ given $\bm{\epsilon}_0$, due to the identity \eqref{eq:var.identity}.
Therefore, it remains to identify the rate $\beta_n$ such that the three conditions \eqref{eq:cond.CCK.M1}, \eqref{eq:cond.CCK.M2}, and \eqref{eq:cond.CCK.E1} hold with high probability.
Define a function $\mathcal{V}(\bm{\epsilon}_0) \in \mathbb{R}^{p^2}$, such that each entry can be written as
\begin{equation*}
    (\mathcal{V}(\bm{\epsilon}_0))_{(s,t)} = (q_n+1)^{-1}\sum_{j=0}^{q_n}\mathbb{E}[(B_j^{\bm{\epsilon}_0})_{(s,t)}^2\mid  \bm{\epsilon}_0].
\end{equation*}
The following lemma provides a concentration bound in the probability space of $\bm{\epsilon}_0$.

\begin{lemma}\label{lm:V.conc}
    Under Condition \ref{cond.A}, for a constant $c_2>0$ related to $C_0$,
    \begin{equation*}
        \mathbb{P}(|\mathcal{V}(\bm{\epsilon}_0) - \mathbb{E}[\mathcal{V}(\bm{\epsilon}_0)]|_\infty > \delta) \leq 4 p^2 \exp\left\{-c_2\left(\frac{n\delta^2}{m^5} \wedge \frac{n^{1/2}\delta}{m^{5/2}}\right)\right\}.
    \end{equation*}
\end{lemma}
\begin{proof}[Proof of Lemma \ref{lm:V.conc}]
    By the identity \eqref{eq:var.identity},
    \begin{equation*}
        (\mathcal{V}(\bm{\epsilon}_0))_{(s,t)} = (q_n+1)^{-1}\sum_{j=1}^{q_n+1} (V_j(\bm{\epsilon}_{j,0}))_{(s,t),(s,t)}.
    \end{equation*}
    From \eqref{eq:F.j.0}, \eqref{eq:F.j.1} and \eqref{eq:F.j.2}, it can be calculated that
    \begin{align*}
        \mathbb{E}[F_{j,0}(\bm{\epsilon}_{j,0})F_{j,1}^\top] = 0,
    \end{align*}
    and
    \begin{align*}
        \mathbb{E}[F_{j,1} (F_{j,2}(\bm{\epsilon}_{j,0}) - \Lambda_{j,2}(\bm{\epsilon}_{j,0}))^\top] = 0.
    \end{align*}
    We can thus rewrite $\mathcal{V}(\bm{\epsilon}_{j,0})$ as
    \begin{equation}
        (q_n+1)(\mathcal{V}(\bm{\epsilon}_{0}))_{(s,t)} = \sum_{j=1}^{q_n}\mathbb{E}[(F_{j,0}(\bm{\epsilon}_{j,0}))_{(s,t)}^2] + \sum_{j=0}^{q_n} \mathbb{E}[(F_{j,1})_{(s,t)}^2]
        + \sum_{j=0}^{q_n} \mathbb{E}[(F_{j,2}(\bm{\epsilon}_{j,0}) - \Lambda_{j,2}(\bm{\epsilon}_{j,0}))_{(s,t)}^2].\label{eq:V.partition}
    \end{equation}
    Note that $\bm{\epsilon}_0$ is only involved in the first and third terms on the right hand side of \eqref{eq:V.partition}, and that the second term is a constant. Recall \eqref{eq:F.j.0} and \eqref{eq:F.j.2}, we can rewrite each entry of $\mathcal{V}(\bm{\epsilon}_0) - \mathbb{E}[\mathcal{V}(\bm{\epsilon}_0)]$ as a centered quadratic form with respect to $\bm{\epsilon}_0$,
    \begin{equation}\label{eq:V.to.Upsilon}
        (q_n+1)(\mathcal{V}(\bm{\epsilon}_0) - \mathbb{E}[\mathcal{V}(\bm{\epsilon}_0)]) = \Upsilon_0(\bm{\epsilon}_0) - \mathbb{E}[\Upsilon_0(\bm{\epsilon}_0)] + \Upsilon_2(\bm{\epsilon}_0) - \mathbb{E}[\Upsilon_2(\bm{\epsilon}_0)].
    \end{equation}
    Here,
    \[
    \Upsilon_0(\bm\epsilon_0) = \sum_{j=1}^{q_n}\sum_{i=3jm+1}^{(3j+1)m}\sum_{k=i-m+1}^{3jm}\sum_{k'=i-m+1}^{3jm}\sum_{l=0}^\infty\sum_{l'=0}^\infty (I_{p^2} + K_{pp}) \bigg[ (D_{u_0} \otimes D_{v_0})\vectorize(L_0) + \text{diag}(L_0) \otimes (u_0 \circ v_0) \bigg]
    \]
    where $K_{pp}$ is the commutation matrix such that $K_{pp} \vectorize(A) = \vectorize(A^\top)$, $u_0 = A_{l+i-k}\epsilon_k$, $v_0 = A_{l'+i-k'}\epsilon_{k'}$, $L_0 = A_l A_{l'}^\top$, $\otimes$ is the Kronecker product, and $\circ$ is the Hadamard product.
    Likewise, we have
    \[
    \Upsilon_2(\bm\epsilon_2) = \sum_{j=0}^{q_n}\sum_{i=(3j+2)m+1}^{(3j+3)m}\sum_{i'=(3j+2)m+1}^{(3j+3)m}\sum_{k=(i\vee i')-m+1}^{(3j+2)m}\sum_{l=0}^\infty\sum_{l'=0}^\infty (I_{p^2} + K_{pp}) \bigg[ (D_{u_2} \otimes D_{v_2})\vectorize(L_2) + \text{diag}(L_2) \otimes (u_2 \circ v_2) \bigg],
    \]
    where $u_2 = A_l \epsilon_i$, $v_2 = A_{l'} \epsilon_{i'}$, and $L_2 = A_{l+i-k} A_{l'+i'-k}^\top$.
    Both quantities can be written as a quadratic form. Let
    \[
    \vec{\bm{\epsilon}}_0 = [\epsilon^\top_{2m+1},\cdots,\epsilon^\top_{3m}, \epsilon^\top_{5m+1},\cdots,\epsilon^\top_{6m},\cdots, \epsilon^\top_{(3q_n+2)m+1}, \cdots, \epsilon^\top_n]^\top,
    \]
    which concatenates $\epsilon_{k}\in\mathbb{R}^d$ for $k\in\mathcal{I}_0$. 
    We then have $(\Upsilon_0(\bm{\epsilon_0}))_{(s,t)} = \vec{\bm{\epsilon}}_0^\top \bm{\Gamma}^{(s,t)}_{0} \vec{\bm{\epsilon}}_0$, where $\bm{\Gamma}^{(s,t)}_{0}$ is a matrix with $|\mathcal{I}_0|\times |\mathcal{I}_0|$ many blocks, each with size $d\times d$. Each block is indexed by $k,k'\in\mathcal{I}_0$. The blocks can be expressed as
    \[
    \bm{\Gamma}_{0,k,k'}^{(s,t)} = \sum_{i=3jm+1}^{((k\wedge k')+m) \wedge n}\sum_{l=0}^\infty \sum_{l'=0}^\infty U_{l+i-k} \left[ J (A_l A_{l'}^\top)_{\{s,t\}} J \right] U_{l'+i-k'}^\top
    \]
    if $k$ and $k'$ belong to the same index subset $\mathcal{I}_{j,0}$ for $1\leq j\leq q_n$. All other blocks are zero.
    Here, $J = \begin{bmatrix} 0 & 1 \\ 1 & 0 \end{bmatrix}$ is the exchange matrix, and $U_k \in \mathbb{R}^{d \times 2}$ is the matrix containing the $s$-th and $t$-th rows of $A_k$ as columns: $U_k = \begin{bmatrix} (A_k)_{s\cdot}^\top & (A_k)_{t\cdot}^\top \end{bmatrix}$.
    $M_{\{s,t\}}$ denotes the $2 \times 2$ submatrix of a matrix $M$ corresponding to indices $\{s,t\}$.
    Similarly, we have $(\Upsilon_2(\bm{\epsilon_0}))_{(s,t)} = \vec{\bm{\epsilon}}_0^\top \bm{\Gamma}^{(s,t)}_{2} \vec{\bm{\epsilon}}_0$, where $\bm{\Gamma}^{(s,t)}_{2}$ is also a matrix with $|\mathcal{I}_0|\times |\mathcal{I}_0|$ many blocks, each with size $d\times d$. Each block is indexed by $i,i'\in\mathcal{I}_0$. The blocks can be expressed as
    \[
    \bm{\Gamma}_{2,i,i'}^{(s,t)} = \sum_{k=(i\vee i')-m+1}^{(3j+2)m \wedge n} \sum_{l=0}^\infty \sum_{l'=0}^\infty U_{l} \left[ J (A_{l+i-k} A_{l'+i'-k}^\top)_{\{s,t\}} J \right] U_{l'}^\top
    \]
    if $i$ and $i'$ belong to the same index subset $\mathcal{I}_{j,0}$ for $1\leq j\leq q_n+1$. All other blocks are zero.
    The Frobenius norms of $\bm{\Gamma}_0^{(s,t)}$ and $\bm{\Gamma}_0^{(s,t)}$ are bounded under Condition \ref{cond.A} by
    \begin{align*}
        |\bm{\Gamma}_0^{(s,t)}|_{\F}^2 &= \sum_{j=1}^{q_n} \sum_{k = (3j-1)m+1}^{3jm} \sum_{k‘ = (3j-1)m+1}^{3jm} |\bm{\Gamma}_{0,k,k'}^{(s,t)}|_{\F}^2\\
        &\lesssim \sum_{j=1}^{q_n}\left(\sum_{k = (3j-1)m+1}^{3jm} \sum_{i=3jm+1}^{((k\wedge k')+m) \wedge n} \left(\sum_{l=0}^{\infty} (l+i-k)^{-\beta}(l\vee 1)^{-\beta}\right)^2\right)^2\\
        &\leq q_n \left(m^2\left(1+\sum_{l=1}^\infty l^{-2\beta}\right)^2\right)^2 = O(q_n m^4)
    \end{align*}
    and
    \begin{align*}
        |\bm{\Gamma}_2^{(s,t)}|_{\F}^2 &= \sum_{j=1}^{q_n+1} \sum_{i = (3j-1)m+1}^{3jm} \sum_{i‘ = (3j-1)m+1}^{3jm} |\bm{\Gamma}_{2,i,i'}^{(s,t)}|_{\F}^2\\
        &\lesssim \sum_{j=1}^{q_n+1}\left(\sum_{i = (3j-1)m+1}^{3jm} \sum_{k=(i\vee i')-m+1}^{(3j-1)m\wedge n} \left(\sum_{l=0}^{\infty} (l+i-k)^{-\beta}(l\vee 1)^{-\beta}\right)^2\right)^2\\
        &\leq (q_n+1) \left(m^2\left(1+\sum_{l=1}^\infty l^{-2\beta}\right)^2\right)^2 = O(q_n m^4).
    \end{align*}
    Therefore, by Lemma \ref{lm:hanson.wright}, for any $1\leq s,t\leq p$,
    \begin{align}
        \mathbb{P}\left(|(\Upsilon_{0}(\bm{\epsilon}_0))_{(s,t)} - \mathbb{E}[(\Upsilon_{0}(\bm{\epsilon}_0))_{(s,t)}]|\geq \delta\right) \leq 2\exp\left(-c_2\left(\frac{\delta^2}{q_n m^4}\wedge \frac{\delta}{q_n^{1/2}m^2}\right)\right),\label{eq:Upsilon.0.conc}\\
        \mathbb{P}\left(|(\Upsilon_{2}(\bm{\epsilon}_0))_{(s,t)} - \mathbb{E}[(\Upsilon_{2}(\bm{\epsilon}_0))_{(s,t)}]|\geq \delta\right) \leq 2\exp\left(-c_2\left(\frac{\delta^2}{q_n m^4}\wedge \frac{\delta}{q_n^{1/2}m^2}\right)\right)\label{eq:Upsilon.2.conc}.
    \end{align}
    Recall \eqref{eq:V.to.Upsilon}, note that $q_n\asymp n/m$, by rescaling $\delta$ to $(q_n+1)\delta$ and applying union bound on \eqref{eq:Upsilon.0.conc} and \eqref{eq:Upsilon.2.conc}, we finish the proof of Lemma \ref{lm:V.conc}.
\end{proof}

For the next step, we provide a bound to
\begin{equation*}
    \mathbb{E}[(\mathcal{V}(\bm{\epsilon}_0))_{(s,t)}] = (q_n+1)^{-1}\sum_{j=0}^{q_n} \mathbb{E}[(B_j)^2_{(s,t)}].
\end{equation*}
Recall that $B_j = F_{j,0} + F_{j,1} + F_{j,2}$. It is trivial to verify that $\mathbb{E}[F_{j,0}F_{j,1}^\top] = 0$, $\mathbb{E}[F_{j,1}F_{j,2}^\top] = 0$, and $\mathbb{E}[F_{j,0}F_{j,2}^\top] = 0$. Thus we have
\begin{equation}
    \mathbb{E}[(B_j)^2_{(s,t)}] = \mathbb{E}[(F_{j,0})^2_{(s,t)}] + \mathbb{E}[(F_{j,1})^2_{(s,t)}] + \mathbb{E}[(F_{j,2})^2_{(s,t)}].\label{eq:Bj.partition}
\end{equation}
Under Condition \ref{cond.A}, we have
\begin{align}
    \begin{split}
    \mathbb{E}[(F_{j,0})^2_{(s,t)}] &= \sum_{i=(3j-1)m+1}^{3jm} \sum_{k=0}^{m-1} \sum_{l=0}^\infty \sum_{l'=0}^\infty \bigg[\bigg. (A_{l+m-k} A_{l'+m-k}^\top)_{ss}(A_l A_{l'}^\top)_{tt}\\
    &\hspace{120pt}+ (A_{l+m-k} A_{l'+m-k}^\top)_{tt}(A_l A_{l'}^\top)_{ss}\\
    &\hspace{120pt}+ (A_{l+m-k} A_{l'+m-k}^\top)_{st}(A_l A_{l'}^\top)_{ts} \\
    &\hspace{120pt}+ (A_{l+m-k} A_{l'+m-k}^\top)_{ts}(A_l A_{l'}^\top)_{st}\bigg.\bigg]\\
    &\lesssim \sum_{i=(3j-1)m+1}^{3jm} \sum_{k=0}^{m-1}\left(\sum_{l=0}^\infty (l\vee 1)^{-2\beta}\right)^2 = O(m^2),\label{eq:E.F0.sq}
    \end{split}
\end{align}
and similarly $\mathbb{E}[(F_{j,1})^2_{(s,t)}] = O(m^2)$, $\mathbb{E}[(F_{j,2})^2_{(s,t)}] = O(m^2)$ up to a constant related to $C_0$.
Thus we have $\mathbb{E}[(\mathcal{V}(\bm{\epsilon}_0))_{(s,t)}] = O(m^2)$.
By Lemma \ref{lm:V.conc}, for any universal constant $\delta>0$, $|\mathcal{V}(\bm{\epsilon}_0) - \mathbb{E}[\mathcal{V}(\bm{\epsilon}_0)]|_\infty \leq \delta m^2$ with probability no less than $1-4p^2\exp(-c_2\delta n^{1/2}/m^{1/2})$, which converges to 1 as $n\rightarrow\infty$.
Therefore, $(\mathcal{V}(\bm{\epsilon}_0))_{(s,t)} \gtrsim m^2$ for any $1\leq s,t\leq p$ with high probability, proving the first condition \eqref{eq:cond.CCK.M1}.
Moreover, we also have $(\mathcal{V}(\bm{\epsilon}_0))_{(s,t)} \lesssim m^2$ for any $1\leq s,t\leq p$ with high probability. By Lemma \ref{lm:Lk.norm}, the second condition \eqref{eq:cond.CCK.M2} holds with $\beta_n \asymp m^2$.

To verify the third condition \eqref{eq:cond.CCK.E1}, by Taylor expansion, Fubini theorem and Lemma \ref{lm:Lk.norm},
\begin{align}
    \begin{split}
    \mathbb{E}[\exptext(|(B_j^{\bm{\epsilon}_0})_{(s,t)}| / m^2)\mid  \bm{\epsilon}_0] = \sum_{k=0}^\infty \frac{\mathbb{E}[|(B_j^{\bm{\epsilon}_0})_{(s,t)}|^k\mid \bm{\epsilon}_0]}{k!m^{2k}}\leq 1+ C_u \sum_{k=1}^\infty \frac{k^k (V_j^*(\bm{\epsilon}_{j,0}, \bm{\epsilon}_{j+1,0}))^{k/2}_{(s,t)}}{k!m^{2k}}.\label{eq:quantity.cond.E1}
    \end{split}
\end{align}
Similar to the proof of Lemma \ref{lm:V.conc}, since
\begin{equation*}
    (V_j^*(\bm{\epsilon}_{j,0}, \bm{\epsilon}_{j+1,0}))_{(s,t)} = \mathbb{E}[(F_{j,0}(\bm{\epsilon}_{j,0}))_{(s,t)}^2] + \mathbb{E}[(F_{j,1})_{(s,t)}^2]
    + \mathbb{E}[(F_{j,2}(\bm{\epsilon}_{j+1,0}) - \Lambda_{j,2}(\bm{\epsilon}_{j+1,0}))_{(s,t)}^2],
\end{equation*}
we can rewrite
\begin{equation*}
        (V_j^*(\bm{\epsilon}_{j,0}, \bm{\epsilon}_{j+1,0}))_{(s,t)} - \mathbb{E}[(V_j^*(\bm{\epsilon}_{j,0}, \bm{\epsilon}_{j+1,0}))_{(s,t)}]
        = \Upsilon_{j,0}(\bm{\epsilon}_{j,0}) - \mathbb{E}[\Upsilon_{j,0}(\bm{\epsilon}_{j,0})] + \Upsilon_{j,2}(\bm{\epsilon}_{j+1,0}) - \mathbb{E}[\Upsilon_{j,2}(\bm{\epsilon}_{j+1,0})],
\end{equation*}
where
\[
\Upsilon_{j,0}(\bm{\epsilon}_{j,0}) = \sum_{i=3jm+1}^{(3j+1)m}\sum_{k=i-m+1}^{3jm}\sum_{k'=i-m+1}^{3jm}\sum_{l=0}^\infty\sum_{l'=0}^\infty (I_{p^2} + K_{pp}) \bigg[ (D_{u_0} \otimes D_{v_0})\vectorize(L_0) + \text{diag}(L_0) \otimes (u_0 \circ v_0) \bigg],
\]
and
\[
\Upsilon_{j,2}(\bm{\epsilon}_{j+1,0}) = \sum_{i=(3j+2)m+1}^{(3j+3)m}\sum_{i'=(3j+2)m+1}^{(3j+3)m}\sum_{k=(i\vee i')-m+1}^{(3j+2)m}\sum_{l=0}^\infty\sum_{l'=0}^\infty (I_{p^2} + K_{pp}) \bigg[ (D_{u_2} \otimes D_{v_2})\vectorize(L_2) + \text{diag}(L_2) \otimes (u_2 \circ v_2) \bigg].
\]
Following a same procedure as in the proof of Lemma \ref{lm:V.conc}, by Lemma \ref{lm:hanson.wright}, we have the following concentration inequality for any $j$ and $s,t$. For any $\eta\geq m^2$,
\begin{equation*}
        \mathbb{P}\left(|(V_j^*(\bm{\epsilon}_{j,0}, \bm{\epsilon}_{j+1,0}))_{(s,t)} - \mathbb{E}[(V_j^*(\bm{\epsilon}_{j,0}, \bm{\epsilon}_{j+1,0}))_{(s,t)}]| > \eta \right)\leq 2\exp\left(-c_2\eta/m^2\right).
\end{equation*}
Note that $\mathbb{E}[(V_j^*(\bm{\epsilon}_{j,0}, \bm{\epsilon}_{j+1,0}))_{(s,t)}] = \mathbb{E}[(B_j)_{(s,t)}^2] = O(m^2)$.
Let us set $\eta = m^{2(1+\epsilon)}$ with any $0<\epsilon<1$.
Then we have $|(V_j^*(\bm{\epsilon}_{j,0}, \bm{\epsilon}_{j+1,0}))_{(s,t)} - \mathbb{E}[(V_j^*(\bm{\epsilon}_{j,0}, \bm{\epsilon}_{j+1,0}))_{(s,t)}]| \lesssim m^{2(1+\epsilon)}$ with probability no less than $1-2\exp(-c_2 m^{2\epsilon})$, which converges to 1 as $n\rightarrow\infty$.
By \eqref{eq:quantity.cond.E1} and Stirling's Formula, it follows that
\begin{equation*}
    \mathbb{E}[\exptext(|(B_j^{\bm{\epsilon}_0})_{(s,t)}| / m^2)\mid  \bm{\epsilon}_0] \leq 1+C_u \sum_{k=1}^\infty \frac{e^k}{\sqrt{2\pi k}m^{k(1-\epsilon)}}
\end{equation*}
with probability no less than $1-2\exp(-c_2 m^{2\epsilon})$.
This right-hand side is further bounded by $1+C_u\frac{1}{(m/e)^{1-\epsilon} - 1}$, which converges to 1 from above as $n\rightarrow\infty$, regardless of the choice of $\epsilon$.
For finite-sample performance, specifically, Condition \eqref{eq:cond.CCK.E1} holds with $\beta_n = m^2$, when $m$ is large enough such that $(m/e)^{1-\epsilon} \geq C_u + 1$.

We conclude that for any choice of $\delta>0$ and $0<\epsilon<1$, with probability no less than $1- 4p^2\exp(-c_2 \delta n^{1/2}/m^{1/2}) -2\exp(-c_2 m^{2\epsilon})$, all three conditions \eqref{eq:cond.CCK.M1}, \eqref{eq:cond.CCK.M2} and \eqref{eq:cond.CCK.E1} are satisfied.
The probability converges to 1 as $n\rightarrow\infty$ regardless of the choice of $\delta$ and $\epsilon$.
This finishes the proof of Lemma \ref{lm:cond.CCK}.

\subsection{Proof of Lemma \ref{lm:uncond.CCK}}\label{sec:pf.uncond.CCK}

To ease the notation, let $U_j = [V_j(\bm{\epsilon}_{j,0})]^{1/2}Y_j$.
It suffices to prove that the following three conditions hold under each $1\leq s,t\leq p$. First,
\begin{equation}\label{eq:cond.CCK2.M1}
    (q_n+1)^{-1}\sum_{j=1}^{q_n+1}\mathbb{E}[(U_j)^2_{(s,t)}] \geq b_1^2
\end{equation}
for a constant $b_1>0$ related to $C_0$ and $c_0$. Second,
\begin{equation}\label{eq:cond.CCK2.M2}
    (q_n+1)^{-1}\sum_{j=1}^{q_n+1}\mathbb{E}[(U_j)_{(s,t)}^4] \leq b_2^2 \beta_n^2
\end{equation}
for a sequence of constants $\beta_n$ with a constant $b_2>0$ related to $C_0$ and $c_0$ such that $b_2 \geq b_1$.
Third, for any $1\leq j\leq q_n+1$,
\begin{equation}\label{eq:cond.CCK2.E1}
    \mathbb{E}[\exp(|(U_j)_{(s,t)}|/\beta_n)] \leq 2.
\end{equation}
Then by Theorem \ref{prop:CCK17}, Lemma \ref{lm:uncond.CCK} immediately follows.
The proof of second and third conditions relies on the following result on the norms of $U_j$, similar to Lemma \ref{lm:Lk.norm}.
\begin{lemma}\label{lm:Lk.norm.U}
    Given any $1\leq j\leq q_n$ and $1\leq s,t\leq p$, there exists some universal constant $C_1>0$, so that for all $k\geq 2$,
    \begin{equation*}
        \|(U_j)_{(s,t)}\|_k \leq C_1 k\|(U_j)_{(s,t)}\|_2.
    \end{equation*}
\end{lemma}
\begin{proof}[Proof of Lemma \ref{lm:Lk.norm.U}]
    Since the Gaussian random variables $\{Y_j\}_{j=1}^{q_n+1}$ are serially independent to $\bm{\epsilon}_0$, by the Tower Rule,
    \begin{equation*}
        \mathbb{E}[(U_j)^2_{(s,t)}] = \mathbb{E}[(V_j(\bm{\bm{\epsilon}}_{j,0}))_{(s,t),(s,t)}] = \|(V_j(\bm{\bm{\epsilon}}_{j,0}))_{(s,t),(s,t)}\|_1.
    \end{equation*}
    Moreover, for any positive integer $k$, the conditioned moments are
    \begin{equation*}
            \mathbb{E}[(U_j)^{2k}_{(s,t)}\mid \bm{\epsilon}_{j,0}] = \sum_{(u_1,v_1)}\cdots\sum_{(u_k,v_k)} (V^{1/2}(\bm{\epsilon}_{j,0}))_{(s,t),(u_1,v_1)}^2 \cdots (V^{1/2}(\bm{\epsilon}_{j,0}))_{(s,t),(u_k,v_k)}^2
            \mathbb{E}[(Y_j)_{(u_1,v_1)}^2 \cdots (Y_j)_{(u_k,v_k)}^2].
    \end{equation*}
    Similar to the proof of Lemma \ref{lm:Lk.norm}, the pairs of indices $(s_1,t_1),\cdots,(s_k, t_k)$ could take $1\leq L\leq k$ different values. Denote them by $(s_{(1)},t_{(1)}), \cdots, (s_{(L)}, t_{(L)})$. For any $1\leq l\leq L$, assume that $\tau_l$ out of $k$ pairs of $(s,t)$ take the value $(s_{(l)}, t_{(l)})$, thus $\tau_1 + \cdots + \tau_L = k$. Therefore we can write
    \begin{align*}
            \mathbb{E}[(U_j)^{2k}_{(s,t)}\mid \bm{\epsilon}_{j,0}] &= \sum_{L=1}^k \sum_{\tau_1+\cdots+\tau_L = k}\sum_{(u_{(1)},v_{(1)})}\cdots \sum_{(u_{(L)},v_{(L)})}\\
            &\quad\quad(V^{1/2}(\bm{\epsilon}_{j,0}))_{(s,t),(u_{(1)},v_{(1)})}^2 \cdots (V^{1/2}(\bm{\epsilon}_{j,0}))_{(s,t),(u_{(L)},v_{(L)})}^2\mathbb{E}[Z^{2\tau_1}]\cdots \mathbb{E}[Z^{2\tau_L}]\\
            &\leq k{k-1 \choose L-1}(2k-1)!!(V(\bm{\epsilon}_{j,0}))_{(s,t),(s,t)}^k.
    \end{align*}
    where $Z$ is standard Gaussian. Note that $(V(\bm{\epsilon}_{j,0}))_{(s,t),(s,t)}$ is a nonnegative quadratic form with respect to $\bm{\epsilon}_{j,0}$.
    By Lemma \ref{lm:Lk.norm}, there exists a universal constant $C^*>0$ such that for any $k\geq 1$, $\|(V(\bm{\epsilon}_{j,0}))_{(s,t),(s,t)}\|_k \leq C^* k \|(V(\bm{\epsilon}_{j,0}))_{(s,t),(s,t)}\|_1$.
    By Tower Rule and Stirling's Formula,
    \begin{align*}
            \|(U_j)_{(s,t)}\|_{2k}^{2k} &= \mathbb{E}[\mathbb{E}[(U_j)^{2k}_{(s,t)}\mid \bm{\epsilon}_{j,0}]] \leq k 2^{k-1} (2k-1)!! \|(V(\bm{\epsilon}_{j,0}))_{(s,t),(s,t)}\|^k_k\\
            &\leq C^* k^{2k+1}(4/e)^k \mathbb{E}[(U_j)^2_{(s,t)}]^{k}\\
            &\leq \left(Ck \|(U_j)_{(s,t)}\|_2\right)^{2k}.
    \end{align*}
    This finishes the proof of Lemma \ref{lm:Lk.norm.U} for even $k$. For odd $k$, the result follows immediately from Cauchy-Schwarz inequality.
\end{proof}

To prove the three conditions, note that
\begin{align}
    \begin{split}
        \mathbb{E}[(U_j)^2_{(s,t)}] &= \mathbb{E}[\mathbb{E}[(U_j)^2_{(s,t)}]\mid \bm{\epsilon}_{j,0}] = \mathbb{E}[(V_j(\bm{\epsilon}_{j,0}))_{(s,t),(s,t)}]\\
        &= \mathbb{E}[\mathbb{E}[(F_{j,0}(\bm{\epsilon}_{j,0}))^2_{(s,t)}\mid \bm{\epsilon}_{j,0}]] + \mathbb{E}[(F_{j,1})_{(s,t)}^2]\\
        &\quad\quad+ \mathbb{E}[\mathbb{E}[(F_{j,2}(\bm{\epsilon}_{j,0}) - \Lambda_{j,2}(\bm{\epsilon}_{j,0}))^2_{(s,t)}\mid \bm{\epsilon}_{j,0}]].\label{eq:EU.partition}
    \end{split}
\end{align}
By Tower Rule, the first term on the right hand side of \eqref{eq:EU.partition} is equal to $\mathbb{E}[(F_{j,0})^2_{(s,t)}]$.
Since $\{\epsilon_i\}_{i\leq n}$ are i.i.d., the third term is equal to
\begin{align*}
    \begin{split}
        \mathbb{E}[\mathbb{E}[(F_{j,2}(\bm{\epsilon}_{j+1,0}) - \Lambda_{j,2}(\bm{\epsilon}_{j+1,0}))^2_{(s,t)}\mid \bm{\epsilon}_{j+1,0}]] = \mathbb{E}[(F_{j,2})^2_{(s,t)}].
    \end{split}
\end{align*}
Therefore by \eqref{eq:Bj.partition},
\begin{equation*}
    \mathbb{E}[(U_j)^2_{(s,t)}] = \mathbb{E}[(F_{j,0})^2_{(s,t)}] + \mathbb{E}[(F_{j,1})^2_{(s,t)}] + \mathbb{E}[(F_{j,2})^2_{(s,t)}] = \mathbb{E}[(B_j)^2_{(s,t)}].
\end{equation*}
Recall \eqref{eq:E.F0.sq}. We have $\mathbb{E}[(F_{j,0})^2_{(s,t)}] = \mathbb{E}[(F_{j,1})^2_{(s,t)}] = \mathbb{E}[(F_{j,2})^2_{(s,t)}] = m\sum_{k=1}^m \zeta_{k,st}$, where
\begin{multline*}
        \zeta_{k,st} = \sum_{l=0}^\infty \sum_{l'=0}^\infty \bigg[\bigg.  (A_{l+k} A_{l'+k}^\top)_{ss}(A_l A_{l'}^\top)_{tt} + (A_{l+k} A_{l'+k}^\top)_{tt}(A_l A_{l'}^\top)_{ss}\\
        + (A_{l+k} A_{l'+k}^\top)_{st}(A_l A_{l'}^\top)_{ts} + (A_{l+k} A_{l'+k}^\top)_{ts}(A_l A_{l'}^\top)_{st}\bigg.\bigg],
\end{multline*}
or alternatively, using Kronecker product $\otimes$, Hadamard product $\circ$, and commutation matrix $K_{pp}$,
\[
\zeta_k = \sum_{l=0}^\infty \sum_{l'=0}^\infty (I_{p^2} + K_{pp}) \bigg[ \text{diag}(L_0) \otimes \text{diag}(L_1) + \vectorize(L_1 \circ L_0^\top) \bigg],
\]
where $L_1 = A_{l+k} A_{l'+k}^\top$, and $L_0 = A_l A_{l'}^\top$.
Under Condition \ref{cond.G}, we have
\begin{equation*}
    \sum_{k=1}^\infty \zeta_{k,st} = \sum_{k=-\infty}^\infty [(\Gamma_k)_{ss}(\Gamma_k)_{tt} + (\Gamma_k)_{st}(\Gamma_k)_{ts}] \geq c_0.
\end{equation*}
Since our choice of $m\rightarrow \infty$, $\sum_{k=1}^m \zeta_{k,st}$ is lower bounded by a constant related to $c_0$. The first condition \eqref{eq:cond.CCK2.M1} holds.

Also by \eqref{eq:E.F0.sq}, all three terms on the right hand side of \eqref{eq:EU.partition} are of $O(m^2)$ order under Condition \ref{cond.A}, up to a constant related to $C_0$.
Therefore $\|(U_j)_{(s,t)}\|_2^2 = O(m^2)$.
By Lemma \ref{lm:Lk.norm.U}, we have $\|(U_j)_{(s,t)}\|_4^4 = O(m^4)$.
The second condition \eqref{eq:cond.CCK2.M2} holds with $\beta_n \asymp m^2$, the same choice as in the proof of Lemma \ref{lm:cond.CCK}.

Finally, for the third condition \eqref{eq:cond.CCK2.E1}, by Taylor expansion, Fubini's Theorem, Lemma \ref{lm:Lk.norm.U} and Stirling's formula, with $\beta_n \asymp m^2$, when $m$ is large,
\begin{equation*}
    \mathbb{E}[\exp(|(U_j)_{(s,t)}|/m^2)] = 1+\sum_{k=1}^\infty \frac{\|(U_j)_{(s,t)}\|_k^k}{k!m^{2k}} \leq 1+C_u\sum_{k=1}^\infty \frac{e^k}{\sqrt{2\pi k}m^{k}} \leq 2.
\end{equation*}
All three conditions \eqref{eq:cond.CCK2.M1}, \eqref{eq:cond.CCK2.M2} and \eqref{eq:cond.CCK2.E1} are satisfied. This finishes the proof of Lemma \ref{lm:uncond.CCK}.

\subsection{Proof of Lemma \ref{lm:GA.comparison}}\label{sec:apdx.GA.comparison}

Recall that both $Z$ and $S_n^Z$ are centered Gaussian random vectors on $\mathbb{R}^{p^2}$.
Denote their respective covariance matrices as $\Sigma_Z$ and $\Sigma_S$.
Note that
\begin{equation*}
    \Sigma_Z = \frac{1}{n}\cov\left(\sum_{i=1}^n \mathcal{X}_i\right).
\end{equation*}
By the discussion in Section \ref{sec:pf.uncond.CCK} and Tower Rule,
\begin{equation}\label{eq:Sigma.S}
        \Sigma_S = \frac{1}{n}\sum_{j=1}^{q_n+1}\mathbb{E}[V_j(\bm{\epsilon}_{j,0})] = \frac{1}{n}\mathbb{E}\left[\sum_{j=0}^{q_n}\var(B_j\mid \bm{\epsilon}_0)\right]
        = \frac{1}{n}\left(\sum_{j=1}^{q_n}\mathbb{E}[F_{j,0}F_{j,0}^\top] + \sum_{j=0}^{q_n}\mathbb{E}[F_{j,1}F_{j,1}^\top] + \sum_{j=0}^{q_n}\mathbb{E}[F_{j,2}F_{j,2}^\top]\right).
\end{equation}
By trivial computation, $\Sigma_Z = \sum_{k=-n+1}^{n-1} \frac{n-|k|}{n} (I_{p^2} + K_{pp}) (\Gamma_k \otimes \Gamma_k)$, or each of their entries can be written as
\begin{equation}\label{eq:covariance.of.Z}
    (\Sigma_Z)_{(s_1,t_1),(s_2,t_2)} = \sum_{k=-n+1}^{n-1}\frac{n-|k|}{n}[(\Gamma_k)_{s_1 s_2}(\Gamma_k)_{t_1 t_2} + (\Gamma_k)_{s_1 t_2}(\Gamma_k)_{t_1 s_2}],
\end{equation}
and $\Sigma_S = \sum_{k=1}^{m-1}\sum_{l=0}^\infty\sum_{l'=0}^\infty \left[ (L_1 \otimes L_0) + (L_0 \otimes L_1) \right] (I_{p^2} + K_{pp})$ (recall that $L_1 = A_{l+k}A^\top_{l'+k}$, and $L_0 = A_l A_{l'}^\top$), or to be written entrywise,
\begin{align*}
        (\Sigma_S)_{(s_1,t_1),(s_2,t_2)} = \sum_{k=1}^{m-1}\sum_{l=0}^\infty\sum_{l'=0}^\infty&\bigg[\bigg. (A_{l+k}A^\top_{l'+k})_{s_1 s_2} (A_l A_{l'}^\top)_{t_1 t_2} \\
        &+ (A_{l+k}A^\top_{l'+k})_{t_1 t_2} (A_l A_{l'}^\top)_{s_1 s_2}\\
        &+ (A_{l+k}A^\top_{l'+k})_{s_1 t_2} (A_l A_{l'}^\top)_{t_1 s_2} \\
        &+ (A_{l+k}A^\top_{l'+k})_{t_1 s_2} (A_l A_{l'}^\top)_{s_1 t_2}\bigg.\bigg].
\end{align*}
Observe that when $n\rightarrow\infty$ and $m\rightarrow\infty$, both $\Sigma_Z$ and $\Sigma_S$ converge to an identical matrix, denoted by $\Sigma^* = \sum_{k=-\infty}^{\infty} (I_{p^2} + K_{pp}) (\Gamma_k \otimes \Gamma_k)$, or to be written entrywise,
\begin{equation*}
    \Sigma^*_{(s_1,t_1), (s_2,t_2)} = \sum_{k=-\infty}^\infty [(\Gamma_k)_{s_1 s_2}(\Gamma_k)_{t_1 t_2} + (\Gamma_k)_{s_1 t_2}(\Gamma_k)_{t_1 s_2}].
\end{equation*}
For each $1\leq s_1,t_1,s_2,t_2\leq p$, we have
\begin{align*}
        &(\Sigma^* - \Sigma_Z )_{(s_1,t_1),(s_2,t_2)}\\
        &= \sum_{k=-\infty}^{-n}\sum_{l=0}^\infty\sum_{l'=0}^\infty[(A_l A_{l+k})_{s_1 s_2}(A_{l'}A_{l'+k})_{t_1 t_2} + (A_l A_{l+k})_{s_1 t_2}(A_{l'}A_{l'+k})_{t_1 s_2}]\\
        &\quad+ \sum_{k=-n+1}^{n-1} \frac{|k|}{n} \sum_{l=0}^\infty\sum_{l'=0}^\infty[(A_l A_{l+k})_{s_1 s_2}(A_{l'}A_{l'+k})_{t_1 t_2} + (A_l A_{l+k})_{s_1 t_2}(A_{l'}A_{l'+k})_{t_1 s_2}]\\
        &\quad+ \sum_{k=n}^\infty \sum_{l=0}^\infty\sum_{l'=0}^\infty[(A_l A_{l+k})_{s_1 s_2}(A_{l'}A_{l'+k})_{t_1 t_2} + (A_l A_{l+k})_{s_1 t_2}(A_{l'}A_{l'+k})_{t_1 s_2}].
\end{align*}
Under Condition \ref{cond.A}, $|\Sigma_Z - \Sigma^*|_\infty$ can be bounded by
\begin{align*}
        |\Sigma^* - \Sigma_Z|_\infty &\lesssim \sum_{k=1}^{n-1}\frac{k}{n}\left(\sum_{l=0}^\infty (l+k)^{-\beta}(l\vee 1)^{-\beta}\right)^2 + \sum_{k=n}^\infty \left(\sum_{l=0}^\infty (l+k)^{-\beta}(l\vee 1)^{-\beta}\right)^2\\
        &\lesssim \sum_{k=1}^{n-1}\frac{k}{n}\left(\sum_{l=0}^{k-1} (l+k)^{-\beta}(l\vee 1)^{-\beta}\right)^2 + \sum_{k=1}^{n-1}\frac{k}{n}\left(\sum_{l=k}^\infty (l+k)^{-\beta}l^{-\beta}\right)^2\\
        &\hspace{10pt}+ \sum_{k=n}^\infty \left(\sum_{l=0}^{k-1} (l+k)^{-\beta}(l\vee 1)^{-\beta}\right)^2 + \sum_{k=n}^\infty\left(\sum_{l=k}^\infty (l+k)^{-\beta}l^{-\beta}\right)^2\\
        &\leq \sum_{k=1}^{n-1}\frac{k}{n} k^{-2\beta}\left(1+\sum_{l=1}^{k-1} l^{-\beta}\right)^2 + \sum_{k=1}^{n-1}\frac{k}{n}\left(\sum_{l=k}^\infty l^{-2\beta}\right)^2\\
        &\hspace{10pt}+ \sum_{k=n}^\infty k^{-2\beta}\left(1+\sum_{l=1}^{k-1} l^{-\beta}\right)^2 + \sum_{k=n}^\infty\left(\sum_{l=k}^\infty l^{-2\beta}\right)^2\\
        &\asymp n^{-1} \sum_{k=1}^{n-1} k^{(-2\beta+1)\vee(-4\beta+3)} + \sum_{k=n}^\infty k^{(-2\beta)\vee (-4\beta+2)}\\
        &\asymp n^{(-2\beta+1)\vee(-4\beta+3)},
\end{align*}
up to a constant related to $C_0$. Meanwhile,
\begin{align*}
        (\Sigma^* - \Sigma_S)_{(s_1,t_1),(s_2,t_2)} = \sum_{k=m}^\infty\sum_{l=0}^\infty\sum_{l'=0}^\infty&\bigg[\bigg. (A_{l+k}A^\top_{l'+k})_{s_1 s_2} (A_l A_{l'}^\top)_{t_1 t_2} \\
        &+ (A_{l+k}A^\top_{l'+k})_{t_1 t_2} (A_l A_{l'}^\top)_{s_1 s_2}\\
        &+ (A_{l+k}A^\top_{l'+k})_{s_1 t_2} (A_l A_{l'}^\top)_{t_1 s_2} \\
        &+ (A_{l+k}A^\top_{l'+k})_{t_1 s_2} (A_l A_{l'}^\top)_{s_1 t_2}\bigg.\bigg].
\end{align*}
Using a similar computation technique as above,
\begin{align}\label{eq:Sigma.S.close}
        |\Sigma^* - \Sigma_S|_\infty &\lesssim \sum_{k=m}^\infty\left(\sum_{l=0}^\infty (l+k)^{-\beta} (1\vee l)^{-\beta}\right)^2 \lesssim m^{(-2\beta+1)\vee(-4\beta+3)},
\end{align}
up to a constant related to $C_0$.
Since $m\leq n$, $|\Sigma_Z - \Sigma_S|_\infty = O\left(m^{(-2\beta+1)\vee(-4\beta+3)}\right)$.
Under Condition \ref{cond.A} and \ref{cond.G}, for all $1\leq s,t\leq p$, the diagonal entries of $\Sigma^*$ are all upper bounded by a constant related to $C_0$ (by Lemma \ref{prop:cond.A.cor}), and lower bounded by $c_0$.
For adequately large $n$, all diagonal entries of $\Sigma_Z$ and $\Sigma_S$ are therefore positive.
Moreover, the minimum of these diagonal entries, $\min_{(s,t)}(\Sigma_Z)_{(s,t),(s,t)}$ and $\min_{(s,t)}(\Sigma_S)_{(s,t),(s,t)}$, are only related to $c_0$.
The maximum of these diagonal entries, $\max_{(s,t)}(\Sigma_Z)_{(s,t),(s,t)}$ and $\max_{(s,t)}(\Sigma_S)_{(s,t),(s,t)}$, are only related to $C_0$.
By Theorem \ref{thm:comparison}, Lemma \ref{lm:GA.comparison} follows.

\subsection{Choice of parameters in (\ref{eq:choice.of.eta.m})} \label{sec:choice.of.eta.m}
There are five terms on the right hand side of \eqref{eq:rho.n.bound.eta}.
The first and second terms decrease with $\eta$, while the third increases with $\eta$.
The second and fifth terms decrease with $m$, while the fourth term increases with $m$.
Note that all terms are related to $p$ and $n$ in terms of polynomials of $\log(pn)$ and $n$.
We shall choose the order of $\eta$ and $m$ with respect to $\log(pn)$ and $n$, such that the upper bound of $\rho(n)$ in \eqref{eq:rho.n.bound.eta} is approximately minimized.
Specifically, assume that up to a constant related to $C_0$ and $c_0$,
\begin{equation*}
    \rho(n) = O\left(\frac{\log^\alpha(pn)}{n^\gamma}\right).
\end{equation*}

\[
p^2\exp\left(-c_1 \eta / Q_1^{1/2}(n)\right) + p^2\exp\left(-c_2 \eta / Q_2^{1/2}(m)\right) + 
    \eta\sqrt{1\vee\log(p/\eta)} +\left(\frac{m^4\log^5(pn)}{n}\right)^{1/4} + \frac{\log(p)}{m^{\tilde\beta/2}}
\]

Next, we find the value of $\alpha>0$ and $\gamma>0$.
For the first term on the right hand side of \eqref{eq:rho.n.bound.eta} to be smaller than this bound, it is required that
\begin{equation}\label{eq:eta.req.1}
    \eta \gtrsim \frac{1}{n^{(1\wedge \tilde{\beta})/2}} \log\left(\frac{p^2n^\gamma}{\log^\alpha(pn)}\right).
\end{equation}
A similar lower bound of $\eta$ is required by the second term,
\begin{equation}\label{eq:eta.req.2}
    \eta \gtrsim \frac{1}{m^{\tilde{\beta}/2}} \log\left(\frac{p^2n^\gamma}{\log^\alpha(pn)}\right).
\end{equation}
Meanwhile, for the third term to be smaller than the same bound, we need
\begin{equation*}
    \eta \lesssim \frac{\log^{\alpha-1/2}(pn)}{n^{\gamma}}.
\end{equation*}
As long as $m^{\tilde{\beta}} \lesssim n$ holds, the lower bound in \eqref{eq:eta.req.2} is tighter than that in \eqref{eq:eta.req.1}, and the lower bound of $\eta$ is specified by \eqref{eq:eta.req.2}.
The lower bound of $\eta$ is required to be smaller than the upper bound, i.e.,
\begin{equation*}
    \frac{1}{n^{(1\wedge \tilde{\beta})/2}} \log\left(\frac{p^2n^\gamma}{\log^\alpha(pn)}\right) \lesssim \frac{\log^{\alpha-1/2}(pn)}{n^{\gamma}},
\end{equation*}
which is equivalent to a lower bound for the choice of $m$,
\begin{equation}\label{eq:m.req.23}
    m \gtrsim \left(\frac{n^\gamma}{\log^{\alpha-3/2}(pn)}\right)^{2/\tilde{\beta}}.
\end{equation}
Similarly, for the fifth term to be smaller than the desired order, we need
\[
    m \gtrsim \left(\frac{n^\gamma}{\log^{\alpha-1}(pn)}\right)^{2/\tilde{\beta}},
\]
which holds as long as \eqref{eq:m.req.23} is satisfied.
Meanwhile, for the fourth term on the right hand side of \eqref{eq:rho.n.bound.eta} to be smaller than $\log^\alpha(pn) / n^\gamma$, an upper bound of $m$ is required,
\begin{equation}\label{eq:m.req.4}
m\lesssim n^{1/4 - \gamma}\log^{\alpha - 5/4}(pn).
\end{equation}
If we let the right hand side of \eqref{eq:m.req.23} to have the same order as \eqref{eq:m.req.4}, then
\begin{equation*}
n^{2\gamma/\tilde{\beta} + \gamma - 1/4} \asymp \log^{\alpha - 5/4 + 2(\alpha-3/2)/\tilde{\beta}}(pn)
\end{equation*}
for any $n$ and $\log(pn)$.
The choice of $\alpha>0$ and $\gamma>0$ is then unique, i.e.,
\begin{equation*}
\alpha = \frac{5\tilde{\beta}+12}{4\tilde{\beta}+8}, \quad\quad \gamma = \frac{\tilde{\beta}}{4\tilde{\beta}+8}.
\end{equation*}
And the order of $\eta$ and $m$ with respect to $\log(pn)$ and $n$ is also unique,
\begin{equation*}
\eta = {O}\left(\frac{\log^{(3\tilde{\beta}+8) / (4\tilde{\beta}+8)}(pn)}{n^{\tilde{\beta} / (4\tilde{\beta}+8)}}\right),\quad m={O}\left((n\log(pn))^{1/(2\tilde{\beta}+4)}\right),
\end{equation*}
up to a constant related to $C_0$ and $c_0$.
We can indeed verify that for any $\beta>3/4$, i.e., $\tilde{\beta}>0$, we have $m\rightarrow\infty$ as $n\rightarrow\infty$, but $m/n\rightarrow 0$.
Moreover, $m^{\tilde{\beta}} \lesssim n^{1/2} \log^{1/2}(pn) \lesssim n$ as long as $\log(p) \lesssim n^2$, which means that the lower bound in \eqref{eq:eta.req.2} is always greater than the lower bound in \eqref{eq:eta.req.1}.

\section{Proof of Section \ref{sec:bootstrap.cov}}

\subsection{Proof of Lemma \ref{lm:B.check.B.tilde}}\label{sec:apdx.B.check.B.tilde}
As an intermediate quantity between $\check{B}_{i,l}$ and $\tilde{B}_{i,l}$, recall that $\mathcal{X}_j = \vectorize(X_j X_j^\top - \Sigma)$, and let $B_{i,l} = \sum_{j=i-l+1}^{i} \mathcal{X}_j$.
Denote $\hat{\Sigma}_{i,l} = l^{-1}\sum_{j=i-l+1}^i X_j X_j^\top$.
Then
\begin{equation*}
    l^{-1/2} (\check{B}_{i,l} - B_{i,l}) = l^{-1/2}\vectorize(\hat{\Sigma}_{i,l} - \Sigma).
\end{equation*}
Another application of Theorem \ref{thm:GA.cov} directly implies that
\begin{equation*}
    \sup_{u\in\mathbb{R}} \left|\mathbb{P}\left(l^{-1/2}|\check{B}_{i,l} - B_{i,l}|_\infty\geq u\right) - \mathbb{P}\left(\sqrt{\frac{l}{n}}|Z|_\infty\geq u\right)\right| \leq C \Psi(p,n).
\end{equation*}
Note that $Z$ is a Gaussian random vector with mean zero. Each entry of $Z$ has a variance upper bounded and lower bounded from 0 under Conditions \ref{cond.A} and \ref{cond.G}; see the discussion in Appendix \ref{sec:apdx.GA.comparison}.
Therefore we have $\mathbb{E}[|Z|_\infty] \lesssim \sqrt{\log(p)}$, up to a constant related to $C_0$ and $c_0$.
By Markov inequality,
\begin{equation*}
    \mathbb{P}\left(|Z|_\infty > \sqrt{\frac{n}{l}}\delta \right) \leq \sqrt{\frac{l}{n}}\cdot\frac{1}{\delta}\mathbb{E}[|Z|_\infty]\leq C\delta^{-1}\sqrt{\frac{l\log(p)}{n}}.
\end{equation*}
This implies
\begin{equation}\label{eq:B.check.B}
    \mathbb{P}\left(l^{-1/2}|\check{B}_{i,l} - B_{i,l}|_\infty > \delta\right) \leq C \Psi(p,n) + C\delta^{-1}\sqrt{\frac{l\log(p)}{n}}.
\end{equation}
Moreover, $\tilde{B}_{i,l,m}$ is a valid $m$-dependent approximation to $B_{i,l}$.
By Lemmas \ref{lm:mtg.approx} and \ref{lm:m.dep.approx},
\begin{equation}\label{eq:B.B.tilde}
        \mathbb{P}\left(l^{-1/2}|B_{i,l} - \tilde{B}_{i,l,m}|_\infty > \delta\right) \leq 2p^2\exp\left(-c_1\left(\frac{\delta^2}{Q_1(l)} \wedge \frac{\delta}{\sqrt{Q_1(l)}}\right)\right)
        + 2p^2\exp\left(-c_2\left(\frac{\delta^2}{Q_2(m)} \wedge \frac{\delta}{\sqrt{Q_2(m)}}\right)\right).
\end{equation}
The proof of Lemma \ref{lm:B.check.B.tilde} is finished by combining \eqref{eq:B.check.B} and \eqref{eq:B.B.tilde}.

\subsection{Proof of Lemma \ref{lm:B.check.GA}}\label{sec:apdx.B.check.GA}
We start by dividing the index set $\{i-l+1,\cdots,i\}$, which is of length $l$, into triadic blocks, each with a maximal length of $3m$.
Let $r_l = \lfloor l/3m \rfloor$.
Define
\begin{equation*}
    \mathcal{J}_{j,k}^{m,i} = \{i-l+(3j+k-1)m+1, \cdots, i-l+(3j+k)m \wedge i\}.
\end{equation*}
The full index set is divided into
\begin{equation*}
    \{i-l+1,\cdots,i\} = \mathcal{J}_{0,1}^{m,i} \cup \mathcal{J}_{0,2}^{m,i} \cup \mathcal{J}_{1,0}^{m,i} \cup \mathcal{J}_{1,1}^{m,i} \cup \mathcal{J}_{1,2}^{m,i} \cup \cdots \cup \mathcal{J}_{r_l,1}^{m,i} \cup \mathcal{J}_{r_l,2}^{m,i} \cup \mathcal{J}_{r_l+1,0}^{m,i}.
\end{equation*}
For each $0\leq j\leq r_l$, let
\begin{equation*}
    \Xi_j^i = \Phi_{j,1}^i + \Phi_{j,2}^i + \Phi_{j+1,0}^i, \quad \text{ where } \Phi_{j,k}^i = \sum_{t\in\mathcal{J}_{j,k}^{m,i}} \tilde{D}_{t,m}.
\end{equation*}
This gives a partition of $\tilde{B}_{i,l,m}$ into $(r_l+1)$ many triadic blocks,
\begin{equation*}
    \tilde{B}_{i,l,m} = \sum_{j=0}^{r_l} \Xi_j^i.
\end{equation*}
Denote the set $\{\epsilon_t: t\in\mathcal{J}_{j,k}^{m,i}\}$ by $\bm{\epsilon}_{j,k}^i$, and $\bigcup_{j=1}^{r_l+1} \bm{\epsilon}_{j,0}^{i}$ by $\bm{\epsilon}_0^i$. For each $j\geq 1$, we have the following conditional independence for neighboring blocks:
\begin{equation*}
    \Xi_{j-1}^i \indep \Xi_j^i\mid \bm{\epsilon}_{j,0}^i.
\end{equation*}
Furthermore, conditioned on $\bm{\epsilon}_0^i$, all blocks $\{\Xi_j^i\}_{j=0}^{r_l}$ are mutually independent. Define
\begin{equation*}
    V_0^{B*}(\bm{\epsilon}_{1,0}^i) = \var(\Xi_0^i\mid \bm{\epsilon}_0^i),
\end{equation*}
and for each $1\leq j\leq r_l$,
\begin{equation*}
    V_j^{B*}(\bm{\epsilon}_{j,0}^i, \bm{\epsilon}_{j+1,0}^i) = \var(\Xi_j^i\mid \bm{\epsilon}_0^i).
\end{equation*}
Consider a set of i.i.d.~standard Gaussian random vectors $Y_0^{B*},Y_1^{B*},\cdots, Y_{r_l}^{B*}\in\mathbb{R}^{p^2}$ that are mutually independent to $\bm{\epsilon}_0^i$.
Conditional on $\bm{\epsilon}_0^i$, the partial sum
\begin{equation*}
    B_{i,l}^* = [V_0^{B*}(\bm{\epsilon}_{1,0}^i)]^{1/2}Y_0^{B*} + \sum_{j=1}^{r_l}[V_j^{B*}(\bm{\epsilon}_{j,0}^i, \bm{\epsilon}_{j+1,0}^i)]^{1/2}Y_j^{B*}
\end{equation*}
is Gaussian with mean 0 and variance
\begin{equation*}
    V_0^{B*}(\bm{\epsilon}_{1,0}^i) + \sum_{j=1}^{r_l} V_j^{B*}(\bm{\epsilon}_{j,0}^i, \bm{\epsilon}_{j+1,0}^i),
\end{equation*}
thus a Gaussian approximation to $\tilde{B}_{i,l,m} - \mathbb{E}[\tilde{B}_{i,l,m}\mid \bm{\epsilon}_0^i]$ conditional on $\bm{\epsilon}_0^i$.
We have the following identity analogous to \eqref{eq:var.identity} with the same proof technique:
\begin{equation*}
    V_0^{B*}(\bm{\epsilon}_{1,0}^i) + \sum_{j=1}^{r_l} V_j^{B*}(\bm{\epsilon}_{j,0}^i, \bm{\epsilon}_{j+1,0}^i) = \sum_{j=0}^{r_l+1} V_j^B(\bm{\epsilon}_{j,0}^i),
\end{equation*}
where for each $1\leq j\leq r_l$,
\begin{equation*}
    V_j^B(\bm{\epsilon}_{j,0}^i) = V_j^{B*}(\bm{\epsilon}_{j,0}^i, \bm{\epsilon}_{j,0}^i),
\end{equation*}
and
\begin{equation*}
    V_{r_l+1}^B(\bm{\epsilon}_{r_l+1,0}^i) = V_0^{B*}(\bm{\epsilon}_{r_l+1,0}^i).
\end{equation*}
For another set of standard Gaussian random vectors $Y_1^B,\cdots,Y_{r_l}^B, Y_{r_l+1}^B\in\mathbb{R}^{p^2}$ that are also independent to $\bm{\epsilon}_{0}^i$, the partial sum
\begin{equation*}
    B_{i,l}^Y = \sum_{j=1}^{r_l+1}[V_j^B(\bm{\epsilon}_{j,0}^i)]^{1/2}Y_j^B
\end{equation*}
is identically distributed to $B_{i,l}^*$ conditional on $\bm{\epsilon}_{0}^i$.
Therefore, $B_{i,l}^Y$ is also a Gaussian approximation to $\tilde{B}_{i,l,m} - \mathbb{E}[\tilde{B}_{i,l,m}\mid \bm{\epsilon}_0^i]$ conditional on $\bm{\epsilon}_0^i$.
Analogous to Lemma \ref{lm:cond.CCK}, for any $\delta>0$ and $0<\epsilon<1$,
\begin{equation}\label{eq:boot.3.3}
        \sup_{u\in\mathbb{R}}\left|\mathbb{P}\left(l^{-1/2}|\tilde{B}_{i,l,m} - \mathbb{E}[\tilde{B}_{i,l,m}\mid \bm{\epsilon}_0^i]|_\infty \leq u\mid \bm{\epsilon}_0^i\right) - \mathbb{P}\left(l^{-1/2}|B_{i,l}^Y|_\infty \leq u\mid \bm{\epsilon}_0^i\right)\right|
        \leq C \left(\frac{m^4\log^5(pl)}{l}\right)^{1/4}
\end{equation}
with probability no less than $1-4p^2\exp(-c_2\delta l^{1/2}/m^{1/2}) - 2\exp(-c_2 m^{2\epsilon})$. Here, $C>0$ is a constant related to $C_0$ and $c_0$, and $c_2$ is a constant related to $C_0$ only.
Note that $\{\bm{\epsilon}_{j,0}^i\}_{j=1}^{r_l+1}$ are mutually independent.
Then the summands in $B_{i,l}^Y$ are also mutually independent.
For each $j\in[r_l+1]$, define $Z_j^B\in\mathbb{R}^{p^2}$, a Gaussian random vector with mean zero and variance $\var([V_j^B(\bm{\epsilon}_{0,j}^i)]^{1/2}Y_j^B)$.
The normalized partial sum $S_l^Z = l^{-1/2}\sum_{j=1}^{r_l+1} Z_j^B$ is the Gaussian approximation to $l^{-1/2}B_{i,l}^Y$.
Analogous to Lemma \ref{lm:uncond.CCK}, we have
\begin{equation}\label{eq:boot.3.4}
    \sup_{u\in\mathbb{R}}\left|\mathbb{P}(l^{-1/2}|B_{i,l}^Y|_\infty \leq u) - \mathbb{P}(|S_l^Z|_\infty \leq u)\right| \leq C\left(\frac{m^4\log^5(pl)}{l}\right)^{1/4}.
\end{equation}
If we take a closer look at the variance of $S_l^Z$, analogous to \eqref{eq:Sigma.S},
\begin{align*}
        \var(S_l^Z) &= \frac{1}{l}\sum_{j=1}^{r_l+1} \mathbb{E}[V_j^B(\bm{\epsilon}_{j,0}^i)] = \frac{1}{l}\mathbb{E}\left[\sum_{j=0}^{r_l}\var(\Xi_j^i\mid \bm{\epsilon}_0^i)\right]\\
        &= \frac{1}{l}\left(\sum_{j=0}^{r_l}\mathbb{E}[\Phi_{j,1}^i \Phi_{j,1}^{i,\top}] + \sum_{j=0}^{r_l}\mathbb{E}[\Phi_{j,2}^i \Phi_{j,2}^{i,\top}] + \sum_{j=0}^{r_l}\mathbb{E}[\Phi_{j+1,0}^i \Phi_{j+1,0}^{i,\top}]\right).
\end{align*}
Since $\tilde{D}_{t,m}$ is stationary, $\var(S_l^Z)$ is identical to $\Sigma_S$.
Lemma \ref{lm:GA.comparison} immediately implies that
\begin{equation}\label{eq:boot.3.5}
    \sup_{u\in\mathbb{R}}\left|\mathbb{P}(|Z|_\infty \leq u) - \mathbb{P}(|S_l^Z|_\infty \leq u)\right| \leq C\frac{\log(p)}{m^{\tilde\beta/2}}.
\end{equation}
Recall the order of $m$ given in \eqref{eq:choice.of.eta.m}.
The rate in \eqref{eq:boot.3.5} is negligible since $l/n\rightarrow 0$.
Following the same procedure as in the proof of Theorem \ref{thm:GA.cov} only with $n$ replaced by $l$, by combining \eqref{eq:boot.3.3} \eqref{eq:boot.3.4} and \eqref{eq:boot.3.5}, we arrive at
\begin{equation*}
    \sup_{u\in\mathbb{R}}\left|\mathbb{P}\left(l^{-1/2}|\tilde{B}_{i,l,m}|_\infty \leq u\right) - \mathbb{P}(|Z|_\infty \leq u)\right| \leq C\frac{\log^{(5\Tilde{\beta}+12)/(4\Tilde{\beta}+8)}(pl)}{l^{\Tilde{\beta}/(4\Tilde{\beta}+8)}}.
\end{equation*}

\subsection{Proof of Lemma \ref{lm:F.check.F.tilde}}\label{sec:apdx.F.check.F.tilde}
    By Markov inequality, for any $u\in\mathbb{R}$ and $\delta>0$,
    \begin{align*}
        \mathbb{P}\left(\hat{F}_{n,l}(u) - \tilde{F}_{n,l,m}(u) > \delta\right) &\leq \delta^{-1}\mathbb{E}[\hat{F}_{n,l}(u) - \tilde{F}_{n,l,m}(u)];\\
        \mathbb{P}\left(\tilde{F}_{n,l,m}(u) - \hat{F}_{n,l}(u) > \delta\right) &\leq \delta^{-1}\mathbb{E}[\tilde{F}_{n,l,m}(u) - \hat{F}_{n,l}(u)].
    \end{align*}
    Combining these inequalities lead to
    \begin{align}
        \begin{split}\label{eq:F.check.F.tilde.bound}
            &\mathbb{P}\left(|\hat{F}_{n,l}(u) - \tilde{F}_{n,l,m}(u)| > \delta\right)\\ &\leq 2\delta^{-1}\left|\mathbb{E}[\hat{F}_{n,l}(u) - \tilde{F}_{n,l,m}(u)]\right|\\
            &= 2\delta^{-1} \left|\mathbb{P}(l^{-1/2}|\check{B}_{i,l}|_\infty \leq u) - \mathbb{P}(l^{-1/2}|\tilde{B}_{i,l,m}|_\infty \leq u)\right|\\
            &\leq 2C\delta^{-1} \Psi(p,l),
        \end{split}
    \end{align}
    where the last inequality follows Corollary \ref{cor:B.check.B.tilde}.
    Note that $C$ is a constant that is not related to $u$.
    Taking $\delta = (\Psi(p,l))^{1-\epsilon}$, we have that for any $u\in\mathbb{R}$, with probability no less than $1-(\Psi(p,l))^\epsilon$,
    \begin{equation*}
        \left|\hat{F}_{n,l}(u) - \tilde{F}_{n,l,m}(u)\right| \leq C (\Psi(p,l))^{1-\epsilon}.
    \end{equation*}
    The proof is finished by taking supremum on both sides of the inequality.

\subsection{Proof of Lemma \ref{lm:B.check.CDF}}\label{sec:apdx.B.check.CDF}
    By the $m$-dependent property of $\tilde{D}_{j,m}$, for any $|i-i'|\geq l+m-1$, $\tilde{B}_{i,l,m}$ is independent to $\tilde{B}_{i',l,m}$.
    We can therefore divide $\{\tilde{B}_{i,l,m}\}_{i=l}^{n}$ into $(l+m-1)$ many groups, such that all entries within the same group are mutually independent.
    Specifically, for each $0\leq k\leq l+m-2$, define the index set
    \begin{equation*}
        \mathcal{G}_k^{l,m} = \left\{l+j(l+m-1)+k; j=0,1,2,\cdots, \left\lceil\frac{n-l-k}{l+m-1}\right\rceil\right\}.
    \end{equation*}
    Then for each $0\leq k\leq m+l-2$, $\{\tilde{B}_{i,l,m}\}_{i\in\mathcal{G}_k^{l,m}}$ are mutually independent.
    Define a within-group empirical CDF,
    \begin{equation*}
        \tilde{F}^{(k)}_{n,l,m}(u) = |\mathcal{G}_k^{l,m}|^{-1}\sum_{i\in\mathcal{G}_k^{l,m}} \mathbf{1}\{l^{-1/2}|\tilde{B}_{i,l,m}|_\infty \leq u\},
    \end{equation*}
    and the empirical CDF $\tilde{F}_{n,l,m}(u)$ can be written as
    \begin{equation*}
        \tilde{F}_{n,l,m}(u) = (n-l+1)^{-1}\sum_{k=0}^{l+m-2} |\mathcal{G}_k^{l,m}|\tilde{F}_{n,l,m}^{(k)}(u).
    \end{equation*}
    By Corollary 1 in \cite{massart1990tight}, for any $\delta>0$,
    \begin{equation*}
        \mathbb{P}\left(\sup_{u\in\mathbb{R}}\left|\tilde{F}^{(k)}_{n,l,m}(u) - \mathbb{P}(l^{-1/2}|\tilde{B}_{i,l,m}|_\infty \leq u)\right| > \delta\right) \leq 2\exp(-2|\mathcal{G}_k^{l,m}|\delta^2).
    \end{equation*}
    Note that for any $k$,
    \begin{equation*}
        |\mathcal{G}_k^{l,m}| \leq \left\lceil \frac{n-l+1}{l+m-1}\right\rceil \leq \frac{n+m}{l+m-1}.
    \end{equation*}
    By triangle inequality, the left hand side of \eqref{eq:B.check.CDF} can be upper bounded by
    \begin{align*}
            &\mathbb{P}\left(\sup_{u\in\mathbb{R}}\left|\tilde{F}_{n,l,m}(u) - \mathbb{P}\left(l^{-1/2}|\tilde{B}_{i,l,m}|_\infty \leq u\right)\right| > \delta \right)\\
            \leq & \sum_{k=0}^{l+m-2} \mathbb{P}\left(\frac{|\mathcal{G}_k^{l,m}|}{n-l+1} \cdot \sup_{u\in\mathbb{R}}\left|\tilde{F}^{(k)}_{n,l,m}(u) - \mathbb{P}(l^{-1/2}|\tilde{B}_{i,l,m}|_\infty \leq u)\right| > \frac{\delta}{l+m-1}\right)\\
            \leq & 2\sum_{k=0}^{l+m-2} \exp\left(-2|\mathcal{G}_k^{l,m}|\left(\frac{n-l+1}{|\mathcal{G}_k^{l,m}|(l+m-1)}\delta\right)^2\right)\\
            \leq & 2(l+m-1)\exp\left(-2\frac{(n-l+1)^2}{(n+m)(l+m-1)}\delta^2\right).
    \end{align*}
    By taking
    \begin{equation*}
        \delta^2 = \frac{(n+m)(l+m-1)}{(n-l+1)^2}\left(\lambda^2 + \frac{1}{2}\log(l+m-1)\right),
    \end{equation*}
    we finish the proof of Lemma \ref{lm:B.check.CDF}.

\section{Proof of Section \ref{sec:GA.prec}}
\subsection{Proof of Lemma \ref{lm:GA.cov.linear}}\label{sec:apdx.GA.cov.linear}

Using the fact that $|\Omega A \Omega|_\infty \leq |\Omega|_1^2 |A|_\infty$ for any matrix $A$, Lemmas \ref{lm:mtg.approx} and \ref{lm:m.dep.approx} immediately implies that
\begin{equation}\label{eq:like.lm.3.1}
    \mathbb{P}\left(|S_n^\Omega-n^{-1/2}T_n^\Omega|_\infty\geq\delta\right)\leq 2p^2\exp\left\{-c_1\left(\frac{\delta^2}{|\Omega|_1^4 Q_1(n)}\wedge\frac{\delta}{|\Omega|_1^2\sqrt{Q_1(n)}}\right)\right\},
\end{equation}
and
\begin{equation}\label{eq:like.lm.3.2}
    \mathbb{P}\left(n^{-1/2}|T_n^\Omega-\tilde{T}_{n,m}^\Omega|_\infty\geq\delta\right)\leq 2p^2\exp\left\{-c_2\left(\frac{\delta^2}{|\Omega|_1^4 Q_2(m)}\wedge\frac{\delta}{|\Omega|_1^2\sqrt{Q_2(m)}}\right)\right\}
\end{equation}
where $c_1>0$ and $c_2>0$ are constants related to $C_0$ and $c_0$.
Next, we can build triadic blocks by the following decomposition of $\tilde{T}_{n,m}^\Omega$ similar to \eqref{eq:T.tilde.decomp},
\begin{equation*}
    \tilde{T}_{n,m}^\Omega = \sum_{j=0}^{q_n} B_j^\Omega, \quad\quad \text{where } B_j^\Omega = \sum_{i=3jm+1}^{3(j+1)m \wedge n} \tilde{D}_{i,m}^\Omega,\quad 0\leq j \leq q_n.
\end{equation*}
These blocks have the same conditional independence structure as discussed in Section \ref{sec:GA.cov}.
Therefore we can define
\begin{align*}
    V^{\Omega *}_0(\bm{\epsilon}_{1,0}) &= \var(B_0^\Omega \mid \bm{\epsilon}_0),\\
    V^{\Omega *}_j(\bm{\epsilon}_{j,0},\bm{\epsilon}_{j+1,0}) &= \var(B_j^\Omega \mid \bm{\epsilon}_0), \quad 1 \leq j \leq q_n.
\end{align*}
Consider i.i.d.~standard Gaussian random vectors $Y_j^* \in \mathbb{R}^{p^2}$ that are independent of $\bm\epsilon_0$. The partial sum
\[
T_n^{\Omega *} = [V^{\Omega *}_0(\bm{\epsilon}_{1,0})]^{1/2}Y^*_0 + \sum_{j=1}^{q_n} [V^{\Omega *}_j(\bm{\epsilon}_{j,0},\bm{\epsilon}_{j+1,0})]^{1/2}Y^*_j
\]
follows a Gaussian distribution with mean 0 and variance
\begin{equation*}
V^{\Omega *}_0(\bm{\epsilon}_{1,0}) + \sum_{j=1}^{q_n} V^{\Omega *}_j(\bm{\epsilon}_{j,0},\bm{\epsilon}_{j+1,0}) = \sum_{j=1}^{q_n+1} V^\Omega_j(\bm{\epsilon}_{j,0}),
\end{equation*}
where $V^\Omega_{q_n+1}(\bm{\epsilon}_{q_n+1,0}) = V^{\Omega *}_0(\bm{\epsilon}_{q_n+1,0})$ and $V^\Omega_j(\bm{\epsilon}_{j,0}) = V^{\Omega *}_j(\bm{\epsilon}_{j,0},\bm{\epsilon}_{j,0})$ for $1 \leq j \leq q_n$.
Similar to \eqref{eq:TnY.def}, for a set of independent standard Gaussian random vectors $Y_j \in \mathbb{R}^{p^2}$, the partial sum
\begin{equation*}
T_n^{\Omega Y} = \sum_{j=1}^{q_n+1} [V_j(\bm{\epsilon}_{j,0})]^{1/2}Y_j
\end{equation*}
is identically distributed to $T_n^{\Omega *}$ conditional on $\bm{\epsilon}_0$.

Therefore, a result similar to Lemma \ref{lm:cond.CCK} holds, with the constants related to the true precision matrix $\Omega$ in addition.
To see this, a careful look at the proof of Lemma \ref{lm:cond.CCK} in Appendix \ref{sec:pf.cond.CCK} reveals that the result can also be extended to the precision case as long as all the bounds still hold with coefficient matrices $A_t$ replaced by $\Omega A_t$.
This boils down to ensuring the following alternative version of Condition \ref{cond.A}:
\begin{condition}\label{cond.A.alt}
    For some $\beta>3/4$ and a constant $C'_0>0$, the coefficient matrices satisfy
    \begin{equation*}
        \max_{1\leq j\leq p}|(\Omega A_{t})_{j\cdot}| = \max_{1\leq j\leq p} \left(\sum_{l=1}^d (\Omega A_{t})^2_{jl}\right)^{1/2} \leq C'_0(1\vee t)^{-\beta}.
    \end{equation*}
\end{condition}
We shall see that Condition \ref{cond.A.alt} is implied by Condition \ref{cond.A}, with $C'_0 = C_0 |\Omega|_1$.
By triangle inequality,
\[
|(\Omega A_t)_{j\cdot}| = \left|\sum_{l=1}^p \Omega_{jl} (A_t)_{l\cdot}\right|_2 \leq \sum_{l=1}^p |\Omega_{jl} (A_t)_{l\cdot}|.
\]
Using H\"older's inequality, the right hand side can be further upper bounded by
\[
\sum_{l=1}^p |\Omega_{jl}| \max_{1\leq k\leq p} |(A_t)_{k\cdot}|.
\]
Note that $\max_{1\leq k\leq p}|(A_t)_{k\cdot}| \leq C_0 (1\vee t)^{-\beta}$, and $\max_{1\leq j\leq p} \sum_{l=1}^p |\Omega_{jl}| = |\Omega|_1$, thus Condition \ref{cond.A.alt} holds with $C'_0 = C_0 |\Omega|_1$.
Therefore, we arrive at the following lemma similar to Lemma \ref{lm:cond.CCK}.
\begin{lemma}\label{lm:cond.CCK.prec}
    For a constant $C>0$ related to $C_0$ and $c_0$,
    \begin{equation}\label{eq:like.lm.3.3}
        \sup_{u\in\mathbb{R}} \left| \mathbb{P}\left(n^{-1/2}|\tilde{T}^\Omega_{n,m} - \mathbb{E}[\tilde{T}^\Omega_{n,m}\mid \bm{\epsilon}_0]|_\infty \geq u \mid \bm{\epsilon_0}\right) - \mathbb{P}\left(n^{-1/2}|T^{\Omega Y}_n|_\infty\geq u\mid\bm{\epsilon}_0\right)\right|
        \leq C\left(\frac{m^4|\Omega|_1^8 \log^5(pn)}{n}\right)^{1/4}
    \end{equation}
    with probability no less than $1-4p^2\exp(-c_2\delta n^{1/2}/m^{1/2}) - 2\exp(-c_2 m^{2\epsilon} |\Omega|_1^{4\epsilon})$, where $\delta > 0$ and $0<\epsilon < 1$ are arbitrary constants, and $c_2>0$ is a constant related to $C_0$.
\end{lemma}
\begin{proof}[Proof of Lemma \ref{lm:cond.CCK.prec}]
    Define $B_j^{\Omega,\bm\epsilon_0} = B_j^\Omega - \mathbb{E}[B_j^\Omega \mid \bm\epsilon_0]$ for brevity.
    Similar to the proof of Lemma \ref{lm:cond.CCK} in Appendix \ref{sec:pf.cond.CCK}, we need to verify the following three conditions for any $1\leq s,t\leq p$:
    \begin{align*}
        & (q_n+1)^{-1}\sum_{j=0}^{q_n} \mathbb{E}[(B_j^{\Omega,\bm\epsilon_0})_{(s,t)}^2 \mid \bm\epsilon_0] \geq b_1^2,\\
        & (q_n+1)^{-1}\sum_{j=0}^{q_n} \mathbb{E}[(B_j^{\Omega,\bm\epsilon_0})_{(s,t)}^4 \mid \bm\epsilon_0] \leq b_2^2(\beta_n^\Omega)^2,\\
        & \mathbb{E}\left[ \exp\left(|(B_j^{\Omega,\bm\epsilon_0})_{(s,t)}| / \beta_n^\Omega\right) \mid \bm\epsilon_0 \right] \leq 2 \quad \text{for all } 0\leq j\leq q_n,
    \end{align*}
    for some constants $0< b_1\leq b_2$ related to $C_0$ and $c_0$, and a sequence of constants $\beta_n^\Omega$.
    It remains to find the rate of $\beta_n^\Omega$.
    Recall that we have found $\beta_n \asymp m^2$ in the proof of Lemma \ref{lm:cond.CCK}, $\beta_n^\Omega$ should also be of the same order with respect to $m$.
    The more important question is how it evolves with $|\Omega|_1$.
    Define a function $\mathcal{V}^\Omega(\bm{\epsilon}_0) \in \mathbb{R}^{p^2}$, such that each entry can be written as
    \begin{equation*}
        (\mathcal{V}^\Omega(\bm{\epsilon}_0))_{(s,t)} = (q_n+1)^{-1}\sum_{j=0}^{q_n}\mathbb{E}[(B_j^{\Omega, \bm{\epsilon}_0})_{(s,t)}^2\mid  \bm{\epsilon}_0].
    \end{equation*}
    Following the same proof as Lemma \ref{lm:V.conc}, with all coefficient matrices $A_j$ replaced by $\Omega A_j$, we arrive at the following result using Condition \ref{cond.A.alt} with $C_0' = C_0|\Omega|_1$.
    For a constant $c_2>0$ related to $C_0$,
    \begin{equation*}
        \mathbb{P}(|\mathcal{V}^\Omega(\bm{\epsilon}_0) - \mathbb{E}[\mathcal{V}^\Omega(\bm{\epsilon}_0)]|_\infty > \delta) \leq 4 p^2 \exp\left\{-c_2\left(\frac{n\delta^2}{|\Omega|_1^8 m^5} \wedge \frac{n^{1/2}\delta}{|\Omega|_1^4 m^{5/2}}\right)\right\}.
    \end{equation*}
    Equivalently, for any choice of universal constant $\delta>0$, with probability no less than $1-4p^2\exp(-c_2\delta n^{1/2}/m^{1/2})$, we have
    \[
    |\mathcal{V}^\Omega(\bm{\epsilon}_0) - \mathbb{E}[\mathcal{V}^\Omega(\bm{\epsilon}_0)]|_\infty \leq \delta m^2 |\Omega|_1^4.
    \]
    Meanwhile, we have for any $1\leq s,t\leq p$
    \[
    \mathbb{E}[(\mathcal{V}^\Omega(\bm{\epsilon}_0))_{(s,t)}] = (q_n+1)^{-1}\sum_{j=0}^{q_n}\mathbb{E}[(B_j^\Omega)_{(s,t)}^2] = O(m^2|\Omega|_1^4).
    \]
    Therefore the first two of the aforementioned three conditions hold with $\beta_n^\Omega \asymp m^2|\Omega|_1^4$.
    For the third condition, by Taylor expansion, Fubini Theorem and Lemma \ref{lm:Lk.norm},
    \[
    \mathbb{E}[\exptext(|(B_j^{\bm{\epsilon}_0})_{(s,t)}| / m^2|\Omega_1|^4)\mid  \bm{\epsilon}_0] = \sum_{k=0}^\infty \frac{\mathbb{E}[|(B_j^{\Omega,\bm{\epsilon}_0})_{(s,t)}|^k\mid \bm{\epsilon}_0]}{k!m^{2k}|\Omega|_1^{4k}} \leq 1 + C_u \sum_{k=1}^\infty \frac{k^k (V_j^{\Omega*}(\bm{\epsilon}_{j,0}, \bm{\epsilon}_{j+1,0}))^{k/2}_{(s,t)}}{k!m^{2k}|\Omega|_1^{4k}}.
    \]
    Similar to the manipulation for verifying \eqref{eq:cond.CCK.E1} in Appendix \ref{sec:pf.cond.CCK}, by replacing $A_j$ with $\Omega A_j$, we can show that for any choice of $0<\epsilon<1$, $|(V_j^{\Omega*}(\bm{\epsilon}_{j,0}, \bm{\epsilon}_{j+1,0}))_{(s,t)} - \mathbb{E}[(V_j^{\Omega*}(\bm{\epsilon}_{j,0}, \bm{\epsilon}_{j+1,0}))_{(s,t)}]| \lesssim (m^2|\Omega|_1^4)^{1+\epsilon}$ with probability no less than $1-2\exp(-c_2 (m^2|\Omega|_1^4)^{\epsilon})$.
    By Stirling's Formula, we have
    \[
    \mathbb{E}[\exptext(|(B_j^{\bm{\epsilon}_0})_{(s,t)}| / m^2|\Omega_1|^4)\mid  \bm{\epsilon}_0] \leq 1+C_u\sum_{k=1}^\infty \frac{e^k}{\sqrt{2\pi k}(m|\Omega|_1^2)^{k(1-\epsilon)}}
    \]
    with probability no less than $1-2\exp(-c_2 (m^2|\Omega|_1^4)^{\epsilon})$.
    This bound converges to one with high probability as $n\rightarrow\infty$.
    For finite-sample performance, the third condition holds with $\beta_n^\Omega = m^2|\Omega|_1^4$ when $m$ is large enough such that $(m|\Omega|_1^2/e)^{1-\epsilon} \geq C_u+1$.
    
    Therefore, with probability no less than $1-4p^2\exp(-c_2\delta n^{1/2} / m^{1/2}) - 2\exp(-c_2 (m^2|\Omega|_1^4)^{\epsilon})$, all three conditions hold simultaneously, with $\beta_n^\Omega \asymp m^2|\Omega|_1^4$.
    This finishes the proof of Lemma \ref{lm:cond.CCK.prec}.
\end{proof}

Condition \ref{cond.A.alt} also ensures a similar result as Lemma \ref{lm:uncond.CCK}.
For each $j\in [q_n+1]$, let $Z_j^\Omega \in \mathbb{R}^{p^2}$ be a Gaussian random vector with zero mean and variance $\var([V_j^\Omega(\bm{\epsilon}_{j,0})]^{1/2}Y_j)$.
Then the normalized partial sum $S_n^{\Omega Z} = n^{-1/2}\sum_{j=1}^{q_n+1} Z_j^\Omega$ can be considered as a Gaussian approximation of $T_n^{\Omega Y}$.
The following result holds.
\begin{lemma}\label{lm:uncond.CCK.prec}
    For a constant $C'>0$ related to $C_0$, $c_0$, $C_e$ and $c_e$,
    \begin{equation}\label{eq:like.lm.3.4}
\sup_{u\in\mathbb{R}}\left| \mathbb{P}(n^{-1/2}|T^{\Omega Y}_n|_\infty \geq u) - \mathbb{P}(|S_n^{\Omega Z}|_\infty \geq u)\right| \leq C'\left(\frac{m^4|\Omega|_1^8 \log^5(pn)}{n}\right)^{1/4}.
\end{equation}
\end{lemma}
\begin{proof}[Proof of Lemma \ref{lm:uncond.CCK.prec}]
    Let $U_j^\Omega = [V_j^\Omega(\bm\epsilon_{j,0})]^{1/2} Y_j$.
    Similar to the proof of Lemma \ref{lm:uncond.CCK} in Appendix \ref{sec:pf.uncond.CCK}, we need to verify the following three conditions for any $1\leq s,t\leq p$:
    \begin{align*}
        & (q_n+1)^{-1}\sum_{j=0}^{q_n} \mathbb{E}[(U_j^\Omega)_{(s,t)}^2 ] \geq b_1^2,\\
        & (q_n+1)^{-1}\sum_{j=0}^{q_n} \mathbb{E}[(U_j^\Omega)_{(s,t)}^4 ] \leq b_2^2(\beta_n^\Omega)^2,\\
        & \mathbb{E}\left[ \exp\left(|(U_j^\Omega)_{(s,t)}| / \beta_n^\Omega\right) \right] \leq 2 \quad \text{for all } 1\leq j\leq q_n+1,
    \end{align*}
    for some constants $0< b_1\leq b_2$ related to $C_0$ and $c_0$, and a sequence of constants $\beta_n^\Omega$.
    It remains to find the rate of $\beta_n^\Omega$.

    For the first condition, recall the proof of Lemma \ref{lm:uncond.CCK} that $\mathbb{E}[(U_j)_{(s,t)}^2] = 3m\sum_{k=1}^m \zeta_{k,st}$.
    The summation on the right hand side converges to the diagonal entries of the long-run covariance matrix of the process $\mathcal X_t$ as $m\rightarrow\infty$, which was lower bounded by Condition \ref{cond.G}.
    Recall that we have already denoted this long-run covariance matrix as $\Sigma^*$.
    Similarly, we have $\mathbb{E}[(U_j^\Omega)_{(s,t)}^2] = 3m\sum_{k=1}^m \zeta_{k,st}^\Omega$, where
    \begin{multline*}
        \zeta_{k,st}^\Omega = \sum_{l=0}^\infty \sum_{l'=0}^\infty \bigg[\bigg.  (\Omega A_{l+k} A_{l'+k}^\top \Omega )_{ss}(\Omega A_l A_{l'}^\top \Omega )_{tt} + (\Omega A_{l+k} A_{l'+k}^\top \Omega )_{tt}(\Omega A_l A_{l'}^\top \Omega )_{ss}\\
        + (\Omega A_{l+k} A_{l'+k}^\top \Omega )_{st}(\Omega A_l A_{l'}^\top \Omega )_{ts} + (\Omega A_{l+k} A_{l'+k}^\top \Omega )_{ts}(\Omega A_l A_{l'}^\top\Omega )_{st}\bigg.\bigg],
    \end{multline*}
    or alternatively, using Kronecker product $\otimes$, Hadamard product $\circ$, and commutation matrix $K_{pp}$,
    \[
    \zeta_k^\Omega = \sum_{l=0}^\infty \sum_{l'=0}^\infty (I_{p^2} + K_{pp}) \bigg[ \text{diag}(L_0^\Omega) \otimes \text{diag}(L_1^\Omega) + \vectorize(L_1^\Omega \circ (L_0^\Omega)^\top) \bigg],
    \]
    where $L_1^\Omega = \Omega A_{l+k} A_{l'+k}^\top \Omega$, and $L_0^\Omega = \Omega A_l A_{l'}^\top \Omega$.
    When $m\rightarrow\infty$, the sum $\sum_{k=1}^m \zeta_{k,st}^\Omega$ converges to the diagonal entries of the long-run covariance matrix of the process $\mathcal X_t^\Omega$.
    By Roth's Column Lemma (\cite{roth34on}, also called the \emph{vec trick}), we have $\mathcal X_t^\Omega = (\Omega\otimes\Omega)\mathcal X_t$.
    Therefore, the long-run covariance matrix of $\mathcal X_t^\Omega$ can be written as $\Sigma^{\Omega*} = (\Omega\otimes\Omega)\Sigma^*(\Omega\otimes\Omega)$.
    Note that for any positive semi-definite matrix, the diagonal entries are lower bounded by its smallest eigenvalue.
    Thus
    \[
    \Sigma^{\Omega*}_{(s,t),(s,t)} \geq \lambda_{\min}(\Sigma^{\Omega*}) \geq \lambda_{\min}(\Sigma^*) \cdot \lambda_{\min}^2(\Omega\otimes\Omega).
    \]
    The long-run covariance matrix $\Sigma^*$ of $\mathcal X_t$ can be expressed by Bartlett's Formula as $\Sigma^* = 4\pi \int_{-\pi}^{\pi} f_X(\omega) \otimes f_X(-\omega) d\omega$.
    Recall that $\Sigma = \int_{-\pi}^\pi f_X(\omega) d\omega$.
    Cauchy-Schwarz Inequality implies $\Sigma^* \succeq C \cdot (\Sigma \otimes \Sigma)$ for some universal constant $C>0$, thus $\lambda_{\min}(\Sigma^*) \geq C \cdot \lambda_{\min}^2(\Sigma)$.
    Thus we have $\Sigma^{\Omega*}_{(s,t),(s,t)} \geq C c_e^2 C_e^{-4}$, which is a constant related to $c_e$ and $C_e$, the lower- and upper-bounds of the true underlying covariance matrix.
    Therefore $\mathbb{E}[(U_j^\Omega)_{(s,t)}^2]$ is lower-bounded by the order of $O(m)$, up to a constant related to $C_e$ and $c_e$. This proves the first condition.

    For the second condition, note that
    \[
    \mathbb{E}[(U_j^\Omega)_{(s,t)}^2] = \mathbb{E}[ \mathbb{E}[ (U_j^\Omega)_{(s,t)}^2 \mid \bm\epsilon_{j,0}]] = \mathbb{E}[(V_j^\Omega(\bm\epsilon_{j,0}))_{(s,t),(s,t)}] = O(m^2|\Omega|_1^4).
    \]
    A similar result like Lemma \ref{lm:Lk.norm.U} holds for $U_j^\Omega$.
    Following the same proof, we can show that there exists some universal constant $C_1>0$ such that for all $k\geq 2$, $\|(U_j^\Omega)_{(s,t)}\|_k \leq C_1 k\|(U_j^\Omega)_{(s,t)}\|_2$.
    Therefore, we arrive at
    \[
    (q_n+1)^{-1}\sum_{j=0}^{q_n} \mathbb{E}[(U_j^\Omega)_{(s,t)}^4 ] \lesssim m^4|\Omega|_1^8.
    \]
    The second condition holds with $\beta_n^\Omega = m^2|\Omega|_1^4$.

    For the third condition, by Taylor Expansion, Funibi's Theorem, the above result on the norms of $U_j^\Omega$, and Stirling's Formula, we have
    \[
    \mathbb{E}[\exp(|(U_j^\Omega)_{(s,t)}|/m^2|\Omega|_1^4)] = 1+\sum_{k=1}^\infty \frac{\|(U_j)_{(s,t)}\|_k^k}{k!(m^2|\Omega|_1^4)^{k}} \leq 1+C_u\sum_{k=1}^\infty \frac{e^k}{\sqrt{2\pi k}(m|\Omega|_1^2)^{k}} \leq 2.
    \]
    Since all three conditions hold simultaneously with $\beta_n^\Omega \asymp m^2|\Omega|_1^4$, the proof of Lemma \ref{lm:uncond.CCK.prec} is finished by Theorem \ref{prop:CCK17}.

\end{proof}

Finally, we provide a Gaussian comparison result between $Z^\mathfrak{S}$ and $S_n^{\Omega Z}$ similar to Lemma \ref{lm:GA.comparison}.
Denote their respective covariance matrices by $\Sigma^\mathfrak{S}$ and $\Sigma_S^\Omega$.
Following the same proof as in Lemma \ref{lm:GA.comparison}, but replacing $A_t$ with $\Omega A_t$, Condition \ref{cond.A.alt} ensures that $|\Sigma^{\Omega*} - \Sigma^\mathfrak{S}|_\infty \lesssim |\Omega|_1^4 n^{-\tilde{\beta}}$ and $|\Sigma^{\Omega*} - \Sigma_S^\Omega|_\infty \lesssim |\Omega|_1^4 m^{-\tilde{\beta}}$.
Therefore $|\Sigma^\mathfrak{S} - \Sigma_S^\Omega|_\infty \lesssim |\Omega|_1^4 m^{-\tilde{\beta}}$.
We have already discussed that all diagonal entries of $\Sigma^{\Omega*}$ are upper-bounded, and lower-bounded by a constant.
By Theorem \ref{thm:comparison}, for a constant $C>0$ related to $C_0$, $c_0$,
\begin{equation}\label{eq:like.lm.3.5}
    \sup_{u\in \mathbb{R}}\left|\mathbb{P}(|Z^\mathfrak{S}|_\infty \geq u) - \mathbb P(|S_n^{\Omega Z}|_\infty \geq u)\right| \leq C |\Omega|_1^2 \frac{\log(p)}{m^{\tilde\beta/2}}.
\end{equation}

The proof of Lemma \ref{lm:GA.cov.linear} is finished by combining \eqref{eq:like.lm.3.1}, \eqref{eq:like.lm.3.2}, \eqref{eq:like.lm.3.3}, \eqref{eq:like.lm.3.4}, \eqref{eq:like.lm.3.5}, and following the same proof of Theorem \ref{thm:GA.cov} shown in Section \ref{sec:GA.cov}.

\subsection{Proof of Lemma \ref{lm:Sigma.op.small}}\label{sec:apdx.Sigma.op.small}
    Note that $\hat{\Sigma}_n - \Sigma$ is symmetric and real-valued, we can write
    \begin{equation*}
        |\hat{\Sigma}_n - \Sigma|_{\op} = \sup_{v\in\mathbb{R}^p, |v|=1} |v^\top (\hat{\Sigma}_n - \Sigma) v|,
    \end{equation*}
    the right-hand side being a supremum over the $p$-dimensional unit Euclidean sphere, denoted by $\mathbb{S}^{p-1}$.
    We can discretize this compact sphere by nets; see Section 5 of \cite{vershynin2010introduction}.
    Let $\mathcal{N}_{1/4}$ be a $1/4$-net of $\mathbb{S}^{p-1}$.
    The cardinality of this net $|\mathcal{N}_{1/4}| \leq 9^p$, and
    \begin{equation*}
        \sup_{v\in\mathbb{R}^p, |v|=1} |v^\top (\hat{\Sigma}_n - \Sigma) v| \leq \frac{4}{3}\max_{v\in\mathcal{N}_{1/4}} |v^\top (\hat{\Sigma}_n - \Sigma) v|.
    \end{equation*}
    Therefore by union bound,
    \begin{align*}
            \mathbb{P}\left(n^{1/2}|\hat{\Sigma}_n - \Sigma|_{\op} > \delta\right) &\leq \mathbb{P}\left(n^{1/2}\max_{v\in\mathcal{N}_{1/4}} |v^\top (\hat{\Sigma}_n - \Sigma)v| > 3\delta/4\right)\\
            &\leq 9^p \sup_{|v|=1}\mathbb{P}\left(n^{1/2} |v^\top (\hat{\Sigma}_n - \Sigma)v| > 3\delta/4\right).
    \end{align*}
    For any $|v| = 1$, let $\xi_i = v^\top X_i\in\mathbb{R}$.
    Note that $X_i\in\mathbb{R}^p$ is jointly Gaussian, therefore $\xi_i$ is a Gaussian process.
    Let $\xi = [\xi_1,\cdots,\xi_n]^\top$. This is also a centered Gaussian random vector in $\mathbb{R}^n$.
    Denote the autocovariance of any $\xi_i$ as
    \begin{equation*}
        \gamma_k^{\xi} = \cov(\xi_i, \xi_{i+k}) = v^\top \Gamma_k v,
    \end{equation*}
    recall that $\Gamma_k$ is the $k$-lag autocovariance matrix of the process $X_i$.
    Thus the covariance matrix of $\xi$ can be written as
    \begin{equation*}
        \Sigma^\xi =
        \begin{bmatrix}
            \gamma^\xi_0 & \gamma^\xi_1 & \gamma^\xi_2 &\cdots& \gamma^\xi_{n-1}\\
            \gamma^\xi_{-1} & \gamma^\xi_0 & \gamma^\xi_1 &\cdots& \gamma^\xi_{n-2}\\
            \gamma^\xi_{-2} & \gamma^\xi_{-1} & \gamma^\xi_0 &\cdots& \gamma^\xi_{n-3}\\
            \vdots & \vdots & \vdots & \ddots & \vdots\\
            \gamma^\xi_{n-1} & \gamma^\xi_{n-2} & \gamma^\xi_{n-3} &\cdots& \gamma^\xi_0
        \end{bmatrix}
    \end{equation*}
    Let $\eta$ be a standard Gaussian random vector on $\mathbb{R}^n$. By Lemma \ref{lm:hanson.wright}, we have
    \begin{align*}
            \mathbb{P}\left(n^{1/2} |v^\top (\hat{\Sigma}_n - \Sigma)v| > 3\delta/4\right) &= \mathbb{P}\left(\left|\sum_{i=1}^n (\xi_i^2 - \mathbb{E}[\xi_i^2])\right| > \frac{3n^{1/2}\delta}{4}\right)\\
            &= \mathbb{P}\left(|\eta^\top \Sigma^\xi \eta - \mathbb{E}[\eta^\top \Sigma^\xi \eta]| > \frac{3n^{1/2}\delta}{4}\right).
    \end{align*}
    To apply Hanson-Wright inequality on the right hand side, it is required to bound the Frobenius norm and operator norm of $\Sigma^\xi$.
    We claim that the order is
    \begin{equation}\label{eq:Sigma.xi.norm}
        |\Sigma^\xi|_{\op} \asymp |\Sigma^\xi|_{\F} \asymp p n^{(1-\beta) \vee 0},
    \end{equation}
    up to a constant related to $C_0$.
    The proof of \eqref{eq:Sigma.xi.norm} is delayed to Appendix \ref{app:Sigma.xi.norm}.
    Lemma \ref{lm:Sigma.op.small} is proved by applying Lemma \ref{lm:hanson.wright}.

\subsection{Proof of Order in \eqref{eq:Sigma.xi.norm}}\label{app:Sigma.xi.norm}

Note that for Gaussian linear process \eqref{LinearProcess},
\begin{align*}
        |\Gamma_k|_{\op}^2 &= \sup_{|v|=1} \left|\Gamma_k v\right|^2 = \sup_{|v|=1} \left|\sum_{t=0}^\infty A_t A_{t+k} v\right|^2\\
        &= \sup_{|v|=1} \sum_{j=1}^p \left[\left(\sum_{t=0}^\infty A_t A_{t+k}\right)_{j\cdot} v\right]^2 \\
        &= \sup_{|v|=1} \sum_{j=1}^p \left[\sum_{l=1}^p\sum_{t=0}^\infty\sum_{u=1}^d (A_t)_{ju}(A_{t+k})_{lu} v_l\right]^2\\
        &\leq  \sup_{|v|=1} \sum_{j=1}^p \left[\sum_{l=1}^p v_l \sum_{t=0}^\infty \left(\sum_{u=1}^d (A_t)_{ju}^2\right)^{1/2} \left(\sum_{u=1}^d (A_{t+k})_{lu}^2\right)^{1/2}\right]^2\\
        &\leq C_0^2 \sup_{|v|=1} \sum_{j=1}^p \left[\sum_{l=1}^p v_l \sum_{t=0}^\infty (1\vee t)^{-\beta} (1\vee (t+|k|))^{-\beta}\right]^2\\
        &\leq C_0^2 p^2 |k|^{-(2\beta \wedge (4\beta-2))}.
\end{align*}
Here, the first inequality follows Cauchy-Schwarz inequality, the second inequality is implied by Condition \ref{cond.A}, and the last follows Cauchy-Schwarz inequality. Thus we have
\begin{equation*}
    |\gamma_k^{\xi}| \lesssim p (1\wedge |k|)^{-(\beta\wedge (2\beta-1))}.
\end{equation*}
This implies the order of both operator norm and Frobenius norm of $\Sigma^\xi$ by
\begin{align*}
        |\Sigma^\xi|_{\op} \leq \max_{1\leq i\leq p} \sum_{j=1}^p |\gamma_{i-j}| \leq 2\sum_{k=-n+1}^{n-1} |\gamma_k^\xi| \lesssim p n^{(1-\beta)\vee 0},
\end{align*}
and
\begin{align*}
        |\Sigma^\xi|_{\F}^2 \leq \sum_{k=-n}^n (n-k)|\gamma^\xi_k|^2 \lesssim p^2 n\sum_{k=-n}^n (1\wedge |k|)^{-(2\beta\wedge (4\beta-2))} \asymp p^2 n^{(2-2\beta)\vee 0}.
\end{align*}

\subsection{Proof of Lemma \ref{lm:GA.comparison.prec}}\label{sec:apdx.GA.comparison.prec}
Denote the covariance matrix of $Z^\mathfrak{S}$ by $\Sigma^\mathfrak{S}$, and the covariance matrix of $Z^\Omega$ by $\Sigma^\Omega$.
For convenience, for the remainder of the proof, denote $\varpi_1 = \vectorize(n^{1/2}\mathfrak{W})$, and $\varpi_2 = -\vectorize(n^{1/2} \Omega \mathfrak{S}\Omega)$.
Then \eqref{eq:residual.tail.bound} can be written as
\begin{equation}\label{eq:residual.tail.bound.redux}
        \mathbb{P}\left(|\varpi_1 - \varpi_2|_\infty >\delta \right) \leq 2\exp \left[-c_2' \left(\frac{n^{((2\beta-1/2) \wedge 3/2)}\delta}{p^2} \wedge \frac{n^{((\beta-1/4) \wedge 3/4)}\delta^{1/2}}{p}\right) + 3p\right]
\end{equation}
for some constant $c'_2>0$ related to $C_0$ and $c_e$.
Using $\mathbb{E}[Y^2] = \int_0^\infty 2t\,\mathbb{P}(Y>t)dt$, we can compute by \eqref{eq:residual.tail.bound.redux} that
\[
\mathbb{E}[|\varpi_1 - \varpi_2|_\infty^2] \lesssim \frac{p^{8}}{n^{(4\beta - 1)\wedge 3}},
\]
up to a constant related to $C_0$ and $c_e$.
Meanwhile,
\[
    \Sigma^\Omega - \Sigma^\mathfrak{S} = \mathbb{E}[\varpi_1 \varpi_1^\top] - \mathbb{E}[\varpi_2 \varpi_2^\top]
    = \mathbb{E}[(\varpi_1-\varpi_2) (\varpi_1-\varpi_2)^\top] + \mathbb{E}[(\varpi_1-\varpi_2) \varpi_2^\top] + \mathbb{E}[\varpi_2 (\varpi_1-\varpi_2)^\top].
\]
By triangle inequality and Cauchy-Schwarz Inqeuality,
\[
    |\Sigma^\Omega - \Sigma^\mathfrak{S}|_\infty \leq \mathbb{E}[|\varpi_1-\varpi_2|_\infty^2] + 2\sqrt{\mathbb{E}[|\varpi_1-\varpi_2|_\infty^2] \max_{1\leq s,t\leq p}\mathbb{E}[(\varpi_2)_{(s,t)}]}.
\]
Note that $\max_{1\leq s,t\leq p}\mathbb{E}[(\varpi_2)^2_{(s,t)}]$ is the maximum diagonal entry of the covariance matrix $\Sigma^\mathfrak{S}$.
When $n$ is large enough, it is upper bounded by $|\Omega|_1^4$ times a constant related only to $C_0$, following a similar computation as \eqref{eq:Sigma.S.close}, with $A_j$ replaced by $\Omega A_j$ and applying Condition \ref{cond.A.alt} instead of Condition \ref{cond.A}.
Therefore,
\[
    |\Sigma^\Omega - \Sigma^\mathfrak{S}|_\infty \lesssim \frac{p^{8}}{n^{(4\beta - 1)\wedge 3}} + \frac{|\Omega|_1^2 p^{4}}{n^{(2\beta - 1/2)\wedge 3/2}}
\]
As long as we want the right hand side to be asymptotically small, the second term dominates.
The assertion in Lemma \ref{lm:GA.comparison.prec} is then proved by an application of Theorem \ref{thm:comparison}.

\section{Proof of Section \ref{sec:bootstrap.prec}}
\subsection{Proof of Inequality \eqref{eq:B.check.diamond}}\label{app:X.check.diamond}
Note that
\begin{align*}
        &\quad \hat{\Omega}_n (X_j X_j^\top -\hat{\Sigma}_n) \hat{\Omega}_n - \Omega (X_j X_j^\top -\hat{\Sigma}_n) \Omega\\
        &= \hat{\Omega}_n (X_j X_j^\top -\hat{\Sigma}_n) \hat{\Omega}_n - \Omega (X_j X_j^\top -\hat{\Sigma}_n) \hat{\Omega}_n  + \Omega (X_j X_j^\top -\hat{\Sigma}_n) \hat{\Omega}_n - \Omega (X_j X_j^\top -\hat{\Sigma}_n)\Omega\\
        &= (\hat{\Omega}_n - \Omega)(X_j X_j^\top -\hat{\Sigma}_n)\hat{\Omega}_n + \Omega (X_j X_j^\top -\hat{\Sigma}_n) (\hat{\Omega}_n - \Omega)\\
        &= (\hat{\Omega}_n - \Omega)(X_j X_j^\top -\hat{\Sigma}_n)\Omega + \Omega (X_j X_j^\top -\hat{\Sigma}_n) (\hat{\Omega}_n - \Omega)  + (\hat{\Omega}_n - \Omega)(X_j X_j^\top -\hat{\Sigma}_n) (\hat{\Omega}_n - \Omega).
\end{align*}
For notational convenience, let
$\check M_{i,l} = \sum_{j=i-l+1}^i (X_j X_j^\top - \hat\Sigma_n)$, $\check M_{i,l}^\Omega = \sum_{j=i-l+1}^i \hat\Omega_n (X_j X_j^\top - \hat\Sigma_n) \hat\Omega_n$, and $\check M_{i,l}^{\Omega,\diamond} = \sum_{j=i-l+1}^i \Omega(X_j X_j^\top - \hat\Sigma_n)\Omega$, such that $\check B_{i,l} = \vectorize(\check M_{i,l})$, $\check{B}^\Omega_{i,l} = \vectorize(\check{M}^\Omega_{i,l})$, and $\check{B}_{i,l}^{\Omega,\diamond} = \vectorize(\check{M}_{i,l}^{\Omega,\diamond})$.
Sum both sides of the first equation over $j$ from $i-l+1$ to $i$, and we get
\[
\check M_{i,l}^\Omega - \check M_{i,l}^{\Omega,\diamond} = (\hat \Omega_n - \Omega)\check M_{i,l}\Omega + \Omega\check M_{i,l}(\hat\Omega_n - \Omega) + (\hat \Omega_n - \Omega)\check M_{i,l} (\hat\Omega_n - \Omega).
\]
Using the triangle inequality and the fact that $|AB|_\infty \leq |A|_1 |B|_\infty$, we have
\begin{align*}
    \begin{split}
        & |\check{B}_{i,l}^{\Omega} - \check{B}_{i,l}^{\Omega,\diamond}|_\infty \leq 2|\Omega|_1 |\check B_{i,l}|_\infty |\hat{\Omega}_n - \Omega|_1 +  |\check B_{i,l}|_\infty |\hat{\Omega}_n - \Omega|_1^2.
    \end{split}
\end{align*}
The proof is finished by noting that $|\hat\Omega_n - \Omega|_1 \leq p|\hat\Omega_n - \Omega|_\infty$.

\subsection{Proof of Lemma \ref{lm:B.check.diamond}}\label{sec:apdx.B.check.diamond}
    We shall bound the two terms on the right hand side of \eqref{eq:B.check.diamond} separately.
    Define $\mathcal B_1 = 2p|\Omega|_1 |\check{B}_{i,l}|_\infty |\hat{\Omega}_n - \Omega|_\infty$, and $\mathcal B_2 = p^2 |\check{B}_{i,l}|_\infty |\hat{\Omega}_n - \Omega|_\infty^2$.
    To provide a tail bound for $\mathcal B_1$ and $\mathcal B_2$, it suffices to use the following two tail bounds.
    First, by Theorem \ref{thm:GA.prec}, we have for any $\delta_1>0$, up to a constant $C'>0$ related to $C_0$, $c_0$, $C_e$, and $c_e$,
    \begin{equation*}
            \mathbb{P}\left(|\hat{\Omega}_n - \Omega|_\infty > \delta_1\right) \leq \mathbb{P}\left(|Z^\Omega|_\infty > n^{1/2}\delta_1\right)
            + C' \Psi^\Omega (p,n).
    \end{equation*}
    Second, by Lemma \ref{lm:B.check.GA} and Corollary \ref{cor:B.check.B.tilde}, up to a constant $C>0$ related to $C_0$ and $c_0$,
    \begin{align*}
            \sup_{u\in\mathbb{R}} \left|\mathbb{P}\left(l^{-1/2}|\check{B}_{i,l}|_\infty \leq u\right) - \mathbb{P}\left(|Z|_\infty \leq u\right)\right| \leq  C\Psi(p,l).
    \end{align*}
    Therefore for any $\delta_2>0$,
    \begin{equation*}
        \mathbb{P}\left(|\check{B}_{i,l}|_\infty > \delta_2 \right) \leq \mathbb{P}\left(|Z|_\infty > l^{-1/2} \delta_2 \right) + C\Psi(p,l).
    \end{equation*}
    We have already discussed that the diagonal entries of covariance matrices of Gaussian random vectors $Z$ and $Z^\Omega$ are upper and lower bounded.
    Therefore $\mathbb{E}[|Z^\Omega|_\infty] \lesssim \sqrt{\log(p)}$
    and $\mathbb{E}[|Z|_\infty] \lesssim \sqrt{\log(p)}$ up to a constant $C>0$ related to $C_0$ and $c_0$. By Markov inequality,
    \begin{align*}
        \mathbb{P}\left(|Z^\Omega|_\infty > n^{1/2}\delta_1\right) &\leq C n^{-1/2} \delta_1^{-1} \log^{1/2}(p),\\
        \mathbb{P}\left(|Z|_\infty > l^{-1/2} \delta_2 \right) &\leq C l^{1/2}\delta_2^{-1} \log^{1/2}(p).
    \end{align*}
    For the bound of $\mathcal B_1$, for any choice of $\delta_1>0$ and $\delta_2>0$ such that $\delta_1\delta_2 = \delta$,
    \begin{align*}
        \mathbb{P}(\mathcal B_1 > 2p|\Omega|_1\delta) &\leq \mathbb{P}(|\hat\Omega_n - \Omega|_\infty >\delta_1) + \mathbb{P}(|\check B_{i,l}|_\infty > \delta_2)\\
        &\leq C n^{-1/2} \delta_1^{-1} \log^{1/2}(p) + C l^{1/2}\delta_2^{-1} \log^{1/2}(p) + C\Psi(p,l) + C' \Psi^\Omega (p,n).
    \end{align*}
    The first two terms are on the same order under the choice $\delta_1 = n^{-1/4}l^{-1/4}\delta^{1/2}$, and $\delta_2 = n^{1/4}l^{1/4}\delta^{1/2}$.
    Therefore, we arrive at
    \[
    \mathbb{P}(\mathcal B_1 > \delta) \leq 2\sqrt{2}C|\Omega|_1^{1/2}\frac{p^{1/2}l^{1/4}\log^{1/2}(p)}{n^{1/4}\delta^{1/2}} + C\Psi(p,l) + C' \Psi^\Omega (p,n).
    \]
    We can also say that $\mathcal B_1$ is of order
    \[
    \mathcal B_1 \lesssim_p \frac{pl^{1/2}}{n^{1/2}}|\Omega|_1\log(p).
    \]
    Likewise, for the bound of $\mathcal B_2$, for any choice of $\delta_3>0$ and $\delta_4>0$ such that $\delta_3^2\delta_4 = \delta$,
    \begin{align*}
        \mathbb{P}(\mathcal B_2 > p^2\delta) &\leq \mathbb{P}(|\hat\Omega_n - \Omega|_\infty >\delta_3) + \mathbb{P}(|\check B_{i,l}|_\infty > \delta_4)\\
        &\leq C n^{-1/2} \delta_3^{-1} \log^{1/2}(p) + C l^{1/2}\delta_4^{-1} \log^{1/2}(p) + C\Psi(p,l) + C' \Psi^\Omega (p,n).
    \end{align*}
    With the choice $\delta_3 = n^{-1/6}l^{-1/6}\delta^{1/3}$ and $\delta_4 = n^{1/6}l^{1/6}\delta^{1/3}$, the first two terms are of the same order. Therefore
    \[
    \mathbb{P}(\mathcal B_2 > \delta) \leq 2C\frac{p^{2/3}l^{1/6}\log^{1/2}(p)}{n^{1/3}\delta^{1/3}} + C\Psi(p,l) + C' \Psi^\Omega (p,n).
    \]
    We can also say that $\mathcal B_2$ is of order
    \[
    \mathcal B_2 \lesssim_p \frac{p^2l^{1/2}}{n}\log^{3/2}(p).
    \]
    As long as $p^2\log p \lesssim n|\Omega|_1^2$ (which holds as long as we want $\Psi^\Omega(p,n)\rightarrow 0$), the desired $|\check{B}_{i,l}^{\Omega} - \check{B}_{i,l}^{\Omega,\diamond}|_\infty$ is dominated by $\mathcal B_1$.
    The proof of Lemma \ref{lm:B.check.diamond} is finished using $l^{-1/2}|\check{B}_{i,l}^{\Omega} - \check{B}_{i,l}^{\Omega,\diamond}|_\infty < 2l^{-1/2}\mathcal B_1$ and taking the tail bound of $\mathcal B_1$.


\subsection{Proof of Corollary \ref{cor:B.check.B.tilde.Omega}}\label{app:cor.B.check.B.tilde.Omega}
The proof is similar to that of Corollary \ref{cor:B.check.B.tilde} in Section \ref{sec:bootstrap.cov}.
By triangle inequality, for any $\eta>0$ we have
\begin{multline*}
    \sup_{u\in\mathbb{R}}\left|\mathbb{P}\left(l^{-1/2}|\check{B}^\Omega_{i,l}|_\infty \leq u\right) - \mathbb{P}\left(l^{-1/2}|\tilde{B}^\Omega_{i,l,m}|_\infty \leq u\right)\right|\\
    \leq \mathbb{P}\left(l^{-1/2}|\check{B}^\Omega_{i,l} - \tilde{B}^\Omega_{i,l,m}|_\infty > \eta\right) + \sup_{u\in\mathbb{R}}\mathbb{P}\left(\left|l^{-1/2}|\tilde{B}^\Omega_{i,l,m}|_\infty - u\right| < \eta\right).
\end{multline*}
The first term of the right hand side is upper bounded by Lemma \ref{lm:B.check.diamond} and \ref{lm:B.check.B.tilde.Omega}.
Meanwhile, the second term is upper bounded by Lemma \ref{lm:B.tilde.GA} and the anti-concentration inequality for Gaussian random vectors given in Theorem 3 of \cite{chernozhukov2015comparison}.
Therefore, the desired Kolmogorov distance is upper bounded by
\begin{multline*}
    4C|\Omega|_1^{1/2}\frac{p^{1/2}\log^{1/2}(p)}{n^{1/4}\eta^{1/2}} + C|\Omega|_1^2\eta^{-1}\sqrt{\frac{l\log(p)}{n}} + C\eta\sqrt{\log(p/\eta)} + (C+C''|\Omega|_1^2) \Psi(p,l) + C\Psi(p,n) + C' \Psi^\Omega (p,n) \\
            + 2p^2\exp\left(-c_1\left(\frac{\eta^2}{|\Omega|_1^4 Q_1(l)} \wedge \frac{\eta}{|\Omega|_1^2\sqrt{Q_1(l)}}\right)\right)
            + 2p^2\exp\left(-c_2\left(\frac{\eta^2}{|\Omega|_1^4 Q_2(m)} \wedge \frac{\eta}{|\Omega|_1^2\sqrt{Q_2(m)}}\right)\right)
\end{multline*}
With the same choice of $\eta$ and $m$ as in \eqref{eq:choice.eta.m.boot}, we are left with the $|\Omega|_1^2 \Psi(p,l)$ and $C'\Psi^\Omega(p,n)$ terms dominating, which finishes the proof of Corollary \ref{cor:B.check.B.tilde.Omega}.

\section{Data-generating Mechanism in Simulation}\label{sec:apdx.simulation}

The simulation of the Gaussian linear process \eqref{LinearProcess} entails an infinite summation. Given that $A_t$ typically decreases as $t$ increases, a practical approach is to simulate the truncated form of $X_i$ as follows: 
\begin{equation*}
X_i^* = \sum_{t=0}^{N-1} A_t \epsilon_{i-t}.
\end{equation*}
\begin{algorithm}[t]
\caption{Fast simulation of Gaussian linear process --- unidimensional}
\label{alg:uni.dim}
\begin{algorithmic}
\REQUIRE{$N$, $n$, $p=1$, and $\bm{a}_N = [a_1, a_2,\cdots, a_N]^\top$}\;
\STATE{Compute FFT of $\bm{a}_N$, $\bm{b}_N = FFT(\bm{a}_N) = [b_1, b_2, \cdots, b_N]^\top$}\;
\STATE{Generate $\bm{e}_N = [\epsilon_n, \epsilon_{n-1},\cdots, \epsilon_{n-N+1}]^\top$ with $N$ many i.i.d.~standard Gaussian random variables}\;
\STATE{Compute FFT of $\bm{e}_N$, $\bm{f}_N = FFT(\bm{e}_N) = [f_1, f_2, \cdots, f_N]^\top$}\;
\STATE{Let $\bm{p}_N$ be the entry-wise product of $\bm{f}_N$ and $\bm{b}_N$, i.e., $\bm{p}_N = [b_1 f_1, b_2 f_2, \cdots, b_N f_N]^\top$}\;
\STATE{Compute the inverse FFT of $\bm{p}_N$, $\bm{W}_N = FFT^{-1}(\bm{p}_N) = [W_1, W_2, \cdots, W_N]^\top$}\;
\ENSURE{$\bm{W}_N$, which can be treated as $K_{N,n}$ many copies of realization of $\bm{X}_n^*$.}
\end{algorithmic}
\end{algorithm}
A large value of $N$ is necessary, particularly for maintaining accuracy in the long-memory regime. Typically, $N$ is chosen to be $n^2$. While it may be appealing to generate i.i.d.~standard Gaussian random vectors $\epsilon_t$ for $-N+1 \leq t \leq n$, this approach incurs a very high computational cost. Instead, we will employ a fast simulation technique using the circulant embedding method and fast Fourier transform (FFT) \cite{wood1994simulation, wu2004simulating}. Consider initially the one-dimensional scenario, $X_i^* = \sum_{t=0}^{N-1} a_t \epsilon_{i-t}$, where $X_i^* \in \mathbb{R}$, $a_t \in \mathbb{R}$, and $\epsilon_t \in \mathbb{R}$. An approximation of $X_i^*$ is provided as
\begin{equation}\label{eq:W}
	\begin{bmatrix}
		W_{1}\\
		W_{2}\\
		\vdots\\
		W_{N}
	\end{bmatrix}
	=
	\begin{bmatrix}
		a_1 & a_2 & \cdots & a_{N-2} & a_{N-1} & a_0\\
		a_2 & a_3 & \cdots & a_{N-1} & a_0 & a_1 \\
		\vdots & \vdots &  & \vdots & \vdots & \vdots &\\
		a_0 & a_1 & \cdots & a_{N-3} & a_{N-2} & a_{N-1}
	\end{bmatrix}
	\begin{bmatrix}
		\epsilon_{n}\\
		\epsilon_{n-1}\\
		\vdots\\
		\epsilon_{n-N+1}\\
	\end{bmatrix}.
\end{equation}
Since \eqref{eq:W} has a circulant structure, it can be simulated with the following method. Let $\bm{a}_N = [a_1, a_2, \cdots, a_{N-1}, a_0]^\top$. For any vector $\bm{x}\in\mathbb{R}^n$, define an operator that generates a circulant matrix from $\bm{x}$, $\circulant: \mathbb{R}^n \rightarrow \mathbb{R}^{n\times n}$, such that
\[
\circulant(\bm{x}) = \begin{bmatrix}
		x_1 & x_2 & \cdots & x_{n-1} & x_{n}\\
		x_2 & x_3 & \cdots & x_n & x_1 \\
		\vdots & \vdots &  & \vdots & \vdots \\
		x_n & x_1 & \cdots & x_{n-2} & x_{n-1}
	\end{bmatrix}.
\]
Thus \eqref{eq:W} can be rewritten as
\begin{equation*}
    \bm{W}_N = \circulant(\bm{a}_N) \bm{e}_N,
\end{equation*}
where $\bm{W}_N = [W_{1}, W_{2}, \cdots, W_{N}]^\top$, and $\bm{e}_N = [\epsilon_{n}, \epsilon_{n-1}, \cdots, \epsilon_{n-N+1}]^\top$.
Note that the FFT of $\bm{e}_N$ can be written as $\bm{f}_N = [f_1, f_2, \cdots, f_N]^\top$.
Meanwhile, the FFT of $\bm{a}_N$ is $\bm{b}_N = [b_1, b_2, \cdots, b_N]^\top$.
If we apply inverse FFT on the entry-wise product of $\bm{b}_N \otimes \bm{f}_N = [b_1 f_1, b_2 f_2, \cdots, b_N f_N]^\top$, the result is precisely $\bm{W}_N$.
We can divide $\bm{W}_N$ into multiple segments of length $n$. For any $1 \leq k \leq K_{N,n}$ where $K_{N,n} = \lfloor N/n \rfloor$, define $\bm{W}^{(k)} = [W_{(k-1)n+1}, W_{(k-1)n+2}, \cdots, W_{kn}]^\top$. According to Theorem 2 in \cite{wu2004simulating}, each $\bm{W}^{(k)}$ is sufficiently close to $\bm{X}^*_n = [X_1^*, X_2^*, \cdots, X_n^*]^\top$. Note that for any $k_1 \neq k_2$, $\bm{W}^{(k_1)}$ and $\bm{W}^{(k_2)}$ are typically not independent. Nonetheless, the dependence is weak enough that we can consider $\bm{W}^{(1)}, \cdots, \bm{W}^{(K_{N,n})}$ as $K_{N,n}$ approximate copies of $\bm{X}^*_n$. Refer to Algorithm \ref{alg:uni.dim}.

\begin{algorithm}[t]
\caption{Fast simulation of Gaussian linear process --- multi-dimensional}
\label{alg:multi.dim}
\begin{algorithmic}
\REQUIRE{$N$, $n$, $p$, and $\bm{a}_{N,j}$ for all $1\leq j\leq p$}\;
\STATE{Generate $\bm{e}_N = [\epsilon_n^\top, \epsilon_{n-1}^\top, \cdots, \epsilon_{n-N+1}^\top]^\top$ with $Nd$ many i.i.d.~standard Gaussian random variables}\;
\STATE{Compute FFT of $\bm{e}_N$, $\bm{f}_N = FFT(\bm{e}_N)$}\;
\FOR{$j=1,2,\cdots,p$}
    \STATE{Compute FFT of $\bm{a}_{N,j}$, $\bm{b}_{N,j} = FFT(\bm{a}_{N,j})$}\;
    \STATE{Let $\bm{p}_{N,j}$ be the entry-wise product of $\bm{f}_N$ and $\bm{b}_{N,j}$, i.e., $\bm{p}_{N,j} = \bm{b}_{N,j} \otimes \bm{f}_N$}\;
    \STATE{Compute the inverse FFT of $\bm{p}_N$, $\bm{Y}^j = FFT^{-1}(\bm{p}_{N,j}) \in \mathbb{R}^{Nd}$}\;
    \STATE{Fill in the $j$-th column of $\bm{W}^{(k)}$ for all $1\leq k\leq K_{N,n}$ as in \eqref{eq:W.k.matrix}, by $W_{i,j} = \bm{Y}^j_{(i-1)d+1}$}\;
    \STATE{Release unused memory space}\;
\ENDFOR
\ENSURE{$\bm{W}^{(1)}, \bm{W}^{(2)}, \cdots, \bm{W}^{(K_{N,n})}$, which can be treated as $K_{N,n}$ many copies of realization of $\bm{X}_n^*$.}
\end{algorithmic}
\end{algorithm}

For the multidimensional case $p>1$, we have
\[
X_i^* = \sum_{t=0}^{N-1} A_t \epsilon_{i-t},
\]
where $X_i^*\in\mathbb{R}^p$, $A_t\in\mathbb{R}^{p\times d}$, and $\epsilon_t\in\mathbb{R}^d$. The proxy of $X_i^*$ is given by $W_i\in\mathbb{R}^p$, where
\begin{equation}\label{eq:W.vec}
	\begin{bmatrix}
		W_{1}\\
		W_{2}\\
		\vdots\\
		W_{N}
	\end{bmatrix}
	=
	\begin{bmatrix}
		A_1 & A_2 & \cdots & A_{N-2} & A_{N-1} & A_0\\
		A_2 & A_3 & \cdots & A_{N-1} & A_0 & A_1 \\
		\vdots & \vdots &  & \vdots & \vdots & \vdots &\\
		A_0 & A_1 & \cdots & A_{N-3} & A_{N-2} & A_{N-1}
	\end{bmatrix}
	\begin{bmatrix}
		\epsilon_{n}\\
		\epsilon_{n-1}\\
		\vdots\\
		\epsilon_{n-N+1}\\
	\end{bmatrix}.
\end{equation}
Since each $A_t$ is a $p \times d$ block matrix, the matrix in \eqref{eq:W.vec} is not necessarily circulant. Thus, Algorithm~\ref{alg:uni.dim} cannot be directly applied. Instead, we simulate $X_i$ feature-by-feature. For each $1 \leq j \leq p$, generate $W_{1j}, W_{2j}, \ldots, W_{Nj}$.
Let $\bm{a}_{N,j} = [(A_1)_{j\cdot}, (A_2)_{j\cdot}, \ldots, (A_{N-1})_{j\cdot}, (A_0)_{j\cdot}]^\top$  be a vector in $\mathbb{R}^{Nd}$.
Construct a $Nd \times Nd$ circulant matrix $\bm{A}^j = \circulant(\bm{a}_{N,j})$. Since $\bm{A}^j$ is circulant, similar to Algorithm \ref{alg:uni.dim}, simulate $\bm{Y}^j = \bm{A}^j \bm{e}_N$, where $\bm{Y}^j \in \mathbb{R}^{Nd}$ and $\bm{e}_N = [\epsilon_n^\top, \epsilon_{n-1}^\top, \ldots, \epsilon_{n-N+1}^\top]^\top \in \mathbb{R}^{Nd}$. For each $1 \leq i \leq N$, $W_{ij} = \bm{Y}^j_{(i-1)d+1}$. Repeat for each $j$ and $1 \leq k \leq K_{N,n}$, letting
\begin{equation}\label{eq:W.k.matrix}
\bm{W}^{(k)} = 
\begin{bmatrix}
    W_{(k-1)n+1, 1} & W_{(k-1)n+1, 2} & \cdots & W_{(k-1)n+1, p}\\
    W_{(k-1)n+2, 1} & W_{(k-1)n+2, 2} & \cdots & W_{(k-1)n+2, p}\\
    \vdots & \vdots && \vdots\\
    W_{kn, 1} & W_{kn, 2} & \cdots & W_{kn, p}\\
\end{bmatrix}.
\end{equation}
We treat $\bm{W}^{(1)}, \bm{W}^{(2)}, \cdots, \bm{W}^{(K_{N,n})}$ as $K_{N,n}$ realizations of $\bm{X}_n^* = [X_1^*, X_2^*, \cdots, X_n^*]^\top \in\mathbb{R}^{n\times p}$. See Algorithm~\ref{alg:multi.dim} for details.

\section{Supplementary Results in Simulation}\label{sec:apdx.supplementary.sim.results}
For a clearer demonstration, we also provide the empirical distributions of all three quantities, the covariance error $n^{1/2}|\hat\Sigma_n - \Sigma|_\infty$, its Gaussian approximation $|Z|_\infty$, and its block bootstrap approximation $l^{-1/2}|\check{B}_{i,l}|_\infty$, in the empirical CDF plot in Figures \ref{fig:E2.ECDF.cov} for the covariance matrix.
Similarly, we provide the empirical distributions of all three quantities, the precision error $n^{1/2}|\hat\Omega_n - \Omega|_\infty$, its (intermediate) Gaussian approximation $|Z^\mathfrak{S}|_\infty$, and its block bootstrap approximation $l^{-1/2}|\check{B}^\Omega_{i,l}|_\infty$, in the empirical CDF plot in Figure \ref{fig:E2.ECDF.prec} for the precision matrix.
We only inspect $\beta=2$ and $\beta=0.9$.
The ultra-long-memory case $\beta=0.55$ is omitted since it violates the theoretical conditions and is expected to fail.

\begin{figure}[t]
    \centering
    \includegraphics[width=0.8\linewidth]{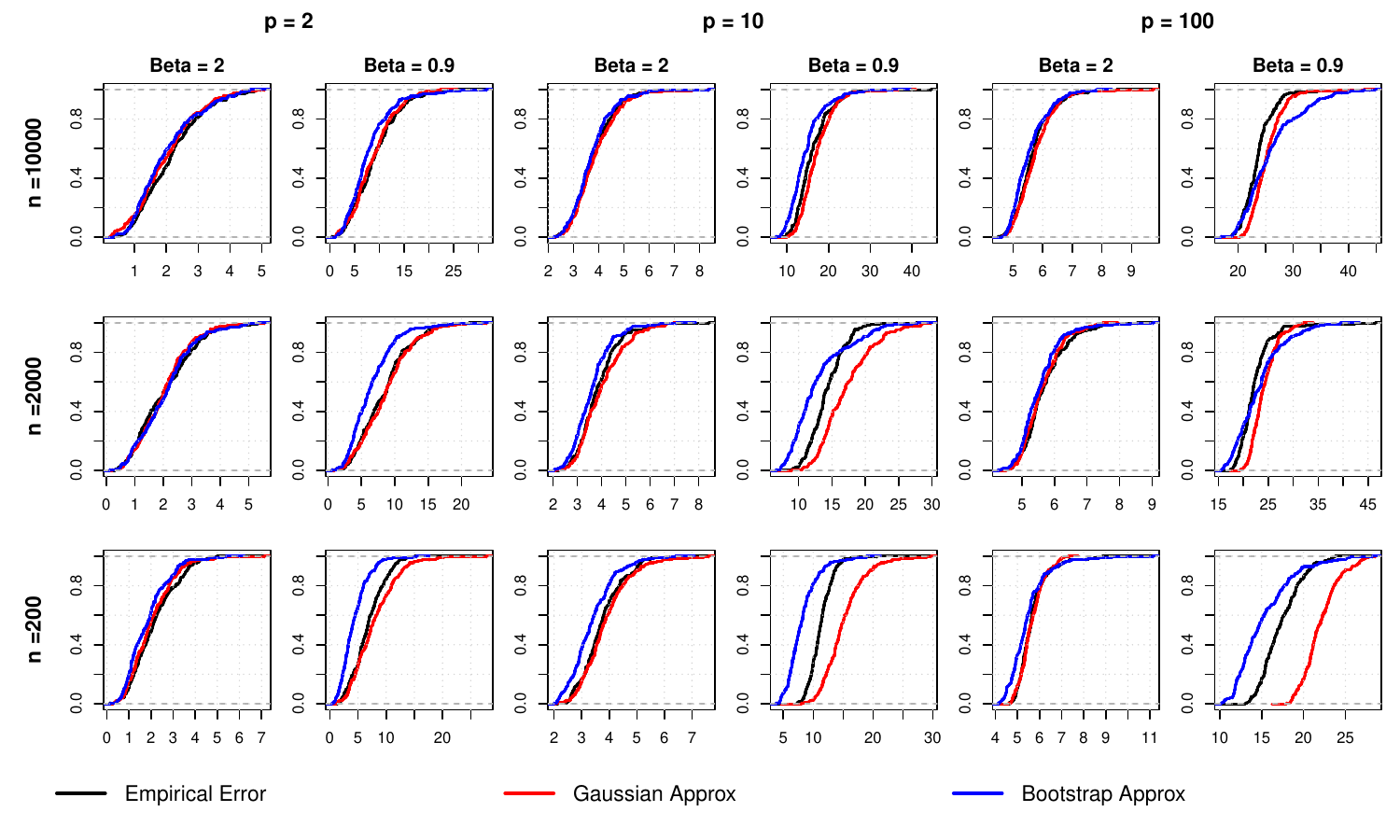}
    \caption{Empirical CDF in the low-dimensional regime, which contains the covariance error $n^{1/2}|\hat\Sigma_n - \Sigma|_\infty$, its Gaussian approximation $|Z|_\infty$, and its block bootstrap approximation $l^{-1/2}|\check{B}_{i,l}|_\infty$, for the short-memory case $\beta=2$ and the long-memory case $\beta=0.9$.}
    \label{fig:E2.ECDF.cov}
\end{figure}

\begin{figure}[t]
    \centering
    \includegraphics[width=0.8\linewidth]{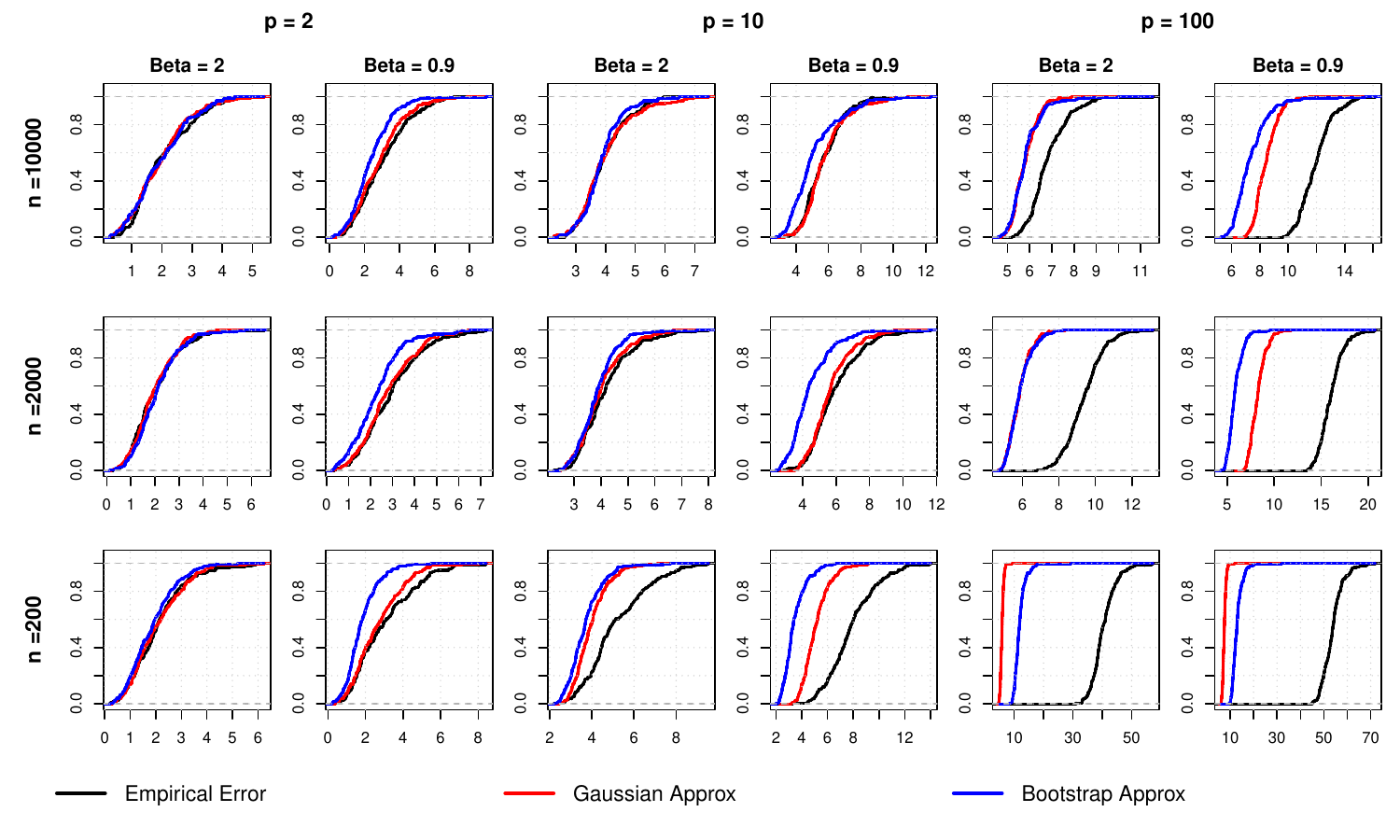}
    \caption{Empirical CDF in the low-dimensional regime, which contains the precision error $n^{1/2}|\hat\Omega_n - \Omega|_\infty$, its (intermediate) Gaussian approximation $|Z^\mathfrak{S}|_\infty$, and its block bootstrap approximation $l^{-1/2}|\check{B}^\Omega_{i,l}|_\infty$, for the short-memory case $\beta=2$ and the long-memory case $\beta=0.9$.}
    \label{fig:E2.ECDF.prec}
\end{figure}

For the empirical CDF plots in the high-dimensional regime, same as the QQ-plots, we still consider $n\in\{200, 400, 800\}$ and $p\in \{2000, 5000, 8000\}$ --- see Figure \ref{fig:E3.ECDF.cov} for the covariance matrix.

\begin{figure}[t]
    \centering
    \includegraphics[width=0.8\linewidth]{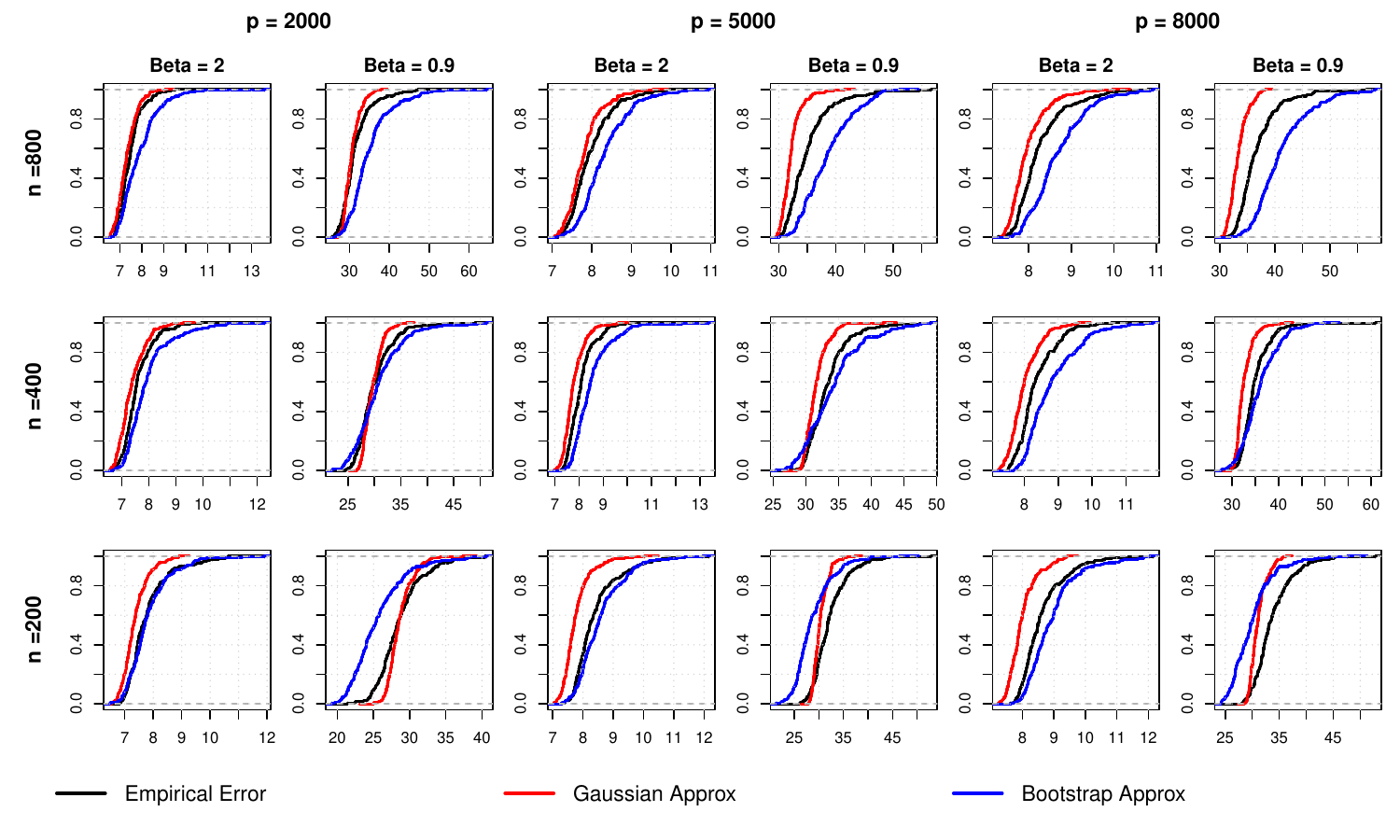}
    \caption{Empirical CDF in the high-dimensional regime, containing the covariance error $n^{1/2}|\hat\Sigma_n - \Sigma|_\infty$, its Gaussian approximation $|Z|_\infty$, and its block bootstrap approximation $l^{-1/2}|\check{B}_{i,l}|_\infty$, for the short-memory case $\beta=2$ and the long-memory case $\beta=0.9$.}
    \label{fig:E3.ECDF.cov}
\end{figure}

Moreover, as a supplement to the Kolmogorov distance tables in the simulation, we provide the corresponding Wasserstein-1 tables here, under the same combination of $n$, $p$, and $\beta$.
For the low-dimensional case, we provide the Wasserstein-1 distance between $n^{1/2}|\hat\Sigma_n - \Sigma|_\infty$ and its Gaussian approximation $|Z|_\infty$ in the GA section of Table \ref{tab:W1.cov}, and the Wasserstein-1 distance between $n^{1/2}|\hat\Sigma_n - \Sigma|_\infty$ and its block bootstrap approximation $l^{-1/2}|\check B_{i,l}|_\infty$ in the Bootstrap section of Table \ref{tab:W1.cov}.
We also provide the Wasserstein-1 distance between $n^{1/2}|\hat\Omega_n - \Omega|_\infty$ and its intermediate Gaussian approximation $|Z^\mathfrak{S}|_\infty$ in the GA section of Table \ref{tab:W1.prec}, and the Wasserstein-1 distance between $n^{1/2}|\hat\Omega_n - \Omega|_\infty$ and its block bootstrap approximation $l^{-1/2}|\check B_{i,l}^\Omega|_\infty$ in the Bootstrap section of Table \ref{tab:W1.prec}.

\begin{table}[t]
    \centering
    \resizebox{\textwidth}{!}{
    \begin{tabular}{llcccccccccccc}
    \toprule
     & & \multicolumn{4}{c}{$\beta=2$} & \multicolumn{4}{c}{$\beta=0.9$} & \multicolumn{4}{c}{$\beta=0.55$} \\
     \cmidrule(lr){3-6} \cmidrule(lr){7-10} \cmidrule(lr){11-14}
    & $n$ & $p=2$ & $p=10$ & $p=30$ & $p=100$ & $p=2$ & $p=10$ & $p=30$ & $p=100$ & $p=2$ & $p=10$ & $p=30$ & $p=100$ \\
    \midrule
    \multirow{6}{*}{GA} & 10000 & 0.14 & 0.07 & 0.06 & 0.08 & 0.41 & 1.02 & 1.28 & 1.60 & 236.82 & 707.31 & 974.09 & 1247.17 \\
     & 4000 & 0.10 & 0.16 & 0.06 & 0.12 & 0.77 & 2.10 & 1.52 & 1.56 & 171.56 & 471.59 & 626.72 & 800.75 \\
     & 2000 & 0.13 & 0.17 & 0.07 & 0.08 & 0.32 & 2.88 & 2.84 & 2.10 & 113.27 & 358.84 & 477.50 & 584.48 \\
     & 1000 & 0.07 & 0.09 & 0.12 & 0.09 & 0.85 & 2.07 & 2.64 & 3.00 & 93.37 & 237.11 & 334.68 & 433.84 \\
     & 500 & 0.11 & 0.10 & 0.10 & 0.09 & 1.14 & 2.71 & 3.67 & 3.23 & 66.83 & 177.92 & 250.05 & 312.28 \\
     & 200 & 0.16 & 0.14 & 0.10 & 0.09 & 1.10 & 3.83 & 4.55 & 4.71 & 38.60 & 112.41 & 155.63 & 196.08 \\
    \midrule
    \multirow{6}{*}{Bootstrap} & 10000 & 0.16 & 0.09 & 0.09 & 0.12 & 1.32 & 1.67 & 1.92 & 2.35 & 375.37 & 418.57 & 378.74 & 316.47 \\
     & 4000 & 0.16 & 0.12 & 0.14 & 0.13 & 1.66 & 1.69 & 2.03 & 2.59 & 248.37 & 269.91 & 247.07 & 206.22 \\
     & 2000 & 0.10 & 0.26 & 0.16 & 0.17 & 1.93 & 2.11 & 2.31 & 1.62 & 185.01 & 192.55 & 173.83 & 151.35 \\
     & 1000 & 0.15 & 0.28 & 0.20 & 0.11 & 2.14 & 2.69 & 1.95 & 1.12 & 132.10 & 143.42 & 130.87 & 108.36 \\
     & 500 & 0.30 & 0.33 & 0.29 & 0.30 & 2.40 & 3.34 & 1.74 & 1.28 & 95.01 & 102.74 & 92.36 & 80.30 \\
     & 200 & 0.33 & 0.37 & 0.32 & 0.26 & 2.46 & 3.14 & 2.47 & 2.24 & 59.35 & 66.06 & 61.18 & 54.03 \\
    \bottomrule
    \end{tabular}
    }
    \caption{Wasserstein-1 distances for covariance matrix in the low-dimensional regime. The GA section compares the covariance error $n^{1/2}|\hat\Sigma_n - \Sigma|_\infty$ against its Gaussian approximation $|Z|_\infty$, while the Bootstrap section compares the covariance error against its block bootstrap approximation $l^{-1/2}|\check{B}_{i,l}|_\infty$. A value closer to zero represents proximity of Gaussian approximation or bootstrap.}
    \label{tab:W1.cov}
\end{table}

\begin{table}[t]
    \centering
    \resizebox{\textwidth}{!}{
    \begin{tabular}{llcccccccccccc}
    \toprule
     & & \multicolumn{4}{c}{$\beta=2$} & \multicolumn{4}{c}{$\beta=0.9$} & \multicolumn{4}{c}{$\beta=0.55$} \\
     \cmidrule(lr){3-6} \cmidrule(lr){7-10} \cmidrule(lr){11-14}
    & $n$ & $p=2$ & $p=10$ & $p=30$ & $p=100$ & $p=2$ & $p=10$ & $p=30$ & $p=100$ & $p=2$ & $p=10$ & $p=30$ & $p=100$ \\
    \midrule
    \multirow{6}{*}{GA} & 10000 & 0.12 & 0.10 & 0.15 & 1.04 & 0.21 & 0.16 & 0.44 & 3.53 & 3.69 & 2.24 & 2.14 & 6.68 \\
     & 4000 & 0.15 & 0.11 & 0.26 & 2.15 & 0.20 & 0.20 & 1.10 & 5.57 & 2.89 & 2.83 & 4.66 & 11.34 \\
     & 2000 & 0.09 & 0.18 & 0.76 & 3.53 & 0.17 & 0.31 & 2.07 & 7.78 & 3.37 & 3.78 & 6.56 & 14.97 \\
     & 1000 & 0.11 & 0.41 & 1.33 & 6.14 & 0.25 & 1.05 & 3.06 & 11.13 & 3.11 & 5.14 & 8.39 & 19.45 \\
     & 500 & 0.10 & 0.53 & 2.06 & 10.66 & 0.20 & 1.49 & 4.35 & 16.98 & 2.87 & 5.90 & 10.34 & 26.51 \\
     & 200 & 0.10 & 1.17 & 5.23 & 34.65 & 0.34 & 2.74 & 8.85 & 46.24 & 2.98 & 7.48 & 15.67 & 59.96 \\
    \midrule
    \multirow{6}{*}{Bootstrap} & 10000 & 0.11 & 0.12 & 0.21 & 1.01 & 0.69 & 0.67 & 0.94 & 4.58 & 10.00 & 14.43 & 18.51 & 27.89 \\
     & 4000 & 0.33 & 0.26 & 0.40 & 2.06 & 0.49 & 1.16 & 2.06 & 7.38 & 7.70 & 11.70 & 15.97 & 25.48 \\
     & 2000 & 0.11 & 0.33 & 0.84 & 3.50 & 0.68 & 1.50 & 3.54 & 10.18 & 6.53 & 10.34 & 14.76 & 25.14 \\
     & 1000 & 0.23 & 0.39 & 1.34 & 5.63 & 1.05 & 2.23 & 5.02 & 13.39 & 5.49 & 9.49 & 14.30 & 25.98 \\
     & 500 & 0.23 & 0.72 & 2.11 & 9.47 & 1.02 & 3.14 & 6.49 & 18.29 & 4.73 & 9.14 & 14.34 & 29.47 \\
     & 200 & 0.24 & 1.44 & 4.64 & 28.57 & 1.20 & 4.41 & 10.17 & 40.71 & 4.00 & 9.14 & 16.80 & 52.64 \\
    \bottomrule
    \end{tabular}
    }
    \caption{Wasserstein-1 distances for precision matrix in the low-dimensional regime. The GA section compares the precision error $n^{1/2}|\hat\Omega_n - \Omega|_\infty$ against its (intermediate) Gaussian approximation $|Z^\mathfrak{S}|_\infty$, while the Bootstrap section compares the precision error against its block bootstrap approximation $l^{-1/2}|\check{B}^\Omega_{i,l}|_\infty$. A value closer to zero represents proximity of Gaussian approximation or bootstrap.}
    \label{tab:W1.prec}
\end{table}

For the high-dimensional case, we provide the Wasserstein-1 distance between $n^{1/2}|\hat\Sigma_n - \Sigma|_\infty$ and its Gaussian approximation $|Z|_\infty$ in the GA section of Table \ref{tab:E3.W1.cov}, and the Wasserstein-1 distance between $n^{1/2}|\hat\Sigma_n - \Sigma|_\infty$ and its block bootstrap approximation $l^{-1/2}|\check B_{i,l}|_\infty$ in the Bootstrap section of Table \ref{tab:E3.W1.cov}.

\begin{table}
    \centering
    \resizebox{\textwidth}{!}{
    \begin{tabular}{llccccccccc}
    \toprule
     & & \multicolumn{3}{c}{$\beta=2$} & \multicolumn{3}{c}{$\beta=0.9$} & \multicolumn{3}{c}{$\beta=0.55$} \\
     \cmidrule(lr){3-5} \cmidrule(lr){6-8} \cmidrule(lr){9-11}
    & $n$ & $p=2000$ & $p=5000$ & $p=8000$ & $p=2000$ & $p=5000$ & $p=8000$ & $p=2000$ & $p=5000$ & $p=8000$ \\
    \midrule
    \multirow{3}{*}{GA} & 800 & 0.15 & 0.17 & 0.29 & 1.15 & 2.95 & 3.46 & 527.45 & 558.36 & 554.90 \\
     & 400 & 0.22 & 0.29 & 0.35 & 1.06 & 1.77 & 2.35 & 379.68 & 393.70 & 399.48 \\
     & 200 & 0.43 & 0.65 & 0.64 & 1.04 & 2.05 & 2.68 & 270.34 & 278.29 & 281.23 \\
    \midrule
    \multirow{3}{*}{Bootstrap} & 800 & 0.43 & 0.35 & 0.42 & 3.53 & 3.43 & 4.33 & 58.76 & 63.53 & 64.38 \\
     & 400 & 0.35 & 0.43 & 0.47 & 1.10 & 1.27 & 1.33 & 45.59 & 47.73 & 51.87 \\
     & 200 & 0.09 & 0.18 & 0.30 & 3.06 & 3.47 & 3.76 & 39.55 & 43.90 & 44.85 \\
    \bottomrule
    \end{tabular}
    }
    \caption{Wasserstein-1 distances for covariance matrix in the high-dimensional regime. The GA section compares the covariance error $n^{1/2}|\hat\Sigma_n - \Sigma|_\infty$ against its Gaussian approximation $|Z|_\infty$, while the Bootstrap section compares the covariance error against its block bootstrap approximation $l^{-1/2}|\check{B}_{i,l}|_\infty$.}
    \label{tab:E3.W1.cov}
\end{table}

\section{Labels of ROIs in the AAL Atlas}
The labels and corresponding names of each ROI in the AAL 116 atlas \cite{tzourio2002automated} is given in Table \ref{table:aal.roi}.

\begin{table}[ht]
\centering
\caption{Labels and Names of Each ROI in the AAL 116 Atlas}
\begin{tabular}{|c|l|c|l|c|l|}
\hline
\textbf{Label} & \textbf{ROI Name} & \textbf{Label} & \textbf{ROI Name} & \textbf{Label} & \textbf{ROI Name} \\
\hline
2001 & Precentral\_L           & 4112 & ParaHippocampal\_R       & 8101 & Heschl\_L \\
2002 & Precentral\_R           & 4201 & Amygdala\_L              & 8102 & Heschl\_R \\
2101 & Frontal\_Sup\_L         & 4202 & Amygdala\_R              & 8111 & Temporal\_Sup\_L \\
2102 & Frontal\_Sup\_R         & 5001 & Calcarine\_L             & 8112 & Temporal\_Sup\_R \\
2111 & Frontal\_Sup\_Orb\_L    & 5002 & Calcarine\_R             & 8121 & Temporal\_Pole\_Sup\_L \\
2112 & Frontal\_Sup\_Orb\_R    & 5011 & Cuneus\_L                & 8122 & Temporal\_Pole\_Sup\_R \\
2201 & Frontal\_Mid\_L         & 5012 & Cuneus\_R                & 8201 & Temporal\_Mid\_L \\
2202 & Frontal\_Mid\_R         & 5021 & Lingual\_L               & 8202 & Temporal\_Mid\_R \\
2211 & Frontal\_Mid\_Orb\_L    & 5022 & Lingual\_R               & 8211 & Temporal\_Pole\_Mid\_L \\
2212 & Frontal\_Mid\_Orb\_R    & 5101 & Occipital\_Sup\_L        & 8212 & Temporal\_Pole\_Mid\_R \\
2301 & Frontal\_Inf\_Oper\_L   & 5102 & Occipital\_Sup\_R        & 8301 & Temporal\_Inf\_L \\
2302 & Frontal\_Inf\_Oper\_R   & 5201 & Occipital\_Mid\_L        & 8302 & Temporal\_Inf\_R \\
2311 & Frontal\_Inf\_Tri\_L    & 5202 & Occipital\_Mid\_R        & 9001 & Cerebelum\_Crus1\_L \\
2312 & Frontal\_Inf\_Tri\_R    & 5301 & Occipital\_Inf\_L        & 9002 & Cerebelum\_Crus1\_R \\
2321 & Frontal\_Inf\_Orb\_L    & 5302 & Occipital\_Inf\_R        & 9011 & Cerebelum\_Crus2\_L \\
2322 & Frontal\_Inf\_Orb\_R    & 5401 & Fusiform\_L              & 9012 & Cerebelum\_Crus2\_R \\
2331 & Rolandic\_Oper\_L       & 5402 & Fusiform\_R              & 9021 & Cerebelum\_3\_L \\
2332 & Rolandic\_Oper\_R       & 6001 & Postcentral\_L           & 9022 & Cerebelum\_3\_R \\
2401 & Supp\_Motor\_Area\_L    & 6002 & Postcentral\_R           & 9031 & Cerebelum\_4\_5\_L \\
2402 & Supp\_Motor\_Area\_R    & 6101 & Parietal\_Sup\_L         & 9032 & Cerebelum\_4\_5\_R \\
2501 & Olfactory\_L            & 6102 & Parietal\_Sup\_R         & 9041 & Cerebelum\_6\_L \\
2502 & Olfactory\_R            & 6201 & Parietal\_Inf\_L         & 9042 & Cerebelum\_6\_R \\
2601 & Frontal\_Sup\_Medial\_L & 6202 & Parietal\_Inf\_R         & 9051 & Cerebelum\_7b\_L \\
2602 & Frontal\_Sup\_Medial\_R & 6211 & SupraMarginal\_L         & 9052 & Cerebelum\_7b\_R \\
2611 & Frontal\_Med\_Orb\_L    & 6212 & SupraMarginal\_R         & 9061 & Cerebelum\_8\_L \\
2612 & Frontal\_Med\_Orb\_R    & 6221 & Angular\_L               & 9062 & Cerebelum\_8\_R \\
2701 & Rectus\_L               & 6222 & Angular\_R               & 9071 & Cerebelum\_9\_L \\
2702 & Rectus\_R               & 6301 & Precuneus\_L             & 9072 & Cerebelum\_9\_R \\
3001 & Insula\_L               & 6302 & Precuneus\_R             & 9081 & Cerebelum\_10\_L \\
3002 & Insula\_R               & 6401 & Paracentral\_Lobule\_L   & 9082 & Cerebelum\_10\_R \\
4001 & Cingulum\_Ant\_L        & 6402 & Paracentral\_Lobule\_R   & 9100 & Vermis\_1\_2 \\
4002 & Cingulum\_Ant\_R        & 7001 & Caudate\_L               & 9110 & Vermis\_3 \\
4011 & Cingulum\_Mid\_L        & 7002 & Caudate\_R               & 9120 & Vermis\_4\_5 \\
4012 & Cingulum\_Mid\_R        & 7011 & Putamen\_L               & 9130 & Vermis\_6 \\
4021 & Cingulum\_Post\_L       & 7012 & Putamen\_R               & 9140 & Vermis\_7 \\
4022 & Cingulum\_Post\_R       & 7021 & Pallidum\_L              & 9150 & Vermis\_8 \\
4101 & Hippocampus\_L          & 7022 & Pallidum\_R              & 9160 & Vermis\_9 \\
4102 & Hippocampus\_R          & 7101 & Thalamus\_L              & 9170 & Vermis\_10 \\
4111 & ParaHippocampal\_L      & 7102 & Thalamus\_R              & & \\
\hline

\end{tabular}\label{table:aal.roi}
\end{table}

\section{Additional Simulation under an Alternative Setup}

This section provides the simulation result under the same setup as in Section \ref{sec:simulation}, but with a different structure on coefficients $A_t$.
We still let $d=p$, but instead of Toeplitz $A_t$, we let $A_t$ be banded.
That is,
\[
(A_t)_{jk} = (t+1)^{-\beta} (|j-k|+1)^{-2} \mathbf{1}\{|j-k|\leq b\}.
\]
In our experiments, the bandwidth is chosen as $b=3$.
Therefore, the true covariance matrix $\Sigma$ will also be a banded matrix with a bandwidth of 6, and the precision matrix $\Omega$ will exhibit a fast decay away from the diagonal.

For the low-dimensional regime, the covariance QQ-plots are given in Figures \ref{fig:E2B.QQ.cov.2}, \ref{fig:E2B.QQ.cov.0.9}, and \ref{fig:E2B.QQ.cov.0.55}.
The covariance CDF plot is given in Figure \ref{fig:E2B.ECDF.cov}.
The precision QQ-plots are given in Figures \ref{fig:E2B.QQ.prec.2}, \ref{fig:E2B.QQ.prec.0.9}, and \ref{fig:E2B.QQ.prec.0.55}.
The precision CDF plot is given in Figure \ref{fig:E2B.ECDF.prec}.
The Kolmogorov distance and W1 distance of Gaussian approximation and bootstrap are given in Tables \ref{tab:E2B.KS.cov} and \ref{tab:E2B.W1.cov} for covariance matrix, and Tables \ref{tab:E2B.KS.prec} and \ref{tab:E2B.W1.prec} for precision matrix.
For the high-dimensional regime, the covariance QQ-plots are given in Figures \ref{fig:E3B.QQ.cov.2}, \ref{fig:E3B.QQ.cov.0.9}, and \ref{fig:E3B.QQ.cov.0.55}.
The covariance CDF plot is given in Figure \ref{fig:E3B.ECDF.cov}.
The Kolmogorov distance and W1 distance of Gaussian approximation and bootstrap for the covariance matrix are given in Tables \ref{tab:E3B.KS.cov} and \ref{tab:E3B.W1.cov}.
The results for $p=2$ stays the same, since it is less than the bandwidth.
The rest of the results remain similar to those in Section \ref{sec:simulation}.

\begin{figure}[p]
    \centering
    
    \includegraphics[width=0.8\linewidth]{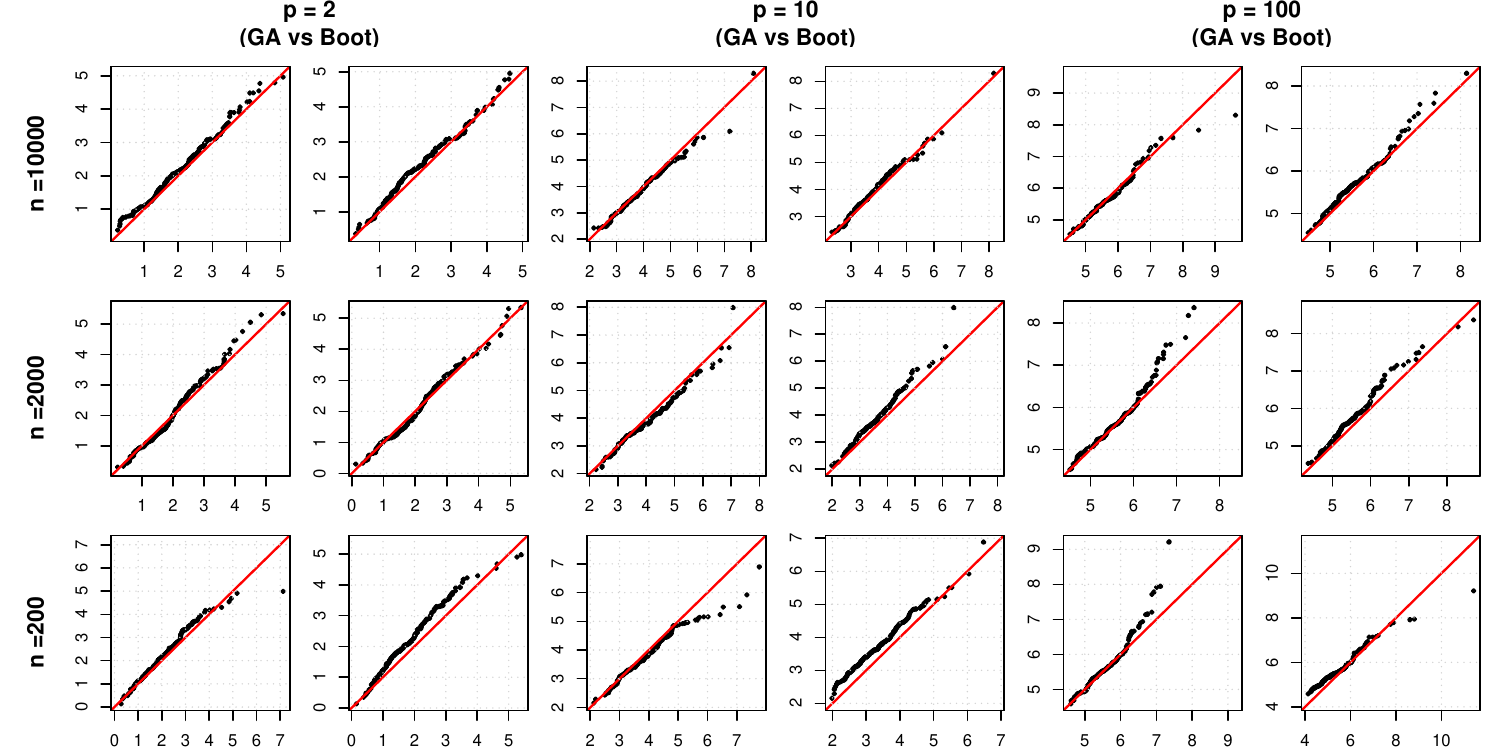}
    \caption{QQ-plots for covariance matrix in the low-dimensional regime, $\beta=2$ (short memory), under the banded setup. Red line represents $y=x$. For each pair of $p$ and $n$ values, the plot on the left compares $n^{1/2}|\hat\Sigma_n - \Sigma|_\infty$ on $y$-axis and $|Z|_\infty$ on $x$-axis.
    Being closer to the $y=x$ line represents a smaller Kolmogorov distance $\rho(n)$ for Gaussian approximation.
    The plot on the right compares $n^{1/2}|\hat\Sigma_n - \Sigma|_\infty$ on $y$-axis and $l^{-1/2}|\check{B}_{i,l}|_\infty$ on $x$-axis.
    Being closer to the $y=x$ line represents a smaller Kolmogorov distance $\rho_B(n,l)$ for block bootstrap.}
    \label{fig:E2B.QQ.cov.2}
    
    \vspace{1cm} 
    
    \includegraphics[width=0.8\linewidth]{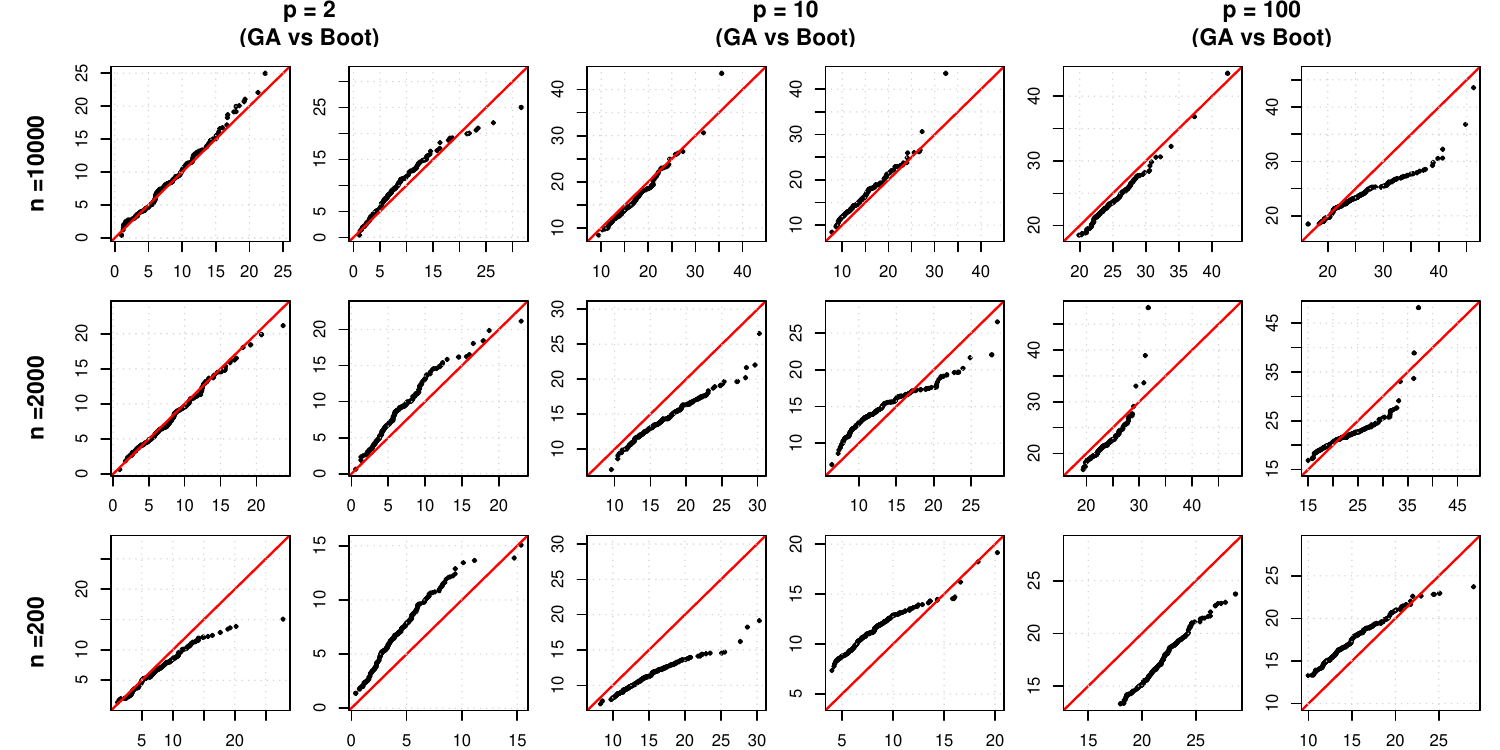}
    \caption{QQ-plots for covariance matrix in the low-dimensional regime, $\beta=0.9$ (long-memory), under the banded setup. Red line represents $y=x$. For each pair of $p$ and $n$ values, the plot on the left compares $n^{1/2}|\hat\Sigma_n - \Sigma|_\infty$ on $y$-axis and $|Z|_\infty$ on $x$-axis.
    Being closer to the $y=x$ line represents a smaller Kolmogorov distance $\rho(n)$ for Gaussian approximation.
    The plot on the right compares $n^{1/2}|\hat\Sigma_n - \Sigma|_\infty$ on $y$-axis and $l^{-1/2}|\check{B}_{i,l}|_\infty$ on $x$-axis.
    Being closer to the $y=x$ line represents a smaller Kolmogorov distance $\rho_B(n,l)$ for block bootstrap.}
    \label{fig:E2B.QQ.cov.0.9}
    
\end{figure}

\begin{figure}[p]
    \centering
    
    \includegraphics[width=0.8\linewidth]{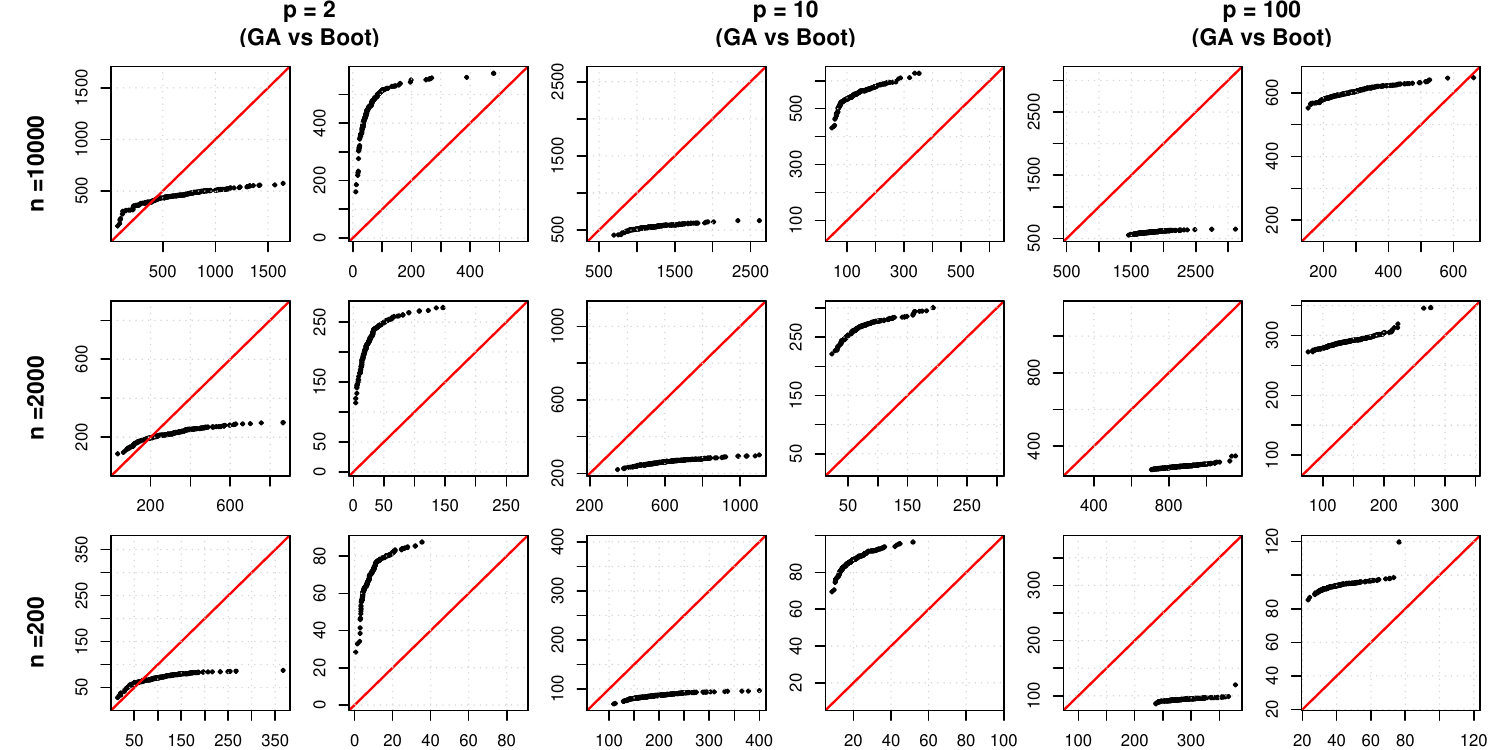}
    \caption{QQ-plots for covariance matrix in the low-dimensional regime, $\beta=0.55$ (ultra-long-memory), under the banded setup. Red line represents $y=x$. For each pair of $p$ and $n$ values, the plot on the left compares $n^{1/2}|\hat\Sigma_n - \Sigma|_\infty$ on $y$-axis and $|Z|_\infty$ on $x$-axis.
    Being closer to the $y=x$ line represents a smaller Kolmogorov distance $\rho(n)$ for Gaussian approximation.
    The plot on the right compares $n^{1/2}|\hat\Sigma_n - \Sigma|_\infty$ on $y$-axis and $l^{-1/2}|\check{B}_{i,l}|_\infty$ on $x$-axis.
    Being closer to the $y=x$ line represents a smaller Kolmogorov distance $\rho_B(n,l)$ for block bootstrap.}
    \label{fig:E2B.QQ.cov.0.55}
    
    \vspace{1cm} 
    
    \includegraphics[width=0.8\linewidth]{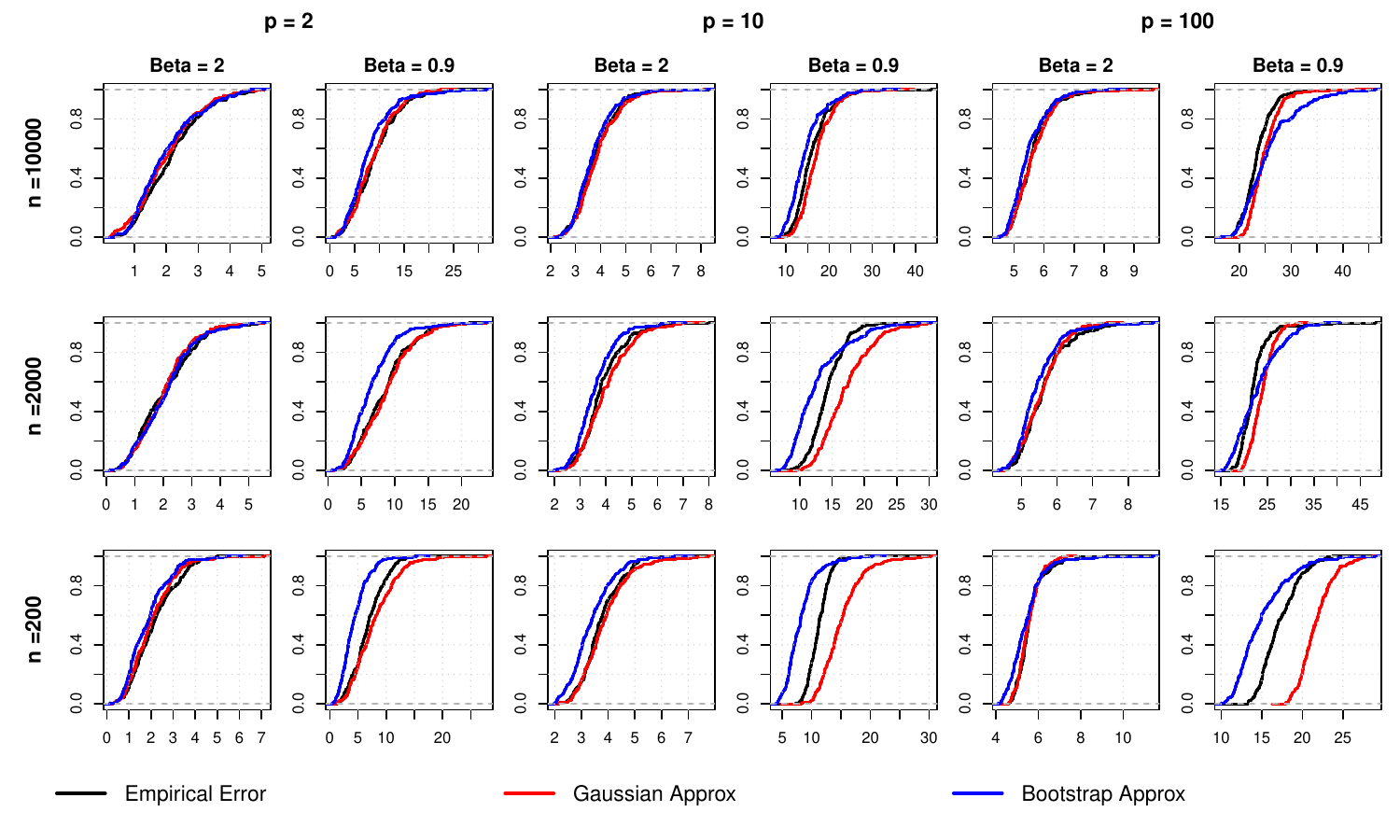}
    \caption{Empirical CDF in the low-dimensional regime under the banded setup, which contains the covariance error $n^{1/2}|\hat\Sigma_n - \Sigma|_\infty$, its Gaussian approximation $|Z|_\infty$, and its block bootstrap approximation $l^{-1/2}|\check{B}_{i,l}|_\infty$, for the short-memory case $\beta=2$ and the long-memory case $\beta=0.9$.}
    \label{fig:E2B.ECDF.cov}
    
\end{figure}

\begin{figure}[p]
    \centering
    
    \includegraphics[width=0.8\linewidth]{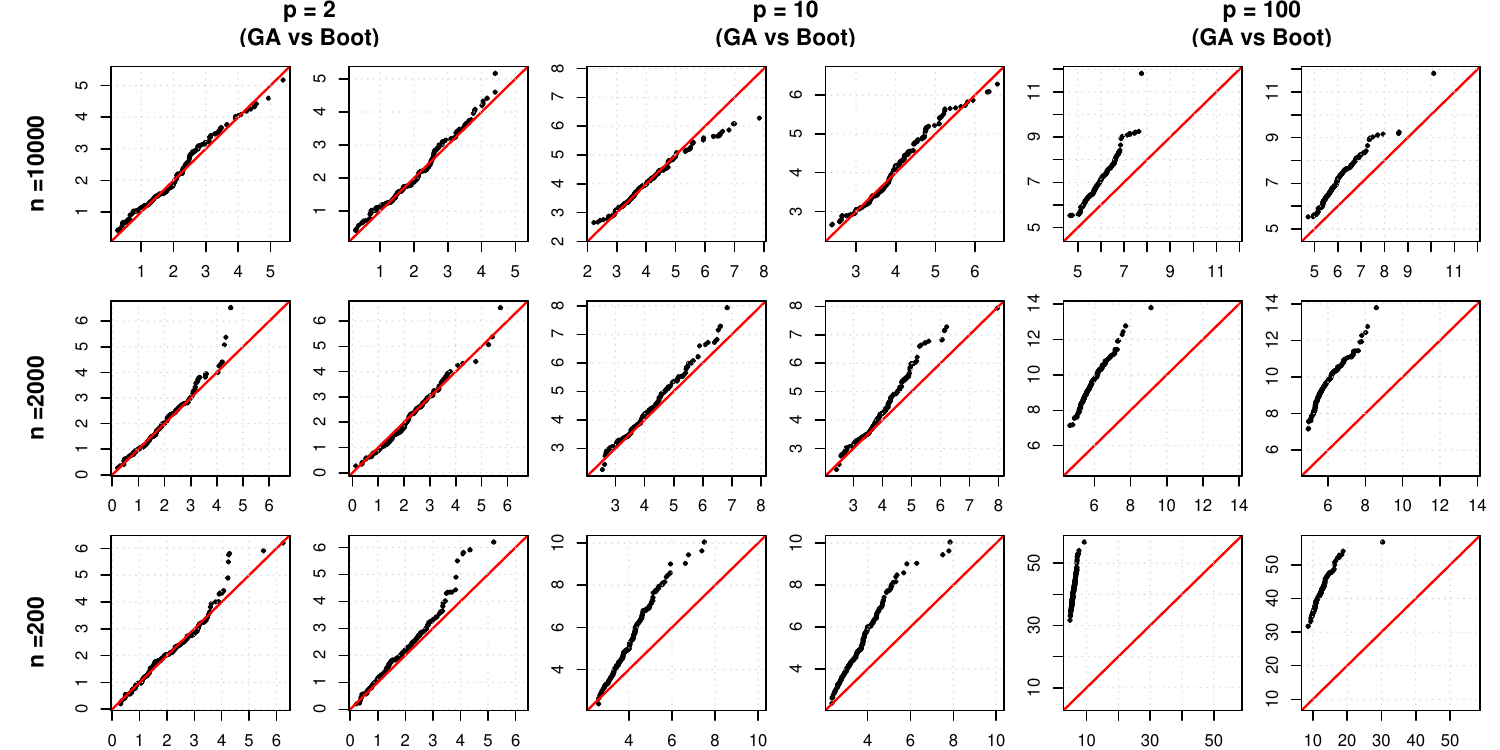}
    \caption{QQ-plots for precision matrix in the low-dimensional regime, $\beta=2$ (short memory), under the banded setup. Red line represents $y=x$. For each pair of $p$ and $n$ values, the plot on the left compares $n^{1/2}|\hat\Omega_n - \Omega|_\infty$ on $y$-axis and $|Z^{\mathfrak{S}}|_\infty$ on $x$-axis.
    Being closer to the $y=x$ line represents a smaller Kolmogorov distance $\rho^\Omega(n)$ for Gaussian approximation.
    The plot on the right compares $n^{1/2}|\hat\Omega_n - \Omega|_\infty$ on $y$-axis and $l^{-1/2}|\check{B}^\Omega_{i,l}|_\infty$ on $x$-axis.
    Being closer to the $y=x$ line represents a smaller Kolmogorov distance $\rho_B^\Omega(n,l)$ for block bootstrap.}
    \label{fig:E2B.QQ.prec.2}
    
    \vspace{1cm} 
    
    \includegraphics[width=0.8\linewidth]{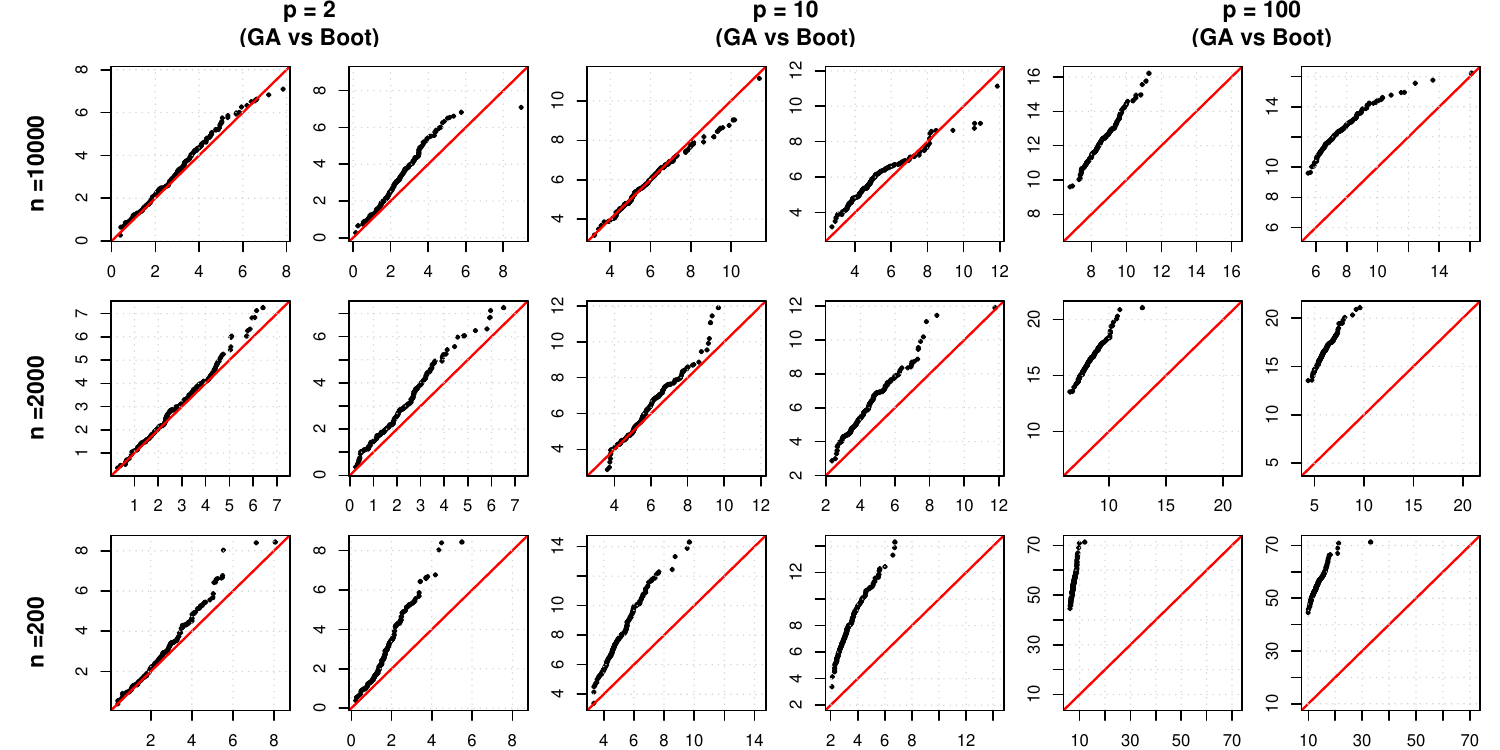}
    \caption{QQ-plots for precision matrix in the low-dimensional regime, $\beta=0.9$ (long-memory), under the banded setup. Red line represents $y=x$. For each pair of $p$ and $n$ values, the plot on the left compares $n^{1/2}|\hat\Omega_n - \Omega|_\infty$ on $y$-axis and $|Z^{\mathfrak{S}}|_\infty$ on $x$-axis.
    Being closer to the $y=x$ line represents a smaller Kolmogorov distance $\rho^\Omega(n)$ for Gaussian approximation.
    The plot on the right compares $n^{1/2}|\hat\Omega_n - \Omega|_\infty$ on $y$-axis and $l^{-1/2}|\check{B}^\Omega_{i,l}|_\infty$ on $x$-axis.
    Being closer to the $y=x$ line represents a smaller Kolmogorov distance $\rho_B^\Omega(n,l)$ for block bootstrap.}
    \label{fig:E2B.QQ.prec.0.9}
    
\end{figure}

\begin{figure}[p]
    \centering
    
    \includegraphics[width=0.8\linewidth]{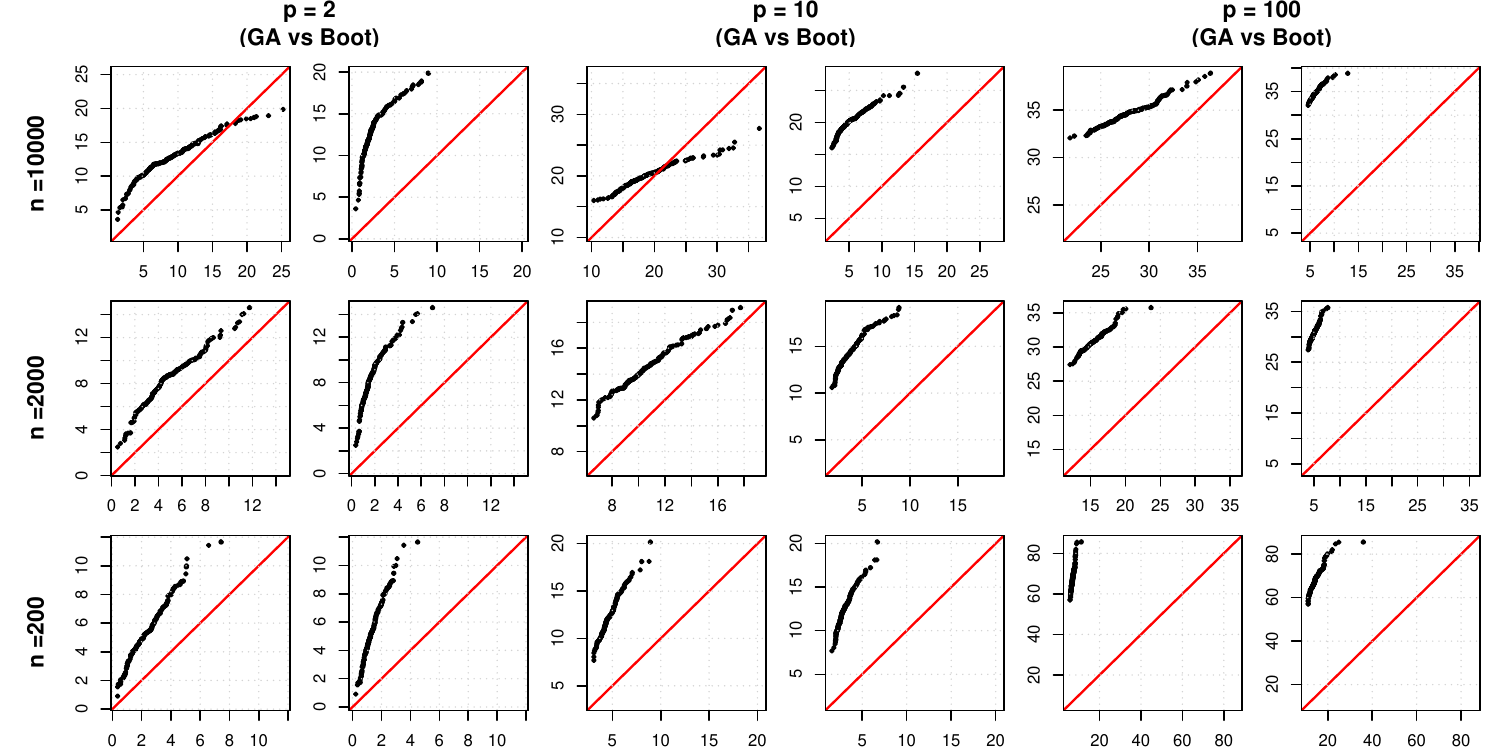}
    \caption{QQ-plots for precision matrix in the low-dimensional regime, $\beta=0.55$ (ultra-long-memory), under the banded setup. Red line represents $y=x$. For each pair of $p$ and $n$ values, the plot on the left compares $n^{1/2}|\hat\Omega_n - \Omega|_\infty$ on $y$-axis and $|Z^{\mathfrak{S}}|_\infty$ on $x$-axis.
    Being closer to the $y=x$ line represents a smaller Kolmogorov distance $\rho^\Omega(n)$ for Gaussian approximation.
    The plot on the right compares $n^{1/2}|\hat\Omega_n - \Omega|_\infty$ on $y$-axis and $l^{-1/2}|\check{B}^\Omega_{i,l}|_\infty$ on $x$-axis.
    Being closer to the $y=x$ line represents a smaller Kolmogorov distance $\rho_B^\Omega(n,l)$ for block bootstrap.}
    \label{fig:E2B.QQ.prec.0.55}
    
    \vspace{1cm} 
    
    \includegraphics[width=0.8\linewidth]{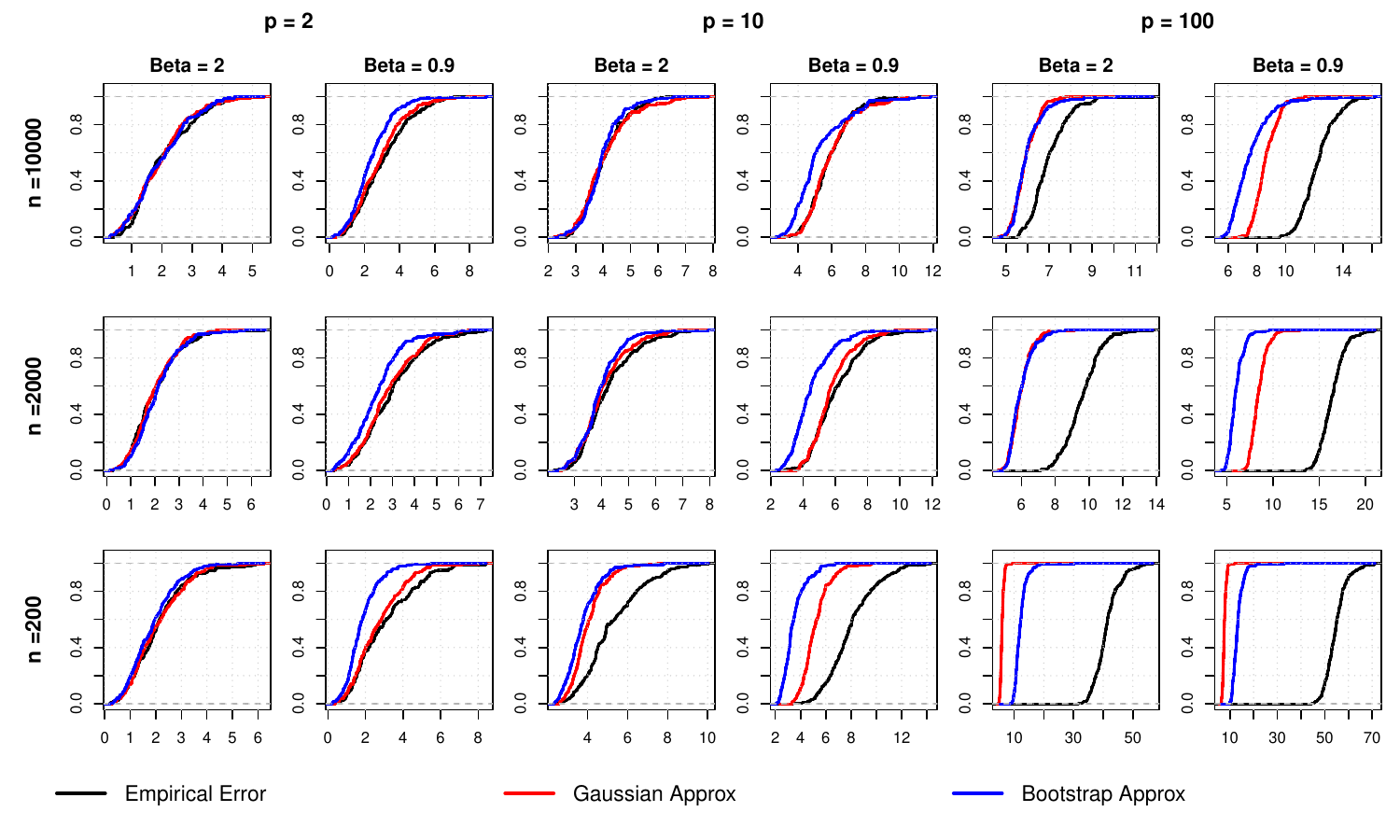}
    \caption{Empirical CDF in the low-dimensional regime under the banded setup, which contains the precision error $n^{1/2}|\hat\Omega_n - \Omega|_\infty$, its (intermediate) Gaussian approximation $|Z^\mathfrak{S}|_\infty$, and its block bootstrap approximation $l^{-1/2}|\check{B}^\Omega_{i,l}|_\infty$, for the short-memory case $\beta=2$ and the long-memory case $\beta=0.9$.}
    \label{fig:E2B.ECDF.prec}
    
\end{figure}

\begin{table}[p]
\centering

\resizebox{\textwidth}{!}{
\begin{tabular}{llcccccccccccc}
\toprule
 & & \multicolumn{4}{c}{$\beta=2$} & \multicolumn{4}{c}{$\beta=0.9$} & \multicolumn{4}{c}{$\beta=0.55$} \\
 \cmidrule(lr){3-6} \cmidrule(lr){7-10} \cmidrule(lr){11-14}
& $n$ & $p=2$ & $p=10$ & $p=30$ & $p=100$ & $p=2$ & $p=10$ & $p=30$ & $p=100$ & $p=2$ & $p=10$ & $p=30$ & $p=100$ \\
\midrule
\multirow{6}{*}{GA} & 10000 & 0.07 & 0.06 & 0.06 & 0.08 & 0.06 & 0.17 & 0.18 & 0.23 & 0.54 & 1.00 & 1.00 & 1.00 \\
 & 4000 & 0.07 & 0.14 & 0.08 & 0.07 & 0.12 & 0.27 & 0.24 & 0.36 & 0.61 & 1.00 & 1.00 & 1.00 \\
 & 2000 & 0.07 & 0.09 & 0.14 & 0.18 & 0.10 & 0.36 & 0.40 & 0.51 & 0.65 & 1.00 & 1.00 & 1.00 \\
 & 1000 & 0.09 & 0.17 & 0.18 & 0.27 & 0.11 & 0.46 & 0.42 & 0.70 & 0.74 & 1.00 & 1.00 & 1.00 \\
 & 500 & 0.07 & 0.23 & 0.32 & 0.45 & 0.12 & 0.57 & 0.55 & 0.90 & 0.81 & 1.00 & 1.00 & 1.00 \\
 & 200 & 0.08 & 0.46 & 0.49 & 0.53 & 0.16 & 0.73 & 0.83 & 0.97 & 0.83 & 1.00 & 1.00 & 1.00 \\
\midrule
\multirow{6}{*}{Bootstrap} & 10000 & 0.12 & 0.08 & 0.08 & 0.16 & 0.18 & 0.22 & 0.22 & 0.24 & 0.96 & 1.00 & 0.99 & 0.99 \\
 & 4000 & 0.10 & 0.10 & 0.11 & 0.17 & 0.22 & 0.34 & 0.28 & 0.32 & 0.98 & 1.00 & 0.98 & 0.98 \\
 & 2000 & 0.10 & 0.11 & 0.16 & 0.21 & 0.28 & 0.42 & 0.40 & 0.47 & 0.99 & 1.00 & 0.97 & 0.97 \\
 & 1000 & 0.12 & 0.23 & 0.21 & 0.26 & 0.36 & 0.58 & 0.43 & 0.66 & 1.00 & 1.00 & 0.96 & 0.96 \\
 & 500 & 0.15 & 0.25 & 0.30 & 0.38 & 0.37 & 0.78 & 0.53 & 0.82 & 1.00 & 1.00 & 0.95 & 0.94 \\
 & 200 & 0.13 & 0.55 & 0.52 & 0.50 & 0.36 & 0.95 & 0.78 & 0.89 & 1.00 & 1.00 & 0.91 & 0.91 \\
\bottomrule
\end{tabular}
}
\caption{Kolmogorov distances for covariance matrix in the low-dimensional regime under the banded setup. The GA section compares the covariance error $n^{1/2}|\hat\Sigma_n - \Sigma|_\infty$ against its Gaussian approximation, while the Bootstrap section compares it against its block bootstrap approximation.}
\label{tab:E2B.KS.cov}

\vspace{0.5cm}

\resizebox{\textwidth}{!}{
\begin{tabular}{llcccccccccccc}
\toprule
 & & \multicolumn{4}{c}{$\beta=2$} & \multicolumn{4}{c}{$\beta=0.9$} & \multicolumn{4}{c}{$\beta=0.55$} \\
 \cmidrule(lr){3-6} \cmidrule(lr){7-10} \cmidrule(lr){11-14}
& $n$ & $p=2$ & $p=10$ & $p=30$ & $p=100$ & $p=2$ & $p=10$ & $p=30$ & $p=100$ & $p=2$ & $p=10$ & $p=30$ & $p=100$ \\
\midrule
\multirow{6}{*}{GA} & 10000 & 0.14 & 0.08 & 0.07 & 0.06 & 0.41 & 1.06 & 1.10 & 1.36 & 236.82 & 708.92 & 956.34 & 1226.22 \\
 & 4000 & 0.10 & 0.16 & 0.07 & 0.08 & 0.77 & 2.03 & 1.48 & 1.65 & 171.56 & 468.28 & 619.69 & 794.77 \\
 & 2000 & 0.15 & 0.10 & 0.12 & 0.18 & 0.52 & 2.16 & 2.45 & 2.76 & 118.82 & 325.26 & 425.23 & 557.48 \\
 & 1000 & 0.14 & 0.17 & 0.16 & 0.28 & 0.69 & 3.01 & 2.50 & 4.14 & 86.87 & 230.12 & 297.80 & 386.41 \\
 & 500 & 0.15 & 0.23 & 0.28 & 0.45 & 0.79 & 3.58 & 3.19 & 5.25 & 60.18 & 163.63 & 210.88 & 270.92 \\
 & 200 & 0.10 & 0.45 & 0.45 & 0.55 & 1.15 & 4.88 & 5.17 & 5.86 & 39.55 & 108.63 & 140.73 & 178.67 \\
\midrule
\multirow{6}{*}{Bootstrap} & 10000 & 0.16 & 0.11 & 0.08 & 0.12 & 1.32 & 1.56 & 1.97 & 2.22 & 375.37 & 414.31 & 370.29 & 311.84 \\
 & 4000 & 0.16 & 0.12 & 0.15 & 0.12 & 1.66 & 1.77 & 1.92 & 2.56 & 248.37 & 267.29 & 242.38 & 202.69 \\
 & 2000 & 0.11 & 0.12 & 0.23 & 0.17 & 1.40 & 1.77 & 1.83 & 2.47 & 161.99 & 178.36 & 157.06 & 131.76 \\
 & 1000 & 0.18 & 0.21 & 0.21 & 0.23 & 2.05 & 2.37 & 2.53 & 2.65 & 115.54 & 120.30 & 104.99 & 88.02 \\
 & 500 & 0.16 & 0.27 & 0.27 & 0.35 & 1.97 & 2.76 & 2.50 & 2.80 & 75.36 & 83.13 & 73.10 & 60.31 \\
 & 200 & 0.09 & 0.46 & 0.36 & 0.44 & 2.81 & 3.49 & 3.25 & 3.14 & 48.74 & 53.64 & 48.37 & 39.23 \\
\bottomrule
\end{tabular}
}
\caption{Wasserstein-1 distances for covariance matrix in the low-dimensional regime under the banded setup. The GA section compares the covariance error $n^{1/2}|\hat\Sigma_n - \Sigma|_\infty$ against its Gaussian approximation, while the Bootstrap section compares it against its block bootstrap approximation.}
\label{tab:E2B.W1.cov}
\end{table}

\begin{table}[p]
\centering

\resizebox{\textwidth}{!}{
\begin{tabular}{llcccccccccccc}
\toprule
 & & \multicolumn{4}{c}{$\beta=2$} & \multicolumn{4}{c}{$\beta=0.9$} & \multicolumn{4}{c}{$\beta=0.55$} \\
 \cmidrule(lr){3-6} \cmidrule(lr){7-10} \cmidrule(lr){11-14}
& $n$ & $p=2$ & $p=10$ & $p=30$ & $p=100$ & $p=2$ & $p=10$ & $p=30$ & $p=100$ & $p=2$ & $p=10$ & $p=30$ & $p=100$ \\
\midrule
\multirow{6}{*}{GA} & 10000 & 0.09 & 0.05 & 0.10 & 0.57 & 0.09 & 0.06 & 0.18 & 0.95 & 0.47 & 0.35 & 0.44 & 0.96 \\
 & 4000 & 0.09 & 0.06 & 0.16 & 0.83 & 0.12 & 0.09 & 0.45 & 0.98 & 0.47 & 0.57 & 0.84 & 0.99 \\
 & 2000 & 0.06 & 0.12 & 0.33 & 0.96 & 0.10 & 0.12 & 0.67 & 1.00 & 0.56 & 0.73 & 0.96 & 1.00 \\
 & 1000 & 0.10 & 0.17 & 0.55 & 1.00 & 0.10 & 0.40 & 0.86 & 1.00 & 0.60 & 0.88 & 1.00 & 1.00 \\
 & 500 & 0.07 & 0.25 & 0.68 & 1.00 & 0.10 & 0.45 & 0.93 & 1.00 & 0.61 & 0.97 & 1.00 & 1.00 \\
 & 200 & 0.06 & 0.41 & 0.99 & 1.00 & 0.12 & 0.68 & 1.00 & 1.00 & 0.69 & 0.99 & 1.00 & 1.00 \\
\midrule
\multirow{6}{*}{Bootstrap} & 10000 & 0.09 & 0.10 & 0.16 & 0.60 & 0.21 & 0.32 & 0.37 & 0.94 & 0.94 & 1.00 & 1.00 & 1.00 \\
 & 4000 & 0.15 & 0.14 & 0.25 & 0.81 & 0.17 & 0.45 & 0.65 & 0.99 & 0.96 & 1.00 & 1.00 & 1.00 \\
 & 2000 & 0.10 & 0.16 & 0.40 & 0.96 & 0.26 & 0.51 & 0.86 & 1.00 & 0.94 & 1.00 & 1.00 & 1.00 \\
 & 1000 & 0.12 & 0.17 & 0.54 & 0.99 & 0.38 & 0.68 & 0.98 & 1.00 & 0.93 & 1.00 & 1.00 & 1.00 \\
 & 500 & 0.16 & 0.34 & 0.74 & 1.00 & 0.38 & 0.84 & 1.00 & 1.00 & 0.93 & 1.00 & 1.00 & 1.00 \\
 & 200 & 0.11 & 0.52 & 0.91 & 1.00 & 0.36 & 0.89 & 1.00 & 1.00 & 0.90 & 1.00 & 1.00 & 1.00 \\
\bottomrule
\end{tabular}
}
\caption{Kolmogorov distances for precision matrix in the low-dimensional regime under the banded setup. The GA section compares the precision error $n^{1/2}|\hat\Omega_n - \Omega|_\infty$ against its (intermediate) Gaussian approximation, while the Bootstrap section compares it against its block bootstrap approximation.}
\label{tab:E2B.KS.prec}

\vspace{0.5cm}

\resizebox{\textwidth}{!}{
\begin{tabular}{llcccccccccccc}
\toprule
 & & \multicolumn{4}{c}{$\beta=2$} & \multicolumn{4}{c}{$\beta=0.9$} & \multicolumn{4}{c}{$\beta=0.55$} \\
 \cmidrule(lr){3-6} \cmidrule(lr){7-10} \cmidrule(lr){11-14}
& $n$ & $p=2$ & $p=10$ & $p=30$ & $p=100$ & $p=2$ & $p=10$ & $p=30$ & $p=100$ & $p=2$ & $p=10$ & $p=30$ & $p=100$ \\
\midrule
\multirow{6}{*}{GA} & 10000 & 0.12 & 0.10 & 0.15 & 1.08 & 0.21 & 0.11 & 0.45 & 3.62 & 3.69 & 2.23 & 2.19 & 6.87 \\
 & 4000 & 0.15 & 0.11 & 0.26 & 2.21 & 0.20 & 0.21 & 1.14 & 5.67 & 2.89 & 2.90 & 4.78 & 11.57 \\
 & 2000 & 0.09 & 0.18 & 0.77 & 3.55 & 0.18 & 0.31 & 2.13 & 8.02 & 3.37 & 3.84 & 6.64 & 15.11 \\
 & 1000 & 0.11 & 0.41 & 1.34 & 6.17 & 0.25 & 1.07 & 3.12 & 11.23 & 3.11 & 5.16 & 8.41 & 19.50 \\
 & 500 & 0.10 & 0.53 & 2.06 & 10.72 & 0.20 & 1.50 & 4.38 & 17.06 & 2.87 & 5.92 & 10.37 & 26.57 \\
 & 200 & 0.10 & 1.18 & 5.25 & 34.69 & 0.35 & 2.75 & 8.87 & 46.29 & 2.98 & 7.49 & 15.68 & 59.98 \\
\midrule
\multirow{6}{*}{Bootstrap} & 10000 & 0.11 & 0.13 & 0.21 & 1.02 & 0.69 & 0.70 & 1.03 & 4.68 & 10.00 & 14.68 & 18.90 & 28.46 \\
 & 4000 & 0.33 & 0.24 & 0.39 & 2.16 & 0.49 & 1.17 & 2.11 & 7.55 & 7.70 & 11.86 & 16.28 & 26.08 \\
 & 2000 & 0.11 & 0.33 & 0.85 & 3.51 & 0.69 & 1.51 & 3.57 & 10.35 & 6.53 & 10.45 & 14.93 & 25.59 \\
 & 1000 & 0.23 & 0.39 & 1.35 & 5.67 & 1.06 & 2.24 & 5.06 & 13.56 & 5.49 & 9.53 & 14.33 & 26.11 \\
 & 500 & 0.23 & 0.72 & 2.13 & 9.54 & 1.02 & 3.15 & 6.51 & 18.42 & 4.73 & 9.17 & 14.39 & 29.56 \\
 & 200 & 0.24 & 1.44 & 4.65 & 28.60 & 1.20 & 4.41 & 10.19 & 40.75 & 4.00 & 9.15 & 16.82 & 52.66 \\
\bottomrule
\end{tabular}
}
\caption{Wasserstein-1 distances for precision matrix in the low-dimensional regime under the banded setup. The GA section compares the precision error $n^{1/2}|\hat\Omega_n - \Omega|_\infty$ against its (intermediate) Gaussian approximation, while the Bootstrap section compares it against its block bootstrap approximation.}
\label{tab:E2B.W1.prec}
\end{table}

\begin{figure}[p]
    \centering
    
    \includegraphics[width=0.8\linewidth]{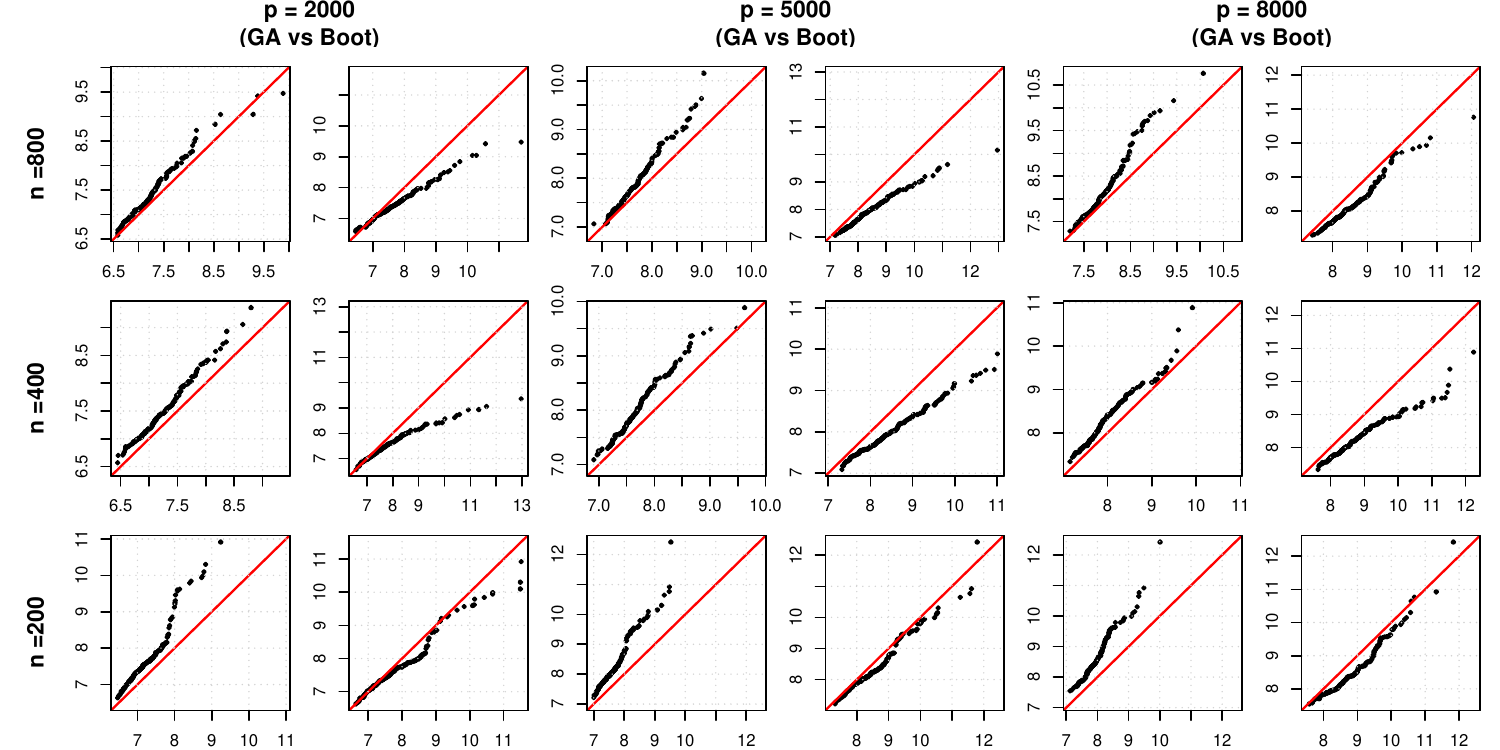}
    \caption{QQ-plots for covariance matrix in the high-dimensional regime, $\beta=2$ (short memory), under the banded setup. Red line represents $y=x$. For each pair of $p$ and $n$ values, the plot on the left compares $n^{1/2}|\hat\Sigma_n - \Sigma|_\infty$ on $y$-axis and $|Z|_\infty$ on $x$-axis.
    Being closer to the $y=x$ line represents a smaller Kolmogorov distance $\rho(n)$ for Gaussian approximation.
    The plot on the right compares $n^{1/2}|\hat\Sigma_n - \Sigma|_\infty$ on $y$-axis and $l^{-1/2}|\check{B}_{i,l}|_\infty$ on $x$-axis.
    Being closer to the $y=x$ line represents a smaller Kolmogorov distance $\rho_B(n,l)$ for block bootstrap.}
    \label{fig:E3B.QQ.cov.2}
    
    \vspace{1cm} 
    
    \includegraphics[width=0.8\linewidth]{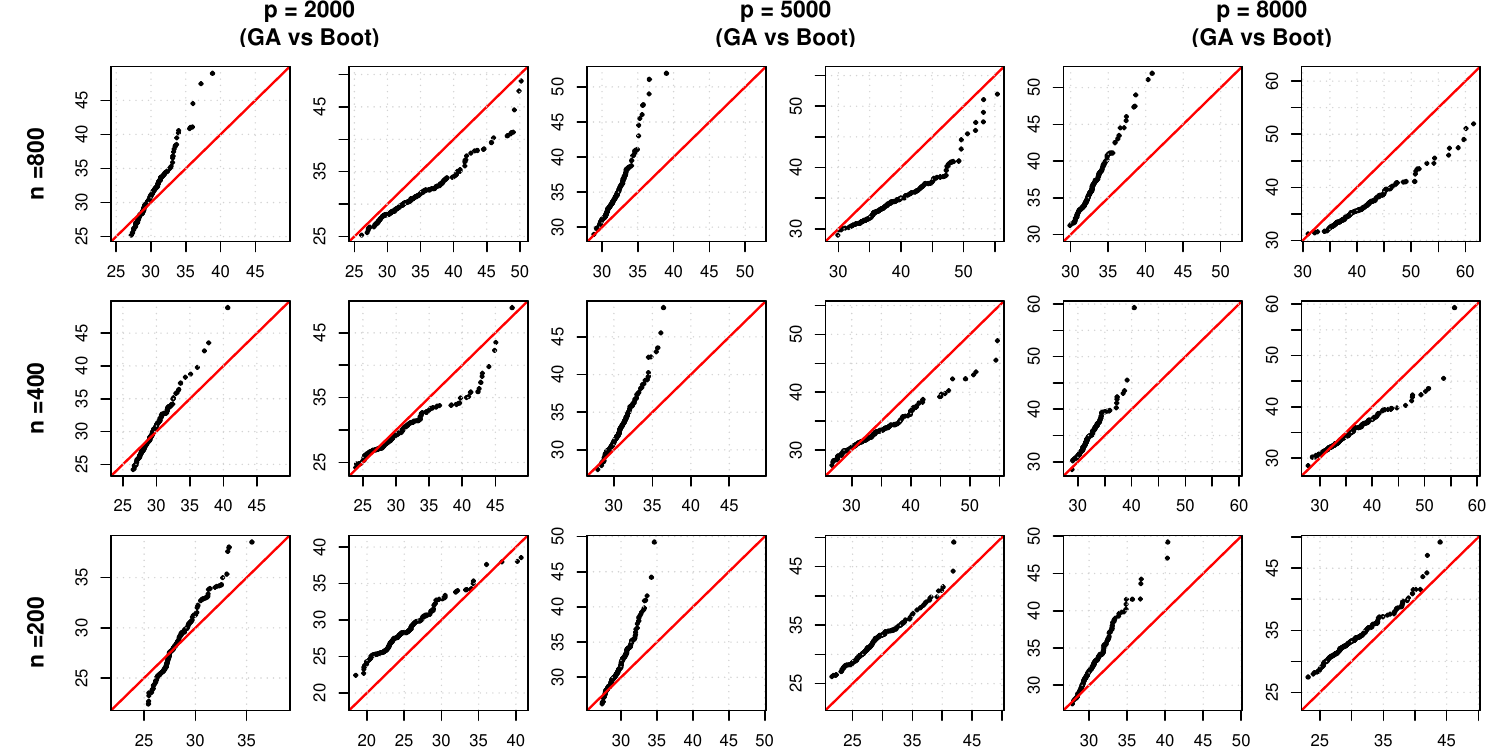}
    \caption{QQ-plots for covariance matrix in the high-dimensional regime, $\beta=0.9$ (long-memory), under the banded setup. Red line represents $y=x$. For each pair of $p$ and $n$ values, the plot on the left compares $n^{1/2}|\hat\Sigma_n - \Sigma|_\infty$ on $y$-axis and $|Z|_\infty$ on $x$-axis.
    Being closer to the $y=x$ line represents a smaller Kolmogorov distance $\rho(n)$ for Gaussian approximation.
    The plot on the right compares $n^{1/2}|\hat\Sigma_n - \Sigma|_\infty$ on $y$-axis and $l^{-1/2}|\check{B}_{i,l}|_\infty$ on $x$-axis.
    Being closer to the $y=x$ line represents a smaller Kolmogorov distance $\rho_B(n,l)$ for block bootstrap.}
    \label{fig:E3B.QQ.cov.0.9}
    
\end{figure}

\begin{figure}[p]
    \centering
    
    \includegraphics[width=0.8\linewidth]{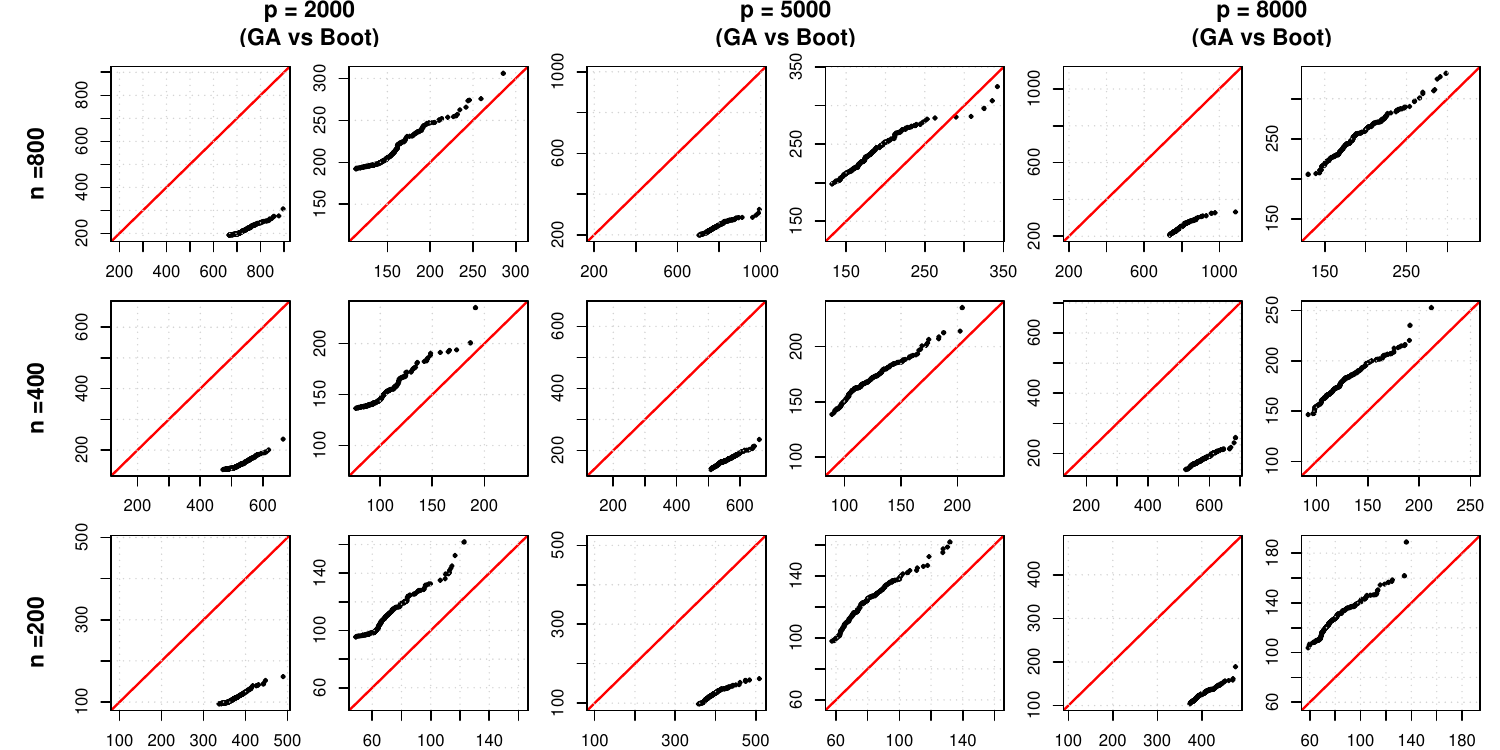}
    \caption{QQ-plots for covariance matrix in the high-dimensional regime, $\beta=0.55$ (ultra-long-memory), under the banded setup. Red line represents $y=x$. For each pair of $p$ and $n$ values, the plot on the left compares $n^{1/2}|\hat\Sigma_n - \Sigma|_\infty$ on $y$-axis and $|Z|_\infty$ on $x$-axis.
    Being closer to the $y=x$ line represents a smaller Kolmogorov distance $\rho(n)$ for Gaussian approximation.
    The plot on the right compares $n^{1/2}|\hat\Sigma_n - \Sigma|_\infty$ on $y$-axis and $l^{-1/2}|\check{B}_{i,l}|_\infty$ on $x$-axis.
    Being closer to the $y=x$ line represents a smaller Kolmogorov distance $\rho_B(n,l)$ for block bootstrap.}
    \label{fig:E3B.QQ.cov.0.55}
    
    \vspace{1cm} 
    
    \includegraphics[width=0.8\linewidth]{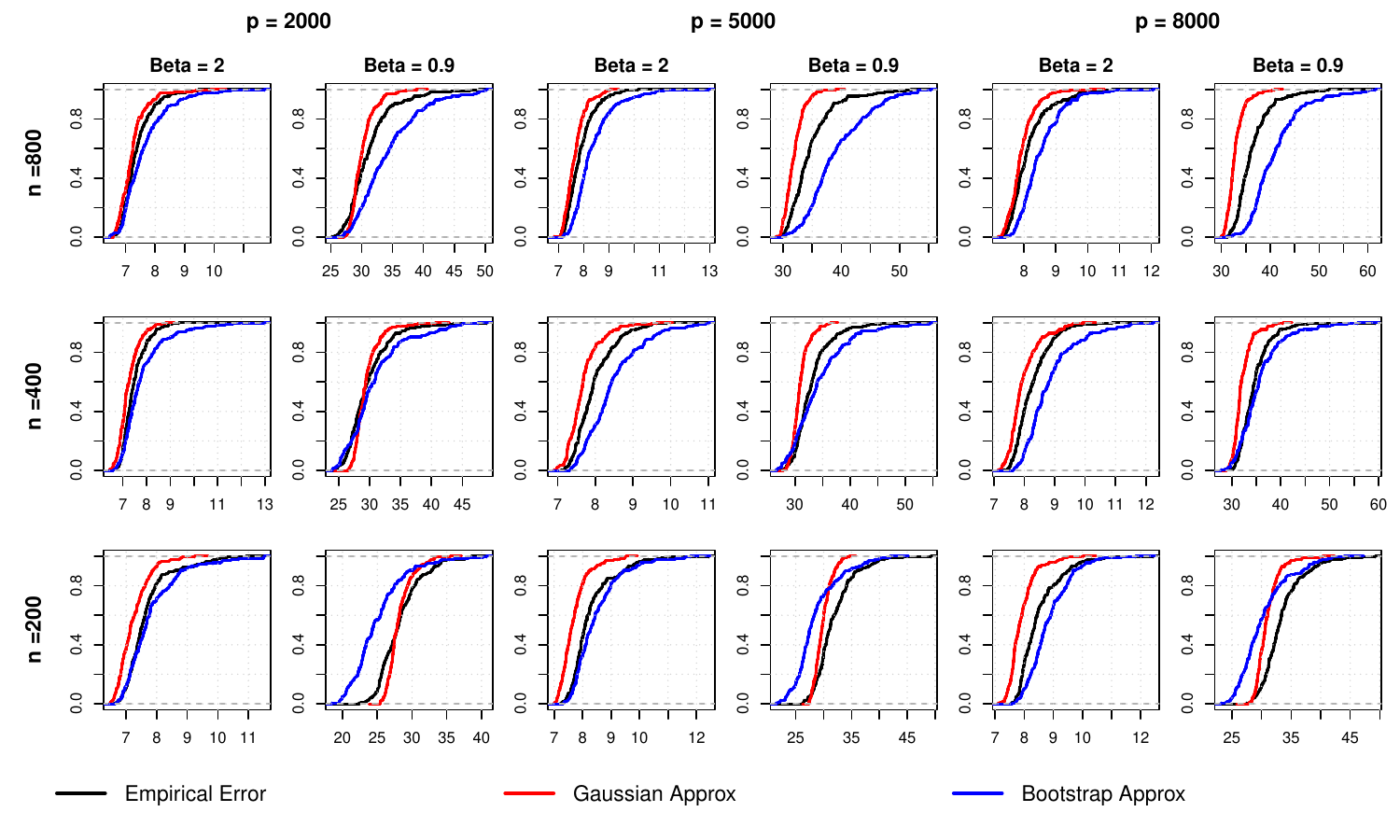}
    \caption{Empirical CDF in the high-dimensional regime under the banded setup, which contains the covariance error $n^{1/2}|\hat\Sigma_n - \Sigma|_\infty$, its Gaussian approximation $|Z|_\infty$, and its block bootstrap approximation $l^{-1/2}|\check{B}_{i,l}|_\infty$, for the short-memory case $\beta=2$ and the long-memory case $\beta=0.9$.}
    \label{fig:E3B.ECDF.cov}
    
\end{figure}

\begin{table}[p]
\centering

\resizebox{\textwidth}{!}{
\begin{tabular}{llccccccccc}
\toprule
 & & \multicolumn{3}{c}{$\beta=2$} & \multicolumn{3}{c}{$\beta=0.9$} & \multicolumn{3}{c}{$\beta=0.55$} \\
 \cmidrule(lr){3-5} \cmidrule(lr){6-8} \cmidrule(lr){9-11}
& $n$ & $p=2000$ & $p=5000$ & $p=8000$ & $p=2000$ & $p=5000$ & $p=8000$ & $p=2000$ & $p=5000$ & $p=8000$ \\
\midrule
\multirow{3}{*}{GA} & 800 & 0.17 & 0.20 & 0.18 & 0.23 & 0.44 & 0.57 & 1.00 & 1.00 & 1.00 \\
 & 400 & 0.24 & 0.32 & 0.29 & 0.20 & 0.41 & 0.44 & 1.00 & 1.00 & 1.00 \\
 & 200 & 0.32 & 0.44 & 0.46 & 0.24 & 0.32 & 0.41 & 1.00 & 1.00 & 1.00 \\
\midrule
\multirow{3}{*}{Bootstrap} & 800 & 0.17 & 0.31 & 0.33 & 0.31 & 0.46 & 0.49 & 0.89 & 0.77 & 0.75 \\
 & 400 & 0.17 & 0.36 & 0.34 & 0.12 & 0.17 & 0.15 & 0.94 & 0.82 & 0.84 \\
 & 200 & 0.12 & 0.17 & 0.30 & 0.50 & 0.50 & 0.48 & 0.92 & 0.91 & 0.92 \\
\bottomrule
\end{tabular}
}
\caption{Kolmogorov distances for covariance matrix in the high-dimensional regime under the banded setup. The GA section compares the covariance error $n^{1/2}|\hat\Sigma_n - \Sigma|_\infty$ against its Gaussian approximation $|Z|_\infty$, while the Bootstrap section compares the covariance error against its block bootstrap approximation $l^{-1/2}|\check{B}_{i,l}|_\infty$.}
\label{tab:E3B.KS.cov}

\vspace{1cm}

\resizebox{\textwidth}{!}{
\begin{tabular}{llccccccccc}
\toprule
 & & \multicolumn{3}{c}{$\beta=2$} & \multicolumn{3}{c}{$\beta=0.9$} & \multicolumn{3}{c}{$\beta=0.55$} \\
 \cmidrule(lr){3-5} \cmidrule(lr){6-8} \cmidrule(lr){9-11}
& $n$ & $p=2000$ & $p=5000$ & $p=8000$ & $p=2000$ & $p=5000$ & $p=8000$ & $p=2000$ & $p=5000$ & $p=8000$ \\
\midrule
\multirow{3}{*}{GA} & 800 & 0.16 & 0.22 & 0.22 & 1.39 & 2.85 & 3.47 & 517.60 & 546.93 & 549.61 \\
 & 400 & 0.24 & 0.30 & 0.30 & 1.07 & 2.10 & 2.51 & 371.12 & 387.89 & 397.64 \\
 & 200 & 0.43 & 0.59 & 0.60 & 1.00 & 2.04 & 2.45 & 264.31 & 274.42 & 276.92 \\
\midrule
\multirow{3}{*}{Bootstrap} & 800 & 0.25 & 0.47 & 0.37 & 3.07 & 4.37 & 4.69 & 56.19 & 54.25 & 61.17 \\
 & 400 & 0.33 & 0.47 & 0.57 & 1.20 & 1.60 & 1.20 & 46.90 & 46.34 & 50.86 \\
 & 200 & 0.16 & 0.19 & 0.34 & 3.26 & 3.27 & 3.17 & 38.67 & 42.40 & 45.52 \\
\bottomrule
\end{tabular}
}
\caption{Wasserstein-1 distances for covariance matrix in the high-dimensional regime under the banded setup. The GA section compares the covariance error $n^{1/2}|\hat\Sigma_n - \Sigma|_\infty$ against its Gaussian approximation $|Z|_\infty$, while the Bootstrap section compares the covariance error against its block bootstrap approximation $l^{-1/2}|\check{B}_{i,l}|_\infty$.}
\label{tab:E3B.W1.cov}

\end{table}

}
\end{document}